\documentclass{amsart} 

\usepackage[centertags]{amsmath} %
\usepackage{amsfonts,amssymb,amsthm,amscd,mathabx,latexsym,cmll,eufrak}
\usepackage[mathscr]{euscript}
\usepackage[all]{xy}
\usepackage{tikz}
\usetikzlibrary{arrows}

\usepackage{hyperref} 

\numberwithin{equation}{subsection}%
\newtheorem{theorem}[equation]{Theorem}
\newtheorem{corollary}[equation]{Corollary}
\newtheorem{proposition}[equation]{Proposition}
\newtheorem{lemma}[equation]{Lemma}
\theoremstyle{definition}
\newtheorem{definition}[equation]{Definition}
\newtheorem{example}[equation]{Example}
\newtheorem{example-special}[equation]{Definition/Example}
\newtheorem{remark}[equation]{Remark}

\newtheorem{assumption}[equation]{Assumption}
\newtheorem{construction}[equation]{Construction}
\newtheorem{notation}[equation]{Notation}

\newcommand{\co}{\colon\thinspace}%
\newcommand{\dual}{\ensuremath{{\sf v}}}%
\newcommand{\mc}[1]{\ensuremath{\mathcal{#1}}}%
\newcommand{\toh}[1]{\ensuremath{\stackrel{#1}{\rightarrow}}}
\newcommand{\colim}{\operatorname*{colim}}%
\newcommand{\coker}{\operatorname{coker\,}}%
\renewcommand{\ker}{\operatorname{ker\,}}
\newcommand{\sk}{\operatorname{sk}}%
\newcommand{\cosk}{\operatorname{cosk}}%
\newcommand{\im}{\operatorname{im}}%
\renewcommand{\Pr}{\mathrm{Prim}}
\newcommand{\holim}{\operatorname*{holim\,}}%
\newcommand{\hocolim}{\operatorname*{hocolim\,}}%
\newcommand{\id}{\operatorname{id}}%
\newcommand{\Tot}{\operatorname{Tot}}%
\newcommand{\Cotor}{\operatorname{Cotor}}%
\newcommand{\surj}{\ensuremath{\twoheadrightarrow}}%
\newcommand{\ul}[1]{\underline{#1}}%
\newcommand{\ol}[1]{\overline{#1}}%
\newcommand{\wh}[1]{\widehat{#1}}%
\newcommand{\oltimes}{\operatorname*{\overline{\otimes}}}%
\newcommand{\ho}[1]{\ensuremath{{\rm Ho}({#1})}}%
\newcommand{\CA}{\ensuremath{{\mathcal{CA}}}}%
\newcommand{\UA}{\ensuremath{{\mathcal{UA}}}}%
\newcommand{\V}{\ensuremath{{\mathcal{V}}}}%
\newcommand{\F}{\ensuremath{\mathbb{F}}}%
\newcommand{\Comod}{{\rm Comod}}
\renewcommand{\Vec}{\ensuremath{{\rm Vec}}}%
\newcommand{\AQ}{\mathrm{AQ}}%
\newcommand{\TAQ}{\mathscr{AQ}}
\newcommand{\Coalg}{{\rm Coalg_{\F}}}%
\newcommand{\grCoalg}{{\rm Coalg_{Vec}}}%
\newcommand{\Fp}{\ensuremath{\mathbb{F}_p}}%
\newcommand{\kfp}[1]{\ensuremath{K(\mathbb{F},{#1})}}%
\newcommand{\bu}{\bullet}%
\newcommand{\naturalpi}[3]{\ensuremath{\pi_{#1}^{\natural}(#2,#3)}}%
\newcommand{\op}{\ensuremath{\mathrm{op}}}%
\newcommand{\Hom}{\ensuremath{{\rm Hom}}}%
\newcommand{\map}{\ensuremath{{\rm map}}}
\newcommand{\olmap}{\ensuremath{\overline{\map}}}%
\newcommand{\Hun}{\ensuremath{\mathcal{H}_{\rm un}}}
\renewcommand{\hom}{\ensuremath{{\rm hom}}}
\newcommand{\olhom}{\operatorname*{\overline{\hom}}}%
\newcommand{\Aut}{\ensuremath{{\rm Aut}}}
\newcommand{\Path}{{\rm Path}}%
\newcommand{\ext}{{\rm ext}}%
\newcommand{\inter}{{\rm int}}%
\newcommand{\cone}{{\rm Cone}}%
\newcommand{\eps}{{\epsilon}}%
\newcommand{\Alg}{\ensuremath{{\rm Alg}}}%
\newcommand{\fib}{\ensuremath{{\rm fib}}}%
\newcommand{\terminal}{\ensuremath{\underline{\mathbb{F}}}}
\newcommand{\set}{\mathrm{Set}}%
\newcommand{\ucat}{\ensuremath{\,\shpos\,}}
\newcommand{\cotower}{\ensuremath{\mathbf{N}}}
\newcommand{\diagram}[2]{ \begin{align} \begin{split} \xymatrix{#1} \end{split} \label{#2} \end{align}}%
\newcommand{\diagr}[1]{ \begin{equation*} \xymatrix{#1} \end{equation*}}%

\begin{document}
\title{The realization space of an unstable coalgebra}

\author{Georg Biedermann}
\address{LAGA, Institut Galil\'{e}e, Universit\'{e} Paris 13, 99 Avenue JB Cl\'{e}ment, 93430 Villetaneuse, France}
\email{biedermann@math.univ-paris13.fr}

\author{Georgios Raptis}
\address{Fakult\"{a}t f\"{u}r Mathematik, Universit\"{a}t Regensburg, 93040 Regensburg, Germany}
\email{georgios.raptis@mathematik.uni-regensburg.de}

\author{Manfred Stelzer}
\address{Institut f\"ur Mathematik, Universit\"at Osnabr\"uck, Albrechtstrasse 28a, D-49076 Osnabr\"uck, Germany}
\email{mstelzer@uni-osnabrueck.de}

\date{\today}

\begin{abstract}
Unstable coalgebras over the Steenrod algebra form a natural target category for singular homology with prime field coefficients. The realization problem asks whether an 
unstable coalgebra is isomorphic to the homology of a topological space. We study the moduli space of such realizations and give a description of this in terms of cohomological 
invariants of the unstable coalgebra. This is accomplished by a thorough comparative study of the homotopy theories of cosimplicial unstable coalgebras and of cosimplicial 
spaces. 
\end{abstract}


\maketitle
\setcounter{tocdepth}{1}
\tableofcontents

\section{Introduction}
Let $p$ be a prime and $C$  an unstable (co)algebra over the Steenrod algebra $\mathcal{A}_p$. The realization problem asks whether $C$ is isomorphic to the singular 
$\mathbb{F}_p$-(co)homology of a topological space $X$. In case the answer turns out to be affirmative, one may further ask how many such spaces $X$ there are up to $\mathbb{F}_p$-homology equivalence. The purpose of this work is to study a moduli space of topological realizations 
associated with a given unstable coalgebra $C$. Our results give a description of this moduli space in terms of a tower of spaces which are determined by cohomological invariants of $C$. As 
a consequence, we obtain obstruction theories for the existence and uniqueness of topological realizations, where the obstructions are defined by Andr\'{e}-Quillen cohomology classes. 
These obstruction theories recover and sharpen results of Blanc~\cite{Blanc:coalg}. Moreover, our results apply also to the case of rational coefficients. We thus provide a unified 
picture of the theory in positive and zero characteristics. 
\vspace*{3mm}

Let us start by briefly putting the problem into historical perspective. 
The realization problem was explicitly posed by Steenrod in~\cite{Steenrod:cohomology-algebra}.
In the rational context, at least if one restricts to simply-connected objects, such realizations always exist by celebrated theorems 
of Quillen~\cite{Quillen:rational} and Sullivan~\cite{sullivan:infinitesimal}. In contrast, there are many deep non-realization theorems known in positive characteristic: 
some of the most notable ones are by Adams~\cite{Adams:Hopf-invariant}, Liulevicius~\cite{Liulevicius:cyclic-powers}, Ravenel~\cite{Ravenel:Arf}, 
Hill-Hopkins-Ravenel~\cite{HHR:Kervaire}, Schwartz~\cite{Schwartz:apropos}, and Gaudens-Schwartz~\cite{Gaudens-Schwartz}.

Moduli spaces parametrizing homotopy types with a given cohomology algebra or homotopy Lie algebra were first constructed in 
rational homotopy theory. The case of cohomology was treated by F\'{e}lix~\cite{Felix:deform}, Lemaire-Sigrist~\cite{Lemaire-Sigrist}, and Schlessinger-Stasheff~\cite{SchleSta}. 
All these works relied on the obstruction theory developed by Halperin-Stasheff~\cite{Halperin-Stasheff:obstructions}. The (moduli) set of equivalence classes of realizations 
was identified with the quotient of a rational variety by the action of a unipotent algebraic group. 
Moreover, Schlessinger and Stasheff \cite{SchleSta} associated to a graded algebra $A$  a differential graded coalgebra  $C_A$, whose set of components, defined in terms of an algebraic notion of homotopy,
parametrizes the different realizations of $A$. This coalgebra represents a moduli space for $A$ which encodes higher order 
 information.
 
An obstruction theory for unstable coalgebras was developed by Blanc in \cite{Blanc:coalg}. He defined obstruction classes for realizing an unstable coalgebra $C$ and proved that the 
vanishing of these classes is necessary and sufficient for the existence of a realization. Furthermore, he defined difference classes which distinguish two given realizations. Both 
kinds of classes  live in certain Andr\'{e}-Quillen cohomology groups associated to $C$. Earlier, Bousfield \cite{Bou:obstructions} developed an obstruction theory for realizing maps 
with obstruction classes in the unstable Adams spectral sequence. For very nice unstable algebras, obstruction theories using the Massey-Peterson machinery \cite{Massey-Peterson:mod2} 
were also introduced by Harper~\cite{Harper:torsion} and McCleary~\cite{MCl:H-spaces} already in the late seventies. 

In a landmark paper, Blanc, Dwyer, and Goerss \cite{BlDG:pi-algebra} studied the moduli space of topological realizations of a given $\Pi$-algebra over the integers.
Again, the components of this moduli space correspond to  different realizations. By means of simplicial resolutions, the authors showed a 
decomposition of this moduli space into a tower of fibrations, thus obtaining obstruction theories for the realization and uniqueness problems for $\Pi$-algebras. 
Their work relied on earlier work of Dwyer, Kan and Stover \cite{DKSt:E2} on resolution (or $E^2$) model categories which was later generalized by Bousfield \cite{Bou:cos}. 
Using analogous methods, Goerss and Hopkins in~\cite{GoHop:moduli}, and \cite{GoHop:moduli2}, studied the moduli space of $E_{\infty}$-algebras in spectra 
which have a prescribed homology with respect to a homology theory. Their results gave rise to some profound applications in stable homotopy theory \cite{Rezk:Hopkins-Miller}. 
\vspace*{3mm}

In this paper, we consider an unstable coalgebra $C$ over the Steenrod algebra and we study the (possibly empty) moduli space of realizations $\mc{M}_{\rm Top}(C)$. We work with unstable 
coalgebras and homology, instead of unstable algebras and cohomology, because one avoids in this way all kinds of issues that are related to (non-)finiteness. In the presence of suitable finiteness assumptions, the 
two viewpoints can be translated into each other. We exhibit a decomposition of $\mc{M}_{\rm Top}(C)$ into a tower of fibrations in which the layers have homotopy groups related to the
Andr\'{e}-Quillen cohomology groups of $C$. Our results apply also to the rational case. 
The decomposition of $\mc{M}_{\rm Top}(C)$ is achieved by means of cosimplicial resolutions and their Postnikov 
decompositions in the cosimplicial direction, following and dualizing the approach of~\cite{BlDG:pi-algebra}. At several steps, the application of this general approach to moduli space 
problems from \cite{BlDG:pi-algebra} requires non-trivial input and this accounts in part for the length of the article. Another reason is that the article is essentially self-contained. 
\vspace*{3mm}

We will now give the precise statements of our main results. Following the work of Dwyer and Kan on moduli spaces in homotopy theory, a moduli space of objects in a certain homotopy theory is defined as the classifying space 
of a category of weak equivalences. Under favorable homotopical assumptions, it turns out that this space encodes the homotopical information of the spaces of homotopy automorphisms of these objects. We refer 
to Appendix \ref{DK-theory} for the necessary background and a detailed list of references. 

Fix a prime field $\mathbb{F}$ and an unstable coalgebra $C$ over the corresponding Steenrod algebra. We consider the category $\mc{W}_{\rm Top}(C)$ whose objects are all spaces with $\mathbb{F}$-homology isomorphic to $C$, as unstable coalgebras, and whose morphisms are the $\mathbb{F}$-homology equivalences. 
The realization space $\mc{M}_{\rm Top}(C)$ is the classifying space of the category $\mc{W}_{\rm Top}(C)$.
For the decomposition of this space into a tower of more accessible spaces, we replace actual topological realizations with their cosimplicial resolutions regarded as objects in an 
appropriate model category of cosimplicial spaces. These resolutions are 
given by products of Eilenberg-MacLane spaces of type $\mathbb{F}$ in each cosimplicial degree. More precisely, they are fibrant replacements in the resolution model 
structure $c \mc{S}^{\mc{G}}$ on the category of cosimplicial spaces $c\mc{S}$ with respect to the class \mc{G} of Eilenberg-MacLane spaces. 

There is a notion of a Postnikov decomposition for such resolutions defined in the cosimplicial direction. It is given by the skeletal filtration. Its terms are characterized by 
connectivity and vanishing conditions similar to the classical case of spaces. A potential $n$-stage for $C$ is then defined to be a cosimplicial space which satisfies such conditions 
so that it may potentially be the $(n+1)$-skeleton of an actual realization. The definition makes sense also for $n=\infty$, giving rise to the notion of an $\infty$-stage. 
The category whose objects are potential $n$-stages for $C$ defines a moduli space $\mc{M}_{n}(C)$ in $c \mc{S}^{\mc{G}}$ and the derived skeleton functor induces a map 
\[\sk_{n}^c:\mc{M}_{n}(C)\to \mc{M}_{n-1}(C).\]
The following theorem addresses the relation between the moduli spaces of $\infty$-stages and genuine realizations of $C$.
\vspace{2mm}

\noindent
{\bf Theorem~\ref{infty-stages-are-reals}.} 
{\it Let $C$ be an unstable coalgebra. Suppose that the homology spectral sequence of each $\infty$-stage $ X^\bu$ for $C$ converges strongly to $H_*(\Tot X^\bu)$.
Then the totalization functor induces a weak 
equivalence $$\mc{M}_{\infty}(C) \simeq \mc{M}_{\rm Top}(C).$$ 
This is the case if the unstable coalgebra $C$ is simply-connected, i.e.\! $C_1 = 0$ and $C_0 = \mathbb{F}$. }
\vspace{2mm}

As a further step towards a decomposition of $\mc{M}_{\infty}(C)$, we show:
\vspace{2mm}

\noindent      
{\bf Theorem~\ref{infty-stages-holim}.}
{\it Let $C$ be an unstable coalgebra. There is a weak equivalence 
  $$\mc{M}_{\infty}(C) \simeq \holim_n \mc{M}_n(C).$$ }

The difference between $\mc{M}_n(C) $ and $\mc{M}_{n-1}(C)$ is explained in the next theorem but we first need to introduce some notation. Given a coabelian unstable coalgebra 
$M \oplus \mathbb{F}$ where $M$ is also a $C$-comodule, there is a cosimplicial unstable coalgebra $K_C(M,n)$ of ``Eilenberg-MacLane type", as explained in Section~\ref{sec:cosimplicial-unst-coalg}. This object co-represents Andr\'{e}-Quillen cohomology. Then we define Andr\'{e}-Quillen spaces 
to be the (derived) mapping spaces in the model category of cosimplicial unstable coalgebras under $C$:
  $$\TAQ^n_C(C;M) : = \map^{\rm der}_{c(C/ \CA)}(K_C(M,n), cC),$$
where  $cC$ denotes the constant cosimplicial object defined by $C$.
The homotopy groups of this mapping space yield the Andr\'{e}-Quillen cohomology of $C$ with coefficients in $M$.
Let us write $C[n]$ for the $C$-comodule and unstable module obtained from $C$ by shifting $n$ degrees up with respect to its internal grading. The automorphism group 
$\Aut_C(C[n])$ acts on $\TAQ^k_C(C; C[n])$, for any $k$, and has a homotopy fixed point at the basepoint which represents the zero Andr\'e-Quillen cohomology 
class. We denote the homotopy quotient by $\widetilde{\TAQ}^{k}_C(C;C[n])$. 
\vspace{2mm}

\noindent      
{\bf Theorem~\ref{main-pullback-square-alternative}.} 
{\it Let $C$ be an unstable coalgebra. For every $n \geq 1$, there is a homotopy pullback square}
\begin{displaymath}
\xymatrix{
\mc{M}_n(C) \ar[rr] \ar[d]^{\sk^c_n} && B\Aut_C(C[n]) \ar[d] \\
\mc{M}_{n-1}(C) \ar[rr] && \widetilde{\TAQ}^{n+2}_C(C;C[n]).
} 
\end{displaymath}
\vspace{2mm}

As a consequence, we obtain:  
\vspace{2mm}

\noindent
{\bf Corollary~\ref{main-pullback-square2}.} 
{\it Let $X^\bu$ be a potential $n$-stage for an unstable coalgebra $C$. Then there is a homotopy pullback square }
\begin{displaymath}
\xymatrix{
\TAQ^{n+1}_C(C; C[n]) \ar[rr] \ar[d] && \mc{M}_n(C) \ar[d]^{\sk^c_n} \\
\ast \ar[rr]^{\sk^c_n X^\bu} && \mc{M}_{n-1}(C). 
}
\end{displaymath}

Finally, the bottom of the tower of moduli spaces can be identified as follows.
\vspace{2mm}

\noindent
{\bf Theorem~\ref{moduli-0-stages}.}
{\it Let $C$ be an unstable coalgebra. There is a weak equivalence 
\begin{center}
$\mc{M}_0(C) \simeq B \Aut(C).$ 
\end{center}
  }
\vspace{2mm}

Before proving these results, we need to develop two main technical tools. The first tool is a cosimplicial version of the spiral exact sequence,
which first appeared in the study of the $E^2$-model structure on pointed simplicial spaces by Dwyer-Kan-Stover~\cite{DKSt:bigraded}. In particular, its structure as a sequence of modules over its zeroth term will be important. The spiral exact sequence establishes a fundamental link between the homotopy theories of cosimplicial spaces and cosimplicial unstable coalgebras.

We also give a detailed account of the homotopy theory of cosimplicial unstable coalgebras. This culminates in a homotopy excision theorem for cosimplicial unstable coalgebras. It leads to the following homotopy excision theorem in $c \mc{S}^{\mc{G}}$, the resolution model category of cosimplicial spaces with respect to Eilenberg-MacLane spaces. 
This is our second main tool.
\vspace*{2mm}

\noindent
{\bf Theorem~\ref{homotopy-excision 2}} (Homotopy excision for cosimplicial spaces).
{\it Let 
\[
 \xymatrix{
 E^\bu \ar[d] \ar[r] & X^\bu \ar[d]^f \\
 Y^\bu \ar[r]^g & Z^\bu
 }
\]
be a homotopy pullback square in $c \mc{S}^{\mathcal{G}}$ where $f$ is $m$-connected and $g$ is $n$-connected. Then the square is homotopy $(m+n)$-cocartesian. }
\vspace*{2mm}

A detailed account of the resolution model category $c \mc{S}^{\mc{G}}$ will be given in Section \ref{sec:cosimplicial-spaces}. Let us simply mention here that a map of cosimplicial spaces is a weak equivalence in 
$c \mc{S}^{\mc{G}}$ if the induced map on homology is a weak equivalence of cosimplicial unstable coalgebras. The relevant connectivity notion in the theorem can be defined in terms of the cohomotopy groups of cosimplicial 
unstable coalgebras. 
\vspace*{3mm}

The paper is structured as follows.

Resolution model categories of cosimplicial objects with respect to a class of injective models \mc{G} are studied in Sections~\ref{resolution-model-cat} and \ref{sec:natural-homotopy-groups}. In Section \ref{resolution-model-cat},
we give a brief review of Bousfield's work \cite{Bou:cos} and then we define and study some useful properties of the class of (quasi-)\mc{G}-cofree maps. In Section \ref{sec:natural-homotopy-groups}, we introduce the natural homotopy 
groups with respect to \mc{G} and compare them with the $E_2$-homotopy groups of a cosimplicial object. The main result is the all-important spiral exact sequence (Theorem~\ref{ses-statement}). 
Finally, we discuss the notion of cosimplicial connectivity with respect to \mc{G}.

Section~\ref{sec:cosimplicial-unst-coalg} is concerned with the homotopy theory of cosimplicial unstable coalgebras. First we review some basic facts about the categories of unstable 
right modules and unstable coalgebras. The model structure on cosimplicial unstable coalgebras is 
an example of a resolution model category. We study the homotopy theory of cosimplicial comodules over a cosimplicial coalgebra and discuss the K\"unneth spectral sequences in this 
setting (Theorem \ref{Kunneth spectral sequence 2}) - which are dual
to the ones proved in \cite{Quillen:HA}. Using these spectral sequences, we obtain a homotopy excision theorem for cosimplicial unstable coalgebras (Theorem \ref{homotopy-excision}). 

In Section \ref{sec:AQ}, we recall the definition of Andr\'{e}-Quillen cohomology of unstable coalgebras following \cite{Quillen:HA}. We construct certain twisted ``Eilenberg-MacLane" objects $K_C(M,n)$ and observe that they co-represent Andr\'{e}-Quillen cohomology. Furthermore, we determine the moduli spaces of these objects which will later be needed for the 
proof of Theorem \ref{main-pullback-square-alternative}. We study the skeletal filtration of a cosimplicial unstable coalgebra and prove that it defines 
a Postnikov decomposition with respect to the cohomotopy groups. Moreover, using the homotopy excision theorem, we prove that each map in this skeletal filtration is a ``principal cofibration'', i.e., it is a pushout of an attachment defined by a map from a twisted Eilenberg-MacLane object (see Proposition \ref{diff. constr. in CA}). 
This yields dual versions of $k$-invariants which enter into the obstruction theory. The application of the homotopy excision theorem is immediate in this case because the maps of the filtration are 
at least $1$-connected; a final subsection analyzes in detail some subtle points of the case of $0$-connected maps which may also be of independent interest. 

Section~\ref{sec:cosimplicial-spaces} is devoted to the resolution model category of cosimplicial spaces with respect to the class of $\F$-GEMs, the generalized Eilenberg-MacLane spaces of type $\mathbb{F}$. 
The K\"unneth theorem for singular homology with field coefficients provides an important link between the topological and the algebraic homotopy theories and its consequences 
are explained. As an example, the homotopy excision theorem for cosimplicial spaces is deduced from the one for cosimplicial unstable coalgebras (Theorem \ref{homotopy-excision 2}). We 
prove the existence of twisted ``Eilenberg-MacLane" objects $L_C(M,n)$ in cosimplicial spaces and study their relation with the objects of type $K_C(M,n)$ 
(Proposition \ref{Representability of L(C,n)}). These new objects appear as layers of the Postnikov decomposition of a cosimplicial space that is obtained 
from its skeletal filtration (see Proposition \ref{diff. constr. in S}). 
Although the homology of an $L$-object is not a $K$-object, these two types of objects are nevertheless similar with respect to their functions in each homotopy theory. Ultimately it is 
this fact that allows a reduction from spaces to unstable coalgebras. 

Section~\ref{moduli-spaces} starts with a discussion of potential $n$-stages defined as cosimplicial spaces that approximate an honest realization. A key result which 
determines when a potential $n$-stage admits an extension to a potential $(n+1)$-stage is easily deduced from results of the previous sections (Theorem \ref{obstruction-theory-step}). 
Then a synthesis in the language of moduli spaces of the results obtained so far, especially on Postnikov decompositions and the relation between $K$- and $L$-objects, 
produces our main results about the 
moduli space of realizations of an unstable coalgebra as listed above. 

The paper ends with three appendices that contain results which are used in the rest of paper but may also be of independent interest.

The spiral exact sequence is constructed in Appendix~\ref{appsec:spiral}. 
This is done in a general cosimplicial unpointed resolution model category. The spiral exact sequence relates the two types of homotopy groups which one associates to a cosimplicial object 
in a resolution model category. An additional algebraic structure on the spiral exact sequence, an action of its zeroth term, is explored in detail. Finally, the associated spiral spectral sequence 
is discussed.

In Appendix~\ref{appsec:H-alg}, we give a description of unstable algebras in terms of algebraic theories. This extends also to rational coefficients.
The result is used to translate the abstract study of the spiral exact sequence in Appendix~\ref{appsec:spiral} to our concrete setting. 
In a final subsection, we identify the spiral spectral sequence as the vector space dual of the homology spectral sequence of a cosimplicial space.

In Appendix~\ref{DK-theory}, we provide a brief survey of some necessary background material on moduli spaces in homotopy theory. The material is mainly drawn 
from several papers of Dwyer and Kan, with some slight modifications and generalizations on a few occasions in order to fit the purposes of this paper.
\vspace*{3mm}

\textbf{Acknowledgments.} This work was mostly done while the authors were at the Institut f\"ur Mathematik, Universit\"at Osnabr\"{u}ck. We would like to warmly acknowledge the support. 
The first author was partially supported by the Centre Henri Lebesgue (programme "Investissements d'avenir" -- ANR-11-LABX-0020-01). This project has also received funding from the European Union’s Horizon 2020 research and innovation programme under the Marie Sklodowska-Curie grant agreement No 661067. The second author was partially supported by 
SFB 1085 - \emph{Higher Invariants}, Universit\"at Regensburg, funded by the DFG. The third author was supported by DFG grant RO 3867/1-1. We would also like to thank Geoffrey Powell for his helpful comments. 
     
\section{Resolution model categories} \label{resolution-model-cat}

We review the theory of resolution model structures due to Dwyer-Kan-Stover~\cite{DKSt:E2} and Bousfield~\cite{Bou:cos}. We are especially interested in the cosimplicial 
and unpointed version from~\cite[Section 12]{Bou:cos}. In Subsection~\ref{cosimplicial-objects}, we recall some basic facts about the categories of cosimplicial 
objects. Subsection~\ref{subsec:res-model-cat} provides a short summary of resolution model structures. 
The last Subsection~\ref{subsec:quasi-G-cofree} is concerned with the class of quasi-\mc{G}-cofree maps. These are analogues of relative cell complexes in the setting of 
a cosimplicial resolution model category and are useful in producing explicit factorizations. 

\subsection{Cosimplicial objects}\label{cosimplicial-objects}
This section fixes some notation and recalls several basic constructions related to cosimplicial objects.

We denote by $\Delta$ the category of finite ordinals. Its objects are given by $\ul{n}=\{0<1<\hdots<n\}$ for all $n\ge 0$. The morphisms are the non-decreasing maps.
Let \mc{M} be a complete and cocomplete category. We denote by $c\mc{M}$ the category of {\it cosimplicial objects} in \mc{M}, i.e. the category of functors from $\Delta$ to \mc{M}. 
For an object $X$ in \mc{M}, let $c(X)$ or $cX$ denote the {\it constant cosimplicial object} at $X$. 

\begin{definition}
For a cosimplicial object $X^\bullet$ in $\mc{M}$ we define its {\it $n$-th matching object} by
  $$ M^nX^\bullet:=\lim_{[n]\twoheadrightarrow [k]\atop k< n}X^k $$
and its {\it $n$-th latching object} by
  $$ L^nX^\bullet:=\colim_{[k]\hookrightarrow [n] \atop k< n} X^k .$$
We recall the definition of the {\it Reedy model structure} on the category $c\mc{M}$ (see, e.g., \cite{GoJar:simp2}). A map $X^\bu\to Y^\bu$ is
\begin{enumerate}
   \item
a {\it Reedy equivalence} if for all $n\ge 0$, the map $X^n\to Y^n$ is a weak equivalence in \mc{M}.
   \item
a {\it Reedy cofibration} if for all $n\ge 0$, the induced map 
  $$X^n\cup_{L^nX^\bu}L^nY^\bu\to Y^n$$ 
is a cofibration in \mc{M}.
   \item
a {\it Reedy fibration} if for all $n\ge 0$, the induced map
  $$X^n\to Y^n\times_{M^nY^\bu}M^nX^\bu $$ 
is a fibration in \mc{M}.
\end{enumerate}
\end{definition}

Let \mc{S} denote the category of simplicial sets. The term ``space'' will refer to a simplicial set unless explicitly stated otherwise. 
\begin{definition}
Assume that \mc{M} is simplicially enriched, tensored and cotensored over \mc{S}. Then the {\it internal simplicial structure} on $c\mc{M}$ is defined by
  $$ (X^\bu\otimes^{\rm int} K)^n= X^n\otimes K \ , \ \hom^{\rm int}(K,X^\bu)^n=\hom(K,X^n) $$
and
  $$ \map^{\rm int}(X^\bu,Y^\bu)_n=\Hom_{\mc{M}}(X^\bu\otimes^{\rm int}\Delta^n,Y^\bu),$$
where $X^\bu$ and $Y^\bu$ are objects in $c\mc{M}$ and $K$ is in \mc{S}. 
\end{definition}

\begin{proposition}\label{lem:c-Reedy-right-Quillen} 
The Reedy model structure on $c\mc{M}$ is compatible with the internal simplicial structure. The adjoint functors 
  $$(-)^0\co c\mc{M}^{\rm Reedy}\rightleftarrows\mc{M}\!: c$$ 
define a simplicial Quillen adjunction.
\end{proposition}
\begin{proof}
See, e.g., ~\cite[2.6]{Bou:cos}.
\end{proof}

For any $n\ge 0$, we denote by $\Delta_{\le n}$ the full subcategory of $\Delta$ given by the objects $\ul{\ell}$ with $0\le\ell\le n$. We denote the category of functors from $\Delta_{\le n}$ to \mc{M} by $c_n\mc{M}$. The inclusion functor $i_n\co\Delta_{\le n}\to\Delta$ induces a restriction functor
  $$ i_n^*\co c\mc{M}\to c_n\mc{M} $$
which has a right adjoint $\rho_n$ and a left adjoint $\lambda_n$. We define the {\it $n$-th skeleton} of a cosimplicial object $X^\bu$ by
  $$ \sk_nX^\bu:=\lambda_ni_n^*X^\bu $$
and the {\it $n$-th coskeleton} of a cosimplicial object $X^\bu$ by
  $$ \cosk_nX^\bu:=\rho_ni_n^*X^\bu. $$
Since $i_n$ is full, we have $(\sk_nX^\bu)^{s}=X^s=(\cosk_nX^\bu)^s$ for $0\le s\le n$. 

\begin{lemma}\label{app:skeletal-inclusion-Reedy-cof}
Let $X^\bu$ be Reedy cofibrant. Then the map $\sk_nX^\bu\to \sk_{n+1}X^\bu$ is a Reedy cofibration in $c\mc{M}$ for all $n\ge 0$. A dual statement holds for coskeleta.
\end{lemma}
\begin{proof}
See, e.g., \cite[Proposition 6.5]{Berger-Moerdijk:extension-Reedy}.
\end{proof}

\begin{proposition}\label{sk-left-Quillen}
The adjoint pair $\sk_n\co c\mc{M} \rightleftarrows c\mc{M}\!:\cosk_n$ is a Quillen adjunction for the Reedy model structure. 
\end{proposition}

\begin{proof}
See, e.g., \cite[Lemma 6.4]{Berger-Moerdijk:extension-Reedy}. 
\end{proof}

\medskip

We review the construction of the external simplicial structure on $c\mc{M}$. This does not require any simplicial structure on \mc{M}. 
Since \mc{M} is complete and cocomplete, it is tensored and cotensored over sets in the following way:
for any set $S$ and any object $X$ of $\mc{M}$, we denote by 
  $$\bigsqcup_{S}X \,\text{ and }\,\prod_{S}X $$
the coproduct and product in \mc{M} of $|S|$ copies of $X$. These constructions are functorial in both variables $X$ and $S$. Thus, for objects $X^\bullet$ of $c\mc{M}$ and $L$ of \mc{S}, there is a functor $\Delta^{\rm op}\times\Delta\to\mc{M}$, defined on objects by
$$([\ell],[m])\mapsto\bigsqcup_{L_\ell}X^m.$$
Now one can consider the coend
  $$ X^\bullet\otimes_{\Delta}L:=\int^{\Delta}\bigsqcup_{L_\ell}X^m \in\mc{M} $$
which is, by definition, the following coequalizer
  $$ \bigsqcup_{[m]\to [\ell]}\bigsqcup_{L_\ell} X^m \rightrightarrows \bigsqcup_{\ell\ge 0}\bigsqcup_{L_\ell} X^\ell \dashrightarrow  X^\bullet\otimes_{\Delta}L,   $$
where the parallel arrows are defined in the obvious way using the maps induced by $[m] \to [\ell]$. Moreover, for a fixed simplicial set $K$ and $X^\bu$ in $c\mc{M}$, there is a bicosimplicial object in \mc{M},
  $$ \Delta\times\Delta\to \mc{M}, (m,n)\mapsto \prod_{K_m}X^n ,$$
that is functorial in $K$ and $X^\bu$.

\begin{definition} \label{Externe simpliziale Struktur}
Let $K$ be a simplicial set and let $X^\bullet$ and $Y^\bullet$ be objects of $c\mc{M}$. Then we define tensor, cotensor and mapping space objects as follows:
\begin{align*}
          ( X^\bullet\otimes^{\rm ext} K)^n  &=\, X^\bullet\otimes_{\Delta} ( K\times\Delta^n) \\ 
          \hom^{\rm ext}( K,X^\bullet)^n &= {\rm diag}\left((m,n)\mapsto \prod_{K_m}X^n\right) =\ \ \prod_{K_n}X^n \\
          \map^{\rm ext}( X^\bullet,Y^\bullet)_n &=\ \Hom_{c\mc{M}}(X^\bullet\otimes^{\rm ext}\Delta^n,Y^\bullet)
\end{align*} 
This structure is called the {\it external simplicial structure} on $c\mc{M}$. We will usually drop the superscripts and often abbreviate
  $$ \hom^{\rm ext}( K,X^\bullet)\ \ \text{ to }\ \ (X^\bu)^K. $$
\end{definition}

\begin{example}\label{exam:ext-mapping-space-constant-target}
The external mapping space takes a simple form if the target is constant. For $X^\bu$ in $c\mc{M}$, $Y$ in \mc{M}, and any $n\ge 0$, we have a canonical isomorphism
  $$ (X^\bu\otimes^{\ext}\Delta^n)^0=X^\bu\otimes_{\Delta}\Delta^n\cong X^n $$
and consequently,
  $$ \map^{\ext}(X^\bu,cY)_n\cong\Hom_{c\mc{M}}(X^\bu\otimes^{\ext}\Delta^n,cY)\cong\Hom_{\mc{M}}(X^n,Y) .$$
\end{example}

\medskip

We will also make use of the pointed variations of these objects. Let $\mc{M}_*$ denote the category $* \ucat \mc{M}$ where $*$ is the terminal object of $\mc{M}$. 

\begin{definition} \label{punktierte externe simpliziale Struktur}
Let $K$ be a pointed simplicial set, $X^\bu$ an object in $c\mc{M}$, and $Y^\bu$ in $c\mc{M}_*$. We define
  $$ X^\bullet\oltimes K  =\, (X^\bullet\otimes^{\rm ext} K)/(X^\bullet\otimes^{\rm ext} \Delta^0)\in c\mc{M}_* $$
which has a canonical basepoint as cofiber. We define 
  $$ \olhom( K,Y^\bullet) =\,\fib\bigl[\hom(K,Y^\bu)\to\hom(\Delta^0,Y^\bu)\cong Y^\bu\bigr]\in c\mc{M}_* $$
as the fiber taken at the basepoint of $Y^\bu$ of the map induced by the basepoint of $K$. The map $K\to\Delta^0$ together with the basepoint of $Y^\bu$ induce a basepoint 
$$\ast\to Y^\bu\to\hom(K,Y^\bu),$$ 
and hence $\olhom(K,Y^\bu)$ is also pointed. We will often abbreviate
  $$ \olhom( K,Y^\bullet)\ \ \text{ to }\ \ K(Y^\bu). $$
\end{definition}

\medskip 

\noindent By forgetting the basepoint of $K(Y^\bu)$, one obtains an adjunction isomorphism
  $$ \Hom_{c\mc{M}_{\ast}}(X^\bu\oltimes K,Y^\bu)\cong\Hom_{c\mc{M}}(X^\bu,K(Y^\bu)).$$

\begin{definition}\label{external-loops}
The collapsed boundary gives the $n$-sphere $S^n:=\Delta^n/\partial\Delta^n$ a canonical basepoint. Let $X^\bullet$ be an object of $c\mc{M}$. 
We define the {\it $s$-th external sus\-pen\-sion} by
  $$ \Sigma^n_{\rm ext}X^\bullet= X^\bullet\oltimes (\Delta^n/\partial\Delta^n).$$
If $X^\bu$ is pointed, we define the {\it $s$-th external loop object} by
   $$ \Omega^n_{\rm ext}X^\bullet=\olhom(\Delta^n/\partial\Delta^n, X^\bullet) .$$
We will often omit the subscript ``\texttt{ext}'' when the reference to the external structure is clear from the context.
Note that $\Sigma^0X^\bu=X^\bu\sqcup\ast$.
\end{definition}

\begin{definition}\label{def:olmap}
Let $X^\bu$ be an object in $c\mc{M}$ and $Y^\bu$ an object in $c\mc{M}_*$.
We denote the external mapping space $\map(X^\bu, Y^\bu)$ in $c\mc{M}$ by 
  $$\olmap^{\ext}(X^\bu, Y^\bu),$$ 
when we view it as a pointed simplicial set whose basepoint is given by the constant map $X^\bu\to\ast\to Y^\bu$. 
We will use this bar notation for special emphasis in the definitions, but will often neglect this distinction in the text 
(as well as the superscript \texttt{ext}).
\end{definition}

\begin{remark}
For an unpointed cosimplicial object $X^\bu$ and a pointed cosimplicial object $Y^\bu$, there is a canonical isomorphism of pointed sets 
  $$ \olmap^{\ext}(X^\bu,Y^\bu)_n\cong\Hom_{c\mc{M}_{\ast}}(X^\bu\oltimes\Delta^n_+,Y^\bu)$$
which is natural in $n$, $X^\bu$, and $Y^\bu$.
\end{remark}

\begin{proposition}\label{prop:external-adjunctions-(un)pointed}
Let $K$ be a simplicial set and let $X^\bullet$ and $Y^\bullet$ be objects of $c\mc{M}$.
\begin{enumerate}
   \item 
There are natural isomorphisms of simplicial sets
\begin{align*} 
   \map_{\mc{S}}\bigl(K,\map(X^\bullet,Y^\bullet)\bigr)&\cong\map(X^\bullet\otimes K,Y^\bullet) \cong\map\bigl(X^\bullet,(Y^\bullet)^K\bigr).
\end{align*}
   \item If $K$ and $Y^\bullet$ are pointed, but not $X^\bullet$, there are natural isomorphisms of pointed simplicial sets
   $$\map_{\mc{S}_*}\bigl(K, \olmap(X^\bu, Y^\bu)\bigr) \cong \map_{c\mc{M}_*}(X^\bu\oltimes K,Y^\bu) \cong \olmap\bigl(X^\bu,K(Y^\bu)\bigr).$$
\end{enumerate}
\end{proposition}

\begin{proof}
Part (1) can be found in the simplicial case in \cite[Theorem II.2.5]{GoJar:simp2} and 
\cite[Example II.2.8(4)]{GoJar:simp2} illustrates the cosimplicial version. Part (2) is an easy consequence by direct inspection.
\end{proof}

When they both exist, the external and internal simplicial structures commute in the following sense.

\begin{proposition}\label{int-ext-commute}
Given simplicial sets $K$ and $L$ and a cosimplicial object $X^\bu$, there are canonical natural isomorphisms
  $$ \hom^{\rm int}\bigl(K,\hom^{\rm ext}(L,X^\bu)\bigr)\cong\hom^{\rm ext}\bigl(L,\hom^{\rm int}(K,X^\bu)\bigr) ,$$
where on the left the internal cotensor is applied degreewise. 
If $L$  and $X^\bu$ are pointed, there are canonical natural isomorphisms
  $$ \hom^{\rm int}\bigl(K,\olhom(L,X^\bu)\bigr)\cong\olhom\bigl(L,\hom^{\rm int}(K,X^\bu)\bigr) .$$
\end{proposition}
\begin{proof}
This follows easily from the definitions.
\end{proof}

Although the external simplicial structure and the Reedy model structure are not compatible, the next lemma states that there is some partial compatibility.

\begin{proposition}\label{Reedy-ext-comp}
Let $f\co X^\bu \to Y^\bu$ be a Reedy cofibration in $c\mc{M}$ and $j\co K\to L$ a cofibration in \mc{S}. Then the map
  $$ (X^\bu \otimes^{\rm ext} L)\cup_{(X^\bu \otimes^{\rm ext} K)} (Y^\bu \otimes^{\rm ext} K)\to Y^\bu \otimes^{\rm ext} L$$
is a Reedy cofibration which is a Reedy equivalence if $f$ is a Reedy equivalence. The pointed analogue referring to $\oltimes$ also holds.
\end{proposition}
\begin{proof}
This is proved for simplicial objects in \cite[VII.2.15]{GoJar:simp2}.
\end{proof}

\subsection{Cosimplicial resolutions}\label{subsec:res-model-cat}

This subsection recalls briefly the relevant notions from Bousfield~\cite{Bou:cos}. Especially relevant is \cite[Section 12]{Bou:cos} 
where the unpointed theory is outlined. 

Let \mc{M} be a left proper model category with terminal object $\ast$. In \cite{Bou:cos}, the factorizations in $\mc{M}$ are not assumed to be functorial, however, 
{\it in this article the definition of a model structure is assumed to include functorial factorizations.} 

Let $\mc{M}_{\ast}$ denote the associated pointed model category whose weak equivalences, cofibrations, and fibrations are defined by forgetting the basepoints.
Let $\ho{\mc{M}}$ and $\ho{\mc{M}_\ast}$ be the respective homotopy categories, and $\ho{\mc{M}}_\ast$ the category $[\ast]\ucat\ho{\mc{M}}$.
In \cite[Lemma 12.1]{Bou:cos}, it is explained that for left proper model categories, the ordinary (derived) loop functor $\Omega\co\ho{\mc{M}_\ast}\to \ho{\mc{M}_\ast}$ also yields objects $\Omega^nY, n\ge 1,$ in $\ho{\mc{M}}_{\ast}$ for each object $Y$ in $\ho{\mc{M}}_{\ast}$. The object $\Omega^n Y$ is defined up to isomorphism and
depends only on the isomorphism class of $Y$. Moreover, the object $\Omega^nY$ admits a group object structure in \ho{\mc{M}} for $n\ge 1$, which is abelian for $n\ge 2$, and this group structure depends only on the isomorphism class of $Y$. For $X$ in \ho{\mc{M}}, we write
   $$ [X, Y]_n = [X, \Omega^nY]=\Hom_{\ho{\mc{M}}}(X, \Omega^nY). $$

\medskip   
   
Let \mc{G} be a class of group objects in \ho{\mc{M}}. Each $G\in\mc{G}$, with its unit map, represents an object of $\ho{\mc{M}}_{\ast}$ and therefore it has an $n$-fold loop 
object $\Omega^n G$ in \ho{\mc{M}}. For $n \geq 0$, we have an associated homotopy functor 
$$[-,G]_n \colon \ho{\mc{M}}^{\op} \to \text{Groups}, \ X \mapsto [X, G]_n.$$ 
A map $i\co A\to B$ in \ho{\mc{M}} is {\it \mc{G}-monic} if
    $$i^*\co[B,G]_n\to [A,G]_n$$ is surjective for each $G\in\mc{G}$ and $n\ge 0$, and an object $Y$ in \ho{\mc{M}} is {\it \mc{G}-injective} if
    $$i^*\co [B,Y]\to [A,Y]$$ 
is surjective for each \mc{G}-monic map $i\co A\to B$ in \ho{\mc{M}}. We denote the class of \mc{G}-injective objects by \mc{G}-inj. 
The homotopy category $\ho{\mc{M}}$ has {\it enough \mc{G}-injectives} when every object is the domain of a \mc{G}-monic map to a \mc{G}-injective target. 

A map in \mc{M} is \mc{G}-monic if its image in $\ho{\mc{M}}$ is $\mc{G}$-monic, and an object in \mc{M} is \mc{G}-injective if its image in $\ho{\mc{M}}$ is \mc{G}-injective. 
A fibration in $\mc{M}$ is {\it \mc{G}-injective} if it has the right lifting property with respect to the \mc{G}-monic cofibrations in \mc{M}. 

\begin{definition}\label{def:injective-models}
  A class \mc{G} of group objects in $\ho{\mc{M}}$ is called a {\it class of injective models} if $\ho{\mc{M}}$ has enough \mc{G}-injectives and for every object in \mc{M}, 
  a \mc{G}-monic morphism into a \mc{G}-injective target can be chosen functorially, see \cite[4.2]{Bou:cos}. 
\end{definition}

Note that $[X^\bullet,G]$ is a simplicial group for any $X^\bu$ in $c \mc{M}$ and $G \in \mc{G}$. A map $X^\bu\to Y^\bu$ in $c\mc{M}$ is called:
\begin{itemize}
   \item[(i)] a {\it \mc{G}-equivalence} if for each $G\in\mc{G}$ and $n \geq 0$, the induced map
   $$ [Y^\bu,G]_n \to[X^\bu,G]_n $$
is a weak equivalence of simplicial groups. 
   \item[(ii)] a {\it \mc{G}-cofibration} if it is a Reedy cofibration and for each $G\in\mc{G}$ and $n \geq 0$, the induced map
   $$ [Y^\bu,G]_n \to [X^\bu,G]_n $$
is a fibration of simplicial groups.
   \item[(iii)] a {\it \mc{G}-fibration} if for each $s\ge 0$, the induced map
   $$ X^s\to Y^s\times_{M^sY^\bu}M^sX^\bu $$
is a \mc{G}-injective fibration.
\end{itemize}
A useful observation is that an object is \mc{G}-cofibrant if and only if it is Reedy cofibrant. In fact, a Reedy cofibration that is degreewise \mc{G}-monic is a \mc{G}-cofibration.

\begin{theorem}[Dwyer-Kan-Stover \cite{DKSt:E2}, Bousfield \cite{Bou:cos}] \label{bousfield}
Let \mc{M} be a left proper model category.
If \mc{G} is a class of injective models for \ho{\mc{M}}, then the category $c\mc{M}$ admits a left proper simplicial model structure given by 
the \mc{G}-equivalences, the \mc{G}-fibrations, and the \mc{G}-cofibrations. The simplicial structure is the external one.
\end{theorem}

This is called the {\it \mc{G}-resolution model category} and will be denoted by $c \mc{M}^{\mc{G}}$.

If \mc{G} is a set, then by~\cite[4.5]{Bou:cos}, there is a functorial \mc{G}-monic map into a \mc{G}-injective target for each object in $\mc{M}$, 
i.e., \mc{G} is a class of injective models and the resolution model category exists. 

Theorem \ref{bousfield} may be regarded as a vast generalization of a theorem of Quillen \cite[II.4]{Quillen:HA} which treats the dual statement in the case of a discrete model category $\mc{M}$ with a set $\mc{G}$ of small projective generators. 

\medskip

Given $X^\bu \in c \mc{M}$ and $G \in \mc{G}$, we call the homotopy groups 
$$\pi_{\ast}[X^\bu, G]$$
the $E_2$\emph{-homotopy groups} of $X^\bu$ with respect to $G \in \mc{G}$. There are long exact sequences of  $E_2$-homotopy groups $\pi_{*}[ -, G]$ associated with a $\mc{G}$-cofibration $X^\bu \to Y^\bu$. The 
relative groups of $(Y^\bu, X^\bu)$ are defined to be the homotopy groups of the simplicial group $[(Y^\bu, X^\bu), (G, \ast)]$, whose $n$-simplices 
are the homotopy classes of morphisms of pairs $(Y^n, X^n) \to (G, \ast)$ in the model category of morphisms of $\mc{M}$. 
Let $C^\bu$ be the pointed cofiber of the \mc{G}-cofibration $X^\bu \hookrightarrow Y^\bu$. Since $\mc{M}$ is left proper by assumption 
(otherwise, simply assume that $X^\bu$ is degreewise cofibrant), we have an isomorphism 
$$[(Y^\bu, X^\bu), (G, \ast)] \cong [(C^\bu, \ast), (G, \ast)] \cong [C^\bu, G]_{\mc{M}_*} .$$

\begin{proposition} \label{G-cofiber-to-LES-2} 
Let $X^\bu \hookrightarrow Y^\bu$ be a $\mc{G}$-cofibration. Suppose that 
  $$[Y^\bu, G]_n \to [X^\bu, G]_n$$ 
is an epimorphism for every $G \in \mc{G}$ and $n \geq 0$. Then for every $G \in \mc{G}$, there is a long exact sequence
$$
\cdots \to \pi_s[(Y^\bu, X^\bu), (G, \ast)] \to \pi_s[Y^\bu, G] \to \pi_s[X^\bu, G] \to \pi_{s-1}[(Y^\bu, X^\bu), (G, \ast)] 
\to \cdots $$
$$\cdots \to \pi_0[(Y^\bu, X^\bu), (G, \ast)] \to \pi_0[Y^\bu, G] \to \pi_0[X^\bu, G] \to 0. $$
\end{proposition}
\begin{proof}
Since $X^\bu \to Y^\bu$ is a $\mc{G}$-cofibration, the induced 
map of simplicial groups
$$[Y^\bu, G] \to [X^\bu, G]$$
is a fibration for every $G \in \mc{G}$. Hence it suffices to identify the fiber of this simplicial map with the simplicial 
group $[(Y^\bu, X^\bu), (G, \ast)]$. There is a long exact sequence of simplicial groups 
$$ \cdots [X^\bu, \Omega G] \to [(Y^\bu, X^\bu), (G, \ast)] \to [Y^\bu, G] \to [X^\bu, G] $$
which, by the assumption, simplifies to the required short exact sequences of simplicial groups 
$$0 \to [(Y^\bu, X^\bu), (G, \ast)] \to [Y^\bu, G] \to [X^\bu, G] \to 0.$$
\end{proof}

The constant functor $c$ obviously maps weak equivalences to 
\mc{G}-equivalences, but $(-)^0\co c\mc{M}^{\mc{G}}\rightleftarrows\mc{M}\!: c$ is {\it not} a Quillen adjunction; 
$c$ does not preserve fibrations.

\begin{lemma}\label{lem:c-preserves-G-inj-fib}
If $X\to Y$ is a \mc{G}-injective fibration in \mc{M}, then $cX\to cY$ is a \mc{G}-fibration. In particular, if $G$ is a fibrant \mc{G}-injective object, then $cG$ is \mc{G}-fibrant.
\end{lemma}

\begin{proposition} \label{Delta-Tot-adjunction}
$\Delta^\bu\times -\co \mc{M}\rightleftarrows c\mc{M}\!:\Tot$ is a Quillen adjunction for both the Reedy and the 
\mc{G}-resolution model structures.
\end{proposition}
\begin{proof}
This is shown in \cite[Proposition 8.1]{Bou:cos}.
\end{proof}

\subsection{\mc{G}-cofree and quasi-\mc{G}-cofree maps}\label{subsec:quasi-G-cofree}

Let $\mc{M}$ be a model category and $\mc{G}$ a class of group objects in $\ho{\mc{M}}$. The $\mc{G}$-cofree maps can be viewed as analogues of relative cell complexes. More precisely,  
the objects of $\mc{G}$ and their external iterated loop objects are used to define co-attachments for the relative coskeletal projections. 

Let $\Delta_{\rm id}\subset\Delta_{\rm surj}\subset\Delta$ be the subcategories of identity and surjective maps respectively. 
The restriction functor 
    $$ \rm{Fun}(\Delta_{\rm surj}, \mc{M}) \to \rm{Fun}(\Delta_{\rm id}, \mc{M})$$
has a right adjoint $R$ which is defined objectwise by $ (R \{X^i\}_{i\ge 0})_j = \prod_{j \twoheadrightarrow i} X^i$ for each sequence $\{X^i\}_{i\ge 0}$ of objects in \mc{M}.

\begin{definition}\label{def:codeg-cofree}
A cosimplicial object $X^\bullet$ in $c\mc{M}$ is called {\it codegeneracy cofree} if its restriction to $\Delta_{\rm surj}$ is of the form $R(\{M^i\}_{i\ge 0})$ for some $\{M^i\}_{i\ge 0} \co \Delta_{\rm id} \to \mc{M}$. 
A map $f \co X^\bullet \to Y^\bullet$ in $c \mc{M}$ is called {\it codegeneracy cofree} if there is a codegeneracy cofree object $F^\bu$ such that when we regard $f$ as a map in $\rm{Fun}(\Delta_{\rm surj}, \mc{M})$,  it can be factored into an isomorphism $X^\bu \cong Y^\bu \times F^\bu$ followed by the projection $Y^\bu \times F^\bu \to Y^\bu$. 
\end{definition}

\begin{definition} \label{def:G-cofree}
An object in $c\mc{M}$ is called {\it \mc{G}-cofree} if it is codegeneracy cofree on a sequence 
$(G_i)_{i\ge 0}$ of fibrant \mc{G}-injective objects.  A map $f: X^\bu \to Y^\bu$ is called {\it\mc{G}-cofree} 
if it is a codegeneracy cofree map, as defined in Definition \ref{def:codeg-cofree}, where $F^\bu$ is a \mc{G}-cofree object.
\end{definition}

The class of \mc{G}-cofree maps is stable under pullbacks and \mc{G}-cofree objects are stable under products. Note that an object $X^\bu$ is \mc{G}-cofree 
if and only if the map $X^\bu\to\ast$ is \mc{G}-cofree. The following proposition explains the analogy to relative cell complexes.

\begin{proposition}\label{co-cell decomposition of cofree maps}
A map $f\co X^\bullet\to Y^\bullet$ in $c\mc{M}$ is \mc{G}-cofree on $(G_s)_{s\ge 0}$ if and only if, for all $s\ge 0$, there are pullback diagrams in $c\mc{M}$
\diagr{ \cosk_s(f) \ar[r]\ar[d]_{\gamma_{s-1}(f)} & \hom(\Delta^s,cG_s) \ar[d]^{i_s^*} \\
        \cosk_{s-1}(f) \ar[r]   & \hom(\partial\Delta^s,cG_s),}
where $i_s\co\partial\Delta^s\to\Delta^s$ is the canonical inclusion.
\end{proposition}

\begin{proof}
This is the dual of \cite[VII.1.14]{GoJar:simp2}.
\end{proof}

\begin{proposition}
Every \mc{G}-cofree map is a \mc{G}-fibration. 
\end{proposition}

\begin{proof}
For a \mc{G}-cofree map $X^\bu\to Y^\bu$ and $s \ge 0$, using the pullback diagram of Proposition \ref{co-cell decomposition of cofree maps} in degree $s$, we see that there exists 
a fibrant \mc{G}-injective object $G_s$ such that the map $X^s \to Y^s \times_{M^s Y^\bu} M^s X^\bu$ is 
isomorphic to the projection $$(Y^s \times_{M^s Y^\bu} M^s X^\bu) \times G_s \to Y^s \times_{M^s Y^\bu} M^s X^\bu.$$ 
These are obviously \mc{G}-injective fibrations. 
\end{proof}

\begin{remark} \label{G-cofree-fact}
In certain cases of $\mc{M}$ and $\mc{G}$, there is a monad $\Gamma\co\mc{M}\to\mc{M}$ with natural transformations $e\co {\rm Id}\to\Gamma$ and $\mu\co\Gamma^2\to\Gamma$ such that (cf. 
\cite[Section 7]{Bou:cos}): 
\begin{enumerate}
\item  $\Gamma$ preserves weak equivalences,
\item  $\Gamma$ takes values in fibrant \mc{G}-injective objects,
\item  the map $X \xrightarrow{e} \Gamma(X)$ is \mc{G}-monic for each $X \in \mc{M}$. 
\end{enumerate} 
Using this monad and the associated cosimplicial resolutions, it is sometimes possible to produce functorial factorizations of maps in $c \mc{M}$ into a \mc{G}-equivalence 
(or trivial \mc{G}-cofibration) and a \mc{G}-cofree map, using a construction similar to \cite[p. 57]{Miller:sullivan} or the methods of \cite[Section 7]{Bou:cos}. 
\end{remark}

The notion of a \mc{G}-cofree map is unnecessarily strict for certain purposes when the model category \mc{M} is not discrete (i.e., when the weak equivalences are not just 
the isomorphisms). This leads us to consider the weaker notion of quasi-\mc{G}-cofree maps which are \mc{G}-cofree only up to Reedy equivalence. This notion also provides a useful description 
of $\mc{G}$-fibrations as retracts of quasi-$\mc{G}$-cofree maps. We also prove a factorization statement (Proposition~\ref{prop:quasi-cofree-replacements}) assuming right 
properness of the underlying model category $\mc{M}$. Another advantage of working with quasi-\mc{G}-cofree maps is that they are better suited for the construction of minimal replacements 
(see Proposition~\ref{n-conn-fact}).

\begin{definition}\label{def:quasi-G-cofree}
A map $f\co X^\bullet\to Y^\bullet$ in $c\mc{M}$ is called {\it quasi-\mc{G}-cofree} if there is a sequence of fibrant \mc{G}-injective objects $(G_s)_{s\ge 0}$ such that for all $s\ge 0$, there are homotopy pullback diagrams 
\diagr{ \cosk_s(f) \ar[r]\ar[d]_{\gamma_{s-1}(f)} & \hom(\Delta^s,cG_s) \ar[d]^{i_s^*} \\
        \cosk_{s-1}(f) \ar[r]   & \hom(\partial\Delta^s,cG_s)}
in $c\mc{M}$ equipped with the Reedy model structure. 
\end{definition}

Quasi-\mc{G}-cofree maps are useful because one usually only needs the fact that these squares are homotopy pullbacks rather 
than pullbacks. The notion seems more appropriate than the notion of a \mc{G}-cofree map when \mc{M} has a more complicated 
model structure than the discrete one. If the model structure is discrete, then both notions coincide. 

\begin{lemma}\label{prop:recognize-quasi-G-cofree}
Let $\mc{M}$ be a right proper model category and $f: X^\bu \to Y^\bu$ a Reedy fibration in $c \mc{M}$. Suppose that for each $s\ge 0$, there is a fibrant $\mc{G}$-injective object $G_s$ together with a factorization 
\[ 
\xymatrix{ && (Y^s \times_{M^sY^\bu} M^sX^\bu) \times G_s \ar[d]^{\rm projection} \\
  X^s \ar@{->>}[urr]^-{\sim} \ar@{->>}[rr] && Y^s \times_{M^sY^\bu} M^sX^\bu.
}
\]
where the first map is a trivial fibration. Then $f$ is a quasi-\mc{G}-cofree map.
\end{lemma}
\begin{proof}
For a map $f: X^\bu \to Y^\bu$,  we denote
  $$ M^sf= Y^s \times_{M^sY^\bu} M^sX^\bu. $$
We also abbreviate $\hom^{\rm ext}(K,-)$ to $(-)^K$. By the definition of the coskeletal tower, there are pullback 
squares as follows 
\diagram{ \cosk_s(f) \ar[r]\ar[d]_{\gamma_{s-1}(f)} &  (c X^s)^{\Delta^s} \ar[d] \\
        \cosk_{s-1}(f) \ar[r]   & (cM^sf)^{\Delta^s} \times_{(cM^sf)^{\partial\Delta^s}} (cX^s)^{\partial \Delta^s}}{diagr:cosimpl-pullback}
If $f$ is a Reedy fibration, then so is the map $c X^s \to c M^s(f)$, by Lemma~\ref{lem:c-Reedy-right-Quillen}. By the partial compatibility of the Reedy model structure with the 
external simplicial enrichment in Proposition~\ref{Reedy-ext-comp}, it follows that the vertical map on the right is a Reedy fibration. 
Since \mc{M} is right proper, this pullback square is a homotopy pullback for the Reedy model structure. Thus, after replacing $(cX^s)^{\Delta^s}$ with the Reedy equivalent 
$\bigl(c (M^s f \times G_s) \bigr)^{\Delta^s}$, we retain a 
homotopy pullback square. Using the assumptions, we achieve the following simplifications on the 
right hand side of Diagram~(\ref{diagr:cosimpl-pullback}): 
\begin{align*} 
   (cM^sf)^{\Delta^s} \times_{(cM^sf)^{\partial\Delta^s}} (cX^s)^{\partial\Delta^s} &\simeq (cM^sf)^{\Delta^s} \times_{(cM^sf)^{\partial\Delta^s}} \bigr(c(M^sf\times G_s)\bigr)^{\partial\Delta^s} \\
   &\cong (cM^sf)^{\Delta^s} \times (cG_s)^{\partial\Delta^s}
\end{align*} 
Projecting away from the term $(cM^sf)^{\Delta^s}$, one obtains a homotopy pullback square:
\diagr{ (c X^s)^{\Delta^s} \ar[d]\ar[r] & (cG_s)^{\Delta^s} \ar[d] \\
        (cM^sf)^{\Delta^s} \times_{(cM^sf)^{\partial\Delta^s}} (cX^s)^{\partial\Delta^s} \ar[r] & (cG_s)^{\partial\Delta^s} 
}
Concatenation with Diagram~(\ref{diagr:cosimpl-pullback}) yields the desired homotopy pullback.
\end{proof}

\begin{remark}
Conversely, if $f: X^\bu \to Y^\bu$ is quasi-\mc{G}-cofree, then there is a weak equivalence $X^s \simeq  M^s(f) \times G_s$ for $s \geq 0$. This can be seen by restricting the homotopy pullback square of Definition 
\ref{def:quasi-G-cofree} to  cosimplicial degree $s$.
\end{remark}

\begin{corollary}
Let \mc{M} be a proper model category and $\mc{G}$ a class of injective models. Then:
\begin{enumerate}
   \item
A composition of a trivial Reedy fibration with a $\mc{G}$-cofree map is quasi-\mc{G}-cofree.
   \item
Every \mc{G}-cofree map is quasi-\mc{G}-cofree.
   \item
Every quasi-\mc{G}-cofree Reedy fibration is a \mc{G}-fibration.
\end{enumerate}
\end{corollary}

\begin{proof}
Assertion (1) follows easily from Lemma~\ref{prop:recognize-quasi-G-cofree}. Assertion (2) is a special case of (1). For assertion (3), note that the map 
  $$ (cG_s)^{\Delta^s}\to (cG_s)^{\partial\Delta^s}  $$
is a \mc{G}-fibration by Lemma~\ref{lem:c-preserves-G-inj-fib} and the compatibility of the external simplicial structure with the \mc{G}-resolution model structure. 
Let $f$ be a quasi-\mc{G}-cofree Reedy fibration. This implies that in the diagram of Definition~\ref{def:quasi-G-cofree}, the map
  $$ \gamma_{s-1}(f)\co\cosk_{s}(f)\to\cosk_{s-1}(f) $$
factors into a Reedy equivalence $c \co\cosk_s(f)\to P^\bu$ followed by a \mc{G}-fibration $q \co P^\bu\to\cosk_{s-1}(f)$ where $P^\bu$ denotes the pullback. 
We factor $c$ further into a trivial Reedy cofibration $c_1$ followed by a trivial Reedy fibration $c_2$. Since $f$ is assumed to be a Reedy fibration, one easily 
finds that the Reedy fibration $\gamma_{s-1}(f)$ is a retract of $q \circ c_2$ using the dotted lift in the following diagram:
\diagr{ \cosk_s(f) \ar@{=}[r]\ar[d]_{c_1} & \cosk_s(f) \ar[d]^{\gamma_{s-1}(f)} \\
        Z^\bu \ar[r]^-{q \circ c_2}\ar@{.>}[ur] & \cosk_{s-1}(f) }
Since $q \circ c_2$ is a \mc{G}-fibration, so is $\gamma_{s-1}(f)$ and then also $f$, too, by induction.
\end{proof}

We can now give the second construction for the factorization of maps. 

\begin{proposition}\label{prop:quasi-cofree-replacements}             
Let \mc{M} be a proper model category and \mc{G} a class of injective models.
Every map in $c\mc{M}$ can be factored functorially into a trivial \mc{G}-cofibration followed by a quasi-\mc{G}-cofree Reedy fibration.
\end{proposition}
\begin{proof}
Let $f: X^\bu \to Y^\bu$ be a map. We construct a cosimplicial object $Z^\bu$ and a factorization 
  $$X^\bu \xrightarrow{j} Z^\bu \xrightarrow{q} Y^\bu$$
by induction on the cosimplicial degree. Let $\gamma\co X^0 \to G_0$ be a functorial $\mc{G}$-monic map into a fibrant $\mc{G}$-injective object. Factorize the map $(f^0,\gamma)$ in $\mc{M}$
  $$X^0 \rightarrowtail Z^0 \stackrel{\sim}{\twoheadrightarrow} Y^0 \times G_0$$
into a cofibration followed by a trivial fibration. Since $X^0 \to G_0$ is \mc{G}-monic, it follows that $X^0\to Y^0\times G_0$ and, as a consequence, $X^0\to Z^0$ are also 
\mc{G}-monic. 

Suppose we have constructed $Z^\bu$ up to cosimplicial degree $< s$. For the inductive step, we need to find a suitable factorization 
  $$L^s Z^\bu \cup_{L^s X^\bu} X^s \to Z^s \to Y^s \times_{M^s Y^\bu} M^s Z^\bu.$$
Let $L^s Z^\bu \cup_{L^s X^\bu} X^s \to G_s$ be a functorial $\mc{G}$-monic map into a fibrant $\mc{G}$-injective object.
Choose a factorization in $\mc{M}$ as before,
  $$L^s Z^\bu \cup_{L^s X^\bu} X^s \rightarrowtail Z^s \stackrel{\sim}{\twoheadrightarrow} (Y^s \times_{M^s Y^\bu} M^s Z^\bu) \times G_s$$
into a $\mc{G}$-monic cofibration followed by a trivial fibration. This completes the inductive definition of $Z^\bu$ and the construction of the factorization. By 
construction, the resulting map $j: X^\bu \to Z^\bu$ has the property 
that for all $s\ge 0$, the map 
  $$ L^s Z^\bu \cup_{L^s X^\bu} X^s \to Z^s$$
is a $\mc{G}$-monic cofibration in $\mc{M}$. This implies that $j$ is a trivial $\mc{G}$-cofibration \cite[Proposition 3.13]{Bou:cos}. 
The map $q: Z^\bu \to Y^\bu$ is a Reedy fibration and satisfies the assumptions of Lemma~\ref{prop:recognize-quasi-G-cofree}. So it is also quasi-\mc{G}-cofree, as required. 
\end{proof}

\begin{corollary} \label{all-about-cofree-maps2}
Let $\mc{M}$ and $\mc{G}$ be as in \emph{Proposition~\ref{prop:quasi-cofree-replacements}}. Then every \mc{G}-fibration is a retract of 
a quasi-\mc{G}-cofree Reedy fibration. If the model structure on \mc{M} is discrete then every \mc{G}-fibration is a retract of a \mc{G}-cofree map. 
\end{corollary}
\begin{proof}
This follows easily from Proposition \ref{prop:quasi-cofree-replacements}.
\end{proof}

\section{Natural homotopy groups}\label{sec:natural-homotopy-groups}

In the theory of resolution model structures, there are two notions of homotopy groups: the $E_2$-homotopy groups, as already defined, and the natural homotopy 
groups. A central tool for relating these two types of homotopy groups is the spiral exact sequence. These homotopy groups and the spiral exact sequence carry 
further algebraic structure that we formalize in the notions of \mc{H}-algebra and $\pi_0$-module. These ideas were introduced in \cite{DKSt:E2, DKSt:bigraded} 
in the ``dual'' setting of simplicial spaces. 

In Subsection~\ref{subsec:natur+spirale}, we define the natural homotopy groups and discuss some of their basic properties. The \mc{H}-algebra structure 
and $\pi_0$-module structure on the $E_2$- and the natural homotopy groups are described in Subsection~\ref{subsec:algebraic-structure-on-G-homotopy-groups}. We state the theorem about the spiral exact 
sequence in Subsection~\ref{subsec:spiral}. Its construction and the proof of the $\pi_0$-module structure are deferred to Appendix~\ref{appsec:spiral}. 

The last Subsection~\ref{subsec:cosimp-conn} introduces cosimplicial connectivity associated to \mc{G} and some useful properties of cosimplicially $n$-connected maps.

\subsection{Basic properties of the natural homotopy groups}\label{subsec:natur+spirale}
Here we will define the natural homotopy groups and the $E_2$-homotopy groups associated with a class of injective models \mc{G}. The comparison between these types of homotopy groups
is given by the spiral exact sequence in Theorem~\ref{ses-statement}. To obtain all the necessary properties we introduce the following conditions.

\begin{assumption}\label{strict-unit}
Let \mc{M} be a left proper model category (with functorial factorizations) and $\mc{G}$ a class of injective models in the sense of Definition~\ref{def:injective-models}. Furthermore,
we assume:
\begin{enumerate}
   \item
Each object $G \in \mc{G}$ is an abelian group object in \ho{\mc{M}}. 
   \item
The unit map of this group structure in \ho{\mc{M}} is given by a map $\ast\to G$ in $\mc{M}$. We regard the objects in $\mc{G}$ as objects of $\mc{M}_*$. 
   \item
Each object $G \in \mc{G}$ is fibrant in $\mc{M}$ (equivalently, in $\mc{M}_*$).
   \item
The class \mc{G} is closed under the loop functor $\Omega$ defined on $\mc{M}_*$.
\end{enumerate}
\end{assumption}

\begin{remark}
Some remarks on these assumptions are in order. 
\begin{enumerate}
   \item
The assumption that all $G\in\mc{G}$ are fibrant simplifies the notation. We do not want to invoke derived mapping spaces. 
   \item
The \mc{G}-resolution model structure is determined by the \mc{G}-injective objects (see \cite[4.1]{Bou:cos}). If $G$ is in \mc{G}, then $\Omega G$ is \mc{G}-injective. 
Hence, the assumption that \mc{G} is closed under internal loops is not necessary. It is convenient though because this way we avoid passing to \mc{G}-injective 
objects.
   \item
The assumption that each homotopy group object $G\in\mc{G}$ is strictly pointed and homotopy abelian is technically convenient in the construction of the spiral exact sequence 
and its $\pi_0$-module structure in Appendix~\ref{appsec:spiral}. It is likely that this assumption can be relaxed.
\end{enumerate}
\end{remark}

Under our assumptions, for any cosimplicial object $X^\bu$ in $c \mc{M}$ and $G$ in $\mc{G}$, the simplicial group $[X^{\bu}, G]$ is abelian. For every $s\ge 0$, we have the functor of $E_2$\emph{-homotopy groups}
\begin{equation}
(X^\bu,G)  \mapsto \pi_s[X^\bullet,G] .
\end{equation}
These groups detect \mc{G}-equivalences by definition. The name is justified by the fact that they define the $E_2$-term of the spiral spectral sequence in~\ref{subsec:sss}.

\medskip 

Following \cite{DKSt:E2},  we introduce another set of invariants.
For every $s\ge 0$, we have a functor
\begin{equation}\label{natural homotopy groups}
       (X^\bu, G)  \mapsto [X^\bullet,\Omega^s_{\rm ext}cG]_{c\mc{M}^{\mc{G}}}=:\naturalpi{s}{X^\bullet}{G} 
\end{equation}
which we call the {\it s-th natural homotopy group} of $X^\bullet$. The 
functoriality properties in $G$ of these two types of homotopy groups will be studied in detail in the following Subsection \ref{subsec:algebraic-structure-on-G-homotopy-groups}.

Thanks to Assumption~\ref{strict-unit}, we can supply the simplicial set $\Hom_{\mc{M}}(X^\bullet,G)$ with a canonical basepoint $X^0\to\ast\to G$. With the notation defined in 
\ref{def:olmap}, there is a natural isomorphism of pointed simplicial sets
\begin{equation} \label{Hom=olmap}
    \Hom_{\mc{M}}(X^\bullet,G)\cong \olmap^{\rm ext}(X^\bullet, cG),
\end{equation}
as explained in Example~\ref{exam:ext-mapping-space-constant-target}.
If $X^\bullet$ is Reedy cofibrant, and since $G$ is fibrant, this simplicial set is fibrant and there is another natural isomorphism using 
Proposition~\ref{prop:external-adjunctions-(un)pointed}:
\begin{equation*}
    \naturalpi{s}{X^\bullet}{G}=[X^\bullet,\Omega^s_{\rm ext}cG]_{c\mc{M}^{\mc{G}}}\cong \pi_s\olmap^{\rm ext}(X^\bullet, cG) 
\end{equation*}

One reason for introducing the natural homotopy groups is the fact that cosimplicial
objects have natural Postnikov decompositions with respect to them. This decomposition is central 
to the approach we take and will be developed in later sections.
Given a Reedy cofibrant object $X^\bu \in c \mc{M}$ and $G \in \mc{G}$, there is a natural isomorphism 
  $$\map(\sk_n X^\bu, cG) \cong \cosk_n \map(X^\bu, cG),$$ 
which is obtained using the isomorphism in (\ref{Hom=olmap}). This shows that 
$$ \emph{\naturalpi{s}{\sk_nX^\bullet}{G}} \cong \left\{
                   \begin{array}{cl}
                     \emph{\naturalpi{s}{X^\bu}{G}} & s < n \\
                             0                      & \hbox{otherwise. }
                   \end{array} 
                                        \right. $$ 
Here, we use that for a Kan complex $K$, $\cosk_{n+1}K$ is a model for the $n$-th Postnikov section. Thus, the skeletal filtration
  $$ \sk_1 X^\bu \to \cdots \to \sk_n X^\bu \to \cdots \to X^\bu$$
defines a Postnikov decomposition of $X^\bu$ with respect to its natural homotopy groups. 

\begin{lemma}
The functor $\sk_n$ preserves trivial \mc{G}-cofibrations and dually, $\cosk_n$ preserves \mc{G}-fibrations.
\end{lemma}

\begin{proof}
This proof is based on the characterization of (trivial) \mc{G}-cofibrations in \cite[Proposition 3.13]{Bou:cos}. 
Let $f: X^\bu \to Y^\bu$ be a trivial $\mc{G}$-cofibration. Then the cofibration 
$$\sk_n(X^\bu)^m \cup_{L^m(\sk_n(X^\bu))} L^m(\sk_n(Y^\bu)) \longrightarrow \sk_n(Y^\bu)^m$$
is $\mc{G}$-monic if $m \leq n$ - since $f$ is a trivial \mc{G}-cofibration - and if $m > n$, it is an isomorphism because 
$$L^m(\sk_n(X^\bu)) \xrightarrow{\cong} \sk_n(X^\bu)^m$$
$$L^m(\sk_n(Y^\bu)) \xrightarrow{\cong} \sk_n(Y^\bu)^m$$
(see \cite[Lemmas 6.2 and 6.3]{Berger-Moerdijk:extension-Reedy}). It follows that $\sk_n(f)$ is a trivial \mc{G}-cofibration. 
\end{proof}

\begin{remark}\label{rem:natural-Postnikov}
We note that that  $\sk_nX^\bu\to X^\bu$ is usually {\it not} a \mc{G}-cofibration. 
This can be seen by comparing the cofiber with the \mc{G}-homotopy cofiber of the map $i_n\co\sk_nX^\bu\to X^\bu$. Since, for Reedy cofibrant $X^\bu$, $\sk_nX^\bu$ is a model for the cosimplicial $(n-1)$-Postnikov section, the \mc{G}-homotopy cofiber of $i_n$ has trivial natural homotopy groups in degrees $< n$, but not necessarily in degree $n$ itself. However, the cofiber $C^\bu$ of $i_n$ has $C^k=\ast$ for all $0\le k\le n$. 
\end{remark}

There are long exact sequences of natural homotopy groups associated with a $\mc{G}$-cofibration, too. 
Given a $\mc{G}$-cofibration $i : X^\bu \hookrightarrow Y^\bu$, we define relative natural homotopy groups as follows:
\begin{align*}
 \naturalpi{s}{(Y^\bu, X^\bullet)}{G}&= \pi_s \Hom_{\mc{M}}((Y^\bu, X^\bullet),(G, \ast)) \cong \pi_s \Hom_{\mc{M}}((C^\bu, \ast), (G, \ast)) \\
           & \cong\pi_s\map_{\mc{M}_*}(C^\bu, cG)
\end{align*}
where $C^\bu$ denotes the canonically pointed cofiber of $i$. We emphasize here that the $n$-simplices of $\Hom_{\mc{M}}\bigl((Y^\bu, X^\bullet),(G, \ast)\bigr)$ are the morphisms $Y^n \to G$ which send $X^n$ to the basepoint of $G$. Clearly, $\pi_s\map_{\mc{M}_*}(C^\bu, cG)$ is 
a natural homotopy group of $C^\bu$ as a cosimplicial object in the pointed category $c\mc{M}_*$.

\begin{proposition} \label{G-cofiber-to-LES} 
Let $X^\bu \hookrightarrow Y^\bu$ be a $\mc{G}$-cofibration. Then for every $G \in \mc{G}$, there is a long exact sequence of natural homotopy groups
$$
\cdots \to \emph{\naturalpi{s}{(Y^\bu, X^\bu)}{G}} \to \emph{\naturalpi{s}{Y^\bu}{G}} \to \emph{\naturalpi{s}{X^\bu}{G}} \to 
\emph{\naturalpi{s-1}{(Y^\bu, X^\bu)}{G}} \to \cdots $$
$$\cdots \to \emph{\naturalpi{0}{(Y^\bu, X^\bu)}{G}} \to \emph{\naturalpi{0}{Y^\bu}{G}} \to \emph{\naturalpi{0}{X^\bu}{G}}. $$
\end{proposition}

\begin{proof}
The induced map $\olmap(Y^\bu, cG) \to \olmap(X^\bu, cG)$ is a Kan fibration with fiber $\Hom((Y^\bu, X^\bu), (G, \ast))$ 
at the basepoint. So there is a long exact sequence of homotopy groups as required. 
\end{proof}

\subsection{Algebraic structures on homotopy groups}
\label{subsec:algebraic-structure-on-G-homotopy-groups}

The algebraic structure of the spiral exact sequence is formulated using the notions of \Hun-algebra, \mc{H}-algebra, and of $\pi_0$-module.  
In Examples~\ref{exam:action-on-natural} and~\ref{exam:naturalpi-with-Omega} we explain how the natural homotopy groups are equipped with these structures. Example~\ref{exam:action-on-bigraded} exhibits these structures on the $E_2$-homotopy groups.

\begin{definition} \label{def:K-algebra-general-unpointed}
Let $\Hun$ be the full subcategory of $\ho{\mc{M}}$ given by finite products of objects in \mc{G}. An {\it \Hun-algebra} is a product-preserving functor 
from $\Hun$ to the category of sets. The morphisms between \Hun-algebras are given by natural transformations. We denote the category of \Hun-algebras by $\Hun\text{-\Alg}$.
\end{definition}

We will also need to consider the pointed version of the notion of an \Hun-algebra. 

\begin{definition}\label{def:K-algebra-general}
Let \mc{H} be the full subcategory of $\ho{\mc{M}_*}$ given by finite products of objects in \mc{G} using Assumption~\ref{strict-unit}(2). An {\it \mc{H}-algebra} is a product-preserving functor from \mc{H} to the category of sets. The morphisms between \mc{H}-algebras are given by natural transformations. We denote the category of \mc{H}-algebras by $\mc{H}\text{-\Alg}$.
\end{definition}

Note that every \Hun-algebra becomes an \mc{H}-algebra via precomposition with the forgetful functor $\mc{H} \to \Hun$. Given an \Hun-algebra $A \co \Hun \to \set$,
we denote the associated \mc{H}-algebra by $\Xi(A): \mc{H} \to \set$.

\begin{remark}\label{rem:pointed-values} 
The terminal object $\ast$ is in \mc{H}, since it is the empty product, therefore \mc{H} is pointed. Every \mc{H}-algebra preserves the terminal 
object and lifts uniquely to a product-preserving functor with values in pointed sets. Clearly, the forgetful functor from product-preserving 
functors with values in pointed sets to \mc{H}-algebras is an equivalence of categories. Moreover, since each $G \in \mc{G}$ is an abelian group 
object, so are the values of each \mc{H}-algebra as well, and the morphisms of \mc{H}-algebras preserve this structure. 
\end{remark}

\begin{example}\label{example:K-algebras}
Our motivating examples are the following:
\begin{enumerate}
   \item
Every object $X$ of \mc{M} yields an \Hun-algebra 
  $$F_X\co\Hun \to \set\, ,\, F_X(-)=[X,-]_{\mc{M}}.$$
If $\mc{M}=\mc{S}$ is the category of spaces and $\mc{G}=\{K_n=K(\F,n)\,|\,n\ge 0\}$, this turns out to be equivalent 
to $H^*(X)$ as an unstable algebra. This claim is proved in Appendix~\ref{appsec:H-alg}.
   \item
For any $s\ge 0$ and $X^\bu$ in $c\mc{M}$, there are simplicial \mc{H}-algebras
  $$ G\mapsto \Omega^s[X^\bu,G]. $$
Applying $\pi_0$ yields an \mc{H}-algebra structure on the $E_2$-homotopy groups
  $$ \pi_s[X^\bu,G]. $$
Note that this defines an \Hun-algebra for $s = 0$. 
  \item
For any $s\ge 0$, the natural homotopy groups
  $$ G\mapsto \naturalpi{s}{X^\bu}{G}=[X^\bu,\Omega^scG]_{c\mc{M}^{\mc{G}}} $$
define \mc{H}-algebras. Moreover, this defines an \Hun-algebra for $s=0$.
\end{enumerate}
\end{example}

The notion of a module over an \Hun-algebra can be defined using the general abstract notion of a Beck module. We first 
recall the abstract definition and then discuss its specialization to \Hun-\Alg.

\begin{definition}\label{def:pi-0-modules}
Let \mc{C} be a category with pullbacks and let $A \in \mc{C}$ be an object. An abelian object in the slice category $\mc{C}/A$ is called an {\it $A$-module}. More explicitly, an $A$-module is 
an object in the slice category 
 $$p\co E\stackrel{\curvearrowleft}{\longrightarrow} A,$$ 
which is equipped with a section $s\co A\to E$ and structure maps in $\mc{C}/A$
$$m \co E \times_A E \to E, \ i \co E \to E$$
that make $p$ into an abelian group object in $\mc{C}/A$. The term {\it Beck module} is also used in the literature.  A {\it morphism of $A$-modules} is a morphism of abelian objects in $\mc{C}/A$.
\end{definition}

\begin{example-special} \label{Hun-algebra-mod}
Let $A$ be an \Hun-algebra. Suppose that $p \co E \to A$ corresponds to an abelian group object in $(\Hun\text{-\Alg})/A$ (or $A$-module). If we restrict to the corresponding \mc{H}-algebras, 
then $\Xi(A)$ is pointed and we have split short exact sequences of abelian groups
$$0 \to K(G) \to \Xi(E)(G) \stackrel{\curvearrowleft}{\rightarrow} \Xi(A)(G) \to 0$$
where the \mc{H}-algebra $K$ is the kernel of $\Xi(p)$ at the basepoint of $\Xi(A)$. Thus, we obtain for each $G \in \mc{G}$ an isomorphism as follows
$$E(G) \cong A(G) \times K(G).$$
These isomorphisms together with the \Hun-algebra structure of $E$ prescribe an action of $A$ on $K$, that is, for each $g \co G_1 \to G_2$ in \Hun, we obtain an action map
$$A(G_1) \times K(G_1) \cong E(G_1) \xrightarrow{E(g)} E(G_2) \cong A(G_2) \times K(G_2) \to K(G_2)$$
and these maps satisfy certain compatibility conditions (cf. \cite[Section 4]{BlDG:pi-algebra}). Moreover, this action of $A$ on $K$ determines $p$ as an abelian group object. We will often identify the datum of an $A$-module 
$p: E \to A$ with the associated action of $A$ on the kernel of $p$. 
\end{example-special} 

Fix an object $X^\bu$ in $c\mc{M}$. Our interest lies in the following modules over the respective \Hun-algebras. 

\begin{example}\label{exam:action-on-natural}
For every $G\in\mc{G}$ and $p > 0$, we have a \mc{G}-fibration sequence 
\begin{equation*}
    \Omega^pcG \xrightarrow{\gamma_\ext}  (cG)^{\Delta^p/\partial\Delta^p} \xrightarrow{\eps_*} cG ,
\end{equation*}
with section $s_\ext\co cG\to (cG)^{\Delta^p/\partial\Delta^p}$. In fact, since $cG$ is an abelian group object in the \mc{G}-resolution homotopy category, 
this fibration is equivalent to a trivial projection via the following diagram: 
  $$\xymatrix@C=45pt{\Omega^pcG\times cG \ar[r]_-{\simeq}^-{\gamma_\ext + s_\ext}\ar[d]_{{\rm pr}} & (cG)^{\Delta^p/\partial\Delta^p} \ar[d]_{\eps_*} \\
        cG \ar@{=}[r]\ar@/_10pt/[u]_-{(0,\id)} & cG \ar@/_10pt/[u]_-{s_\ext} }$$
where both squares with maps to and from $cG$ commute. The induced map 
$$[X^\bu,(cG)^{\Delta^p/\partial\Delta^p}]_{c\mc{M}^{\mc{G}}} \xrightarrow{p} \naturalpi{0}{X^\bu}{G} = \co \pi_0$$ 
defines an \Hun-algebra over the \Hun-algebra $\pi_0$ and is equipped with an abelian group structure that comes from the cogroup structure of  $\Delta^p / \partial \Delta^p$ 
in the homotopy category of pointed spaces. If we restrict to the underlying \mc{H}-algebras, we obtain a split short exact sequence of \mc{H}-algebras:
$$ 0\to \naturalpi{p}{X^\bu}{G}\to [X^\bu,(cG)^{\Delta^p/\partial\Delta^p}]_{c\mc{M}^{\mc{G}}} \stackrel{\curvearrowleft}{\longrightarrow} \naturalpi{0}{X^\bu}{G}\to 0.$$
The \Hun-algebra structure on $G \mapsto [X^\bu,(cG)^{\Delta^p/\partial\Delta^p}]_{c\mc{M}^{\mc{G}}}$ and the isomorphisms
$$[X^\bu,(cG)^{\Delta^p/\partial\Delta^p}]_{c\mc{M}^{\mc{G}}} \cong \naturalpi{p}{X^\bu}{G}\times\naturalpi{0}{X^\bu}{G}$$
prescribe an action of $\pi_0$ on the natural homotopy groups $\{G \mapsto \naturalpi{p}{X^\bu}{G}\}$, as indicated in Example \ref{Hun-algebra-mod}.
\end{example}

\begin{definition}
Given a fibrant simplicial set $K$, we define
  $$ \Omega_+^pK=\map_{\mc{S}}(\Delta^p/\partial\Delta^p,K). $$ 
The functor $\Omega_+=\Omega_+^1$ is the free loop space functor often denoted by $L$ or $\Lambda$. 
Note the isomorphism of pointed simplicial sets
  $$ \map^{\ext}(X^\bu,(cG)^{\Delta^p/\partial\Delta^p})\cong \Omega_+^p\map(X^\bu,cG). $$
\end{definition}

\begin{example}\label{exam:action-on-bigraded} 
The split cofiber sequence of pointed simplicial sets
  $$ S^0\stackrel{\curvearrowleft}{\longrightarrow}(\Delta^p/\partial\Delta^p)_+\to \Delta^p/\partial\Delta^p, $$
where the splitting collapses $\Delta^p/\partial\Delta^p$ to the non-basepoint in $S^0$, induces the split fibration sequence of simplicial abelian groups:
\begin{equation}\label{eqn:action-on-bigraded} 
   \Omega^p[X^\bu,G] \to \Omega_+^p[X^\bu,G] \stackrel{\curvearrowleft}{\longrightarrow} [X^\bu,G].
\end{equation}
Note that this is natural in $G \in \mc{H}$. Then, after passing to the path components, we obtain a split short exact sequence of \mc{H}-algebras:
$$0 \to \pi_p[X^\bu,G] \to \pi_0(\Omega_+^p [X^\bu,G])  \stackrel{\curvearrowleft}{\longrightarrow} \pi_0[X^\bu,G] \to 0.$$
This defines an abelian object over $\pi_0$ and consequently, $\pi_p[X^\bu,G]$ is endowed with an action by the \Hun-algebra $\pi_{0}[X^\bu, G]$. 
In this way, the $E_2$-homotopy groups of $X^\bu$ become $\pi_0$-modules.
\end{example}

\begin{example}\label{exam:naturalpi-with-Omega}
The \mc{H}-algebra $\naturalpi{p}{X^\bu}{\Omega G}$ is not just a module over the $\mc{H}$-algebra $\naturalpi{0}{X^\bu}{\Omega G}$, but also a module over the \Hun-algebra $\pi_0=\naturalpi{0}{X^\bu}{G}$. The reason is that $\naturalpi{0}{X^\bu}{\Omega G}$ is already a 
$\pi_0$-module via the split $\mc{G}$-fibration sequence (see Definition~\ref{def:Path-P-Omegabar})
$$\Omega G \to LG \stackrel{\curvearrowleft}{\longrightarrow} G.$$

\noindent To exhibit this structure more precisely, we first need some preparation. We deal with this in detail because it can be confusing and it will be needed in Appendix \ref{appsec:spiral} (see 
Corollary~\ref{cor:shift-is-pi-0}).

The representing object $\Omega^p c(\Omega G)$ of $\naturalpi{p}{X^\bu}{\Omega G}$ is the total fiber of the following commutative square of \mc{G}-fibrations:
\begin{equation}
\begin{split}\label{totalfiber} 
   \xymatrix{ c(LG)^{S^p} \ar[r]\ar[d] & c(LG) \ar[d] \\
                (cG)^{S^p} \ar[r] & cG}
\end{split}
\end{equation}
In other words, there is a fiber sequence
  $$ \Omega^p c(\Omega G) \to c(LG)^{S^p}\to c(LG)\times_{cG} (cG)^{S^p} .$$
In Diagram~\eqref{totalfiber} the squares that are defined by the respective sections over $c G$, either to $c(LG)^{S^p}$ or the 
pullback, commute. Hence there is a pullback square
\begin{equation*}
   \xymatrix{  P \ar[r]\ar[d]^{p} &  c(LG)^{S^p} \ar[d] \\
              cG \ar[r] & c(LG)\times_{cG}(cG)^{S^p} }
\end{equation*}
where $p$ is a \mc{G}-fibration with fiber $\Omega^p c(\Omega G)$ and a section which is derived from pulling back the section $cG \to c(LG)^{S^p}$.
Then the required split short exact sequence is as follows, 
  $$ 0\to \naturalpi{p}{X^\bu}{\Omega G}\to [X^\bu, P]_{c\mc{M}^{\mc{G}}}\stackrel{\curvearrowleft}{\longrightarrow} \naturalpi{0}{X^\bu}{G}\to 0.$$
In the case $p=0$, this is explicitly given by
  $$ 0\to \naturalpi{0}{X^\bu}{\Omega G}\to [X^\bu, c(LG)]_{c\mc{M}^{\mc{G}}}\stackrel{\curvearrowleft}{\longrightarrow} \naturalpi{0}{X^\bu}{G}\to 0.$$
\end{example}

\subsection{The spiral exact sequence}\label{subsec:spiral}
The main result of Section~\ref{sec:natural-homotopy-groups} is the following long exact sequence which contains the comparison between the natural homotopy groups and the $E_2$-homotopy groups. 

\begin{theorem}[The spiral exact sequence, \ref{spiral-ex-sequ}~and~\ref{thm:how--spiral-exact-seq-looks}] \label{ses-statement}
There is an isomorphism of \Hun-algebras 
    $$ \naturalpi{0}{X^\bullet}{G}\cong\pi_0[X^\bullet,G]=:\pi_0$$ 
and a long exact sequence of \mc{H}-algebras and $\pi_0$-modules
\begin{align*}
    ...& \to \naturalpi{s-1}{X^\bullet}{\Omega G}\to\naturalpi{s}{X^\bullet}{G}\to\pi_s[X^\bullet,G]\to \naturalpi{s-2}{X^\bullet}{\Omega G} \to ... \\
    ...& \to \pi_{2}[X^\bullet,G]\to \naturalpi{0}{X^\bullet}{\Omega G}\to \naturalpi{1}{X^\bullet}{G}\to \pi_1[X^\bullet,G] \to 0,
\end{align*}
where $\Omega$ is the internal loop functor.
\end{theorem}
 
The construction of this {\it spiral exact sequence} will be deferred to Appendix~\ref{appsec:spiral} and Theorem~\ref{thm:how--spiral-exact-seq-looks}. The $\pi_0$-module structure 
is proved in Theorem~\ref{thm:pi-0-module-structure}. 

In the case where \mc{M} is the category of spaces and \mc{G} the set of finite products of $\F$-Eilenberg-MacLane spaces, an 
\Hun-algebra is the same as an unstable algebra by Theorem~\ref{thm:UA-equiv-H-alg}. Moreover, by Corollary~\ref{cor:inverse-to-h}, the module structure over 
$\pi_0$ on the terms $\pi_s[X^\bu,G]$ corresponds to the usual action of the unstable algebra $\pi_0H^*(X^\bu)$ on (the unstable module) $\pi_sH^*(X^\bu)$.

\begin{corollary} \label{naturale G-Aequivalenzen} 
A map $X^\bullet\to Y^\bullet$ is a \mc{G}-equivalence if and only if it induces iso\-mor\-phisms
$ \emph{\naturalpi{s}{Y^\bullet}{G}}\to\emph{\naturalpi{s}{X^\bullet}{G}}$
for all $s\ge 0$ and $G\in\mc{G}$.
\end{corollary}

\begin{proof}
This follows immediately from the spiral exact sequence by induction on the degree of connectivity over the whole class \mc{G} and using the Five Lemma. 
\end{proof}

\subsection{Cosimplicial connectivity}\label{subsec:cosimp-conn}
The following definition generalizes the notion of connectivity of maps to general resolution model categories. The main result of this subsection 
establishes yet another factorization property which takes into account the cosimplicial connectivity of a map and gives a more controlled 
way of replacing it by a cofree map.

Let $\mc{M}$ be a left proper model category and \mc{G} a class of injective models. In this subsection, we also require that Assumption~\ref{strict-unit} is satisfied. 
We have the following notion of connectivity for maps in the resolution model category
$c \mc{M}^{\mc{G}}$. 

\begin{definition}\label{def. n-cocon.}
Let $n \geq 0$. 
\begin{itemize}
\item[(a)] A morphism $f: X^\bu \to Y^\bu$ is called {\it cosimplicially $n$-connected} if the induced map
  $$ \naturalpi{s}{Y^\bu}{G}\to\naturalpi{s}{X^\bu}{G} $$
is an isomorphism for $0\le s< n$ and an epimorphism for $s=n$ for all $G\in \mc{G}$.
\item[(b)] A {\it pointed object $(C^\bu, \ast)$  of $c\mc{M}^{\mc{G}}$ is cosimplicially $n$-connected} if the map $\ast \to C^\bu$ is cosimplicially $n$-connected.
\item[(c)] An arbitrary object $X^\bu$ of $c\mc{M}^{\mc{G}}$ is called {\it cosimplicially $n$-connected} if the canonical map $X^\bu \to \ast$ is cosimplicially $n$-connected.
\end{itemize}
\end{definition}

A pointed object $(C^\bu, \ast)$  of $c\mc{M}^{\mc{G}}$ is cosimplicially $n$-connected if and only if $C^\bu$ is cosimplicially $(n-1)$-connected.
Every pointed object is cosimplicially $0$-connected.

Cosimplicial connectivity has to be clearly distinguished and is completely different from connectivity notions internal to the category \mc{M}.
However, when it is clear from the context that a connectivity statement refers to the cosimplicial direction, rather than an internal direction, we will 
omit the word ``cosimplicially''. Obviously, the notion also depends on the class \mc{G} which will also be suppressed from the notation. 
We emphasize that cosimplicial connectivity is in\-va\-riant under $\mc{G}$-equivalences. 

Cosimplicial connectivity can be formulated in several equivalent ways. We state some of them in the following lemma. 
The restriction to $\mc{G}$-cofibrations is, of course, simply for technical convenience. 

\begin{lemma}\label{lem:equivalent-formulations-n-connected}
Let $f: X^\bu \to Y^{\bu}$ be a $\mc{G}$-cofibration in $c\mc{M}^{\mc{G}}$ and $(C^\bu, \ast)$ its pointed cofiber.
Then the following assertions are equivalent:
\begin{enumerate}
\item[(i)] The map $f$ is $n$-connected.
\item[(ii)] The map $f$ is $0$-connected and $(C^\bu, \ast)$ is $n$-connected. 
\item[(iii)] The map $f$ is $0$-connected and the induced map
  $$\pi_s[C^\bu, G]_{\mc{M}} \to \pi_s[\ast, G]_{\mc{M}} $$
is an isomorphism for $0\le s < n$ and an epimorphism for $s = n$ for all $G \in \mc{G}$.
\item[(iv)] The induced map
  $$ \pi_{s}[Y^\bu,G]_{\mc{M}}\to\pi_{s}[X^\bu,G]_{\mc{M}} $$
is an isomorphism for $0\le s< n$ and an epimorphism for $s=n$ for all $G
\in \mc{G}$.
\end{enumerate}
\end{lemma}
\begin{proof}
Consider the cofiber sequence
  $$ X^\bu\toh{f} Y^\bu\to C^\bu $$
This gives rise to a homotopy fiber sequence (cf. Proposition
\ref{G-cofiber-to-LES})
  $$ \Hom_{\mc{M}}((C^\bu, \ast), (G, \ast)) \to \map(Y^\bu,cG) \to
\map(X^\bu,cG) $$
By Proposition \ref{G-cofiber-to-LES}, (i) is equivalent to
\begin{itemize}
\item[(ii)']  The map $f$ is $0$-connected and $\pi_s \Hom_{\mc{M}}((C^\bu,\ast),
(G, \ast)) = 0$ for all $0 \le s < n$ and $G \in \mc{G}$.
\end{itemize}
which is clearly equivalent to (ii), by another application of Proposition
\ref{G-cofiber-to-LES} to the map $\ast \to C^\bu$. Thus (i) and (ii) are
equivalent.
The equivalence between
(ii) and (iii) follows inductively from comparing the spiral exact sequences of
$C^\bu$ and $\ast$ and applying the Five lemma.
There is a homotopy fiber sequence (cf. Proposition \ref{G-cofiber-to-LES-2})
  $$[(C^\bu, \ast),(G, \ast)] \to [Y^\bu,G] \to [X^\bu,G]. $$
By Proposition \ref{G-cofiber-to-LES-2}, (iv) is equivalent to
\begin{itemize}
\item[(iii)']  The map $f$ is $0$-connected and $\pi_s[(C^\bu,\ast), (G, \ast)] =
0$ for all $0 \le s < n$ and $G \in \mc{G}$,
\end{itemize}
which is clearly equivalent to (iii), by another application of
Proposition \ref{G-cofiber-to-LES-2} to the map $\ast \to C^\bu$.
\end{proof}

\begin{proposition} \label{n-conn-fact}
Let $\mc{M}$ be a proper model category and \mc{G} a class of injective models satisfying \emph{Assumption~\ref{strict-unit}}. Then each cosimplicially $n$-connected map between Reedy cofibrant objects can be written functorially as a 
composition $qj$, where $j$ is a $\mc{G}$-equivalence and $q$ is a quasi-$\mc{G}$-cofree map on objects $(G_s)_{s \geq 0}$ such that $G_s = \ast$ for $s \leq n$. 
\end{proposition}

The proof is a modification of the proof of Proposition~\ref{prop:quasi-cofree-replacements}. 

\begin{proof}
Let $f: X^\bu \to Y^\bu$ be an $n$-connected map between Reedy cofibrant objects.
We will produce a factorization of $f$
  $$ X^\bu\toh{j} Z^\bu\toh{q} Y^\bu $$
such that
\begin{enumerate}
   \item $j$ is a \mc{G}-equivalence, and
   \item $q$ is quasi-\mc{G}-cofree on $(G_s)_{s \geq 0}$ such that $G_s = \ast$ for $s \leq n$. 
\end{enumerate}
For $0\le s\le n$, we set $Z^s=Y^s$ and $j^s :=f^s$. Suppose we have constructed $Z^\bu$ up to cosimplicial degree $s -1 \ge n$. Then we choose functorially a \mc{G}-monic 
cofibration  
  $$ \alpha_s\co X^s\cup_{L^sX^\bu}L^sZ^\bu\to G_s $$
into a fibrant \mc{G}-injective object, and factor functorially the canonical map
  $$ \beta_s\co X^s\cup_{L^sX^\bu}L^sZ^\bu\to (Y^s\times_{M^sY^\bu}M^sZ^\bu)\times G_s $$
into a cofibration  
  $$ l^s(j) \co X^s\cup_{L^sX^\bu}L^sZ^\bu\to Z^s $$
followed by a trivial fibration
  $$ \gamma_s\co Z^s\to (Y^s\times_{M^sY^\bu}M^sZ^\bu)\times G_s. $$
Since $\alpha_s$ is \mc{G}-monic, so are $\beta_s$ and $l^s(j)$. Let
  $$m^s(q) \co Z^s\to Y^s\times_{M^sY^\bu}M^sZ^\bu $$
be the composition ${\rm pr}_{1}\circ\gamma_s$, where ${\rm pr}_{1}$ is the projection onto the first factor. 
The maps $l^s(j)$ and $m^s(q)$ extend the factorization to the cosimplicial degree $s$ and this process yields 
inductively a factorization of $f$.

The resulting map $q\co Z^\bu\to Y^\bu$ is quasi-\mc{G}-cofree by Lemma~\ref{prop:recognize-quasi-G-cofree}. Here we use the right properness of \mc{M}. 
By construction, it is quasi-\mc{G}-cofree on a sequence $\{G_s\}_{s\ge 0}$ with $G_s=\ast$ for $0\le s\le n$. 

It remains to show that $j$ is a \mc{G}-equivalence. By the connectivity assumption and Lemma~\ref{lem:equivalent-formulations-n-connected}, there are isomorphisms 
  $$ \naturalpi{s}{Y^\bu}{G}=\naturalpi{s}{Z^\bu}{G}\cong \naturalpi{s}{X^\bu}{G}, $$ 
for $0\le s< n$, and an epimorphism for $s=n$, because these groups only depend on $Z^m=Y^m$ for $0\le m\le n$.
For every $s>n$, the map $l^s(j)$ is a \mc{G}-monic cofibration. Thus, for all fibrant $G\in\mc{G}$, the induced map
  $$ \Hom(Z^s,G)\to \Hom(X^s\cup_{L^sX^\bu}L^sZ^\bu,G) $$
is surjective. This means that every diagram of the form
\diagr{ \partial\Delta^s \ar[r]\ar[d] & \Hom(Z^\bu,G)=\map(Z^\bu,cG) \ar[d] \\
        \Delta^s \ar[r]\ar@{.>}[ur] & \Hom(X^\bu,G)=\map(X^\bu,cG)}
admits a lift. Since both $X^\bu$ and $Y^\bu$ are Reedy cofibrant, $Z^\bu$ is also Reedy cofibrant and the mapping spaces on the right 
in the diagram above are Kan complexes. Therefore, the map
  $$ \pi_s^{\natural}(j)\co\naturalpi{s}{Z^\bu}{G}\to\naturalpi{s}{X^\bu}{G} $$
is injective for all $s\ge n$. 
For $s>n$, it is also surjective, because of the diagram
\diagr{ \partial\Delta^s \ar[r]\ar[d] & \ast \ar[d]\ar[r] & \map(Z^\bu,cG) \ar[d] \\
\Delta^s \ar[r]\ar@{.>}@/^5pt/[urr]^-<<<<{h_1} & \Delta^s/\partial\Delta^s \ar[r]\ar[ur]_{h_2}& \map(X^\bu,cG) }
and the lift $h_1$ yields a lift $h_2$. This proves the claim that $\pi_s^{\natural}(j)$ is an isomorphism for all $s\ge 0$.
\end{proof}

\section{Cosimplicial unstable coalgebras}\label{sec:cosimplicial-unst-coalg}

In this section, we study the homotopy theory of cosimplicial unstable coalgebras. First we recall some basic definitions and facts about unstable modules and coalgebras over 
the Steenrod algebra. For more details, see Schwartz~\cite{Schwartz:book} and Lannes~\cite{Lannes:fonctiennels}. We also discuss the resolution model structure on $c\CA$ by 
applying the methods of Section \ref{resolution-model-cat}. 

Then we digress and discuss the proof of a related model structure on cosimplicial comodules over a general cosimplicial coalgebra. This model category 
provides the context in which we can construct K\"unneth spectral sequences for the derived cotensor product of comodules. They are dual to the ones constructed by Quillen in~\cite[II.6, Theorem 6]{Quillen:HA}. 

These spectral sequences will be used to prove a homotopy excision theorem for cosimplicial unstable coalgebras (Theorem~\ref{homotopy-excision}). It is the main result needed for Pro\-po\-si\-tion~\ref{diff. constr. in CA}, which is the analogue of the ``difference construction'' in the sense of Blanc-Dwyer-Goerss \cite{BlDG:pi-algebra}. 

\subsection{Preliminaries} \label{unstable-coalgebras}

Let $p$ be a prime or $0$. Let $\mathbb{F} = \mathbb{F}_p$ denote the prime field of characteristic $p$ with the convention that 
$\mathbb{F}_0 =\mathbb{Q}$. Let $\Vec$ denote the category of $\mathbb{Z}$-graded $\mathbb{F}$-vector spaces which are trivial in 
negative degrees. 
The degree of a homogeneous element $x$ of a graded vector space $V$ will be indicated by $|x|$. 
For $q\ge 0$, we write
  $$V\mapsto V[q]$$
for the functor which increases the grading by $q$, i.e.\! $(V[q])_r = V_{r-q}$. 

The graded tensor product of graded vector spaces $V$ and $W$, defined by 
  $$(V \otimes W)_q = \bigoplus_{k + l = q} V_k \otimes W_l,$$
is a symmetric monoidal pairing on $\Vec$. Its symmetry isomorphism involves the Koszul sign rule
  $$ v\otimes w \mapsto (-1)^{|v|\cdot |w|}w\otimes v. $$
Moreover, this pairing is closed in the sense that there exist
internal Hom-objects defined as follows
  $$\underline{\Hom}_{\Vec}(V, W)_q = \Hom_{\Vec}( V[q], W).$$

Let $\mathcal{A}= \mathcal{A}_p$ be  the Steenrod algebra at the prime $p$, with the convention that $\mathcal{A}_0 =\mathbb{Q}$, concentrated in degree 0. Since we are going to work with homology instead of cohomology, we use homological grading on $\mathcal{A}$ and let $\mathcal{A}^i$, which consists of cohomology operations raising the cohomological degree by $i$, sit in degree $-i$.

We recall the definition of an unstable $\mathcal{A}$-module. 
\begin{definition} \label{unstable right module}
An {\it unstable right $\mathcal{A}$-module} consists of
\begin{enumerate}
   \item[(a)] a graded $\mathbb{F}$-vector space $M \in \Vec$,
   \item[(b)] a  homomorphism \[M \otimes \mathcal{A} \to M\] which in addition to the usual module properties, also satisfies the following instability conditions:
\begin{enumerate}
\item[(i)] \hspace{0,9mm}$xSq^n =0$ \ for $p=2$ and $ |x| < 2n$\\
\item[(ii)] \hspace{2,2mm}$xP^n =0$ \ for $p$ odd and $ |x| < 2pn$\\
\item[(iii)] $x\beta P^n =0$ \ for $p$ odd and $ |x| = 2pn + 1$.
\end{enumerate}
\end{enumerate}
We write $\mathcal{U}$ for the category of unstable right $\mathcal{A}$-modules. It is easy to see that $\mathcal{U}$ is an abelian category. 
In the rational case, \mc{U} is just the category $\Vec$ of graded rational vector spaces.
\end{definition}

For $M,N\in \mathcal{U}$, the tensor product $M\otimes N$ of the underlying graded vector spaces can be given an $\mathcal{A}$-module structure via the Cartan formula. 
Moreover, the resulting $\mathcal{A}$-module is an object of $\mathcal{U}$, i.e., it satisfies the instability conditions.

\begin{remark}\label{Vec-U}
The categories $\Vec$ and $\mathcal{U}$ are connected via an adjunction where the left adjoint is the forgetful functor. The right adjoint sends a graded vector space $V$ 
to the maximal unstable submodule of the $\mathcal{A}$-module $\underline{\Hom}_{\Vec}(\mathcal{A}, V)$, with $\mathcal{A}$-action given by $(f \cdot a)(a') = f(a \cdot a')$. 
The objects in the image of this right adjoint are injective. Thus, the abelian category $\mathcal{U}$ has enough injectives. 
\end{remark}

Let $\Coalg$ denote the category of cocommutative, coassociative, counital coalgebras over $\mathbb{F}$. There is a functor 
  $$\F \co {\rm Sets}\to \Coalg,$$ 
which sends a set $S$ to the free $\mathbb{F}$-module $\mathbb{F}(S)$ generated by $S$, with comultiplication determined by 
$\Delta(s)=s \otimes s$ and counit by $\epsilon (s)=1$, for $s \in S$. The functor $\F$ admits a right adjoint $\pi_0$ which sends a coalgebra  
$(C, \Delta_C, \epsilon)$ to the set of \emph{set-like elements}, i.e.\! the elements $c\in C$ which satisfy $\Delta_C (c)=c\otimes c$. An object 
$C\in \Coalg$ is called \emph{set-like} if it is isomorphic to $\mathbb{F}(S)$ for some set $S$. The names
\emph{discrete} \cite{Bou:obstructions} and \emph{group-like} \cite{sweedler:hopf} are also used in the literature for  this concept. Note that there 
is a canonical isomorphism $\pi_0 \mathbb{F}(S) \cong S$, whence it follows that the category of set-like coalgebras is equivalent to the category of sets. 

Let $\grCoalg$ denote the category of cocommutative, coassociative, counital coalgebras in $\Vec$ with the 
respect to the  graded tensor product. We recall the definition of an unstable coalgebra. 

\begin{definition} \label{unstable coalgebra}
An {\it unstable coalgebra} over the mod $p$ Steenrod algebra $\mathcal{A}$ consists of an unstable right $\mathcal{A}$-module 
$C$ whose underlying graded vector space is also a cocommutative coalgebra $(C, \Delta_C, \epsilon)$ in $\grCoalg$ such that the two structures satisfy the following compatibility conditions:
\begin{enumerate}
\item[(a)] The comultiplication $\Delta_C :C\to C\otimes C$ is a morphism in $\mathcal{U}$, 
\item[(b)] The $p$-th root or Verschiebung map  $\xi :C_{pn}\to C_n $ (dual to the Frobenius) satisfies
\begin{align*}
   \xi (x) &= xSq^n\ \  \text{ for } p=2 \text{ and } |x|=2n\\
   \xi (x) &= xP^{n/2} \,\,  \text{ for } p >2, n \text{ even and } |x| = pn.
\end{align*}
\end{enumerate}
\end{definition}  

\begin{remark} \label{set-like-deg-0}
The $0$-th degree of an unstable coalgebra is set-like as an $\mathbb{F}$-coalgebra. This follows from the fact that the $p$-th root map is the identity in degree $0$ by \ref{unstable coalgebra}(b) for $n=0$. 
A proof in the dual context of $p$-Boolean algebras can be found in~\cite[Appendix]{kuhn:generic}. Indeed, by duality, this implies the result for finite dimensional $\F$-coalgebras, which then extends to all $\F$-coalgebras since each $\F$-coalgebra is the filtered colimit of its finite dimensional subcoalgebras. 
\end{remark}

For $\F=\mathbb{Q}$, condition (b) above does not make sense. We have the following definition in this case.

\begin{definition} \label{set-like-deg-0-2} 
An \emph{unstable coalgebra over $\mathbb{Q}$} is a cocommutative and counital graded $\mathbb{Q}$-coalgebra 
$C_* = \{C_n\}_{n \geq 0} \in \grCoalg$ such that $C_0 \in \mathrm{Coalg}_{\mathbb{Q}}$ is set-like.
\end{definition}

Let $\CA$ be the category of unstable coalgebras over the Steenrod algebra $\mathcal{A}$. The category $\CA$ is complete and cocomplete and has a generating set given by the finite (= finite dimensional in finitely many degrees) coalgebras. As these objects are also finitely presentable, the category $\CA$ is finitely presentable. 

The categorical product of a pair of coalgebras is given by the tensor product. It is easy to verify that in $\CA$, too, a finite product is given by the corresponding tensor 
product of unstable modules endowed with the canonical coalgebra structure. 

Note that the terminal object is given by the unit $\terminal$ of the monoidal pairing, i.e.\! the unstable module which is concentrated in 
degree $0$ where it is the field $\mathbb{F}$ with the canonical coalgebra structure. 

A basepoint of an unstable coalgebra $C$ is given by a section $\sigma: \terminal \to C$ of the counit map $\epsilon : C \to \terminal$.
Evidently, a basepoint is specified by a set-like element of $C$ in 
degree $0$.
Let $\CA_* = \terminal / \CA$ denote the category of pointed unstable coalgebras. More generally, for a given unstable coalgebra
$C$, we denote as usual the associated under-category by $C / \CA$. There is an adjunction 
\begin{equation} \label{slice-adjunction}
\CA \rightleftarrows C / \CA
\end{equation}
where the left adjoint is defined by taking the coproduct with $C$, i.e.\! $D \mapsto D \oplus C$ for $D \in \CA$, and the right 
adjoint is the forgetful functor. 

The categories $\mathcal{U}$ and $\CA$ are connected by an adjunction
\begin{equation} \label{U-CA}
\CA \rightleftarrows \mathcal{U}
\end{equation}
where the right adjoint defines the \emph{cofree unstable coalgebra of an unstable $\mathcal{A}$-module} and the left adjoint is 
the forgetful functor. The adjunction in Remark~\ref{Vec-U} can be composed with \eqref{U-CA} and yields an adjunction
\begin{equation} \label{cofree-adjunction} 
J\co \CA \rightleftarrows \Vec\! : G
\end{equation}
where $J$ is the forgetful functor and the right adjoint $G$ defines the \emph{cofree unstable coalgebra} associated with 
a graded vector space. 
In general, an unstable coalgebra is called \emph{injective} if it is a retract of an object of the form 
$G(V)$. More generally, for $C \in \CA$ there is an induced adjunction between the under-categories
\begin{equation}\label{eqn:J:G}
J\co C / \CA \rightleftarrows  J(C) / \Vec\! :G. 
\end{equation}

Next we recall a description of the functor $G$ in terms of Eilenberg-MacLane spaces (Theorem~\ref{explicit-description-G}). 
This provides the fundamental link for the comparison between $\CA$ and the homotopy category of spaces.

\subsection{Homology and $\F$-GEMs} 
Let \kfp{m} denote the Eilenberg-MacLane space representing the singular cohomology functor $X \mapsto H^m(X,\mathbb{F})$. 
We denote the singular homology functor with coefficients in $\mathbb{F}$ by
  $$ H_*\co\mc{S} \to \Vec\, ,\  X\mapsto\bigoplus_{m\ge 0}H_m(X;\mathbb{F})=H_*(X).$$
The graded homology $H_*(X)$ of a simplicial set $X$ is naturally an unstable coalgebra where the comultiplication is induced by the diagonal map $X \to X \times X$ combined with the K\"{u}nneth isomorphism.
Thus we view singular homology as a functor
\begin{equation*} 
H_* \co\mc{S} \to\CA.
\end{equation*}
We note that there are also pointed versions of the above where we consider the unreduced homology of a pointed space
as a pointed unstable coalgebra. 

The classical K\"unneth theorem can be stated in the following way.

\begin{theorem} \label{enhanced Kuenneth}
The homology functor $H_*\co\mc{S}\to\CA$ preserves finite products.
\end{theorem}

An arbitrary product of spaces of type \kfp{m}, for possibly different $m$ but fixed \F, is called a generalized Eilenberg-MacLane 
space or \emph{$\F$-GEM} for brevity.  Given a graded vector space $V$, we denote the associated $\F$-GEM by 
  $$K(V) := \prod_{n\ge 0} K(V_n, n).$$
In this case the homology functor preserves even infinite products. The following fact is stated for $\mathbb{F}_p$  in \cite{Bou:obstructions}.

\begin{theorem}\label{explicit-description-G}
The unstable coalgebra $H_*(K(V))$ is 
the cofree unstable coalgebra associated with $V$, i.e.
\begin{equation*}
   G(V) \cong  H_*(K(V)).
\end{equation*}  
\end{theorem}

\begin{proof} The homology of $K(V_0, 0)$ is cofree because every unstable coalgebra is set-like in degree $0$ and we have an isomorphism $\pi_0 \mathbb{F}(S) \cong S$. Hence, by Theorem \ref{enhanced Kuenneth} it is enough to consider a connected $\F$-GEM. 

By the classical computations of Cartan~\cite{Cartan} 
and Serre~\cite{Serre} the unstable coalgebra $H_*(K(\mathbb{F},n))$ is the cofree unstable coalgebra associated to the graded vector space which is $\mathbb{F}$ in degree $n$ and trivial otherwise. Using Theorem \ref{enhanced Kuenneth} and the fact that $G$ preserves products, this also yields the claim in the case of finite products of $\mathbb{F}$-Eilenberg-MacLane spaces. 

The general case in positive characteristic follows using the properties of the K\"unneth spectral sequence for an infinite product as explained in \cite[4.4]{Bou: homology SS}.

We give a separate argument for the case $\mathbb{F}=\mathbb{Q}$ which relies on the structure theory of connected graded Hopf algebras over the rationals. The homology $C : = H_*(K(V))$ is a connected graded commutative and cocommutative Hopf algebra. By the theorem of Cartier-Milnor-Moore~\cite[Appendix B.4.5]{Quillen:rational}, 
it is the universal enveloping algebra $U\Pr(C)$ on its primitives. By the theorem of Poincare-Birkhoff-Witt~\cite[Appendix B.2.3]{Quillen:rational}, the underlying graded 
coalgebra is isomorphic to the symmetric algebra $S\Pr(C)$ on $\Pr(C)$. This is the graded cofree coalgebra on $\Pr(C)$ by \cite[Appendix B.4.1]{Quillen:rational}. Now it is  
a well-known fact in rational homotopy theory that the primitives in loop space homology can be identified with the homotopy groups. As a consequence, $V = \Pr(C)$ which completes the proof of the assertion.    
\end{proof}

As a consequence we have for any $C \in \CA$ a natural isomorphism 
\begin{equation} \label{catan-serre-calculations}
\Hom_{\CA}\bigl(C,H_*(K(V))\bigr) \cong \Hom_{\Vec} (J(C), V).
\end{equation}
For any $\F$-GEM $G$ and space $X$, we obtain a natural isomorphism 
\begin{equation} \label{representing property of GEMs}
[X,  G] \cong \Hom_{\CA}\bigl(H_*(X),H_*(G)\bigr).
\end{equation}
Combining the last two isomorphisms, we note that the functor $K: \Vec \to \ho{\mc{S}}$ 
is a right adjoint to the homology functor $J \circ H_*\co \ho{\mc{S}} \to \Vec$.
Another direct consequence is the following proposition.

\begin{proposition}\label{cogenerating property of GEMs}
The set of objects $\{H_*(\kfp{m})\, |\, m \geq 0 \}$ is a cogenerating set for the category $\CA$. In particular, for a map $f\co C \to D$ in $\CA$ the following statements are equivalent:
\begin{enumerate}
\item The map $f$ is an isomorphism \emph{(}respectively, monomorphism\emph{)}.
\item The induced map 
   $$f^*\co \Hom_{\CA}\bigl(D,H_*\kfp{m}\bigr) \to \Hom_{\CA}\bigl(C,H_*\kfp{m}\bigr)$$ 
is an isomorphism \emph{(}respectively, epimorphism\emph{)} for every $m \geq 0$.
\end{enumerate}
\end{proposition}

\subsection{Cosimplicial resolutions of unstable coalgebras}
\label{subsec:cos-res-unst-coalg}

We may regard $\CA$ as a proper model category endowed with the discrete model structure. Consider 
\begin{equation}\label{eqn:def-of-mcE}  
   \mc{E}=\{H_*\bigl( \kfp{m} \bigr) \, |\,  m \geq 0 \} 
\end{equation}
as a set of group objects in $\CA$. The group structure is induced by 
the group structure of the Eilenberg-MacLane spaces. By Proposition~\ref{cogenerating property of GEMs}, a map in $\CA$ is \mc{E}-monic if and only if it is injective. Moreover, every unstable coalgebra can be embedded into a product of objects from $\mc{E}$. 
An unstable coalgebra is \mc{E}-injective if and only if it is a retract of a product of objects from $\mc{E}$ (cf. \cite[Lemma 3.10]{Bou:cos}). 

\begin{theorem}\label{thm:E-resolution-model-structure}
There is a proper simplicial model structure $c\CA^{\mc{E}}$ on the category of cosimplicial unstable coalgebras, whose weak equivalences are the $\mc{E}$-equivalences, the cofibrations are the $\mc{E}$-cofibrations, and the fibrations are the $\mc{E}$-fibrations.
\end{theorem}

\begin{proof}
This is an application of Theorem \ref{bousfield} or the dual of \cite[II.4, Theorem 4]{Quillen:HA}. Right properness will be
shown in Corollary \ref{right-proper}.
\end{proof}

\begin{remark} \label{F-equivalences-explicit}
The underlying model structure on $\CA$ is discrete. This has the following consequences:
\begin{enumerate}
   \item
Every object in $c\CA$ is Reedy cofibrant, and therefore \mc{E}-cofibrant.
   \item
For any $s\ge 0$ and $H\in\mc{E}$, the canonical map
  $$ \naturalpi{s}{C^\bu}{H}\to\pi_s[C^\bu,H] $$
is a natural isomorphism. Furthermore, there are natural isomorphisms
  $$ \pi_{s}[C^\bu,H_* \kfp{m}] \cong \bigl(\pi^s (C^\bu)_m\bigr)^{\dual} $$ 
where $m$ on the left denotes the internal grading and $(-)^\dual$ is the dual vector space. This shows that in $c\CA$ we deal essentially with only one invariant, namely, $\pi^*C^\bu$.
\end{enumerate}
\end{remark}

The cohomotopy groups of a cosimplicial (unstable) coalgebra $C^\bu$ carry additional structure which is induced from the comultiplication. More specifically, the graded vector space $\pi^0(C^\bu)$ of a cosimplicial (unstable) coalgebra $C^\bu$ is again an (unstable) coalgebra. Moreover, the higher cohomotopy groups $\pi^s(C^\bu)$, for $s \geq  1$, are (unstable) $\pi^0(C^\bu)$-comodules where the comodule structure is essentially induced from the unique degeneracy map $C^s \to C^0$ and the comultiplication. Furthermore, these are also coabelian as unstable modules. (The definition of coabelian objects will be recall in Subsection~\ref{coabelian-objects}.) See also the Appendix of~\cite{Blanc-Stover} for the dual statement. 

Using the Dold-Kan correspondence, the normalized cochains functor $N : c\Vec \to {\rm Ch}^+(\Vec)$ is an equivalence from $c \Vec$ to the category ${\rm Ch}^+(\Vec)$ of cochain complexes in $\Vec$. The category ${\rm Ch}^+(\Vec)$ admits 
a  model structure where the weak equivalences are the quasi-isomorphisms, the cofibrations are the monomorphisms in positive degrees, and the fibrations are the epimorphisms. The Dold-Kan equivalence can be used to transport a model structure on $c\Vec$ from the model category ${\rm Ch}^+(\Vec)$.  The resulting model structure on $c \Vec$ is an example of a resolution model category as explained in \cite[4.4]{Bou:cos}. Moreover, with respect to this model structure, the adjunction obtained in (\ref{cofree-adjunction}): 
  $$J : c\CA^{\mc{E}} \rightleftarrows c\Vec : G $$
is a Quillen adjunction. One easily derives the following characterizations 
(cf. \cite[4.4]{Bou:cos}).

\begin{proposition} \label{F-equivalences-explicit2}
Let $f: C^\bu \to D^\bu$ be a map in $c \CA$. 
\begin{enumerate}
\item 
$f$ is an $\mc{E}$-equivalence if and only if it induces isomorphisms 
  $$\pi^n(f): \pi^n(C^\bu) \cong \pi^n(D^\bu)$$ 
for all $n \geq 0$, or equivalently, if the induced map between the associated normalized cochain complexes of graded vector spaces 
is a quasi-isomorphism.
 \item 
$f$ is an $\mc{E}$-cofibration if and only if the induced map 
  $$Nf : N C^\bu \to N D^\bu$$ 
between the associated cochain complexes is injective in positive degrees.
\end{enumerate}
\end{proposition}

\begin{remark}
The normalized cochain complex of a cocommutative coalgebra $C^\bu$ in $c\Vec$ carries the structure of a cocommutative coalgebra in ${\rm Ch}^+(\Vec)$ which is induced by the comultiplication of $C^\bu$ and the shuffle map (cf. \cite[II, 6.6-6.7]{Quillen:HA}). 
\end{remark}

\subsection{Cosimplicial comodules over a cosimplicial coalgebra} 
In this subsection, we discuss the homotopy theory of cosimplicial comodules over a cosimplicial $\mathbb{F}$-coalgebra $C^\bu$. The results here are mostly dual to results about simplicial modules over a simplicial ring due to Quillen \cite{Quillen:HA}. 

Let $C^{\bu}$ be a cosimplicial graded $\mathbb{F}$-coalgebra, i.e., a cocommutative, counital coalgebra object in $c\Vec$. Note that $C^\bu$ carries a cosimplicial and an internal grading.

\begin{definition}
A {\it (cosimplicial) $C^\bu$-comodule} $M^\bu$ is a cosimplicial graded $\F$-vector space $M^\bu$ together with maps $M^n \to C^n \otimes_{\F} M^n$, for each 
$n \geq 0$, which are compatible with the cosimplicial structure maps and endow $M^n$ with a (left) $C^n$-comodule structure. We denote the category of $C^\bu$-comodules by $\Comod_{C^\bu}$. 
\end{definition}

The category $\Comod_{C^\bu}$ is an abelian category whose objects are equipped with a cosimplicial and an internal grading. The forgetful functor $U$ admits a right adjoint 
  $$U\!: \Comod_{C^\bu} \rightleftarrows c\Vec\!: C^\bu \otimes -$$ 
which is comonadic: $\Comod_{C^\bu}$ is a category of coalgebras with respect to a comonad on $c \Vec$.

We wish to promote this adjunction to a Quillen adjunction by transferring the model structure from $c \Vec$ to $\Comod_{C^\bu}$. The standard methods for transferring a (cofibrantly generated) model structure along an adjunction do not apply to this case because here the transfer is in the other direction, from right to left along the adjunction. 
Moreover, $\Comod_{C^\bu}$ is not a category of cosimplicial objects in a model category, so the resulting model category is not, strictly speaking, an example of a resolution model category. We give a direct proof for the existence of the induced model structure on $\Comod_{C^{\bu}}$, essentially by dualizing arguments from \cite{Quillen:HA}, see also \cite{Blanc-model str}.

A map $f: X^\bu \to Y^\bu$ in $\Comod_{C^\bu}$ is:
\begin{itemize}
\item a \emph{weak equivalence} (resp. \emph{cofibration}) if $U(f)$ is a weak equivalence (resp. cofibration) in $c\Vec$.
\item a \emph{cofree map} if it is the transfinite (pre-)composition of an inverse diagram $F: \lambda^{\op} \to \Comod_{C^\bu}$ for some ordinal $\lambda$,
$$X^\bu \to \cdots \to F(\alpha + 1) \xrightarrow{p_{\alpha}} F(\alpha) \to \cdots \to F(1) \xrightarrow{p_0} F(0) = Y^\bu$$  
such that for each $\alpha < \lambda$ there is a fibration $q_{\alpha}: V^\bu_{\alpha} \to W^\bu_{\alpha}$ in $c\Vec$ and a pullback square
\[
\xymatrix{
F(\alpha+1) \ar[r] \ar[d]^{p_{\alpha}} & C^\bu \otimes V^\bu_{\alpha} \ar[d]^{C^\bu \otimes q_{\alpha}} \\
F(\alpha) \ar[r] & C^\bu \otimes W^\bu_{\alpha}
}
\]
and for each limit ordinal $\alpha \leq \lambda$, $F(\alpha) \to \varprojlim_{\kappa < \alpha} F(\kappa)$ is an isomorphism. (This is equivalent to saying that 
$f$ is cellular in $(\Comod_{C^\bu})^{\op}$ with respect to the class of maps of the form $C^\bu \otimes q$ where $q$ is a fibration in $c \Vec$.) 
\item a \emph{fibration} if it is a retract of a cofree map. 
\end{itemize}

\begin{theorem} \label{Comod-model-structure}
These classes of weak equivalences, cofibrations and fibrations define a proper simplicial model structure 
on $\Comod_{C^\bu}$.
\end{theorem}
\begin{proof}
It is clear that $\Comod_{C^\bu}$ has colimits - they are lifted from $c \Vec$. Following \cite{Porst}, it can be shown that $\Comod_{C^\bu}$ is locally presentable and therefore also complete. 
It can be checked directly that finite limits are lifted from finite limits in $c \Vec$ due to the exactness of the graded tensor product. The ``2-out-of-3'' and retract axioms are obvious. The factorization axioms are proved in Lemmas~\ref{lem:comod-cof-triv-fib} and~\ref{lem:comod-triv-cof-fib} below. In both factorizations the second map is actually a cofree map. For the lifting axiom in the case of a trivial cofibration
and a fibration, it suffices to consider a lifting problem as follows
\[
\xymatrix{
A^\bu \ar[r] \ar[d]^j & C^\bu \otimes V^\bu \ar[d]^{C^\bu \otimes q} \\
B^\bu \ar[r] & C^\bu \otimes W^\bu 
}
\]
where $j$ is a trivial cofibration, and $q$ a fibration in $c\Vec$. This square admits a lift because the adjoint diagram in the model category
$c \Vec$ admits a lift
\[
\xymatrix{
U(A^\bu) \ar[r] \ar[d]^{U(j)} & V^\bu \ar[d]^{q} \\
U(B^\bu) \ar[r] \ar[ru] & W^\bu 
}
\]
For the other half of the lifting axiom, consider a commutative square
\[
\xymatrix{
X^\bu \ar[r]^f \ar[d]^i & E^\bu \ar[d]^{p} \\
D^\bu \ar[r]^g & Y^\bu 
}
\]
where $i$ is a cofibration and $p$ a trivial fibration. The proof of Lemma \ref{lem:comod-cof-triv-fib} below gives a factorization 
$p = p' j$ into a cofibration followed by a trivial fibration which has the following form
$$E^\bu \stackrel{\simeq}{\hookrightarrow} Y^\bu \oplus (C^\bu \otimes Z^\bu) \stackrel{\simeq}{\twoheadrightarrow} Y^\bu$$
where $Z^\bu$ is weakly trivial in $c \Vec$ and the last map is the projection. By the ``2-out-of-3'' property, $j$ is a weak equivalence.  The commutative square 
\[
\xymatrix{
X^\bu \ar[r]^(.3){j f} \ar[d]^i & Y^\bu \oplus (C^\bu \otimes Z^\bu) \ar[d]^{p'} \\
D^\bu \ar[r]^g & Y^\bu
}
\]
admits a lift $h: D^\bu \to Y^\bu \oplus (C^\bu \otimes Z^\bu)$ which is induced by an extension of $U(X^\bu) \to Z^\bu$ 
to $U(D^\bu)$ in $c \Vec$. Moreover, given the half of the lifting axiom that is already proved, the commutative square 
\[
\xymatrix{
E^\bu \ar@{=}[r] \ar[d]^j &  E^\bu \ar[d]^p \\
Y^\bu \oplus (C^\bu \otimes Z^\bu) \ar[r]^(.7){p'} & Y^\bu
}
\]
admits a lift $h': Y^\bu \oplus (C^\bu \otimes Z^\bu) \to E^\bu$ - which shows that $p$ a retract of $p'$. The composition 
$h'': = h' h: D^\bu \to E^\bu$ is a required lift for the original square. 

Properness is inherited from $c\Vec$. The simplicial structure is the external one - this is defined dually to the one in \cite{Quillen:HA}, see also Subsection~\ref{cosimplicial-objects}. The compatibility
of the simplicial structure with the model structure can easily be deduced from the corresponding compatibility in $c \Vec$. 
\end{proof}

We complete the proof of Theorem \ref{Comod-model-structure} with proofs of the following two lemmas. 

\begin{lemma}\label{lem:comod-cof-triv-fib}
Every map in $\Comod_{C^\bu}$ can be factored functorially into a cofibration followed by a cofree map which is also a weak equivalence.
\end{lemma}

\begin{proof}
Let $f\co X^\bu\to Y^\bu$ be a map in $\Comod_{C^\bu}$.
Choose a functorial factorization in $c\Vec$ of the zero map  $U(X^\bu) \xrightarrow{j} Z^\bu\xrightarrow{q} 0$, into a cofibration 
$j$ followed by a trivial fibration $q$.
Consider $i: X^\bu \to C^\bu \otimes Z^\bu$, the adjoint of $j$. Then $f$ factors as 
\[
X^\bu \xrightarrow{(f,i)} Y^\bu\oplus (C^\bu \otimes Z^\bu) \xrightarrow{p} Y^\bu \] 
where $p$ is the projection. The map $p$ is obviously cofree. It is also a weak equivalence since $C^\bu \otimes Z^\bu$ is weakly trivial. 
\end{proof}

\begin{lemma}\label{lem:comod-triv-cof-fib}
Every map in $\Comod_{C^\bu}$ can be factored functorially into a trivial cofibration followed by a cofree map.
\end{lemma}

\begin{proof}
Let $f\co X^\bu\to Y^\bu $ be a map in $\Comod_{C^\bu}$. 
First, we show that we can factor $f$ functorially into a cofibration which is injective on cohomotopy groups 
followed by a cofree map.
There is a functorial factorization of $U(f)$ in $c\Vec$
$$U(X^\bu) \xrightarrow{(f,i')} U(Y^\bu) \oplus W^\bu \xrightarrow{q} U(Y^\bu)$$
where $i'$ is a cofibration and induces monomorphisms on cohomotopy groups. The map $q$ is the projection.
Then the map $i$ in $\Comod_{C^\bu}$ adjoint to $i'$ gives the desired factorization:
\[
X^\bu \xrightarrow{(f,i)} Z^\bu : = Y^\bu\oplus (C^\bu \otimes W^\bu) \xrightarrow{p} Y^\bu 
\] 
where $p$ denotes the projection away from $C^\bu \otimes W^\bu$. Since $i'$ factors through $U(i)$, it follows 
that $i$ also induces a monomorphism on cohomotopy groups.

Thus it suffices to show that a cofibration $f: X^\bu \to Y^\bu$ which is injective on cohomotopy groups can be factored functorially into a trivial cofibration 
followed by a cofree map. We proceed with an inductive construction. Set $i_{0} = f$. For the inductive step, suppose that we are given a cofibration 
$i_n : X^\bu \to Z^\bu_n$ which is injective on cohomotopy in all degrees and surjective in degrees less than $n$. We claim that there is a functorial factorization 
of $i_n$ into a cofibration $i_{n+1}$ whose connectivity is improved by one, followed by a cofree map. 

Let $W^\bu$ be the cofiber of $i_n$ in $\Comod_{C^\bu}$ (or $c\Vec$) and $q\co Z^\bu_n \to W^\bu$ the canonical map. By assumption, we have that $\pi^*(W^\bu)= 0$
for $* < n$ and $\pi^s(Z^\bu_n) \to \pi^s(W^\bu)$ is surjective with kernel $\pi^s(X^\bu)$ for all $s \geq 0$. Let $0 \xrightarrow{\simeq} P(W^\bu) \twoheadrightarrow U(W^\bu)$ be a functorial 
(trivial cofibration, fibration)-factorization in $c\Vec$. Consider the pullback in $\Comod_{C^\bu}$,
\[
\xymatrix{
& Z^\bu_{n+1} \ar[d]^{p_n} \ar[rr] && C^\bu \otimes P(W^\bu) \ar[d] \\
X^\bu \ar[r]^{i_n} \ar@{-->}[ru]^{i_{n+1}} & Z^\bu_n \ar[r] & W^\bu \ar[r] & C^\bu \otimes W^\bu 
}
\]
There is a canonical factorization of $i_n$ through $Z^\bu_{n+1}$ since the bottom composition is the zero map. We claim that this factorization $i_n = p_n i_{n+1}$ 
has the desired properties. Obviously $p_n$ is cofree, $i_{n+1}$ is injective on cohomotopy in all degrees, and it also is surjective in degrees less 
than $n$. Since $C^\bu \otimes P(W^\bu)$ is weakly trivial and $W^\bu$ 
is $(n-1)$-connected, there is an exact sequence 
$$ 0 \to \pi^n(Z^\bu_{n+1}) \to \pi^n(Z^\bu_n) \to \pi^n(W^\bu)$$
from which it follows that the monomorphism $\pi^n(X^\bu) \to \pi^n(Z^\bu_{n+1})$ is actually an isomorphism. 

By this procedure, we obtain inductively a sequence of maps $(i_{n+1} ,p_n)$, for each $n \geq 0$, such that $i_{n} = p_n i_{n+1}$. The desired 
factorization of $f$ is obtained by first passing to the limit  
\[
X^\bu \xrightarrow{i} \varprojlim_n Z_n^\bu \xrightarrow{p} Y^\bu 
\]
and then applying Lemma \ref{lem:comod-cof-triv-fib} to factor $i = p' j$ into a cofibration $j$ followed by a cofree map $p'$ which is a weak equivalence. 
By construction, the map $p$ is cofree and therefore so is $pp'$. The inverse system of cohomotopy groups induced by the tower of objects 
$\{Z^\bu_n\}$ in $\Comod_{C^\bu}$ satisfies the Mittag-Leffler condition since the maps become isomorphisms eventually. Hence 
$$\pi^* (\varprojlim_n Z_n^\bu ) \xrightarrow{\cong} \varprojlim_n \pi^*(Z^\bu_n),$$
so $i$ is a weak equivalence, and therefore so is $j$. Then the factorization $f = (pp') j$ has the required properties. 
\end{proof}

\subsection{Spectral sequences}
Let $C$ be a cocommutative coalgebra over $\mathbb{F}$, and let $M$ and $N$ be two $C$-comodules. The cotensor product is the equalizer of the diagram
\[ 
M\ \Box_C\ N \to M \otimes N \rightrightarrows M \otimes C\otimes N 
\] 
where the maps on the right are defined using the two comodule structure maps. A $C$-comodule $M$ is called \emph{cofree} if it is of the form  $C \otimes V$ for some 
$\F$-vector space $V$. $M$ is called  {\it injective} if it is a retract of a cofree $C$-comodule. We have a canonical isomorphism 
$(C \otimes V)\ \Box_C\ N \cong V \otimes N$. There are analogous definitions and properties in the context of graded vector spaces 
and graded coalgebras. For background material and the general properties of the cotensor product,  we refer to Milnor-Moore~\cite{Milnor-Moore}, Eilenberg-Moore~\cite{Eilenberg-Moore}, Neisendorfer~\cite{Neis:alg-methods}, and Doi~\cite{Doi}. 

The definition of the cotensor product extends pointwise to the context of $C^\bu$-comodules $B^\bu$ and $A^\bu$ over a cosimplicial cocommutative coalgebra 
$C^\bu$ in $\Vec$. The cotensor product with a $C^\bu$-comodule $B^\bu$,
$$ B^\bu \ \Box_{C^\bu} \ - \co \Comod_{C^\bu} \to c\Vec,$$
is left exact and therefore admits right derived functors denoted $\Cotor^p_{C^\bu}(B^\bu, -)$. Note that the functor  $\Cotor$ comes with a trigrading, i.e., 
$\Cotor^p_{C^\bu}(B^\bu, -)$ is bigraded for all $p \geq 0$. As usual, we will usually suppress the internal grading from the notation.

Here we will be interested in the derived cotensor product and its relation with $\Cotor^*_{C^\bu}(B^{\bu}, -)$ via a coalgebraic 
version of the K\"unneth spectral sequence. We define the derived cotensor product 
$$B^\bu \stackrel{R}{\Box}_{C^\bu} A^\bu : = I^\bu\ \Box_{C^\bu}\ J^\bu$$
where $B^\bu \xrightarrow{\simeq} I^\bu$ and $A^\bu \xrightarrow{\simeq} J^\bu$ are functorial fibrant replacements in $\Comod_{C^\bu}$. It can be shown by 
standard homotopical algebra arguments that the derived cotensor product is invariant under weak equivalences in both variables. Indeed, any two fibrant 
replacements are (cochain) homotopy equivalent and such equivalences are preserved by the cotensor product. 

The following theorem is the coalgebraic version of \cite[II.6, Theorem 6]{Quillen:HA}. We only give a sketch of the proof since the arguments are dual to those 
of \cite{Quillen:HA}.

\begin{theorem}[K\"unneth spectral sequences] \label{Kunneth spectral sequence 2}
Let $C^\bu$ be a cocommutative coalgebra in $c\Vec$ and $A^\bu$, $B^\bu$ two $C^\bu$-comodules.
Then there are first quadrant spectral sequences
\begin{itemize}
\item[(a)] $E_2^{p, q} = \pi^p({\rm \Cotor}_{C^\bu}^q(B^\bu, A^\bu))\ \Longrightarrow\ \pi^{p+q}  (B^\bu \stackrel{R}{\Box}_{C^\bu} A^{\bu})$ \\

\item[(b)] $E_2^{p, q} = {\rm \Cotor}_{\pi^{*}(C^\bu)}^p (\pi^*(B^\bu), \pi^*(A^\bu))^q\ \Longrightarrow\ \pi^{p+q}(B^\bu \stackrel{R}{\Box}_{C^\bu} A^{\bu})$ \\

\item[(c)] $E_2^{p, q} = \pi^p( \pi^q(B^\bu) \stackrel{R}{\Box}_{C^\bu} A^\bu)\ \Longrightarrow\ \pi^{p+q}(B^\bu \stackrel{R}{\Box}_{C^\bu} A^{\bu}) $ \\

\item[(d)] $E_2^{p, q} = \pi^p( B^\bu \stackrel{R}{\Box}_{C^\bu} \pi^q(A^\bu))\ \Longrightarrow\ \pi^{p+q}(B^\bu \stackrel{R}{\Box}_{C^\bu} A^{\bu})$ 
\end{itemize}
which are natural in $A^\bu$, $B^\bu$ and $C^\bu$.

The edge homomorphism of \emph{(a)} is 
  $$\pi^* (B^\bu\ \Box_{ C^\bu}\ A^\bu ) \to   \pi^* (B^\bu\ \stackrel{R}{\Box}_{ C^\bu}\ A^\bu )$$ 
and is induced by the canonical map 
  $$B^\bu\ \Box_{C^\bu}\ A^\bu \to B^\bu\ \stackrel{R}{\Box}_{ C^\bu}\ A^\bu.$$ 
This canonical map is a weak equivalence if $\Cotor^p_{C_n}(B_n ,A_n )=0$ for $p>0$ and $n \geq 0$. 
\end{theorem}

The notation in (b) and (c) requires some explanation. In (b), $\pi^*(B^\bu)$, a graded coalgebra in $\Vec$, becomes a $\pi^*(C^\bu)$-comodule using the $C^\bu$-comodule 
structure of $B^\bu$ and the shuffle map. In (c), $\pi^*(B^\bu)$ denotes the constant object $c \pi^*(B^\bu) \in c\Vec$. The $C^\bu$-comodule structure
and the shuffle map makes this a comodule over the constant cosimplicial coalgebra $\pi^0(C^\bu)$. We then regard it as $C^\bu$-comodule via the 
canonical coalgebra map $c\pi^0(C^\bu) \to C^\bu$.

\begin{proof}(Sketch)
(a) One constructs inductively  an exact sequence in $\Comod_{C^\bu}$:
\begin{equation}\label{C-Comod-resolution}
    B^\bu \to I^\bu_0 \to I^\bu_1\to \ldots
\end{equation}
where $B_0^\bu =B^\bu$, $B^\bu_p \stackrel{\simeq}{\hookrightarrow} I^\bu_p$ is a fibrant replacement of $B^\bu_p$ in $\Comod_{C^\bu}$, and  
$$B_{p+1}^\bu = \coker(B^\bu_{p}\to I^\bu_p).$$ 
Then $I^\bu_p$ is weakly trivial for all $p > 0$, i.e.\! $\pi^s(I^\bu_p) = 0$ for all $s \geq 0$ and $p > 0$. Since $I^\bu_p$ is fibrant, it 
follows that $I^\bu_p\ \Box_{C^\bu}\ A^\bu$ is also weakly trivial, i.e.\! $\pi^* (I^\bu_p\ \Box_{C^\bu}\ A^\bu )=0$ for all $p > 0$. Apply the 
normalized cochains functor with respect to the cosimplicial direction of the complex~(\ref{C-Comod-resolution}) to obtain a double complex 
  $$N_v^\bu (I^\bu_{\ast}\ \Box_{C^{\bu}}\ A^\bu)$$
in which the vertical direction corresponds to the cosimplicial direction. Since $I^n_p$ is an injective $C^n$-comodule for all $n$, by the description 
of the fibrant objects in $\Comod_{C^\bu}$, $I^n_{\ast} $ is an injective resolution of the $C^n$-comodule $B^n$. Hence, 
  $$H^q N_v^p (I^\bu_{\ast}\ \Box_{C^{\bu}}\ A^\bu ) = N^p \Cotor^q_{C^\bu}(B^\bu ,A^\bu).$$ 
So there is a spectral sequence associated to this double complex with $E_2$-term
\begin{equation}\label{double complex}
 E^{p,q}_2 = H^p_v H^q_h N_v^\bu (I^\bu_{\ast}\ \Box_{C^{\bu}}\ A^\bu) = \pi^p( \Cotor^q_{C^\bu}(B^\bu, A^\bu))
\end{equation}  
where the subscripts $v,h$ denote the vertical and horizontal directions respectively.
To see that the spectral sequence converges to $\pi^{p+q}(I_0^\bu \Box_{C^\bu}A^\bu )$, it suffices to note that
$$H^p_h H^q_v N_v^\bu (I^\bu_{\ast}\ \Box_{C^{\bu}}\ A^\bu) = 
                  \left\{ \begin{array}{cl}
                                   0                & q > 0 \\
             \pi^p (I^\bu_{0}\ \Box_{C^{\bu}}\ A^\bu )  & q=0.
             \end{array}\right. $$
Resolving $A^\bu \to J^\bu_{\ast}$  instead of $B^\bu$, we obtain a similar spectral sequence. This new spectral sequence degenerates if we replace $B^\bu$ with $J_0^\bu$. 
This means that the map 
  $$A^\bu\ \Box_{C^\bu}\ I^\bu_0 \xrightarrow{\simeq} J^\bu_0\ \Box_{C^\bu}\ I^\bu_0$$
is a weak equivalence and therefore 
  $$B^\bu \stackrel{R}{\Box}_{C^\bu} A^{\bu} \simeq I_0^\bu\ \Box_{C^\bu}\ A^\bu .$$ 
This completes the construction of the spectral sequence and proves the last claim (cf. \cite[II.6, 6.10]{Quillen:HA}).

(b) Consider the category $c\Comod_{C^\bu}$ of cosimplicial objects over $C^\bu$-comodules endowed with the resolution model structure 
with respect to the class $\mc{G}$ of cofree $C^\bu$-comodules. We choose a $\mc{G}$-fibrant replacement $B^\bu \to I^\ast_\bu$,
which we may in addition assume to be cofree in $c\Comod_{C^\bu}$.
By (a), we have
\begin{equation}\label{Tot}
   \Tot(I^\ast_\bu\ \Box_{C^\bu}\ A^{\bu})={\rm diag}(I^\ast_\bu\ \Box_{C^\bu}\ A^{\bu}) \simeq B^\bu \stackrel{R}{\Box}_{C^\bu} A^{\bu}.
\end{equation}
Consider the spectral sequence with $E_2^{p,q}=\pi^p_h\pi^q_v$ of the associated double complex where the original cosimplicial 
direction (the $\bu$-direction) is viewed as the vertical one. Since the fibrant replacement is cofree, the normalized cochain complex 
of $\pi^\bu(I^\bu_\ast)$, along the horizontal $\ast$-direction, defines an injective $\pi^\bu(C^\bu )$-resolution of $\pi^\bu (B^\bu )$ 
(see the proof of \cite[II.6, Lemma 1]{Quillen:HA}). Then 
  $$E_2^{p,q}= \pi^p(\pi^\bu(I^\ast_\bu)\ \Box_{\pi^\bu(C^\bu)}\ \pi^\bu (A^{\bu}))^q $$ 
by arguments  dual to \cite[II.6, Lemma 1]{Quillen:HA}. As a consequence, the spectral sequence has the required 
$E_2$-term and it converges to the derived cotensor product by (\ref{Tot}).

(c) For a $C^\bu$-comodule $D^\bu$, we consider a truncation functor defined by 
$$T(D^\bu) =D^\bu /c \pi^0 (D^\bu ).$$
We define cone and suspension functors, $\mathrm{Cone}$ and $\Sigma$, using the simplicial structure as follows,
\begin{align*} 
   \mathrm{Cone}(D^\bu) &= D^\bu \otimes \Delta^1 / D^\bu \otimes \Delta^0 \\
   \Sigma(D^\bu) & = D^\bu \otimes \Delta^1 / D^\bu \otimes \partial \Delta^1.
\end{align*}
Note that $\mathrm{Cone}(D^\bu)$ is weakly trivial for each $D^\bu$. There is a natural exact sequence of $C^\bu$-comodules
\[c\pi^0 (D^\bu )\to D^\bu \to \mathrm{Cone}(TD^\bu) \to \Sigma(TD^\bu)\]
For $k\geq 1$, we have natural isomorphisms 
  $$\pi^k (TD^\bu )\cong \pi^k (D^\bu )\ \ \text{ and }\ \ \pi^{k-1} (\Sigma D^\bu )\cong \pi^k (D^\bu ).$$
Consequently, we obtain exact sequences for each $k \geq 0$:
\[c\pi^k (B^\bu )\to \Sigma^{k} (B^\bu) \to \mathrm{Cone}(T\Sigma^{k} B^\bu) \to \Sigma (T \Sigma^{k}B^\bu) \simeq \Sigma^{k+1}(B^\bu).\]
Let $A^\bu \to I^\bu$ be a fibrant replacement. Since the functor $-\ \Box_{C^\bu}\ I^\bu $ is exact, we get long exact sequences for $k\geq 0$:
\begin{align*}
   \hdots\rightarrow \pi^{n-2} (\Sigma^{k+1}B^\bu\ \Box_{C^\bu}\ I^\bu)&\to \pi^n (c\pi^{k}(B^\bu )\ \Box_{C^\bu}\ I^\bu)\to \pi^{n} (\Sigma^k B^\bu\ \Box_{C^\bu}\ I^\bu) \\
                                                          &\to \pi^{n-1} (\Sigma^{k+1}B^\bu\ \Box_{C^\bu}\ I^\bu)\rightarrow\hdots 
\end{align*}
which can be spliced together to obtain an exact couple with 
  $$D_2^{p,q}=\pi^p (\Sigma^q B^\bu\ \Box_{C^\bu }\ I^\bu)\ \ \text{ and }\ \ E^{p,q}_2=\pi^p (\pi^q (B^\bu )\ \Box_{C^\bu}\ I^\bu ).$$
The spectral sequence defined by the exact couple has the required $E_2$-term and converges to $\pi^*(B^\bu \stackrel{R}{\Box}_{C^\bu} I^\bu) \cong 
\pi^*(B^\bu \stackrel{R}{\Box}_{C^\bu} A^\bu)$ by (a) above.
\end{proof} 

These spectral sequences for the derived cotensor product of $C^\bu$-comodules will be useful in the study of homotopy pullbacks in $c \CA^{\mc{E}}$. We note that 
given a pullback of cocommutative graded $\mathbb{F}$-coalgebras,
\[
\xymatrix{
E \ar[r] \ar[d] & A \ar[d]^f \\
B \ar[r]^g & C
}
\]
then the maps $f$ and $g$ define $C$-comodule structures on $A$ and $B$, respectively. Moreover, the cotensor product $B\, \Box_C\, A$ is naturally a cocommutative 
graded coalgebra. This uses the left exactness of the tensor product and the fact that the intersection of two subcoalgebras is again a coalgebra. It is 
easy to verify that the graded coalgebra $B\, \Box_C\, A$ can be identified with the pullback $E$. Thus, we obtain the following corollary.

\begin{corollary}\label{right-proper}
The model category $c \CA^{\mc{E}}$ is right proper. 
\end{corollary}
\begin{proof}
Consider a pullback square in $c \CA^{\mc{E}}$ 
\[
 \xymatrix{
E^{\bullet} \ar[r]^{\tilde{g}} \ar[d] &  A^{\bullet}\ar[d]^{f} \\
B^{\bullet} \ar[r]^{g}& C^{\bullet} 
}
\]
where $g$ is a weak equivalence and $f$ a fibration in $c\CA^{\mc{E}}$. We claim that $\tilde{g}$ is 
also a weak equivalence. By the characterization of fibrations in Proposition~\ref{all-about-cofree-maps2},
it suffices to consider the case where $f$ comes from a pullback square as follows
\[
\xymatrix{
A^{\bullet}\ar[d]_{f} \ar[r] & (cG(V))^{\Delta^s} \ar[d] \\
C^{\bullet} \ar[r]^-{h} & (cG(V))^{\partial \Delta^s} 
}
\]
for some map $h$ and $V \in \Vec$. Regarding $B^\bu$ and $A^\bu$ as $C^\bu$-comodules, we observe that 
$E^\bu = B^\bu\ \Box_{C^\bu}\ A^\bu$. Since $A^n$ is cofree as $C^n$-comodule, then the last statement of 
Theorem \ref{Kunneth spectral sequence 2} implies that $E^\bu \simeq B^\bu \stackrel{R}{\Box}_{C^\bu} A^\bu$ and the result follows. 
\end{proof}

\subsection{Homotopy excision} 
In this subsection, we prove a homotopy excision theorem in $c\CA^{\mc{E}}$ using the results on the homotopy theory of $C^\bu$-comodules. 

The homotopy excision theorem says that a homotopy pullback square in $c\CA^{\mc{E}}$ is also a homotopy pushout in a cosimplicial range depending 
on the connectivities of the maps involved. In a way familiar from classical obstruction theory of spaces, this theorem will be essential in later sections 
in identifying obstructions to extending maps and in describing the moduli spaces of such extensions. 

\begin{definition}\label{def:(co)cartesian}
A commutative square in $c\CA^{\mc{E}}$
  $$\xymatrix{ 
  E^\bu \ar[r] \ar[d] &  A^\bu\ar[d] \\
               B^\bu \ar[r] & C^\bu
               }
  $$
is called
\begin{itemize} 
\item[(a)] {\it homotopy $n$-cocartesian} if the canonical map 
  $$ \hocolim(A^\bu \leftarrow E^\bu \rightarrow B^\bu) \to C^\bu $$
is cosimplicially $n$-connected. 
\item[(b)] {\it homotopy $n$-cartesian} if the canonical map
$$E^\bu \to \holim(A^\bu \rightarrow C^\bu \leftarrow B^\bu)$$
is cosimplicially $n$-connected.   
\end{itemize}
Similarly we define homotopy $n$-(co)cartesian squares in $\Comod_{C^\bu}$ or $c\Vec$ by considering instead the homotopy (co)limits
in the respective categories.
\end{definition}

\begin{remark}
The general notion of ``cosimplicially $n$-connected'' is defined in Definition~\ref{def. n-cocon.}. In this particular case, 
a map $f\co D^\bu \to C^\bu$ in $c\CA^{\mc{E}}$ (or $c \Vec$) is cosimplicially $n$-connected if and only if it induces isomorphisms on 
cohomotopy groups in degrees $< n$ and a monomorphism in degree $n$ (cf. Remark \ref{F-equivalences-explicit} 
and Proposition \ref{F-equivalences-explicit2}). 
\end{remark}

\begin{lemma} \label{pull-n-conn}
Consider a homotopy pullback square in $c\CA^{\mc{E}}$ \emph{(}or $c\Vec$\emph{)}: 
$$
\xymatrix{
E^{\bullet} \ar[r] \ar[d]^g &  A^{\bullet}\ar[d]^{f} \\
B^{\bullet} \ar[r] & C^{\bullet} .
}
$$
If $f$ is $n$-connected, then so is $g$. 
\end{lemma}
\begin{proof}
We may assume that $f$ is an $\mc{E}$-cofree map and $E^\bu = B^\bu \Box_{C^\bu} A^\bu$. Then the result follows 
easily from Theorem \ref{Kunneth spectral sequence 2}(b). The case of $c\Vec$ is well-known and can be shown using
standard results from homological algebra. 
\end{proof}

\begin{remark}\label{rem:coskeletal-connection}
Let $p\co E^\bu \to B^\bu$ be an $\mc{E}$-cofree map in $c\CA^{\mc{E}}$. Then we have a (homotopy) pullback square 
$$ 
\xymatrix{ 
\cosk_s(p) \ar[r]\ar[d]_{\gamma_{s-1}(p)} & \hom(\Delta^s,cG_s) \ar[d] \\
\cosk_{s-1}(p) \ar[r]   & \hom(\partial\Delta^s,cG_s) .
}
$$
By inspection, the right vertical map is $(s-1)$-connected. Hence, the left vertical map is $(s-1)$-connected. 
\end{remark}

There is a forgetful functor $V \co c\CA /C^\bu \to \Comod_{C^\bu}$
which sends $f\co D^\bu \to C^\bu$ to $D^\bu$ regarded as a $C^\bu$-comodule with structure map $(f, \mathrm{id})\co D^\bu \to C^\bu \otimes D^\bu$. The functor $V$ admits a right adjoint, but we will not need this fact here. Note that homotopy pullbacks in $\Comod_{C^\bu}$ define also homotopy pushouts, and that $V$ preserves homotopy pushouts. Using these facts, we will first express the homotopy excision property in $c\CA^{\mc{E}}$ as a comparison between taking homotopy pullbacks in $c \CA^{\mc{E}}$ and in $\Comod_{C^\bu}$, respectively.  

\begin{theorem}[Homotopy excision for cosimplicial unstable coalgebras] \label{homotopy-excision}
Let $$
\xymatrix{
E^{\bullet} \ar[r] \ar[d] &  A^{\bullet}\ar[d]^{f} \\
B^{\bullet} \ar[r]^{g}& C^{\bullet}
}
$$
be a homotopy pullback square in $c\CA^{\mc{E}}$ where $f$ is $m$-connected and $g$ is $n$-connected.
Then
\begin{itemize}
 \item[(a)] the square is homotopy $(m+n+1)$-cartesian in $\Comod_{C^\bu}$,
 \item[(b)] the square is homotopy $(m+n)$-cocartesian in $c \CA^{\mc{E}}$. 
\end{itemize}
\end{theorem}
\begin{proof}
(a) The idea of the proof is to compare $E^\bu$ with the homotopy pullback in $\Comod_{C^\bu}$. By Proposition~\ref{n-conn-fact},
we may assume that $f$ and $g$ are $\mc{E}$-cofree maps of the form
\begin{align*} 
   f&\co A^\bu\to\hdots\to\cosk_{m+1}(f)\xrightarrow{f_{m}}\cosk_m(f)=C^\bu \\
   g&\co B^\bu\to\hdots\to\cosk_{n+1}(g)\xrightarrow{g_{n}}\cosk_n(g)=C^\bu 
\end{align*}
and $E^\bu$ is the strict pullback of $f$ and $g$. Therefore, we have that $E^\bu = B^\bu\ \Box_{C^\bu}\ A^\bu$ as $C^\bu$-comodules. 
Let $\widehat{E}^\bu := A^\bu \oplus_{C^\bu} B^\bu$ denote the pullback of $f$ and $g$ in $C^\bu$-comodules. Then it suffices to show that the canonical map 
$$c\co E^\bu \to \widehat{E}^\bu$$ 
is $(m+n+1)$-connected, i.e., the induced map $ \pi^sE^\bu\to\pi^s\wh{E}^\bu $ is an isomorphism for $s\le m+n$ and injective for 
$s=m+n+1$. 

We first consider the case where each of the maps $f$ and $g$ is defined by a single coattachment. This means that we have pullback squares in 
$c\CA$ as follows
\begin{equation*} \label{special-case-excision}
\xymatrix{
E^{\bullet} \ar[r] \ar[d] &  A^{\bullet}\ar[d]^{f} \ar[r] & cG(V)^{\Delta^{k}} \ar[d] \\
B^{\bullet} \ar[r]^{g} \ar[d] & C^{\bullet} \ar[r] \ar[d] & cG(V)^{\partial\Delta^{k}} \\
cG(W)^{\Delta^{l}} \ar[r] & cG(W)^{\partial\Delta^{l}}
}
\end{equation*}
where $k \geq m+1$ and $l \geq n+1$. Note that in each cosimplicial degree $s \geq 0$, there are isomorphisms of graded vector spaces
  $$A^s\cong C^s \otimes \bigl(\Omega^{k}cG(V)\bigr)^s$$ 
and
  $$B^s\cong C^s \otimes \bigl(\Omega^{l}cG(W)\bigr)^s $$ 
such that both $f$ and $g$ are isomorphic to the projections onto $C^s$. Thus, for all $s\ge 0$, there are isomorphisms of graded vector spaces 
\begin{align*}
   E^s   & \cong C^s\ \otimes\ \bigl(\Omega^{k}cG(V)\bigr)^s\ \otimes\ \bigr(\Omega^{l}cG(W)\bigr)^s \\
         &  \cong \left(C^s \otimes \bigl(\Omega^{k}G(V)\bigr)^s\right) \Box_{C^s} \left(C^s \otimes \bigl(\Omega^{l}G(W)\bigr)^s\right).
\end{align*}
Let $\ol{\Omega^{k}}cG(V)= \ker\left[\Omega^{k}cG(V) \to c\F \right]$. Then we have isomorphisms of graded vector spaces, for all $s \geq 0$,  
\begin{align*} \label{identification-2}
   \widehat{E}^s \cong C^s\ \oplus\  \left(C^s \otimes \bigl(\overline{\Omega^{k}}cG(V)\bigr)^s\right)\ \oplus\ \left(C^s \otimes \bigl(\overline{\Omega^{l}}cG(W)\bigr)^s\right).
\end{align*}
Using these identifications of $E^s$ and $\widehat{E}^s$, it is easy to see that the map $c$ is given by the canonical projection. As a consequence, we can identify 
the kernel of $c$ as follows
\begin{align*} 
   (\ker c)^s & \cong C^s\ \otimes\ \bigl(\overline{\Omega^{k}}G(V)\bigr)^s\ \otimes\ \bigl(\overline{\Omega^{l}}G(W)\bigr)^s \\
         &\cong \left(C^s \otimes \bigl(\overline{\Omega^{k}}G(V)\bigr)^s\right) \Box_{C^s} \left(C^s \otimes \bigl(\overline{\Omega^{l}}G(W)\bigr)^s\right).
\end{align*}
We estimate the connectivity of the cotensor product of these two $C^\bu$-comodules, 
$$\overline{A}^\bu :=\ker\left[f\co A^\bu \to C^\bu\right] = C^\bu \otimes \ol{\Omega^{k}}cG(V)$$  
$$\overline{B}^\bu :=\ker\left[g\co B^\bu \to C^\bu\right] = C^\bu \otimes \ol{\Omega^{l}}cG(W),$$
using the K\"unneth spectral sequence from Theorem~\ref{Kunneth spectral sequence 2}(b). We have
 \[E_2^{p,q}=\Cotor^p_{\pi^{\ast}(C^{\bu})}(\pi^{\ast}(\overline{B}^\bu),\pi^{\ast}(\overline{A}^\bu))^q\]
and the two arguments vanish in degrees less than or equal to $m$ and $n$ respectively. It follows that $E_2^{p,q}=0$ for $p+q\leq m+n+1$, and therefore 
$$\pi^*(\ker(c))=0$$
for $*\leq m+n+1.$ Since $\ker(c) \to E^\bu \xrightarrow{c} \widehat{E}^\bu$ defines a short exact sequence in $c\Vec$, we conclude that the map $c\co E^\bu\to\wh{E}^\bu$ is $(m+n+1)$-connected. This proves the claim in the special case of single coattachments. 

The general case follows inductively. For each new coattachment associated with $f$ or with $g$, the same argument applies to show that we obtain a new 
homotopy $(m+n+1)$-cartesian square. It is easy to check that homotopy $k$-cartesian squares are closed under composition which then completes the inductive
step. Also, Lemma \ref{pull-n-conn} shows that it suffices to consider only the coattachments up to a finite coskeletal degree. 

For Part (b), we note that homotopy colimits in $c \CA$ can be computed in the underlying category of cosimplicial graded vector spaces (or in $\Comod_{C^\bu}$) since the forgetful 
functor from $c\CA$ to $c\Vec$ (or $\Comod_{C^\bu}$) is left Quillen and detects colimits. The statement now follows from (a) by comparing the long exact sequences of cohomotopy groups.
\end{proof}

\section{Andr\'{e}-Quillen cohomology} \label{sec:AQ}

In this section, we discuss Andr\'{e}-Quillen cohomology in the context of unstable coalgebras. The basic definitions and constructions are recalled in Subsections \ref{coabelian-objects} 
and \ref{AQ-cohomology}. Then we introduce the objects of type $K_C(M,n)$ which play the role of twisted Eilenberg-MacLane spaces in this context, and also represent Andr\'{e}-Quillen cohomology. 

In Subsection \ref{Objects of type K}, we identify the homotopy types of the moduli spaces of $K$-objects. These results, together with the homotopy excision theorem 
from Section~\ref{sec:cosimplicial-unst-coalg}, are used in Subsection~\ref{postnikov_decomp_unst_coalg} (and Subsection~\ref{an_extension}) to analyze the Postnikov-type 
skeletal filtration of a cosimplicial unstable coalgebra. 

\subsection{Coabelian objects} \label{coabelian-objects}

We denote by $\V$ the full subcategory of all unstable right $\mathcal{A}$-modules $M$ such that 
\begin{align*}
  x P^n &= 0  \text{ for } |x|\le 2pn \text{ and } p \text{ odd, or } \\
 x Sq^n &= 0 \text{ for } |x|\le 2n \text{ and } p=2. 
\end{align*}
This subcategory is equivalent to the category of coabelian cogroup objects in $\CA_*$, i.e.\! the unstable coalgebras $C$ for which the diagonal $\mathcal{U}$-homomorphism $C \to C \oplus C$ is a morphism in $\CA_*$. This property implies that the coalgebra structure must be trivial. The relations above are forced by Definition \ref{unstable coalgebra}(b) so that the trivial coalgebra structure on $M \oplus\, \terminal$ is \emph{unstable}. These instability conditions imply, in particular, that $M$ is trivial in non-positive degrees. The category $\V$ is an abelian subcategory of $\mathcal{U}$ which has enough injectives (cf. \cite[8.5]{Bou:obstructions}). 

Given an unstable coalgebra $C$, a $C$-comodule is an unstable module $M \in \mathcal{U}$ equipped with a map in $\mathcal{U}$
  $$\Delta_M : M\to C \otimes M , $$ 
which satisfies the obvious comodule properties. The category of $C$-comodules is denoted by  $\mathcal{U}C$. 

Let $\V C$ be the full subcategory of the $C$-comodules which are in $\V$. This subcategory $\V C$ is equivalent to the category of coabelian cogroup objects in the category $C / \CA$ and, consequently, there is an adjunction as follows
   $$\iota_C: \V C \rightleftarrows C / \CA : Ab_C.$$ 
The left adjoint is defined by $\iota_C (M) =C \oplus M$, whose comultiplication is specified by the comultiplication of $C$ and by 
  $$\Delta_M + \tau \Delta_M: M \to (C \otimes M) \oplus (M \otimes C) , $$
where $\tau$ denotes the twist map. This combined comultiplication can also be expressed in terms of \emph{derivations}: for 
$M \in \V C$ and $D \in C/ \CA$, there is a natural isomorphism
\begin{equation} \label{derivations}
{\rm Der}_{\CA} (M, D) \cong \Hom_{C/ \CA}(\iota_C(M), D),
\end{equation}
where the $D$-comodule structure on $M$ is defined by the given map $f: C \to D$. 

The right adjoint $Ab_C$ is called the coabelianization functor and carries an object $f:C \to D$ to the kernel of the following map in $\mathcal{U}C$:
\[ {\rm Id} \otimes \Delta_D - ( {\rm Id} \otimes f\otimes {\rm Id} )(\Delta_C \otimes {\rm Id})- ({\rm Id} \otimes \tau)( {\rm Id} 
\otimes f \otimes {\rm Id})(\Delta_C \otimes {\rm Id}): C \otimes D \to C \otimes D \otimes D. \]

This definition generalizes the notion of primitive elements to the relative setting. We recall the definition of primitive elements. 
 
\begin{definition} Let $C$ be an object in $\grCoalg$ with a basepoint $\F \to C$. An element $x \in C$ is called \emph{primitive} if 
  $$\Delta_C(x) = 1 \otimes x + x \otimes 1 , $$
where $1$ denotes the image of the basepoint of $C$ at $1 \in \F$. The comodule $\Pr(C)$ is the sub-comodule of primitive elements
  $$\Pr(C) = \{x \in C \ | \ \Delta_C(x)= 1 \otimes x + x \otimes 1 \}.$$ 
\end{definition}

In the case where $C = \terminal$, the coabelianization functor $Ab_{\terminal}: \CA_* \to \V$ takes a pointed unstable 
coalgebra $D$ to the sub-comodule of primitives $\Pr(D)$ with the induced unstable $\mathcal{A}$-module structure.

An important example of coabelian objects is given in the following
\begin{definition}\label{def:internal-shift} 
For an object $M$ in $\mc{U}C$, we define its {\it internal shift} by
  $$ M[1]=M\otimes\widetilde{H}_*(S^1). $$
In other words, we have $M[1]_0=0$ and, for $n\ge 1$, $M[1]_n=M_{n-1}$ with shifted Steenrod algebra action. The $C$-coaction is defined as follows. For $m\in M_{n-1}$, we write $\ul{m}\in M[1]_n$. If $\Delta(m)=\sum_{i}c_i\otimes m_i$, then 
  $$ M[1]\to C\otimes M[1]\ , \ \ul{m}\mapsto \sum_{i}c_i\otimes \ul{m_i}\ .$$
This yields a functor from $\mc{U}C$ to $\mc{V}C$ because the stronger instability condition is automatically satisfied after 
shifting. Inductively, we define $M[n]=(M[n-1])[1]$. 
\end{definition}

Since an unstable coalgebra $C$ is a comodule over itself, the internal shift can be viewed as a functor from $\CA$ to $\mc{V}$. One obviously has a natural isomorphism 
$H_*(\Sigma X_+)\cong H_*(X)[1]$.

\subsection{Andr\'e-Quillen cohomology} \label{AQ-cohomology}
Following the homotopical approach to homology initiated by Quillen \cite{Quillen:HA}, the Andr\'{e}-Quillen cohomology groups 
of an unstable coalgebra are defined as the right derived functors of the coabelianization functor. For an unstable coalgebra $C$, we have an adjunction 
$$\iota_C\co \V C \rightleftarrows C/\CA:\! Ab_C, $$
which is given essentially by the inclusion of the full subcategory of coabelian cogroup objects in $C/ \CA$. Passing to the 
respective categories of cosimplicial objects, this can be extended to a Quillen adjunction: 
$$\iota_C\co c \V C \rightleftarrows c(C/ \CA) :\! Ab_C.$$
The category on the right is the under-category $c C / c\CA$ and is endowed with the model structure induced by $c \CA^{\mc{E}}$. 
The model category on the left is the standard pointed model category of cosimplicial objects in an abelian category (cf. \cite[4.4]{Bou:cos}).
As remarked in \cite{Quillen:HA}, for a given object $D \in C/ \CA$, we may regard the derived coabelianization 
$$(\mathbb{R}Ab_C)(D)$$
as the cohomology of $D$. Then, given an object $M \in \V C$, the $0$-th Andr\'{e}-Quillen cohomology group of $D$ with 
coefficients in $M$ is defined to be
$$\AQ^0_{C}(D; M) = [ cM, (\mathbb{R}Ab_C)(cD)].$$
In general, the $n$-th Andr\'{e}-Quillen cohomology group is 
\begin{equation} \label{AQ-groups} 
\AQ^n_{C}(D;M) = [\Omega^n (cM), (\mathbb{R}Ab_C)(cD)] \cong \pi^n \Hom_{\V C}(M, (\mathbb{R} Ab_C)(cD)) ,
\end{equation} 
where $\Omega$ denotes the derived loop functor in the pointed model category $c \V C$. These cohomology groups are 
also the non-additive right derived functors of the functor
$$\Hom_{\V C} (M, Ab_C(-))\co C / \CA \to {\rm Vec}.$$
We recall \eqref{derivations} that this is the same as the functor of derivations 
$${\rm Der}_{\CA}(M, D)\co C / \CA \to {\rm Vec}.$$
For an object $D \in C/ \CA$ and fibrant replacement $D \to H^{\bullet}$ in $c(C/ \CA)$, the $n$-th derived functor is defined to be 
(cf. \cite[5.5]{Bou:cos})
$${\Big(}R^n \Hom_{\V C} \bigl(M, Ab_C(-)\bigr){\Big)}(D) = \pi^n \Hom_{\V C}\bigl(M, Ab_C(H^\bullet)\bigr).$$
It follows from standard homotopical algebra arguments that this is independent of the choice of fibrant replacement up 
to natural isomorphism. 
A particularly convenient choice of fibrant replacement comes from the cosimplicial resolution 
defined by the monad on $C / \CA$ associated with the adjunction~(\ref{eqn:J:G})
$$J\co C / \CA \rightleftarrows J(C) / \Vec:\! G.$$
Therefore, Andr\'e-Quillen cohomology can also be regarded as monadic (or triple) cohomology, where $M$ is the choice of coabelian coefficients. 

\subsection{Objects of type $K_C(M,n)$} \label{Objects of type K} The adjoint of \eqref{AQ-groups} gives a natural isomorphism
  $$\AQ^n_C(D ; M) \cong [(\mathbb{L} \iota_C) (\Omega^n (cM)), cD],$$
where the right hand side denotes morphisms in $\ho{c(C/\CA)^{\mc{E}}}$.
This can be regarded as a representability theorem for Andr\'{e}-Quillen cohomology. Since the canonical natural transformation in $\ho{c \V C}$:
  $$\Sigma \Omega \to {\rm Id}$$
is a natural isomorphism, there are also natural isomorphisms 
\begin{equation} \label{AQ-groups 2}
\AQ^n_C(D ; M) \cong [(\mathbb{L} \iota_C) (\Sigma^k \Omega^{n + k} (cM)), cD].
\end{equation}

The representing object $(\mathbb{L} \iota_C)(\Omega^n (cM))$ is described explicitly as follows. First, the normalized 
cochain complex associated to $\Omega^n (cM)$ is quasi-isomorphic to the complex which has $M$ in degree $n$ and is 
trivial everywhere else. By the Dold-Kan correspondence, this means that $\Omega^n(cM)$ is given up to weak equivalence 
by a sum of copies of $M$ in each cosimplicial degree. 
Since every object in $c \V C$ is cofibrant, the functor $\iota_C$ preserves the weak equivalences, and so we may choose 
$\iota_C(\Omega^n(cM))$ for the value of the derived functor at $\Omega^n(cM)$. Explicitly, this consists of the semi-direct 
product of copies of $M$ with the coalgebra $C$ in each cosimplicial degree. The cohomotopy groups of the resulting object are
$$ \pi^s {\Big(}\iota_C\bigl(\Omega^n(cM)\bigr){\Big)} \cong\left\{
                   \begin{array}{cl}
                             M     & s=n \\
                             C     & s=0 \\
                             0     & \hbox{otherwise }
                   \end{array} 
                                        \right. $$
for all $n>0$. For $n=0$, clearly $\pi^0 (\iota_C(cM)))= c(\iota_C(M))$ with vanishing higher cohomotopy groups.
These are isomorphisms of unstable coalgebras and $C$-comodules respectively. Since $\V C$ is pointed, we obtain a 
retraction map $\iota_C(\Omega^n (cM)) \to cC$ in the homotopy category. It turns out that objects satisfying all 
these properties are homotopically unique (Proposition \ref{Moduli of K(C,k)}). First we formulate precisely the definition 
of this type of objects. 

\begin{definition}
Let $C$ be an unstable coalgebra. An object $D^\bu \in c \CA$ is said to be of {\it type $K(C,0)$} if it weakly equivalent to
the constant cosimplicial object $cC$. 
\end{definition}

\begin{definition} \label{K_{B}(M,n)} 
Let $C \in \CA $,  $M \in \V C$, and $n \geq 1$. An object $D^\bullet$ of $c \CA$ is said to be 
of {\it type $K_{C}(M,n)$} if the following are satisfied:
\begin{itemize} 
\item[(a)] there are isomorphisms of coalgebras and $C$-comodules respectively:
$$ \pi^sD^\bullet\cong\left\{
                   \begin{array}{cl}
                             C    & s=0 \\
                             M    & s=n \\
                             0    & \hbox{otherwise,}
                   \end{array} 
                                        \right. $$
\item[(b)] there is a map $D^\bu \to D_0^\bu$ to an object of type $K(C,0)$ such that the composite $\sk_1(D^\bu) \to D^\bu \to D_0^\bu$ is a weak equivalence. (Only the existence and not a choice of such a map is required here.)
\end{itemize} 
Occasionally, we will also use the notation {\it $K_C(M,0)$} to denote an object of type $K(\iota_C(M), 0)$.
\end{definition}

It is clear from the definition that an object of type $K_C(M,n)$ can be roughly regarded, up to weak 
equivalence, both as an object under and over $cC$, but the choices involved will be non-canonical. 
It will often be necessary to include such a choice in the structure, and view the resulting object as an 
object of a slice category instead. 

\begin{definition} \label{structured-K-obj}
Let $C \in \CA $ and  $M \in \V C$. 
\begin{itemize}
 \item[(a)] A \emph{pointed} object of type $K_C(M,n)$, $n \geq 1$, is a pair $(D^\bu, i)$ where $D^\bu$ is an object of 
 type $K_C(M,n)$ and $i\co cC \to D^\bu$ is a map which induces an isomorphism on $\pi^0$-groups. 
 \item[(b)] A \emph{structured} object of type $K_C(M,n)$, $n \geq 1$, is a pair $(D^\bu, \eta)$ where
$D^\bu$ is an object of type $K_C(M,n)$ and $\eta\co D^\bu \to D_0^\bu$ is a map as in (b) above. 
\item[(c)] A \emph{pointed} object of type $K_C(M,0)$ is a pair $(D^\bu, i)$ where $D^\bu$ is an object of 
 type $K(\iota_C(M),0)$ and $i\co cC \to D^\bu$ is a map such that $\pi^0(i)$ can be identified up to isomorphism  
 with the standard inclusion $C \to \iota_C(M)$.
 \item[(d)] A \emph{structured object} of type $K_C(M,0)$ is a pair $(D^\bu, \eta)$ where $D^\bu$ is an object 
of type $K(\iota_C(M), 0)$ and $\eta\co D^\bu \to D^\bu_0$ is a map to an object of type $K(C,0)$ such that 
$\pi^0(\eta)$ can be identified up to isomorphism with the canonical projection $\iota_C(M) \to C$. 
\end{itemize}
\end{definition}

\begin{remark} \label{shifting-K-obj}
Let $(D^\bullet, \eta)$ be a structured object of type $K_C(M,n)$ and $n > 0$. Then the homotopy pushout of the maps 
$$D_0^\bu \stackrel{\eta}{\leftarrow} D^\bu \stackrel{\eta}{\rightarrow} D_0^\bu$$
is a structured object of type $K_C(M, n-1)$.  
\end{remark}

\begin{construction} \label{alternative-direct-construction}
As explained above, the objects $\iota_C(\Omega^n (cM))$ are pointed and structured objects of type $K_C(M,n)$.
We sketch a different construction of objects of type $K_C(M,n)$ based on the homotopy excision theorem.
The construction is by induction. For $n=0$, the constant coalgebra $cC$ is an object of type $K(C, 0)$ by definition.
Assume that an 
object $D^\bu$ of type $K_C(M,n-1)$ has been constructed for $n > 1$. Consider the homotopy pullback square
\[
 \xymatrix{
 E^\bu \ar[d] \ar[r] & \sk_1(D^\bu) \ar[d] \\
 \sk_1(D^\bu) \ar[r] & D^\bu ,
 }
\]
where $\sk_1(D^\bu)$ is an object of type $K(C,0)$. Then, by homotopy excision (Theorem~\ref{homotopy-excision}), $\sk_{n+1}(E^\bu)$ is an object of 
type $K_C(M, n)$ (see also Proposition \ref{diff. constr. in CA} below). We can always regard this as structured by making a \emph{choice} 
among one of the maps to $\sk_1(D^\bu)$. The case $n=1$ is somewhat special. In this case, we consider the following homotopy pullback square
\[
\xymatrix{
 E^\bu \ar[r] \ar[d] & cC \ar[d] \\
 cC \ar[r] & c(\iota_C(M))
 }
\]
and an application of homotopy excision shows that $\sk_2(E^\bu)$ is of type $K_C(M,1)$. To be more precise, one applies Theorem 
\ref{refined diff. constr. in CA} and Corollary \ref{difference in dim 1}  below and uses the fact that $C\ \Box_{\iota_C(M)}\ C=C$. 
\end{construction} 

Next we show the homotopical uniqueness of $K$-objects, and moreover determine the homotopy type of the moduli space of structured $K$-objects. 
For $C$ an unstable coalgebra, $M$ a $C$-comodule in $\V C$, and $n \geq 0$, we denote this moduli space by 
  $$\mc{M}\bigl(K_C(M,n) \twoheadrightarrow K(C,0)\bigr).$$
This is an example of a moduli space for maps where the arrow $\twoheadrightarrow$ indicates that these maps satisfy an 
additional property. It is the classifying space of the category $\mc{W}\bigl(K_C(M,n) \twoheadrightarrow K(C,0)\bigr)$ whose objects are structured objects of type $K_C(M,n)$:
  $$D^\bu \xrightarrow{\eta} D^\bu_0$$
and morphisms are square diagrams as follows
\[
 \xymatrix{
 D^{\bullet} \ar[r]^{\eta} \ar[d]^-{\simeq} & D_0^{\bullet} \ar[d]^{\simeq} \\
 E^{\bullet} \ar[r]^{\eta'} & E_0^{\bullet} ,
}
\]
where the vertical arrows are weak equivalences in $c \CA^{\mc{E}}$. Note that we have also allowed the case $n=0$. We will also use the notation 
$\mc{W}(K(C,0))$ and $\mc{M}(K(C,0))$ to denote respectively the category and moduli space for objects weakly equivalent to $cC$.

An automorphism of $\iota_C(M)$ is an isomorphism $\phi\co \iota_C(M) \to \iota_C(M)$ which is compatible with an isomorphism $\phi_0$ 
of $C$ along the projection map:
\[
 \xymatrix{
 \iota_C(M) \ar[r]^{\phi}_{\cong} \ar[d] & \iota_C(M) \ar[d] \\
 C \ar[r]^{\phi_0}_{\cong} & C .
 }
\]
This group of automorphisms will be denoted by $\Aut_C(M)$. 

\begin{proposition} \label{Moduli of K(C,k)}
Let $C \in \CA$, $M \in \V C$, and $n \ge 0$. Then:
\begin{itemize} 
\item[(a)] There is a weak equivalence $ \mc{M}(K(C,0))\simeq B\Aut(C).$
\item[(b)] There is a weak equivalence $\mc{M}\bigl(K_C(M, 0) \twoheadrightarrow K(C,0)\bigr) \simeq B \Aut_C(M).$
\item[(c)] There is a weak equivalence
   $$\mc{M}\bigl(K_C(M, n +1) \twoheadrightarrow K(C,0)\bigr) \simeq \mc{M}\bigl(K_C(M, n) \twoheadrightarrow K(C,0)\bigr).$$
\end{itemize}
\end{proposition}
\begin{proof}
(a) Let $\mc{W}(C)$ denote the subcategory of unstable coalgebras isomorphic to $C$ and isomorphisms between them. This is 
a connected groupoid, so its classifying space is weakly equivalent to $B \Aut(C)$. There is a pair of functors 
  $$c\co \mc{W}(C) \rightleftarrows \mc{W}(K(C,0)):\! \pi^0 $$
and natural weak equivalences 
  $${\rm Id} \longrightarrow \pi^0 \circ c\ \ \text{ and }\ \ c \circ \pi^0 \longrightarrow {\rm Id} , $$
which show that this (adjoint) pair of functors induces an inverse pair of homotopy equivalences $\mc{M}(K(C,0)) \simeq B(\mc{W}(C))$, as required.

(b) Let $\CA^{\to}$ denote the category of morphisms and commutative squares in $\CA$. Let $\mc{W}(C; M)$ be the subcategory of $\CA^{\to}$ whose objects are 
isomorphic to the canonical projection
$$\iota_C(M) \to C$$
and whose morphisms are isomorphisms between such objects. Since this is a connected groupoid, we have 
$$B(\mc{W}(C;M)) \simeq B \Aut_C(M).$$
There is a pair of functors 
$$c\co \mc{W}(C;M) \rightleftarrows \mc{W}\bigl(K_C(M,0) \twoheadrightarrow K(C,0)\bigr):\! \pi^0$$
and the two composites are naturally weakly equivalent to the respective identity functors as before. The required 
result follows similarly. 

(c) There is a pair of functors 
$$F\co \mc{W}\bigl(K_C(M, n+1) \twoheadrightarrow K(C,0)\bigr) \rightleftarrows \mc{W}\bigl(K_C(M, n) \twoheadrightarrow K(C,0)\bigr):\! G$$
defined as follows:
\begin{itemize}
   \item[(i)] $F$ sends an object $\eta\co D^\bu \to D^\bu_0$ to the homotopy pushout of the diagram 
$$D^\bu_0 \stackrel{\eta}{\leftarrow} D^\bu \stackrel{\eta}{\rightarrow} D^\bu_0$$
as structured object of type $K_C(M, n)$ (see Remark \ref{shifting-K-obj}).
  \item[(ii)] Let $\eta\co D^\bu \to D^\bu_0$ be an object of $\mc{W}\bigl(K_C(M, n) \twoheadrightarrow K(C,0)\bigr)$. Suppose first that $n > 0$. In this case, we form the homotopy pullback
\[
 \xymatrix{
 E^\bu \ar[r] \ar[d]_{q} & \sk_1(D^\bu) \ar[d]^i \\
 \sk_1(D^\bu) \ar[r]^i & D^\bu ,
 }
\]
where $i$ is the natural inclusion map. $G$ sends the object $\eta$ to the natural map 
$$\sk_{n+2}(E^\bu) \rightarrow E^\bu \xrightarrow{q} \sk_1(D^\bu)$$
which is an object of $\mc{W}\bigl(K_C(M, n+1) \twoheadrightarrow K(C,0)\bigr)$ by Theorem \ref{homotopy-excision} and 
Construction  \ref{alternative-direct-construction} (see also Proposition \ref{diff. constr. in CA}). The case $n=0$ is treated 
similarly, following Construction \ref{alternative-direct-construction}, by first forming the homotopy pullback 
\[
 \xymatrix{
 E^\bu \ar[r] \ar[d]_{q} & cC \ar[d]^i \\
 cC \ar[r]^i & c \iota_C(M) .
 }
\]
\end{itemize}
There are zigzags of natural weak equivalences connecting the composite functors $F \circ G$ and $G \circ F$ to the 
respective identity functors. These are defined by using the properties of homotopy pushouts and homotopy pullbacks 
and by applying the skeleton functors. Thus, this pair of functors induces a pair of inverse homotopy equivalences 
between the classifying spaces as required. 
\end{proof}

As a consequence, we have the following

\begin{corollary}
Let $C \in \CA$, $M \in \V C$, and $n \ge 0$. Then there is a weak equivalence 
  $$\mc{M}\bigl(K_C(M,n) \twoheadrightarrow K(C,0)\bigr) \simeq B \Aut_C(M).$$
In particular, the moduli spaces of structured $K$-objects are path-connected.  
\end{corollary}

Since objects of type $K_C(M,n)$ are homotopically unique, we will denote by $K_C(M,n)$ a choice of such an object, 
even though it will be non-canonical, whenever we are only interested in the actual homotopy type.

\begin{theorem} \label{representability of AQ}
Let $C \in \CA$, $M \in \V C$, $D \in C/ \CA$, and $n, k \geq 0$. For every structured pointed object of type $K_C(M,n + k)$, there are isomorphisms
\begin{align*}
   \AQ^n_C(D ; M) &\cong \pi_{0} \map^{\rm der}_{c(C/\CA)}\bigl(K_C(M, n), cD\bigr) \\
                 &\cong \pi_{k} \map^{\rm der}_{c(C/\CA)}\bigl(K_C(M, n+ k), cD\bigr).
\end{align*}
Here $\map^{\rm der}(-,-)$ denotes the derived mapping space.
\end{theorem}

\begin{proof}
Without loss of generality, we may choose $\iota_C(\Omega^{n + k}(cM))$ to be the model for an object of type 
$K_C(M, n+ k)$. Then the result follows from \eqref{AQ-groups 2}. We note that in general the mapping space is 
given a (non-canonical) basepoint by the structure map of $K_C(M,n + k)$ to an object of 
type $K(C, 0)$. 
\end{proof}

\subsection{Postnikov decompositions} \label{postnikov_decomp_unst_coalg}
The skeletal filtration of a cosimplicial unstable coalgebra is formally analogous to the Postnikov tower of 
a space. In this subsection, we prove that this filtration is \emph{principal}, i.e. it is defined in terms of 
attaching maps. This is the analogue of the ``difference construction'' from \cite[Proposition 6.3]{BlDG:pi-algebra}.
It is a consequence of the following proposition, which is an immediate application of Theorem \ref{homotopy-excision}.

\begin{proposition}  \label{diff. constr. in CA}
Let $f\co A^\bu \to C^\bu $ be an $n$-connected map in $c\CA^{\mc{E}}$, $n \geq 1$, and let $C = \pi^0(C^\bu)$. Let $D^\bu$ denote 
the homotopy cofiber of $f$ and $M = \pi^n(D^\bu)$. Consider the homotopy pullback square
\[
\xymatrix{
 E^\bu \ar[d] \ar[r] & A^\bu \ar[d]^{f} \\
 \sk_1 (C^\bu)  \ar[r] & C^\bu .
}
\]
Then:
\begin{itemize}
\item[(a)] $M$ is a $C$-comodule and $\sk_{n+2}(E^\bu)$ together with the canonical map 
$$\sk_{n+2}(E^\bu) \to \sk_1(C^\bu)$$
define a structured object of type $K_{C}(M,n+1).$
\item[(b)] If $\pi^sD^\bu=0$ for all $s>n$, then the diagram   
\[
\xymatrix{
 \sk_{n+2}(E^\bu) \ar[d] \ar[r] & A^\bu \ar[d]^{f} \\
 \sk_1(C^\bu)  \ar[r] & C^\bu 
}
\]
is a homotopy pushout.
\end{itemize}
\end{proposition}

As a consequence of Proposition \ref{diff. constr. in CA}, for every $C^\bu \in c \CA^{\mc{E}}$, the skeletal filtration 
$$ \sk_1(C^\bu) \to \sk_2(C^\bu) \to \cdots \to \sk_n(C^\bu) \to \cdots \to C^\bu$$
is defined by homotopy pushouts as follows
\[
 \xymatrix{
 K_C(M, n + 1) \ar[r]^(0.6){w_n} \ar[d] & \sk_n(C^\bu) \ar[d] \\
 K(C,0) \ar[r] & \sk_{n+1}(C^\bu),
 }
\]
where $C = \pi^0(C^\bu)$, $M = \pi^n(C^\bu)$, and the collection of maps $w_n$ may be regarded as 
analogues of the Postnikov $k$-invariants in this context. 

The next proposition reformulates and generalizes Proposition \ref{diff. constr. in CA} in terms of moduli spaces.
For $C^\bu \in c\CA^{\mc{E}}$ such that $\sk_n(C^\bu) \simeq C^\bu$, for some $n \geq 1$, and $M$ a coabelian $\pi^0(C^\bu)$-comodule, 
let $$\mc{W}(C^\bu + (M,n))$$ be the category with objects cosimplicial unstable coalgebras $D^\bu$ such that:
\begin{itemize} 
\item[(i)] $\sk_{n+1}(D^\bu) \simeq D^\bu$, 
\item[(ii)] $\sk_n(D^\bu)$ is weakly equivalent to $C^\bu$,
\item[(iii)] $\pi^n(D^\bu)$ is isomorphic to $M$ as comodules. 
\end{itemize}
The morphisms are given by weak equivalences of cosimplicial unstable coalgebras. The classifying space of this category, denoted by $$\mc{M}(C^\bu + (M,n)),$$
is the moduli space of $(n+1)$-skeletal extensions of the $n$-skeletal object $C^\bu$ by the $\pi^0(C^\bu)$-comodule $M$. 

We define 
  $$ \mc{W}\bigl(K(C,0) \twoheadleftarrow K_C(M, n+1) \stackrel{{\rm 1-con}}{\rightsquigarrow} C^\bu\bigr)$$
to be the category whose objects are diagrams $T\leftarrow U\to V$ in $c\CA$ such that $U\to T$ is a structured object of type $K_C(M,n+1)$, $V$ is \mc{E}-equivalent to $C^\bu$, and $U\to V$ is cosimplicially $1$-connected. The morphisms are weak equivalences of diagrams. Then we define 
 $$\mc{M}\bigl(K(C,0) \twoheadleftarrow K_C(M, n+1) \stackrel{{\rm 1-con}}{\rightsquigarrow} C^\bu\bigr)$$ 
as the classifying space of this category.
Constructions of such moduli spaces are also discussed in Subsection~\ref{subsec:moduli-spaces}.

\begin{proposition} \label{Moduli+Diff. = !}
Let $n \geq 1$ and suppose that $C^\bu$ is an object of $c \CA^{\mc{E}}$ such that $\sk_{n} C^\bu \xrightarrow{\simeq} C^\bu$. 
Let $M$ be an object of $\V C$ where $C=\pi^0(C^\bu)$. Then there is a natural weak equivalence 
  $$\mc{M}\bigl(C^\bu + (M,n)\bigr) \simeq \mc{M}\bigl(K(C,0) \twoheadleftarrow K_C(M, n+1) \stackrel{{\rm 1-con}}{\rightsquigarrow} C^\bu\bigr).$$
\end{proposition}
\begin{proof}
There is a pair of functors 
  $$F\co \mc{W}\bigl(K(C,0) \twoheadleftarrow K_C(M, n+1) \stackrel{{\rm 1-con}}{\rightsquigarrow} C^\bu\bigr) \rightleftarrows \mc{W}\bigl(C^\bu + (M,n)\bigr):\! G, $$ 
where:
\begin{itemize}
 \item[(i)] $F$ is defined to be the homotopy pushout of the diagram. This homotopy pushout is in $\mc{W}(C^\bu + (M,n))$ using the long exact 
 sequence of cohomotopy groups (cf. Proposition \ref{G-cofiber-to-LES}). 
 \item[(ii)] Given an object $D^\bu \in \mc{W}(C^\bu + (M,n))$, form the homotopy pullback 
 \[
  \xymatrix{
  \sk_{n+2}(E^\bu) \ar[r]^(.6)i & E^\bu  \ar[d] \ar[r] & \sk_n(D^\bu) \ar[d] \\
  & \sk_1(D^\bu) \ar[r] & D^\bu .
  }
 \]
The functor $G$ sends $D^\bu$ to the diagram $$\sk_1(D^\bu) \leftarrow \sk_{n+2}(E^\bu) \rightarrow \sk_n(D^\bu)$$
 which, by Proposition \ref{diff. constr. in CA}, is a diagram of the required type. 
\end{itemize}
The composites $F \circ G$ and $G \circ F$ are connected to the respective identity functors via zigzags of natural weak 
equivalences. These are defined by using the properties of homotopy pushouts and homotopy pullbacks and by applying 
the skeleton functors. Hence these functors induce a pair of inverse homotopy equivalences, as required.  
\end{proof}

\subsection{An extension of Proposition \ref{diff. constr. in CA}} \label{an_extension}
We discuss a generalization of Proposition \ref{diff. constr. in CA} to the case of maps which are $0$-connected but do not necessarily induce isomorphisms on $\pi^0$. This subsection will not be used in the rest of the paper.

Recall that $\grCoalg$ denotes the category of cocommutative, counital graded coalgebras over $\mathbb{F}$, that is, the category 
of coalgebras in $\Vec$ with respect to the graded tensor product. 

\begin{definition} Let  $K \rightarrowtail C \leftarrowtail L$ be inclusions of coalgebras in $\grCoalg$. We obtain a pushout of 
$C$-comodules 
\[
\xymatrix{
C \ar[r] \ar[d] & C/K \ar[d]  \ar[dr] \ar[drr] \\
C/L \ar[r] \ar@/_1pc/[rr] \ar@/_2pc/[rrr] & C/(K \oplus L) \ar@{..>}[r]^-{\phi} & C/K\ \Box_C\ C/L \ar[r] & C/K \otimes C/L ,
 }
\] 
\vspace*{5mm}

\noindent
where the $C$-comodule structures and the outer maps are defined by the comultiplication of $C$. Thus we obtain a canonical dotted arrow 
  $$\phi\co C/(K \oplus L)\to C/K \ \Box_C\  C/L$$ 
in $\Vec$. Then we define:
\begin{itemize}
\item[(a)] $C/K\ast_C C/L$ to be the kernel of $\phi$. 
\item[(b)] $C/K\circ_C C/L$  to be the image of $\phi$. 
\end{itemize}
\end{definition}

\begin{remark} The construction of $C/K\ast_C C/L$ is dual to the construction which associates to two ideals $I$,$J$ in a commutative 
ring $R$ the quotient $(I\cap J)/(IJ)$. The constructions $C/(K \oplus L)$ and $C/K \circ_C C/L$ are dual to $I\cap J$ and $IJ$ respectively.  
\end{remark}

\begin{theorem} \label{refined diff. constr. in CA}
Let  $f\co A^\bu \to C^\bu $ be an $n$-connected map in $c\CA^{\mc{E}}$, $n \geq 0$. 
Suppose that $L$ is an unstable coalgebra and $g\co K(L,0)\to B^\bu $ a $0$-connected map. Consider the 
homotopy pullback diagram
\[
\xymatrix{
 E^\bu \ar[d] \ar[r] & {A}^\bu \ar[d]^{f} \\
 {K(L ,0) }  \ar[r]^{g} & {C}^\bu .
}
\]
We denote the bigraded coalgebras 
$$ \mathscr{A}^*:= \pi^\ast (A^\bu )\ \ \text{ and }\ \ \mathscr{C}^*:= \pi^\ast (C^\bu )$$ 
in order to simplify the notation. Then:
\begin{itemize}
\item[(a)] The object $\sk_{n+2}E^\bu $ is of type $K_{D}(N,n+1)$, where 
\begin{itemize}
 \item[(i)] the unstable coalgebra $D$ is isomorphic to $\mathscr{A}^0\ \Box_{\mathscr{C}^0}\ L$.
\item[(ii)] the $D$-comodule $N$ fits into a short exact sequence as follows,
 \[
0 \to (\mathscr{C}/L\ \ast_{\mathscr{C}}\ \mathscr{C}/\mathscr{A})^n \to N \to (L\ \Box_{\mathscr{C}}\ \mathscr{A})^{n+1} \to 0
\]
If $f$ induces a monomorphism in degree $n+1$, then $(L\ \Box_{\mathscr{C}}\ \mathscr{A})^{n+1}=0$. 
\end{itemize}
 \item[(b)] If $\pi^s(C^\bu)$ vanishes for $s > n$, then the diagram   
\[
\xymatrix{ \sk_{n+2}E^\bu \ar[d] \ar[r] & {A}^\bu \ar[d]^{f} \\
           {K(L ,0) }  \ar[r]^{g} & {C}^\bu }
\]
is a homotopy pushout.
\end{itemize} 
\end{theorem}

A formula similar to (a) is obtained by Massey and Peterson in~\cite[Theorem 4.1]{Massey-Peterson:mod2}.
For the proof of this theorem, we will need the following technical lemmas.

\begin{lemma}\label{Cotor} 
Let $B \xrightarrow{g} C \xleftarrow{f} A$ be a diagram in $\grCoalg$. Assume that 
\begin{enumerate}
   \item[(a)] $B$ is concentrated in degree $0$.
   \item[(b)] $f$ is an isomorphism in degrees $<n$ and a monomorphism in degree $n$.  
\end{enumerate}
Then $\Cotor^p_{C}(B ,A )_q=0$ for all pairs $(p,q) \neq (1,n)$ with $p+q \leq n+1$ and $p>0$.
\end{lemma} 
\begin{proof}
Recall that
  $$ \Cotor^p_{C}(B, A )_q = H_p (B\ \Box_{C}\ I^*)_q $$
for any injective $C$-resolution $A \xrightarrow{\simeq} I^*$ of $A \in \Comod_{C}$. 
An injective $C$-resolution $I^*$ of $A$ can be given such that 
$A \to I^0$ agrees with $f$ in internal degrees $q \leq n$. Moreover, we can assume that the resolution 
has the property that $(I^p)_q = 0$ 
for $0 \leq q < n$ and $p > 0$.  Then  
  $$(A\ \Box_{C}\ I^p)_q = 0$$ 
for all $(p,q)$ with $p > 0$ and $0 \leq q < n$ and the assertion follows.
\end{proof}

\begin{lemma}\label{cotor computation} 
Let $K \rightarrowtail C \leftarrowtail L$ be inclusions of coalgebras \emph{(}ungraded, graded or bigraded\emph{)}.
Then there is an isomorphism 
  $$\mathrm{Cotor}^1_{C}(K,L) \cong (C /K \ast_C C/L).$$
\end{lemma}

\begin{proof} Applying the functor $- \Box_C\ L$ to the short exact sequence of $C$-comodules:
  $$0 \to K \to C \to C/K \to 0$$
we obtain an exact sequence
  $$0 \to K \cap L \to L \to C/K \ \Box_C\  L \to \mathrm{Cotor}^1_{C}(K, L) \to 0.$$
On the other hand, applying the functor $- \Box_C\ (C/K)$ to the sequence, 
  $$0 \to L \to C \to C/L \to 0,$$
we get an exact sequence as follows,
  $$0 \to L \ \Box_C\  (C/K) \to C/K \to (C/L) \ \Box_C\  (C/K) \to \cdots$$
Note then that 
  $$(C/K) \ast_C (C/K): = \ker\bigl[(C/K \oplus L) \to (C/K) \otimes (C/L)\bigr]$$
is the image of the composition 
  $$L \ \Box_C\  (C/K) \to C/K \to (C/K \oplus L)$$
This shows that $(C/K) \ast_C (C/L)$ fits in a short exact sequence 
  $$0 \to (L/ K \cap L) \to L \ \Box_C\  (C/K) \to (C/K) \ast_C (C/L) \to 0$$
and the claim follows. 
\end{proof}

\begin{remark} Lemma~\ref{cotor computation}  above is the (graded) dual of the well known formula 
  $$\mathrm{Tor}_R^1(I,J)\cong (I\cap J)/(IJ),$$
where $I$ and $J$ are ideals in a commutative ring $R$.
\end{remark}

\noindent \textbf{Proof of Theorem~\ref{refined diff. constr. in CA}.} We use the spectral sequence (b) from Theorem~\ref{Kunneth spectral sequence 2} for the computation of $\pi^{s}(E^\bu )$ when $s \leq n+1$. The identification of 
$\pi^0(E^\bu)$ is then immediate. As a consequence of (a bigraded variant of) Lemma \ref{Cotor}, we have
  $$\Cotor_{\mathscr{C}}^p(L, \mathscr{A})^q=0,$$ 
as a graded vector space, for all $(p, q) \neq (1,n)$ with $p+q \leq n+1$ and $p>0$. This implies that $\pi^s(E^\bu)$ 
is trivial for all $0 < s < n+1$. Since the cotensor product is left exact, there is a monomorphism of graded vector 
spaces,
$$(L\ \Box_{\mathscr{C}}\ \mathscr{A})^q \rightarrowtail (L\ \Box_{\mathscr{C}}\ \mathscr{C})^q$$
for all $q \leq n$. The map is also injective for $q=n+1$ if $\pi^{n+1}(f)$ is injective. But 
$(L\ \Box_{\mathscr{C}}\ \mathscr{C})^q=0$ for all $q > 0$, and consequently, 
$$(L\ \Box_{\mathscr{C}}\ \mathscr{A})^{q}=0$$
for all $0 < q \leq n$, because $L$ is concentrated in degree $0$. By Lemma \ref{cotor computation}, we have an isomorphism 
  $$\Cotor^1_{\mathscr{C}} (L,\mathscr{A})^n \cong (\mathscr{C}/L \ast_{\mathscr{C}} \mathscr{C}/\mathscr{A})^n$$ 
This graded vector space and $(L\ \Box_{\mathscr{C}}\ \mathscr{A})^{n+1}$ are the only potentially non-trivial objects of 
total degree $n+1$ in the $E_2$-page. Since there is no place for non-trivial differentials in this degree, the proof of 
statement (a) is complete.

Claim (b) follows from the long exact sequence of cohomotopy groups and the Five Lemma.  \qed \\

We list a few immediate consequences of Theorem~\ref{refined diff. constr. in CA}. We will need a more general notion of primitivity.

\begin{definition} 
Let $C\in (\F/\grCoalg)$ be a pointed graded coalgebra and $N$ a $C$-comodule. The sub-comodule $\Pr_C (N)$ of the 
primitives under the $C$-coaction is the kernel of the map 
\[N \xrightarrow{\Delta_N} C\otimes N \to  (C/\mathbb{F}) \otimes N.\]
This is the subset of elements $n \in N$ such that $\Delta_N (n)= 1 \otimes n$.
\end{definition}

\begin{corollary} \label{trivial coaction} 
Consider the special case of \emph{Theorem~\ref{refined diff. constr. in CA}} where:
\begin{itemize} 
 \item $L = \mathbb{F}$ \emph{(}therefore $E^\bu$ is the homotopy fiber of $f$ at the chosen basepoint\emph{)}.
 \item $\pi^{n+1}(f)$ is injective.
\end{itemize}
Let $M = \mathscr{C}^n/\mathscr{A}^n$ as a $\mathscr{C}^0$-comodule where $A$ and $C$ are compatibly pointed. Then there is an isomorphism 
$$\pi^{n+1}(E^\bu) \cong \Pr_{\mathscr{C}^0}(M).$$
If $M$ is a trivial $\mathscr{C}^0$-comodule, then there is an isomorphism $\pi^{n+1}(E^\bu) \cong {M}.$
\end{corollary}
\begin{proof}
It suffices to note that $( \mathscr{C}^0/\mathbb{F} \ast_{\mathscr{C}^0} M) \cong \Pr_{\mathscr{C}^0}(M)$.
\end{proof}

\begin{corollary}\label{difference in dim 1} 
Let $C \in \CA$, $M \in \V C$ and consider the homotopy pullback
\[
 \xymatrix{
 E^\bu \ar[r] \ar[d] & cC \ar[d] \\
 cC \ar[r] & c (\iota_C(M))
 }
\]
in $c\CA^{\mc{E}}$, where $C \to \iota_C(M)$ is the canonical inclusion. Then there is an isomorphism of $C$-comodules
$\pi^{1}(E^\bu) \cong {M}.$
\end{corollary}
\begin{proof} Let $D = \iota_C(M)$. It is enough to show that $M \ast_{D} M \cong M$. We have 
  $$M \ast_{D} M \cong \ker\{\phi\co M\to M\ \Box_{D}\ M\}.$$
The map $M \xrightarrow{\phi} M\ \Box_{D}\ M \subset M \otimes M$ can be factored as follows:
\[
\xymatrix{
 M \ar[d] \ar[r] & M \otimes M \\
 D \ar[r]^-{\Delta_D} & D \otimes D . \ar[u] 
 }
\]
This shows that the map is trivial by definition of $\Delta_D$ on $M$.   
\end{proof}

\begin{corollary}
Let $C^\bu$ be a pointed cosimplicial unstable coalgebra and $n\ge 1$. Then there is an isomorphism
  $$ \pi^n(\Omega^nC^\bu) \cong \Pr(\pi^0 C^\bu). $$ 
In particular, for a pointed unstable coalgebra $C$, there is an isomorphism 
  $$ \pi^n(\Omega^n cC) \cong\Pr(C).$$
\end{corollary}

\begin{proof}
Corollary~\ref{trivial coaction} implies that $\pi^1(\Omega C^\bu) \cong \Pr(\pi^0(C^\bu))$. Then the 
result follows by induction.
\end{proof}

\section{Cosimplicial spaces}\label{sec:cosimplicial-spaces}

In this section, we study the resolution model structure on cosimplicial spaces with respect to the class of Eilenberg-MacLane space $K(\F,n)$.  The close connection with the resolution model category $c\CA^{\mc{E}}$, resulting from the representing property of the Eilenberg-MacLane spaces, is used in order to apply the homotopy theory developed in the previous sections to this model category of cosimplicial spaces. In particular, we deduce a homotopy excision 
theorem for cosimplicial spaces from the homotopy excision theorem 
for cosimplicial unstable coalgebras. This is used in the construction of Postnikov decompositions of cosimplicial spaces which will be important for our obstruction theory in the next section.  

The resolution model category of cosimplicial spaces is introduced in 
Subsection \ref{sec:cosimplicial-resolutions-spaces} and some basic facts on the comparison with the model category of cosimplicial unstable coalgebras are discussed. In Subsection 
\ref{sec:Kuenneth-thm}, using the K\"unneth theorem, we prove that 
the homology functor from cosimplicial spaces to cosimplicial unstable coalgebras preserves homotopy pullbacks and deduce from this the homotopy excision theorem for cosimplicial spaces (Theorem \ref{homotopy-excision 2}). 

In Subsection \ref{section-L-objects}, we define and construct the objects of type $L_C(M,n)$ which are analogues of the twisted Eilenberg-MacLane spaces in the context of cosimplicial spaces with respect to the natural homotopy groups. These are 
closely related to the objects of type $K_C(M,n)$ as we show in Proposition \ref{Representability of L(C,n)}. In the final Subsection 
\ref{section-postnikov-decomp-spaces}, we discuss the properties of the Postnikov decompositions of cosimplicial spaces that arise from the skeletal filtration. Lastly, we identify the moduli space of extensions of a cosimplicial space by attaching an $L$-object in terms of the moduli space of extensions of the associated cosimplicial unstable coalgebra by attaching a $K$-object (Theorem  \ref{Moduli-Comparison-of-Spaces}). 

\begin{notation} 
Let $\sk^c_n(-)\co c \mc{S} \to c\mc{S}$ denote the derived skeleton functor defined as the skeleton of a functorial Reedy 
cofibrant replacement. 
\end{notation}

\subsection{Cosimplicial resolutions of spaces} \label{sec:cosimplicial-resolutions-spaces}
Let $\mc{S}$ be the usual model category of simplicial sets and let $\F$ be any prime field. Consider 
  $$\mc{G}=\{\kfp{m}\, |\, m\ge 0\}$$ 
as a set of homotopy (abelian) group objects for the category $\mc{S}$. Recall that $\kfp{m}$ denotes a (fibrant) Eilenberg-MacLane space with only non-trivial homotopy group $\mathbb{F}$ 
in degree $m$ and that $\mc{E}=H_*(\mc{G})$, where \mc{E} was defined in~(\ref{eqn:def-of-mcE}). 

A map $f \co X \to Y$ in $\ho{\mc{S}}$ is $\mc{G}$-monic if and only if the induced map on  $\F$-cohomology is surjective or the induced map $H_*(f) \co H_*X \to H_*Y$ is injective. 
Recall that an $\F$-GEM is a product of spaces of type \kfp{m} for possibly various values of $m$. Every $\F$-GEM is $\mc{G}$-injective. 

\begin{lemma} A fibrant $X \in \mc{S}$ is $\mc{G}$-injective if and only if it is a retract of a fibrant space weakly equivalent to an $\F$-GEM. If a fibration $f\co X \to Y$ is \mc{G}-injective, then the induced map $H_*(f)\co H_*X\to H_*Y$ is \mc{E}-injective. 
\end{lemma}

\begin{proof}
$\mc{G}$-injectives are closed under retracts and weak equivalences, so the \emph{if} part is immediate. There 
is a canonical $\mc{G}$-monic map (cf. \cite[4.5]{Bou:cos})
$$X \rightarrow \prod_{m \geq 0} \prod_{X \to \kfp{m}} \kfp{m}$$
which we can functorially factor as the composition of a $\mc{G}$-monic cofibration $i\co X \hookrightarrow Z_X$ followed 
by a trivial fibration. In particular, $Z_X$ is again fibrant. If $X$ is 
$\mc{G}$-injective and fibrant, then the square 
\[
 \xymatrix{
 X \ar@{=}[r] \ar[d]^i & X \ar[d] \\
 Z_X \ar[r] & \ast
 }
\]
admits a lift, which means that $X$ is a retract of $Z_X$. This shows the \emph{only if} part of the first statement.
The second statement then follows from the fact that every $\mc{G}$-injective fibration $f\co X \to Y$ is a retract of a fibration of the form (see \cite[Lemma 3.10]{Bou:cos})
  $$E \stackrel{\simeq}{\twoheadrightarrow} Y \times G \stackrel{\pi_1}{\twoheadrightarrow} Y,$$
where $G$ is weakly equivalent to an $\F$-GEM, and the K\"{u}nneth isomorphism. 
\end{proof}

Note that there are enough $\mc{G}$-injectives in the sense of Section \ref{resolution-model-cat}. We will consider the \mc{G}-resolution model structure on cosimplicial spaces 
in the sense of Section \ref{resolution-model-cat}. The following theorem is a direct application of Theorem \ref{bousfield}.

\begin{theorem}
There is a proper simplicial model category $c\mc{S}^{\mc{G}}$ where the underlying category is the category of cosimplicial 
spaces, the weak equivalences are the $\mc{G}$-equivalences, the cofibrations are the $\mc{G}$-cofibrations and the fibrations 
are the $\mc{G}$-fibrations. 
\end{theorem}

\begin{proof}
Apart from the right properness, this is an application of Theorem \ref{bousfield}. Right properness will be shown in Corollary \ref{right-proper2}.
\end{proof}

As a consequence of the representing property \eqref{representing property of GEMs}, we have the following characterizations of $\mc{G}$-equivalences
and $n$-connected maps in $c \mc{S}^{\mc{G}}$ (see also Lemma \ref{lem:equivalent-formulations-n-connected}.)

\begin{proposition}
A map of cosimplicial spaces $f\co X^\bu \to Y^\bu$ is a $\mc{G}$-equivalence if and only if 
$H_*(f)\co H_*(X^\bu) \to H_*(Y^\bu)$ is an $\mc{E}$-equivalence of cosimplicial unstable coalgebras. Moreover,  
$f$ is $n$-connected in $c\mc{S}^\mc{G}$ if and only if the induced map $H_*(f)$ is $n$-connected in $c\CA^{\mc{E}}$.
\end{proposition}

In particular, the functor $H_*$ descends to a functor
$$ H_*\co\ho{c\mc{S}^{\mc{G}}}\to\ho{c\CA^{\mc{E}}}$$
which also preserves finite products as a consequence of Theorem \ref{enhanced Kuenneth}. Also immediate is the following characterization of 
$\mc{G}$-cofibrations.

\begin{proposition}
A map of cosimplicial spaces $f\co X^\bullet\to Y^\bullet$ is a \mc{G}-cofibration if and only if it 
 is a Reedy cofibration and the map $H_*(f)\co H_*(X^\bullet) \to H_*(Y^\bullet)$ is an \mc{E}-cofibration.
\end{proposition}

These observations provide some evidence for the utility of this model category in our realization problem. First, 
the weak equivalences are designed so that they are actually detected in $c \CA^{\mc{E}}$ while at the same time 
a map of constant cosimplicial spaces $cX \to cY$ is a $\mc{G}$-equivalence if and only if the underlying map of spaces 
$X \to Y$ is a homology equivalence. Second, the $\mc{G}$-equivalences are detected on the $E_2$-page of the homology spectral sequence of a cosimplicial space (see~\cite{BK:sscoeffring} and~\cite{Bou: homology SS})
  $$ E^2_{s,t}=\pi^sH_t(X^\bu)~\Longrightarrow~H_{t-s}(\Tot X^\bu). $$
Thus, assuming convergence of the latter, we retain control of the properties
of the resulting objects after totalization.

\begin{remark} \label{free-monad-res}
There is monad $\F\co \mc{S} \to \mc{S}$ that sends an unpointed simplicial set $X$ to the simplicial $\F$-vector space $\F[X]$ generated by $X$.
It is closely related to the Bousfield-Kan monad~\cite[Section 2.1]{BK:sscoeffring} on pointed simplicial sets; in fact, they become the same if one equips $X$ with a disjoint 
basepoint. This monad yields weak resolutions in $c \mc{S}^{\mc{G}}$ in the sense of \cite{Bou:cos} (see \cite[Section 7]{Bou:cos}), cf. Remark \ref{G-cofree-fact}.
\end{remark}

\begin{proposition} \label{G-ho po to H-ho po}
The homology functor $H_* \co c\mc{S}^{\mc{G}} \to c\CA^{\mc{E}}$ preserves the cofibers of $0$-connected $\mc{G}$-cofibrations, that is, when the induced map on $\pi^0 H_*(-)$ is injective. 
\end{proposition}

\begin{proof}
Let $i\co X^\bu \hookrightarrow Y^\bu$ be a $0$-connected $\mc{G}$-cofibration and let $Z^\bu$ be its cofiber. The 
induced map $H_*(i)\co H_*(X^\bu) \to H_*(Y^\bu)$ 
is an \mc{E}-cofibration and we denote its cofiber by $C^\bu$. We show that the canonical map $q\co C^\bu \to H_*(Z^\bu)$ 
is an isomorphism. Note that $\Hom(C^\bu, H_*\kfp{m})$ fits into a pullback square of simplicial groups
\[
\xymatrix{
\Hom(C^\bu, H_* \kfp{m}) \ar[r] \ar[d] & \Hom(H_*Y^\bu, H_*\kfp{m}) \ar[d]^{H_*(i)^*} \\
\Hom (c \terminal, H_* \kfp{m}) \ar[r] & \Hom(H_*X^\bu, H_*\kfp{m}).
}
\]
for all $m\ge 0$.
Since $i$ is a \mc{G}-cofibration, the right hand vertical map is a Kan fibration and the square is also a homotopy pullback. By 
\eqref{representing property of GEMs}, $H_*(i)^*$ can be identified with the 
fibration of simplicial groups
$$i^*\co [Y^\bu, \kfp{m}] \to [X^\bu, \kfp{m}].$$ 
Since $i$ is $0$-connected, the map $i^*$ is a $\pi_0$-surjective Kan fibration and, hence, surjective. Thus, for all $k, m \geq 0$, there are short exact sequences
$$0 \to \tilde{H}^m(Z^k) \to H^m(Y^k) \to H^m(X^k) \to 0.$$
Applying \eqref{representing property of GEMs} again, we conclude that the following is also a pullback square of 
simplicial groups 
\[
\xymatrix{
\Hom\bigl(H_*(Z^\bu), H_* (\kfp{m})\bigr) \ar[r] \ar[d] & \Hom\bigl(H_*(Y^\bu), H_*(\kfp{m})\bigr) \ar[d]^{H_*(i)^*} \\
\Hom \bigl(c \terminal, H_* (\kfp{m})\bigr) \ar[r] & \Hom\bigl(H_*(X^\bu), H_*(\kfp{m})\bigr).
}
\]
So the map $q\co C^\bu \to H_*(Z^\bu)$ induces natural isomorphisms
$$\Hom(C^\bu, H_* \kfp{m}) \cong \Hom(H_*(Z^\bu), H_*\kfp{m})$$
which implies that $q$ an isomorphism by Proposition \ref{cogenerating property of GEMs}. 
\end{proof}

More generally, we have the following statement about preservation of homotopy pushouts. Homotopy $n$-cocartesian squares were defined in~\ref{def:(co)cartesian}.

\begin{proposition} \label{G-ho po to H-ho po 2}
Let 
\[
 \xymatrix{
 X^\bu \ar[d] \ar[r]^i & Y^\bu \ar[d] \\
 W^\bu \ar[r] & Z^\bu
 }
\]
be a homotopy $n$-cocartesian square in $c \mc{S}^{\mathcal{G}}$ where the map $i$ is $0$-connected. 
Then the square 
\[
 \xymatrix{
H_*(X^\bu) \ar[r]^{H_*(i)} \ar[d] & H_*(Y^\bu) \ar[d] \\
H_*(W^\bu) \ar[r] & H_*(Z^\bu)
 }
\]
is homotopy $n$-cocartesian in $c \CA^{\mc{E}}$. 
\end{proposition}
\begin{proof}
We may assume that $i$ is a $\mc{G}$-cofibration and consider the (homotopy) pushout 
\[
 \xymatrix{
  X^\bu \ar[d] \ar[r]^i & Y^\bu \ar[d] \\
 W^\bu \ar[r] & P^\bu.
 }
\]
Let $P^\bu \to Z^\bu$ be the canonical map, which is $n$-connected by assumption. Then the induced map $H_*(P^\bu) \to H_*(Z^\bu)$ is 
$n$-connected. Let $C^\bu$ be the (homotopy) pushout of the maps 
\begin{equation} \label{C-pushout}
H_*(Y^\bu) \leftarrow H_*(X^\bu) \rightarrow H_*(W^\bu)
\end{equation}
and $C^\bu \to H_*(P^\bu)$ the canonical map. Consider the diagram 
\[
 \xymatrix{
 X^\bu \ar[d] \ar[r]^i & Y^\bu \ar[d] \ar[r] & Q_1^\bu \ar[d]^{\simeq} \\
 W^\bu \ar[r]^{i'} & P^\bu \ar[r] & Q_2^\bu\ ,
 }
\]
where the horizontal sequences are (homotopy) cofiber sequences and so the right hand side vertical map is a 
$\mc{G}$-equivalence. The map $i'$ is again $0$-connected and hence by Proposition \ref{G-ho po to H-ho po}, 
the sequence of maps 
$$H_*(W^\bu) \to H_*(P^\bu) \to H_*(Q_2^\bu)$$
is a (homotopy) cofiber sequence in $c \CA^{\mc{E}}$. Now consider the following diagram
\[
\xymatrix{
H_*(W^\bu) \ar[r] \ar@{=}[d] & C^\bu \ar[d] \ar[r] & H_*(Q_1^\bu) \ar[d]^{\simeq}  \\
H_*(W^\bu) \ar[r]^{i'} & H_*(P^\bu) \ar[r] & H_*(Q_2^\bu) \ ,
}
\]
where $C^\bu$ is the pushout of \eqref{C-pushout}, the top sequence is a (homotopy) cofiber sequence by Proposition \ref{G-ho po to H-ho po}, 
and the right vertical map is an $\mc{E}$-equivalence. It follows that $C^\bu \to H_*(P^\bu)$ is an $\mc{E}$-equivalence. Since $H_*(P^\bu) \to H_*(Z^\bu)$ 
is $n$-connected, it follows that $C^\bu \to H_*(Z^\bu)$ is also $n$-connected, as required. 
\end{proof}

\subsection{Consequences of the K\"unneth theorem} \label{sec:Kuenneth-thm}
In this subsection, we prove that the homology functor $H_*$ preserves homotopy pullbacks as an application of the K\"unneth theorem (Proposition \ref{G-ho pb to H-ho pb}). An important 
consequence of this fact is the homotopy excision theorem for cosimplicial spaces (Theorem~\ref{homotopy-excision 2}). 

The following lemma shows that $H_*$ preserves cotensors with a finite simplicial set. It generalizes Theorem \ref{enhanced Kuenneth}. 

\begin{lemma}\label{H_* and external cotensors}
Let $K$ be a finite simplicial set and $X^\bullet$ in $c\mc{S}$. Then there is a natural isomorphism
   $$ H_*\bigl(\hom(K,X^\bullet)\bigr)\cong\hom\bigl(K,H_*(X^\bullet)\bigr).$$
\end{lemma}

\begin{proof}
There is a natural isomorphism in each cosimplicial degree $s\ge 0$:
   $$ H_*\bigl(\hom(K,X^\bullet)^s\bigr)=H_*\Big(\prod_{K_s}X^s\Big)\cong\bigotimes_{K_s}H_*(X^s)=\hom\bigl(K,H_*(X^\bullet)\bigr)^s. $$
These isomorphisms are clearly compatible with the cosimplicial structure maps. 
\end{proof}

\begin{proposition}\label{G-cofree to H-cofree}
The functor $H_*\co c\mc{S}^{\mc{G}} \to c\CA^{\mc{E}}$ sends a \mc{G}-cofree map on $(G_s)_{s \geq 0}$ to an \mc{E}-cofree map
on $(H_*G_s)_{s \geq 0}$ with the induced co-attaching maps. Moreover, $H_*$ sends pullbacks along $\mc{G}$-cofree maps to pullbacks.
\end{proposition}

\begin{proof}
Let $f\co X^\bullet\to Y^\bullet$ in $c\mc{S}$ be a \mc{G}-cofree map on \mc{G}-injectives $(G_s)_{s\ge 0}$. By Proposition~\ref{co-cell decomposition of cofree maps}, 
there is a pullback diagram in $c\mc{S}$ for each $s \geq 0$,
\diagr{ \cosk_s(f) \ar[r]\ar[d] & \hom(\Delta^s,cG_s) \ar[d] \\
        \cosk_{s-1}(f) \ar[r] & \hom(\partial\Delta^s,cG_s). }
For every cosimplicial degree $k\ge 0$, we have isomorphisms   $$\hom(\Delta^s,G_s)^k\cong\prod_{(\partial\Delta^s)^k}G_s \times\prod_{k\surj s}G_s$$
and the right vertical map is given by the projection onto the first factor. Hence, for every cosimplicial degree $k \geq 0$, the left vertical map is isomorphic to the projection
   $$ \cosk_{s-1}(f)^k \times \prod_{k\surj s}G_s \to \cosk_{s-1}(f)^k.$$
By Lemma \ref{H_* and external cotensors}, we conclude that the following square
\diagr{ H_*(\cosk_s (f)) \ar[r]\ar[d] & \hom(\Delta^s, cH_*G_s) \ar[d] \\
        H_*(\cosk_{s-1}(f)) \ar[r] & \hom(\partial\Delta^s, cH_*G_s) }
is a pullback diagram in $c\CA$. Then $H_*(\cosk_s(f)) \cong \cosk_s(H_*(f))$ and the map 
$H_*(f)\co H_*(X^\bullet)\to H_*(Y^\bullet)$ is \mc{E}-cofree as required. The second claim is similar. 
\end{proof}

\begin{corollary}\label{G-fib to H-fib}
The homology functor $H_*\co c\mc{S}^{\mc{G}} \to c\CA^{\mc{E}}$ sends quasi-\mc{G}-cofree maps to 
 $\mc{E}$-cofree maps and $\mc{G}$-fibrations to $\mc{E}$-fibrations. Moreover, $H_*$ sends 
 pullbacks along a quasi-\mc{G}-cofree Reedy fibration to pullbacks. 
\end{corollary}
\begin{proof}
Let $f\co X^\bu \to Y^\bu$ be a quasi-\mc{G}-cofree map. Then there are \mc{G}-injective objects $(G_s)_{s \geq 0}$ and diagrams 
\[
 \xymatrix{
 \cosk_s(f) \ar[r]^-{\simeq} \ar[dr]_{\gamma_{s-1}(f)\phantom{x}} & P^\bu \ar[r] \ar[d] &\hom(\Delta^s,cG_s) \ar[d] \\
       & \cosk_{s-1}(f) \ar[r]   & \hom(\partial\Delta^s,cG_s),}
\]
where the square is a pullback, and thus also Reedy homotopy pullback, and the map indicated by $\simeq$ is a 
Reedy equivalence. Applying $H_*$ gives a diagram
\[
 \xymatrix{
 H_*(\cosk_s(f)) \ar[r]^{\cong} \ar[dr]_-{H_*(\gamma_{s-1}(f))\phantom{xxxx}} & H_*(P^\bu) \ar[r] \ar[d] &\hom(\Delta^s,cH_*(G_s)) \ar[d] \\
       & H_*(\cosk_{s-1}(f)) \ar[r]   & \hom(\partial\Delta^s,cH_*(G_s)),}
\]
where the square is a pullback by Proposition \ref{G-cofree to H-cofree} and the indicated map is an isomorphism. It follows that $H_*(f)$ is $\mc{E}$-cofree. If 
$f\co X^\bu \to Y^\bu$ is a $\mc{G}$-fibration, then by Proposition \ref{all-about-cofree-maps2}, $f$ is a retract of a quasi-\mc{G}-cofree map
and the second claim follows. 

For the final claim, consider a pullback square 
\[
 \xymatrix{
 E^\bu \ar[r] \ar[d]_{f'} & X^\bu \ar[d]^f \\
 Z^\bu \ar[r] & Y^\bu,
 }
\]
where $f$ is a quasi-\mc{G}-cofree Reedy fibration. Then there are \mc{G}-injective objects $(G_s)_{s \geq 0}$ and diagrams 
\[
 \xymatrix{
\cosk_s(f') \ar@{-->}[ddr] \ar[rr] \ar[dr]^-{\simeq} && \cosk_s(f) \ar@{-->}[ddr] \ar[dr]^-{\simeq} \ar@/^7pt/[drr] \\ 
&  Q^\bu \ar[rr] \ar[d] && P^\bu \ar[r] \ar[d] & \hom(\Delta^s,cG_s) \ar[d] \\
 & \cosk_{s-1}(f')  \ar[rr] && \cosk_{s-1}(f) \ar[r]   & \hom(\partial\Delta^s,cG_s),
 }
\]
where the two front squares and the back square are pullbacks by the definition of $Q^\bu$. The map 
$\cosk_s(f) \to \cosk_{s-1}(f)$ is a Reedy fibration, which shows that the back square is a Reedy homotopy pullback. Since 
the front square is also a Reedy homotopy pullback, it follows that the top square is again a Reedy homotopy pullback. This 
implies that the top left map is a Reedy equivalence. Applying the homology functor $H_*$, these Reedy equivalences become 
isomorphisms in $c \CA^{\mc{E}}$. Then the result follows from Lemma \ref{G-cofree to H-cofree}. 
\end{proof}

\begin{proposition}\label{G-ho pb to H-ho pb}
The homology functor $H_* \co c\mc{S}^{\mc{G}} \to c\CA^{\mc{E}}$ preserves homotopy pullbacks.
\end{proposition}

\begin{proof}
By Proposition~\ref{prop:quasi-cofree-replacements}, any homotopy pullback in the \mc{G}-resolution model structure can be replaced up to \mc{G}-equivalence 
by a strict pullback diagram where the right hand vertical map is a quasi-\mc{G}-cofree Reedy fibration.
Since $H_*$ sends \mc{G}-equivalences to \mc{E}-equi\-va\-len\-ces, the result follows from Proposition \ref{G-fib to H-fib}. 
\end{proof}

\begin{corollary} \label{right-proper2}
The model category $c \mc{S}^{\mc{G}}$ is right proper.
\end{corollary}

\begin{proof}
By Corollary \ref{all-about-cofree-maps2}, it suffices to show that the pullback of a 
$\mc{G}$-equivalence along a quasi-$\mc{G}$-cofree Reedy fibration is again a $\mc{G}$-equivalence.
This follows from Corollary \ref{G-fib to H-fib} and Corollary \ref{right-proper}.
\end{proof}

From Propositions \ref{G-ho po to H-ho po 2} and \ref{G-ho pb to H-ho pb} we also obtain the analogue of Theorem~\ref{homotopy-excision}.

\begin{theorem}[Homotopy excision for cosimplicial spaces] \label{homotopy-excision 2}
Let 
\[
 \xymatrix{
 E^\bu \ar[d] \ar[r] & X^\bu \ar[d]^f \\
 Y^\bu \ar[r]^g & Z^\bu
 }
\]
be a homotopy pullback square in $c \mc{S}^{\mathcal{G}}$ where $f$ is $m$-connected and $g$ is $n$-connected for $m, n \geq 0$. Then the square is homotopy $(m+n)$-cocartesian. 
\end{theorem}

\begin{proof}
Let $W^\bu$ be the homotopy pushout of the diagram 
$$Y^\bu \leftarrow E^\bu \rightarrow X^\bu$$ 
and $W^\bu \to Z^\bu$ the canonical map. We need to show that the induced map 
$H_*(W^\bu) \to H_*(Z^\bu)$ is $(m+n)$-connected. 
By Proposition \ref{G-ho pb to H-ho pb}, the induced square 
\[
 \xymatrix{
 H_*(E^\bu) \ar[r] \ar[d] & H_*(X^\bu) \ar[d]^{H_*(f)} \\
 H_*(Y^\bu) \ar[r]^{H_*(g)} & H_*(Z^\bu) 
}
\]
is a homotopy pullback. The homotopy excision theorem for cosimplicial unstable coalgebras (Theorem~\ref{homotopy-excision}) shows that the square is 
also homotopy $(m+n)$-cocartesian. The same theorem (or Lemma \ref{pull-n-conn}) shows also that $H_*(E^\bu) \to H_*(X^\bu)$ is $n$-connected 
and $H_*(E^\bu) \to H_*(Y^\bu)$ is $m$-connected. Then, by Proposition \ref{G-ho po to H-ho po 2}, $H_*(W)$ is the homotopy pushout of
\[
H_*(Y^\bu) \leftarrow H_*(E^\bu) \rightarrow H_*(X),
\]
and therefore the canonical map $H_*(W^\bu) \to H_*(Z^\bu)$ is $(m+n)$-connected. 
\end{proof}

An important consequence of homotopy excision is the following version of the Freudenthal suspension theorem for the $\mc{G}$-resolution model category of cosimplicial spaces. 
In the following statements, $\Sigma_{c}$ and $\Omega_{f}$ denote the derived suspension and loop functors on the homotopy category $\ho{c\mc{M}_*^{\mc{G}}}$ of 
{\it pointed cosimplicial spaces}.

\begin{corollary} 
Let $X^{\bu}$ be a pointed cosimplicial space which is cosimplicially $n$-connected, i.e.\! $\pi^s H_t (X^\bu)= 0$ for $0 < s < n$, $\pi^0 H_0(X^\bu) = \F$, and $\pi^0H_t(X^\bu)=0$ for all $t>0$. Then the canonical map
  $$\Sigma_{c} \Omega_{f} (X^\bu) \to X^\bu$$
is $2n$-connected. 
\end{corollary}

In particular, we have

\begin{corollary} 
Let $X^{\bu}$ be as above and assume in addition that $\pi^s H_*(X^\bu) = 0$ for $s > 2n-1$. Then the canonical map
 $$\Sigma_{c} \sk^c_{2n+1} \Omega_{f} (X^\bu) \to X^\bu$$
is a $\mc{G}$-equivalence. In particular, $X^\bu$ admits a desuspension in $c\mc{S}^{\mc{G}}$.
\end{corollary}

\subsection{Objects of type $L_C(M,n)$} \label{section-L-objects}

In this subsection we define and construct objects of type $L_C(M,n)$ in $c \mc{S}^{\mc{G}}$. They are formal analogues 
of the twisted Eilenberg-MacLane spaces in the setting of this resolution model category. These objects are closely related
to the objects of type $K_C(M,n)$. This relationship is a central fact that will allow the `algebraization' of our 
obstruction theory. 

Recall the general notion of an \Hun-algebra as defined in 
Subsection~\ref{subsec:algebraic-structure-on-G-homotopy-groups}. 
In our context, this structure is equivalent to the structure of an unstable algebra (Theorem~\ref{thm:UA-equiv-H-alg}). 

\begin{definition} \label{L(C,0)}
Let $C$ be an object of $\CA$. An object $X^\bullet$ in $c\mc{S}^{\mc{G}}$ is said to be of {\it type $L(C,0)$} if 
there are isomorphisms of \Hun-algebras and \mc{H}-algebras respectively,
    $$ \naturalpi{s}{X^\bullet}{G}\cong\left\{
                   \begin{array}{cl}
                       \Hom_{\CA}(C,H_*G) & s=0 \\
                                 0        & \hbox{otherwise.}
                   \end{array} 
                                        \right. $$                                  
\end{definition}

\begin{definition} \label{L(C,n)}
Let $C \in \CA$, $M\in \V C$, and $n \geq 1$.
An object $X^\bullet$ in $c\mc{S}^{\mc{G}}$ is said to be of {\it type $L_C (M,n)$} if the following are satisfied:
\begin{itemize} 
\item[(a)] there are isomorphisms
$$\naturalpi{s}{X^\bullet}{G} \cong \left\{
                   \begin{array}{cl}
                   \Hom_{\CA}(C, H_*G)             & s=0 \\
                   \Hom_{C / \CA}(\iota_C(M),H_*G) & s=n \\ 
                             0                  & \hbox{otherwise}
                   \end{array} 
                                        \right. $$ 
where $H_*G$ is viewed as an object under $C$ with respect to the map 
  $$C \xrightarrow{\epsilon} \terminal \to H_*G. $$  
The first isomorphism is required to be an isomorphism of $\Hun$-algebras. 
The second one has to be an isomorphism of $\pi_0$-modules which are defined in \ref{def:pi-0-modules} and made explicit in 
Subsection~\ref{subsec:unstable-alg-equals-H-alg}.. 
\item[(b)] there is a map $\eta\co X^\bu \to X_0^\bu$ to an object of type $L(C,0)$ such that the composite map 
$\sk^c_1 X^\bu \to X^{\bullet} \to X_0^\bu$ is a $\mc{G}$-equivalence. 
\end{itemize}
A \emph{structured} object of type $L_C(M,n)$ is a pair $(X^\bu, \eta)$ where
$X^\bu$ is an object of type $L_C(M,n)$ and $\eta\co X^\bu \to X_0^\bu$ is a map as in (b) above. 
\end{definition}

We have the following equivalent characterization for an object of type $L_C(M,n)$. We will make use of the internal shift functor $M\mapsto M[1]$ which was 
defined in~\ref{def:internal-shift}. Recall that $\sk^c_n(-)$ denotes the derived skeleton functor.

\begin{lemma} \label{Rechnung mit Spiralsequenz}
Let $C \in \CA$, $M\in \V C$, and $n \geq 1$. For a cosimplicial space $X^\bullet \in c\mc{S}^{\mc{G}}$, the following statements are 
equivalent:
\begin{enumerate}
\item $X^\bu$ is of type $L_C (M,n)$.
\item  $X^\bu$ is $(n+1)$-skeletal, i.e.\! $\sk^c_{n+1}(X^\bu) \simeq X^\bu$, and 
\begin{itemize} 
 \item[(a)] there are isomorphisms
    $$ \pi^sH_*(X^\bullet)\cong\left\{
                   \begin{array}{ll}
                             C & s = 0 \\
                           C[1]& s = 2 \\
                             M & s=n \\
                           M[1]& s = n+2 \\
                             0 & \hbox{otherwise. }
                   \end{array} 
                                        \right. $$
\emph{(}if $n \neq 2$, otherwise modify accordingly\emph{)} where the first is an isomorphism of unstable coalgebras 
and the others are isomorphisms of $C$-comodules.
\item[(b)] as above. 
\end{itemize}
\end{enumerate}
\end{lemma}
\begin{proof}
The equivalence of the two statements is an easy consequence of the spiral exact sequence once one recalls the following facts. First, in Appendix~\ref{appsec:H-alg},
it is explained that an \Hun-algebra can be identified with an unstable algebra. Moreover, the $\pi_0$-module structure translates into the usual $\pi_0H^*(X^\bu)$-action on 
$\pi_sH^*(X^\bu)$. Secondly, since we work over a field $\F$ the setting dualizes nicely to unstable coalgebras and comodules.  
\end{proof}

We will often make use of the following key observation.

\begin{proposition}\label{prop:recognizing-L-objects}
Let $C \in \CA$, $M\in \V C$, and $n \geq 1$. Suppose $X^\bu$ is a cosimplicial space such that 
\begin{itemize}  
\item[(a)] there are isomorphisms 
$$ \pi^sH_*(X^\bullet)\cong\left\{
                   \begin{array}{ll}
                             C & s = 0 \\
                           C[1]& s = 2 \\
                             M & s=n   \\
                             0 & 0< s < n\, ,\ s \neq 2
                   \end{array} 
                                        \right. $$
\emph{(}if $n \neq 2$, otherwise modify accordingly\emph{)}, where the first is an isomorphism of unstable coalgebras and the others are isomorphisms of $C$-comodules.
\item[(b)] as above. 
\end{itemize}
Then $\sk^c_{n+1}(X^\bu)$ is an object of type $L_C(M,n)$. 
\end{proposition}
\begin{proof}
Passing to the degreewise $\F$-vector space duals $(-)^{\dual}$, we obtain from (a) isomorphisms between the terms $\pi_sH^*(X^\bu)$ and the corresponding duals on the right side for $0\le s\le n$. We consider the spiral exact sequence for $\sk^c_{n+1}(X^\bu)$ whose natural homotopy groups vanish above dimension $n$. This forces an isomorphism
  $$ \bigl(\{\pi_{n+2}H^m(\sk^c_{n+1}(X^\bu))\}_{m \ge 0}\bigr)\cong\bigl(\{\naturalpi{n}{X^\bu}{K(\F,m-1)}\}_{m\ge 0}\bigr)^{\dual}\cong M[1] $$
as $C$-comodules. Here the middle term is a $C$-comodule by Proposition~\ref{lem:pi0-identified}. By the same lemma, the 
left side has the usual $C\cong\pi_0H^*(X^\bu)$-comodule structure. So we have checked condition (2) from Lemma~\ref{Rechnung mit Spiralsequenz}.
\end{proof}

Similarly to $K$-objects, it will also be convenient to allow objects of type $L_C(M,n)$ for $n = 0$.

\begin{definition}
Let $C \in \CA$ and $M\in \V C$. An object of \emph{type $L_C(M, 0)$} is simply an object of type $L(\iota_C(M), 0)$. A \emph{structured 
object} of type $L_C(M, 0)$ is a pair $(X^\bu, \eta)$ where $X^\bu$ is an object of type $L_C(M,0)$ and $\eta\co X^\bu \to X_0^\bu$ is a map 
to an object of type $L(C,0)$ which induces up to isomorphism the natural projection $\iota_C(M) \to C$ on $\pi^0 H_*$-groups. 
\end{definition}

We now show how to construct objects of type $L_C(M,n)$. 

\begin{construction} Let $C \in \CA$. We first construct an object of type $L(C,0)$. The start of a 
cofree resolution of $C$ gives a pullback square (in $\Vec$ - but the maps are in $\CA$)
\[ 
\xymatrix{
 C \ar[r] \ar[d] & I^0 \ar[d]^{d_0} \\
 I^0 \ar[r]^{d_1} & I^1 .
} 
\]
There exist $\F$-GEMs $G^0$ and $G^1$ and  maps $d'_0,d'_1\co G^0 \to G^1$ such that $d_j = H_*(d'_j)$. Consider the homotopy pullback square in $c\mc{S}^{\mc{G}}$
$$
\xymatrix{ 
X^\bu \ar[r]\ar[d] & c G^0 \ar[d] \\ 
c G^0 \ar[r] & c G^1
}
$$
By Proposition \ref{G-ho pb to H-ho pb} and Theorem~\ref{homotopy-excision}, the induced square after applying $H_*$ 
is homotopy $1$-cartesian in $c\Vec$. This means that $\pi^0H_*(X^\bu)$ is isomorphic to $C$ and the isomorphism is induced by a canonical map $cC \to H_*(X^\bu)$ in $\ho{c\CA^{\mc{E}}}$. It follows that $\sk^c_{1} X^\bu$ is of type $L(C,0)$. 
\end{construction}

\begin{construction}
Let $C \in \CA$ and $M \in \V C$. We give a construction of an object of
type $L_C(M,1)$. Consider a homotopy pullback square as follows:
\[
\xymatrix{
E^\bullet \ar[r] \ar[d] & L(C, 0) \ar[d] \\
L(C, 0) \ar[r] & L(\iota_C(M), 0)
}
\]
where the indicated $L$-objects are given by the construction above. 
We claim that
$\sk_2 H_*(E^\bullet)$ is an object of type $K_C(M, 1)$.
In fact, this
requires a little bit more than what is immediately deducible from the
homotopy excision theorem because the latter only shows
that $\pi^1 H_*(E^\bullet)$ injects into $M$. However, there is a map
of squares from the square below (cf. Construction \ref{alternative-direct-construction}):
\[
\xymatrix{
K_C(M, 1) \ar[r] \ar[d] & K(C, 0) \ar[d] \\
K(C, 0) \ar[r] & K(\iota_C(M), 0)
}
\]
to the square induced by the homotopy pullback above
\[
\xymatrix{
\sk_2 H_*(E^\bullet) \ar[r] \ar[d] & H_*L(C, 0) \ar[d] \\
H_*L(C, 0) \ar[r] & H_*L(\iota_C(M), 0)
}
\]
which implies that $M$ is also a retract of $\pi^1 H_*(E^\bullet)$, thus proving the claim.

Now we claim that $\sk_2^c(E^\bullet)$ is of type $L_C(M,1)$. The spiral
exact sequence shows the required isomorphism on
$\naturalpi{0}{-}{-}$-groups. Moreover, there is an epimorphism
$$\naturalpi{1}{\sk_2^c E^\bullet}{G} \to \pi_1 [\sk_2^c E^\bullet, G]
\cong \Hom_{C / \CA}(\iota_C(M),H_*G).$$
The last isomorphism is a consequence of the spiral exact sequence
and the fact that $\sk_2^c(E^\bu) \to E^\bu$ induces isomorphisms on
the first two natural homotopy groups. Thus it suffices to show
that this epimorphism is actually an isomorphism. This follows from
the spiral exact sequence after we note that there is a commutative
diagram
\[
\xymatrix{
\naturalpi{0}{L(C,0)}{\Omega G} \ar[d]^{\cong} \ar[r] &
\naturalpi{1}{L(C,0)}{G} = 0 \ar[d] \\
\naturalpi{0}{\sk_2^c E^\bu}{\Omega G} \ar[r] & \naturalpi{1}{\sk_2^c E^\bu}{G}
}
\]
which is induced by the map $\sk_2^c E^\bu \to E^\bu \to L(C,0)$ and
therefore the bottom connecting map is trivial.
\end{construction}

\begin{construction}
Let $ n > 1$, $C \in \CA$ and $M \in \V C$. We give an inductive construction of an object of type $L_C(M,n)$. Let $X^\bu$ be of type $L_C(M, n-1)$ and consider the homotopy pullback square 
\[
\xymatrix{
 E^\bu \ar[r] \ar[d] & \sk^c_1(X^\bu) \ar[d] \\
 \sk^c_1(X^\bu) \ar[r] & X^\bu
 }
\]
where $\sk_1^c(X^\bu)$ is an object of type $L(C,0)$. Then, by Theorem \ref{homotopy-excision 2}, it follows 
easily that $\sk^c_{n+1}(E^\bu)$ is an object of type $L_C(M, n)$. We can always regard this as structured by 
declaring one of the maps to $\sk^c_1(X^\bu)$ to be the structure map.
\end{construction}

\begin{remark}
Let $(X^\bullet, \eta\co X^\bu\to X_0^\bu)$ be a structured object of type $L_C(M,n)$ and $n > 0$. Then the homotopy pushout of the maps 
$$X_0^\bu \stackrel{\eta}{\leftarrow} X^\bu \stackrel{\eta}{\rightarrow} X_0^\bu$$
is a structured object of type $L_C(M, n-1)$.  
\end{remark}

It is a consequence of Proposition \ref{Moduli of L(C,k)} below that structured objects of type $L_C(M, n)$ are homotopically unique. To show this, we will compare objects of type $L_C(M,n)$ with objects of type $K_C(M, n)$. 

We fix a cofibrant choice of an object of type $L_C(M,n)$, denoted $L_C(M, n)$. 
Although $H_*(L_C(M,n))$ is not an object of type $K_C(M,n)$, we can extract such an object as follows. Consider the homotopy pullback square: 
\[
 \xymatrix{
 D^\bu \ar[r]^-{j} \ar[d] & H_*(L_C(M,n)) \ar[d] \\
 \sk_1H_*(L(C,0)) \ar[r] & H_*(L(C,0))
 }
\]
where the right hand side vertical map is the structure map and $\sk_1H_*(L(C,0))$ is an object of type $K(C,0)$. 
Then, by the homotopy excision theorem (Theorem \ref{homotopy-excision}), we can conclude that $\sk_{n+1}^c(D^\bu)$ is a structured object of type $K_C(M,n)$. 
Set $K_C(M,n) : = \sk_{n+1}^c(D^\bu)$. There is a functor
$$\phi( X^\bullet)\co \mc{W}^{\rm c}_{\rm Hom}\bigl(L_C(M,n), X^\bu\bigr) \to \mc{W}^{\rm c}_{\rm Hom}\bigl(K_C(M,n), H_*X^\bu\bigr)$$ 
which is defined on objects by sending 
$$L_C(M, n) \xrightarrow{f} U \stackrel{\simeq}{\longleftarrow} X^\bu$$
to the object
$$K_C(M, n) \to H_*(U) \stackrel{\simeq}{\longleftarrow} H_*(X^\bu)$$
where the first map is canonically defined by $j$ and $f$. See Remark~\ref{X-cof-zigzag} for the definition of $\mc{W}^{\rm c}_{\rm Hom}(-, -)$ and its properties.

\begin{proposition} \label{Representability of L(C,n)}
The functor  $\phi(X^\bu)$ induces a weak equivalence of classifying spaces.
\end{proposition}

\begin{proof}
We can always replace $X^\bullet$ up to \mc{G}-equivalence by a quasi-\mc{G}-cofree object using Proposition \ref{prop:quasi-cofree-replacements}. So we may assume that 
$X^\bullet$ is quasi-\mc{G}-cofree. Since both source 
and target of the map $\phi(X^\bullet)$ preserve homotopy pullbacks in $X^\bu$, we can use inductively the decomposition in Definition \ref{def:quasi-G-cofree}, for 
the quasi-\mc{G}-cofree map $X^\bu \to \ast$, in order to reduce the proof to the special case 
  $$X^\bullet= \hom^{\rm ext}(\partial \Delta^s, cG)$$ 
for a fibrant \mc{G}-injective object $G$. 
Using Proposition \ref{DK-theory4} and Remark \ref{X-cof-zigzag}, we pass to the corresponding (derived) simplicial mapping spaces.  
We have
   $$\map^{\rm der}\bigl(L_C(M,n),(cG)^{\partial \Delta^s}\bigr) \simeq \map^{\rm der}\bigl(\partial \Delta^s, \map^{\rm der}(L_C(M,n), cG)\bigr)$$
and
  $$\map^{\rm der}\bigl(K_C(M,n), (cH_*G)^{\partial \Delta^s}\bigr) \simeq \map^{\rm der}\bigl(\partial \Delta^s, \map^{\rm der}(K_C(M,n), cH_*G)\bigr)$$
The map $\phi((cG)^{\partial \Delta^s})$ is homotopic to $(\phi(cG))^{\partial \Delta^s}$. Therefore, it suffices to consider only the 
case of $\phi(cG)$. Inspection shows that the maps on the two non-trivial homotopy groups of these mapping spaces, $\pi_0$ and $\pi_n$, are the duals of the 
maps on $\pi^0$ and $\pi^n$, respectively, induced by $K_C(M,n) \to D^\bu \to H_*(L_C(M,n))$, thus they are isomorphisms.
\end{proof}

We record an obvious variation of the weak equivalence $\phi(X^\bu)$ for later use. Let $\theta$ denote a property of maps in $c \CA^{\mc{H}}$ which is invariant under weak 
equivalences. Examples of such properties are: (a) the map induces isomorphisms on certain cohomotopy groups, or (b) that it is a weak equivalence (the latter appears later 
in Section \ref{moduli-spaces}). Let 
$$\mc{W}^{\rm c}_{\theta}\bigl(L_C(M,n), X^\bullet\bigr) \subset \mc{W}^{\rm c}_{\rm Hom}\bigl(L_C(M,n), X^\bullet\bigr)$$ 
denote the full subcategory defined by objects $f\co L_C(M,n) \to U \simeq  X^\bullet$ such that
$$\phi(X^\bu)(f): K_C(M, n) \to H_*(U) \simeq H_*(X^\bu)$$ 
has property $\theta$. Similarly let 
  $$\mc{W}^{\rm c}_{\theta}\bigl(K_C(M,n), D^\bullet\bigr) \subset \mc{W}^{\rm c}_{\rm Hom}\bigl(K_C(M,n), D^\bullet\bigr)$$ 
denote the full subcategory defined by maps $u\co K_C(M, n) \to V \simeq D^\bullet$ which have property $\theta$.
Note that each of these subcategories is a (possibly empty) union of connected components of the corresponding categories. 
Then the following is an immediate consequence of Proposition \ref{Representability of L(C,n)}. 

\begin{proposition} \label{Representability of L(C,n)-2} 
Let $\theta$ be a property of maps in $c \CA^{\mc{H}}$ which is invariant under weak equivalences. Then the functor $\phi(X^\bu)$ restricts to a weak equivalence between 
the classifying spaces of the subcategories associated with $\theta$,
$$B\mc{W}^{\rm c}_{\theta}\bigl(L_C(M,n), X^\bullet\bigr) \xrightarrow{\simeq} B\mc{W}^{\rm c}_{\theta}\bigl(K_C(M,n), D^\bullet\bigr).$$
\end{proposition}

\begin{remark}
Fix a structure map $L_C(M,n) \to L(C,0)$ and a section up to homotopy $L(C, 0) \to L_C(M, n)$. Let $Y^\bu$ be a cosimplicial space and $L(C,0) \to Y^\bu$ a map in $c \mc{S}^{\mc{G}}$. 
Using standard homotopical algebra arguments, it can be shown that Proposition \ref{Representability of L(C,n)} extends to yield weak equivalences between mapping spaces in the respective
slice categories. That is, there are weak equivalences
  $$ \map^{\rm der}_{L(C,0)/c \mc{S}^{\mc{G}}}\bigl(L_C(M, n), Y^\bu \bigr) \xrightarrow{\simeq} \map^{\rm der}_{cC/c\CA}\bigl(K_C(M, n), H_*(Y^\bu) \bigr).$$
In particular, it follows that the structured pointed object $L_C(M, n)$ represents Andr\'{e}-Quillen cohomology in cosimplicial spaces by applying 
Theorem~\ref{representability of AQ}.
\end{remark}

We now determine the homotopy type of the moduli space of structured objects of type $L_C(M,n)$:
  $$\mc{M}\bigl(L_C(M, n) \twoheadrightarrow L(C,0)\bigr).$$
The definition of this space is completely analogous to the corresponding moduli space of $K$-objects, see Subsection~\ref{Objects of type K}. It is the classifying space of the category $\mc{W}\bigl(L_C(M,n) \twoheadrightarrow L(C,0)\bigr)$ whose objects are structured objects of type $L_C(M,n)$:
  $$X^\bu \xrightarrow{\eta} X^\bu_0$$
and morphisms are square diagrams as follows
\[
 \xymatrix{
 X^{\bullet} \ar[r]^{\eta} \ar[d]_{\simeq} & X_0^{\bullet} \ar[d]^{\simeq} \\
 Y^{\bullet} \ar[r]^{\eta'} & Y_0^{\bullet}
}
\]
where the vertical arrows are weak equivalences in $c \mc{S}^{\mc{G}}$.

\begin{proposition} \label{Moduli of L(C,k)}
Let $C \in \CA$, $M \in \V C$, and $n \ge 0$. Then:
\begin{itemize} 
\item[(a)] There is a weak equivalence $\mc{M}(L(C,0)) \simeq B\Aut(C).$
\item[(b)] There is a weak equivalence $\mc{M}(L_C(M, 0) \twoheadrightarrow L(C,0)) \simeq B \Aut_C(M)$.
\item[(c)] There is a weak equivalence
   $$\mc{M}\bigl(L_C(M, n+1) \twoheadrightarrow L(C,0)\bigr) \simeq \mc{M}\bigl(L_C(M, n) \twoheadrightarrow L(C,0)\bigr).$$
\end{itemize}  
\end{proposition}
\begin{proof}
(a) By Proposition \ref{Moduli of K(C,k)}, it suffices to show that $\mc{M}(L(C,0)) \simeq \mc{M}(K(C,0))$. We apply
Proposition \ref{Representability of L(C,n)} and compare directly the spaces of homotopy automorphisms of these 
two objects. By Proposition \ref{Representability of L(C,n)}, we have a weak equivalence 
$$\map^{\rm der}\bigl(L(C,0), L(C,0)\bigr) \simeq \map^{\rm der}\bigl(K(C,0), H_*(L(C,0))\bigr).$$
There is a functor 
$$F\co \mc{W}^{\rm c}_{\rm Hom}\bigl(K(C,0), H_*(L(C,0))\bigr) \to \mc{W}_{\rm Hom}\bigl(K(C,0), \sk_2H_*(L(C,0))\bigr)$$
which takes a zigzag $K(C,0) \rightarrow U \stackrel{\simeq}{\longleftarrow} H_*(L(C,0))$ to the zigzag
$$K(C,0) \stackrel{\simeq}{\longleftarrow} \sk_2(K(C,0)) \to \sk_2 U \stackrel{\simeq}{\longleftarrow} \sk_2 H_*(L(C,0))$$
where the last object is of type $K(C,0)$. It is easy to see that this functor defines a homotopy inverse to 
the obvious map induced by $\sk_2 H_*(L(C,0)) \to H_*(L(C,0))$. Then the result follows by passing to the 
appropriate components. 

(b) Let
  $$\mathscr{MAP}_{\twoheadrightarrow}\bigl(L_C(M,0), L(C,0)\bigr)~\text{ and }~\mathscr{MAP}_{\twoheadrightarrow}\bigl(K_C(M,0), K(C,0)\bigr)$$
be the classifying spaces of the categories 
  $$\mc{W}^{\rm c}_{\twoheadrightarrow}\bigl(L_C(M,0), L(C,0)\bigr)~\text{ and }~\mc{W}^{\rm c}_{\twoheadrightarrow}\bigl(K_C(M,0), K(C,0)\bigr)$$ 
respectively, where $\twoheadrightarrow$ denotes the property that the map induces up to isomorphism the canonical projection on $\pi^0$ 
(see Proposition \ref{Representability of L(C,n)-2}).
The results of Appendix \ref{DK-theory} (Theorem \ref{calculus-with-moduli-spaces}) show that there is a homotopy fiber sequence
\[
\xymatrix{ 
\mathscr{MAP}_{\twoheadrightarrow}\bigl(L_C(M,0), L(C,0)\bigr) \ar[r] & \mc{M}\bigl(L_C(M, 0)\twoheadrightarrow L(C,0)\bigr) \ar[d] \\
& \mc{M}\bigl(L_C(M, 0)\bigr) \times \mc{M}\bigl(L(C,0)\bigl)
}
\]
The functor $\sk_2 H_*$ defines a map from this homotopy fiber sequence to the
associated homotopy fiber sequence of $K$-objects:
\[
\xymatrix{
\mathscr{MAP}_{\twoheadrightarrow}\bigl(K_C(M,0), K(C,0)\bigr) \ar[r] & \mc{M}\bigl(K_C(M, 0)\twoheadrightarrow K(C,0)\bigr)\ar[d] \\
& \mc{M}\bigl(K_C(M, 0)\bigr) \times \mc{M}\bigl(K(C,0)\bigr)
}
\]
which, by (a), is a weak equivalence on base spaces and fibers. The result
then follows from Proposition \ref{Moduli of K(C,k)}(b).

(c) Similarly to Proposition \ref{Moduli of K(C,k)} using Theorem \ref{homotopy-excision 2} (or Proposition \ref{diff. constr. in S}
below). 
\end{proof} 

As a consequence, we have the following

\begin{corollary}
Let $C \in \CA$, $M \in \V C$ and $n \ge 0$. Then there is a weak equivalence 
$$\mc{M}\bigl(L_C(M,n) \twoheadrightarrow L(C,0)\bigr) \simeq B \Aut_C(M).$$
In particular, the moduli spaces of structured $L$-objects are path-connected.  
\end{corollary}

\subsection{Postnikov decompositions} \label{section-postnikov-decomp-spaces}

Similarly to the case of unstable coalgebras, the homotopy excision theorem for 
cosimplicial spaces shows that the skeletal filtration of a cosimplicial space is principal, that is,
the successive inclusions can be described as homotopy pushouts of a certain type. This is a direct 
consequence of the following proposition.

\begin{proposition} \label{diff. constr. in S} 
Let $f\co X^\bu \to Y^\bu$ be an $n$-connected map in $c \mc{S}^{\mc{G}}$, $n\geq 1$, and $C = \pi^0 H_*(Y^\bu)$. 
Let $Z^\bu$ be the homotopy cofiber of $f$ and $M =\pi^n (H_* (Z^\bu ))$. Consider the homotopy pullback 
\[
\xymatrix{
E^\bu \ar[d]_{f'} \ar[r] & X^\bu \ar[d]^{f} \\
\sk^c_1(Y^\bu)  \ar[r] & Y^\bu 
}
\]
Then:
\begin{itemize}
\item[(a)] $M$ is naturally a $C$-comodule and the object $\sk^c_{n+2}(E^\bu)$ together with the canonical map 
$\sk^c_{n+2}(E^\bu) \to \sk^c_1(Y^\bu)$ define a structured object of type $L_{C}(M,n+1)$. 
\item[(b)] If $\naturalpi{s}{(Z^\bullet, \ast)}{G}$ vanishes for all $s \neq n$, then the diagram 
\[
\xymatrix{
\sk^c_{n+2}(E^\bu) \ar[d] \ar[r] & X^\bu \ar[d]^{f} \\
\sk^c_1(Y^\bu)  \ar[r] & Y^\bu 
}
\] 
is a homotopy pushout.
\end{itemize}
\end{proposition}

\begin{proof}
(a) This is a consequence of the homotopy excision theorems. The induced square 
\[
 \xymatrix{
 H_*(E^\bu) \ar[d]_{H_*(f')} \ar[r] & H_*(X^\bu) \ar[d]^{H_*(f)} \\
H_*(\sk^c_1(Y^\bu))  \ar[r] & H_*(Y^\bu) 
}
\]
is a homotopy pullback by Proposition \ref{G-ho pb to H-ho pb}. Since $f$ is $n$-connected, Theorem~\ref{homotopy-excision}(a) shows that $f'$ is 
$n$-connected and $\pi^{n+1}H_*(E^\bu)$ is isomorphic to $M$. It follows from Lemma~\ref{Rechnung mit Spiralsequenz} and 
Proposition~\ref{prop:recognizing-L-objects} that $\sk^c_{n+2}(E^\bu)$ is an object of type $L_C(M, n+1)$.

Part (b) follows easily from the long exact sequence in Proposition~\ref{G-cofiber-to-LES}. 
\end{proof}

Similarly to the situation in $c\CA^{\mc{E}}$ in Subsection~\ref{postnikov_decomp_unst_coalg}, it follows that for every $X^\bu \in c \mc{S}^{\mc{G}}$, the skeletal filtration 
$$ \sk^c_1(X^\bu) \to \sk^c_2(X^\bu) \to \cdots \to \sk^c_n(X^\bu) \to \cdots \to X^\bu$$
is defined by homotopy pushouts as follows
\[
 \xymatrix{
 L_C(M, n + 1) \ar[r]^-{w_n} \ar[d] & \sk_n^c(X^\bu) \ar[d] \\
 L(C,0) \ar[r] & \sk_{n+1}^c(X^\bu)
 }
\]
where $C = \pi^0(H_*(X^\bu))$ and $M$ identifies the $n$-th natural homotopy group of $X^\bu$. The collection of maps $w_n$ may be regarded as analogues 
of the Postnikov $k$-invariants in the context of the resolution model category $c \mc{S}^{\mc{G}}$.

The next proposition reformulates and generalizes Proposition \ref{diff. constr. in S} in terms of moduli spaces. It is analogous to Proposition \ref{Moduli+Diff. = !}. 
First we need to introduce some new notation. 

For $n \geq 1$, let $X^\bu$ be a cosimplicial space such that the inclusion $\sk^c_n(X^\bu)\stackrel{\simeq}{\longrightarrow} X^\bu$ is a \mc{G}-equivalence. 
Let $C=\pi^0(H_*(X^\bu))$ and $M$ be a coabelian $C$-comodule. Then we define
  $$\mc{W}(X^\bu + (M,n))$$ 
to be the category with objects cosimplicial spaces $Y^\bu$ such that:
\begin{itemize} 
\item[(i)] $Y^\bu$ is $(n+1)$-skeletal, i.e., $\sk^c_{n+1}(Y^\bu) \simeq Y^\bu$, 
\item[(ii)] the derived $n$-skeleton $\sk^c_n(Y^\bu)$ is $\mc{G}$-equivalent to $X^\bu$,
\item[(iii)] $\pi^n(H_*(Y^\bu))$ is isomorphic to $M$ as a $C$-comodule.
\end{itemize}
The morphisms are given by $\mc{G}$-equivalences between cosimplicial spaces. The classifying space of this category, denoted by 
  $$\mc{M}(X^\bu + (M,n)),$$
is the moduli space of $(n+1)$-skeletal extensions of the $n$-skeletal object $X^\bu$ by the comodule $M$.

We define 
  $$ \mc{M}\bigl(L(C,0) \twoheadleftarrow L_C(M, n+1) \stackrel{{\rm 1-con}}{\rightsquigarrow} X^\bu\bigr),$$
analogously to Proposition \ref{Moduli+Diff. = !}, as the classifying space of the category
  $$ \mc{W}\bigl(L(C,0) \twoheadleftarrow L_C(M, n+1) \stackrel{{\rm 1-con}}{\rightsquigarrow} X^\bu\bigr).$$
This category has the following objects: diagrams $T\leftarrow U\to V$ in $c\mc{S}^{\mc{G}}$ 
such that $U\to T$ is a structured object of type $L_C(M,n+1)$, $V$ is \mc{G}-equivalent to 
$X^\bu$, and $U\to V$ is cosimplicially $1$-connected. The morphisms are given by weak equivalences between such diagrams. We refer to Subsection~\ref{subsec:moduli-spaces} for some background about these moduli spaces.

\begin{proposition} \label{Moduli+Diff. = !!}
Let $n \geq 1$ and suppose that $X^\bu$ is an object of $c\mc{S}^{\mc{G}}$ such that $\sk^c_{n} X^\bu \stackrel{\sim}{\to} X^\bu$. 
Let $M$ be an object of $\V C$, where $C=\pi^0H_*(X^\bu)$. Then there is a natural weak equivalence 
$$\mc{M}\bigl(X^\bu + (M,n)\bigr) \simeq \mc{M}\bigl(L(C,0) \twoheadleftarrow L_C(M, n+1) \stackrel{{\rm 1-con}}{\rightsquigarrow} X^\bu\bigr).$$
\end{proposition}

\begin{proof}
Similar to the proof of Proposition \ref{Moduli+Diff. = !} using Proposition \ref{diff. constr. in S}.   
\end{proof}

The relationship between $L$- and $K$-objects can be further refined to show that `attachments' of structured $L$-objects 
are determined `algebraically' by the corresponding `attachments' of structured $K$-objects. This can be expressed elegantly
in terms of a homotopy pullback of the respective moduli spaces. 
See Appendix~\ref{DK-theory} for the definition and properties 
of these moduli spaces of maps. 

\begin{theorem} \label{Moduli-Comparison-of-Spaces} 
Let $X^\bu \in c \mc{S}^{\mc{G}}$, $C \in \CA$, and $M \in \V C$. Then there is a homotopy pullback square 
for every $n \geq 0$:
\[
\xymatrix{
\mc{M}\bigl(L(C,0) \twoheadleftarrow L_C(M, n) \stackrel{{\rm 1-con}}{\rightsquigarrow} X^\bu\bigr) \ar[r] \ar[d] & \mc{M}\bigl(K(C,0) \twoheadleftarrow K_C(M,n) \stackrel{{\rm 1-con}}{\rightsquigarrow} H_*(X^\bu)\bigr) \ar[d] \\
\mc{M}(X^\bu) \ar[r] & \mc{M}(H_*(X^\bu)).
}
\]
\end{theorem}

\begin{proof}
The top map is induced by a functor 
$$\mc{W}\bigl(L(C, 0) \twoheadleftarrow L_C(M, n) \stackrel{{\rm 1-con}}{\rightsquigarrow} X^\bu\bigr) \to \mc{W}\bigl(K(C, 0) \twoheadleftarrow K_C(M,n) \stackrel{{\rm 1-con}}{\rightsquigarrow} H_*(X^\bu)\bigr)$$
which is defined following the recipe for the definition of the functor $\phi(X^\bu)$ and Proposition \ref{Representability of L(C,n)}. Consider the 
following factorization of the diagram
\[
\xymatrix{
\mc{M}\bigl(L(C, 0) \twoheadleftarrow L_C(M, n) \stackrel{{\rm 1-con}}{\rightsquigarrow} X^\bu\bigr) \ar[r] \ar[d] & \mc{M}\bigl(K(C, 0) \twoheadleftarrow K_C(M,n) \stackrel{{\rm 1-con}}{\rightsquigarrow} H_*(X^\bu)\bigr) \ar[d] \\
\mc{M}\bigl(L(C,0) \twoheadleftarrow L_C(M,n)\bigr) \times \mc{M}(X^\bu) \ar[r] \ar[d] & \mc{M}\bigl(K(C,0) \twoheadleftarrow K_C(M,n)\bigr) \times \mc{M}(H_*(X^\bu)) \ar[d] \\
\mc{M}(X^\bu) \ar[r] & \mc{M}(H_*(X^\bu)).
}
\]
It suffices to show that both squares are homotopy pullbacks. By Propositions~\ref{Representability of L(C,n)} and \ref{Representability of L(C,n)-2} 
and Theorem~\ref{calculus-with-moduli-spaces}, the induced map between the homotopy fibers of the top pair of vertical maps is a weak equivalence. 
Since the spaces in the middle are path-connected, the top square is a homotopy pullback. The bottom square is a homotopy pullback because the map
  $$ \mc{M}\bigl(L(C,0) \twoheadleftarrow L_C(M,n)\bigr)\stackrel{\simeq}{\longrightarrow}\mc{M}\bigl(K(C,0) \twoheadleftarrow K_C(M,n)\bigr), $$
induced by $\phi(X^\bu)$, is a weak equivalence by previous results on the moduli spaces of structured $K$- and $L$-objects (Proposition~\ref{Moduli of K(C,k)} 
and Proposition~\ref{Moduli of L(C,k)}).
\end{proof}

\begin{remark} \label{Moduli-Comparison-of-Spaces-2} 
Let 
  $$\mc{M}\bigl(L(C, 0) \twoheadleftarrow L_C(M,n) \Rightarrow X^\bu\bigr)$$ 
denote the moduli subspace of maps $W \leftarrow U \stackrel{f}{\rightarrow} V$, where $U \rightarrow W$ is a structured object of 
type $L_C(M,n)$, $V \simeq X^\bu$, and $\phi(V)(f)$ is a weak equivalence. Similarly, let 
  $$\mc{M}\bigl(K(C, 0) \twoheadleftarrow K_C(M,n) \xrightarrow{\sim} C^\bu\bigr)$$ 
denote the 
moduli subspace of maps $W \leftarrow U \stackrel{f}{\rightarrow} V$, where $U \rightarrow W$ is a structured object of type $K_C(M,n)$, 
$V \simeq C^\bu$, and $f$ is a weak equivalence. By Proposition~\ref{Representability of L(C,n)-2}, it follows that
the corresponding statement where we replace the arrows $\rightsquigarrow$ with arrows $\Rightarrow$ in Theorem \ref{Moduli-Comparison-of-Spaces} 
is also true with the same proof. 
\end{remark}

\section{Moduli spaces of topological realizations} \label{moduli-spaces}

In this section, we give a description of the moduli space $\mc{M}_{\rm Top}(C)$ of topological realizations 
of an unstable coalgebra $C$. This moduli space is a space whose set of path components is the set of 
non-equivalent realizations, and the homotopy type of each component is that of the homotopy automorphisms 
of the corresponding realization. That is, $\mc{M}_{\rm Top}(C)$ is homotopy equivalent to 
$$\bigsqcup_{X} B \mathrm{Aut^h}(X)$$
where the disjoint union is indexed over spaces $X$ with $H_*(X) \cong C$, one in each equivalence class. The actual 
definition of $\mc{M}_{\rm Top}(C)$ is given as a moduli space in the sense of Appendix \ref{DK-theory}. This point 
of view, due to Dwyer and Kan, is essential in what follows. We emphasize that we  work here with the Bousfield 
localization of spaces at $H_*(-, \mathbb{F})$-equivalences, and accordingly equivalence classes of realizations 
are understood in this localized sense. 

The description of the realization space $\mc{M}_{\rm Top}(C)$ is given in terms of a tower of moduli spaces of approximate realizations
$$\mc{M}_{\infty}(C) \to \ldots \to \mc{M}_n(C) \to \mc{M}_{n-1}(C) \to \ldots \to \mc{M}_0(C).$$
The precise meaning of this approximation is given by the notion of a \emph{potential $n$-stage for $C$} which will be explained 
in Subsection \ref{potential n-stages}. This tower of moduli spaces is determined recursively by the Andr\'{e}-Quillen 
cohomology spaces of $C$, i.e.\! spaces whose homotopy groups are Andr\'{e}-Quillen cohomology groups of $C$. The main results 
about the homotopy types of these moduli spaces are obtained in Subsection \ref{main-results}. In Subsection \ref{marked-moduli}, 
we discuss some variations of these moduli spaces where the objects are also equipped with appropriate markings by isomorphisms. 
Finally, in Subsection \ref{obstruction-theories}, we discuss how these results readily yield obstruction theories for the existence and 
uniqueness of realizations in terms of the Andr\'{e}-Quillen cohomology of the unstable coalgebra. 

The arguments of this section are heavily based on the homotopy excision theorems of the previous sections. This means 
that we are going to make frequent use of homotopy pullbacks and pushouts. Since these constructions will be required
to be functorial for the arguments, we assume from the start fixed functorial models for such constructions, which can be 
made using standard methods of homotopical algebra, and omit the details pertaining to this or related issues.

\subsection{Potential $n$-stages} \label{potential n-stages} 
Our description of the realization space of an unstable coalgebra will be given in terms of a sequence of moduli spaces 
of cosimplicial spaces which, in a certain sense, approximate actual topological realizations regarded as constant 
cosimplicial objects. The meaning of this approximation is expressed in the model category $c\mc{S}^{\mc{G}}$, and not 
in the model category of spaces. 

If $X$ is a realization of $C$, i.e.\! $H_*(X) \cong C$, then the associated cosimplicial object $cX \in c\mc{S}^{\mc{G}}$ 
obviously has the property that $H_*(cX)$ is an object of type $K(C,0)$. More generally, it will be shown in Theorem 
\ref{infty-stages-are-reals} that if $X^{\bu}$ is a (fibrant) cosimplicial space such that $H_*(X^\bu)$ is an object of 
type $K(C,0)$, then $\mathrm{Tot}(X^\bu)$ 
is a realization of $C$ \emph{assuming} the convergence of the homology spectral sequence for Tot. This motivates the 
following notion of a cosimplicial space realizing (a cosimplicial resolution of) $C$. 

\begin{definition}\label{def:infty-stage}
A cosimplicial space $X^\bu$ is called an \emph{$\infty$-stage} for $C$ if $H_*(X^\bu)$ is an object of type $K(C,0)$. 
\end{definition}

Before we state the following definition of an approximate $\infty$-stage, let us first discuss how one could 
attempt to construct an $\infty$-stage. The cosimplicial space $L(C,0)$, which is already at our disposal, 
satisfies 
  $$H_*(L(C,0)) \in \mc{W}\bigl(K(C,0) + (C[1], 2)\bigr).$$ 
A natural strategy then is to study the obstruction for killing the non-trivial $C$-comodule in $\pi^2$ of $H_*(L(C,0))$. 
Assuming this can be done, then the spiral exact sequence will force a new non-trivial $C$-comodule in $\pi^3$ of a 
recognizable form. This observation motivates the following definition. 

\begin{definition}\label{def:n-stage}
A cosimplicial space $X^\bu$ is called a \emph{potential $n$-stage} for $C$, where $n \geq 0$, if 
\begin{itemize}
 \item[(a)]
$H_*(X^\bu) \in \mc{W}\bigl(cC + (C[n+1], n + 2)\bigr)$, i.e.,
$$\pi^s H_*(X^\bullet) \cong \left\{
                   \begin{array}{cl}
                       C      & s = 0 \\ 
                       C[n+1] & s=n + 2 \\
                       0      & \hbox{otherwise, }
                   \end{array} 
                                        \right.$$
where the isomorphisms are between unstable coalgebras and $C$-comodules respectively.
 \item[(b)] $\naturalpi{s}{X^\bu}{G} = 0$ for $s > n$ and $G\in\mc{G}$.
  \end{itemize} 
We call $X^\bu$ a \emph{weak potential $n$-stage} in case it satisfies only (a).
\end{definition}

Potential $n$-stages were introduced by Blanc \cite[5.6]{Blanc:coalg} under the name of Postnikov sections or approximations. The name 
here is inspired by the corresponding notion introduced in \cite{BlDG:pi-algebra} for the realization problem of a $\Pi$-algebra. 

We have the following recognition property for potential $n$-stages.

\begin{proposition} \label{alt-def-potential}
A cosimplicial space $X^\bu$ is a potential $n$-stage for $C$  if and only if the following three conditions are 
satisfied:
\begin{itemize}
 \item[(a)] There is an isomorphism of \Hun-algebras
 $$\naturalpi{0}{X^\bu}{G} \cong \Hom_{\CA}(C, H_*(G)).$$
 \item[(b)] $\naturalpi{s}{X^\bu}{G} = 0$ for $s > n$ and $G\in\mc{G}$.
 \item[(c)] $\pi_s[X^\bu, G] = 0$ for $1 \leq s \leq n+1$ and $G\in\mc{G}$. 
\end{itemize}
\end{proposition}
\begin{proof}
There is an equivalence of \Hun-algebras and unstable algebras by Theorem~\ref{thm:UA-equiv-H-alg} with an inverse functor $u\co\Hun{\rm -Alg}\to\UA$ described in 
Corollary~\ref{cor:inverse-to-h}. Proposition~\ref{lem:pi0-identified} identifies the relevant $\pi_0$-module structures with $C^\dual\cong u(\pi_{0})$-module structures, where $\pi_0=\naturalpi{0}{X^\bu}{-}$.

Using the spiral exact sequence from Theorem \ref{ses-statement} and the cogenerating property of the Eilenberg-MacLane spaces $K(\F,m)\in\mc{G}$ in Proposition 
\ref{cogenerating property of GEMs}, condition (a) is equivalent 
to $\pi^0 H_*(X^\bu) \cong C$, as unstable coalgebras. Similarly, condition (c) is equivalent to $\pi^s H_*(X^\bu) = 0$ for $1 \leq s \leq n+1$. So both (a) and (c) are certainly 
true for potential $n$-stages, and (b) holds by definition. 

The converse is similar. The spiral exact sequence yields isomorphisms of $\pi_{0}$-modules
\begin{equation} \label{nat_htpy_potential_n-stage}
\naturalpi{s}{X^\bu}{-} \cong \Hom_{\CA}\bigl(C, H_*(\Omega^s (-))\bigr)
\end{equation}
for all $s \leq n$. The spiral exact sequence continues with
$$\cdots \to \naturalpi{n+2}{X^\bu}{G} \to \pi_{n+2}[X^\bu, G] \to \naturalpi{n}{X^\bu}{\Omega G} \to \naturalpi{n+1}{X^\bu}{G} \to \cdots$$
By assumption, it follows that the middle connecting map is an isomorphism. Then using \eqref{nat_htpy_potential_n-stage}, we obtain an isomorphism of $C$-comodules 
$$\pi^{n+2} H_*(X^\bullet) \cong C[n+1].$$
\end{proof}

From Proposition \ref{alt-def-potential}, together with Lemma \ref{lem:equivalent-formulations-n-connected}, it follows that if $X^\bu$ is a potential $n$-stage for $C$, then 
$\sk^c_{m+1}(X^\bu)$ is a potential $m$-stage for $C$ for all $m \leq n$. Here $\sk^c_n(-)$ denotes the derived $n$-skeleton, i.e., the $n$-skeleton
of a functorial $\mc{G}$-cofibrant replacement. Thus, a potential $n$-stage should be thought of as the $n$-skeletal
truncation in $c\mc{S}^{\mc{G}}$ of a potential $\infty$-stage. 

\begin{proposition} \label{atomicity}
Let $C$ be an unstable coalgebra and  $f:X^\bu \to Y^\bu$ a map between potential $n$-stages for $C$. Suppose that $f$ induces isomorphisms on $\pi^0 H_*$-groups. Then $f$ is a \mc{G}-equivalence.  
\end{proposition}

\begin{proof} Using the naturality of the spiral exact sequence, we see that the map $f$ induces isomorphisms $\naturalpi{s}{Y^\bu}{G}\xrightarrow{\cong} \naturalpi{s}{X^\bu}{G}$ for all $s \geq 0$ and $G \in \mc{G}$. 
\end{proof}

When the characteristic of the ground field is positive, property (b) in the definition of a potential $n$-stage follows automatically. The key ingredients for the proof of this reduction are the extra structure of the spiral exact sequence obtained in Appendices~\ref{appsec:spiral} and \ref{appsec:H-alg} and a theorem of Goodwillie \cite{Good:cosimp} on the homology spectral sequence of a cosimplicial space. 

\begin{theorem} \label{weak_potential_stage}
Let $\mathbb{F} = \mathbb{F}_p$ for $p>0$. Let $C \in \CA$ be a connected unstable coalgebra and $X^\bu$ a weak potential n-stage for $C$. Then  $X^\bu$ is a potential n-stage for $C$.
\end{theorem}

\begin{proof} 
By the spiral exact sequence, there are isomorphisms of $\naturalpi{0}{X^\bu}{-}$-modules 
  $$\naturalpi{s}{X^\bu}{-} \cong \Hom_{\CA}\bigl(C, H_*(\Omega^s (-))\bigr)$$ 
for all $s \leq n$. The spiral exact sequence continues as follows
  $$\cdots \to \naturalpi{n+2}{X^\bu}{G} \to \pi_{n+2}[X^\bu, G] \to \naturalpi{n}{X^\bu}{\Omega G} \to \naturalpi{n+1}{X^\bu}{G} \to 0$$
where the middle connecting map can be identified with the dual of a map of $C$-comodules from $C[n+1]$ to itself. We claim that this is an isomorphism.

From the description of the homology spectral sequence of a cosimplicial space given by Bousfield in \cite{Bou: homology SS}, one sees that this comodule map can be identified with the differential $d_{n+2}$ in the  spectral sequence of $X^\bu$. (Here we identify $\naturalpi{n}{X^\bu}{\Omega G}$ with $\naturalpi{0}{X^\bu}{\Omega^{n+1} G}$ via the isomorphisms in the spiral exact sequence.)

According to the main result of \cite{Good:cosimp}, elements in $E^r_{p,q}$ for $p>q$ do not survive to $E_{p,q}^\infty$ in the homology spectral sequence for $\mathbb{F}_p$. 
Hence the summand $\Sigma^{n+1}\mathbb{F}_p \subseteq C[n+1]$ of $E^{n+2}_{n+2,n+1}$, given by the internal shift of the summand in degree $0$ of $C$, is mapped isomorphically by $d_{n+2}$.
But a $C$-comodule endomorphism of $C[n+1]$ which is nontrivial in degree $n+1$ must be an isomorphism. It follows that the connecting map in the spiral exact sequence above is an isomorphism and, inductively, we also obtain the vanishing of the higher natural homotopy groups.     
\end{proof}

\begin{remark} For $p=2$, a proof of Goodwillie's theorem \cite{Good:cosimp} was given earlier by Dwyer using his construction of higher divided power operations \cite{D-div}. 
These operations exist also at odd primes (Bousfield \cite{Bou:op}), but to the best of our knowledge only for $p=2$ have these operations been studied in the literature in connection with the homology spectral sequence. The interaction of these operations with the Steenrod action forces the $E^{\infty}$-terms to carry an unstable module structure. 
But an unstable module is trivial in negative degrees by the instability condition and the relation $Sq^0 = \id$.
\end{remark}

\begin{remark} 
Theorem~\ref{weak_potential_stage} fails in characteristic $0$. 
Let $C^\bu$ be the bigraded coalgebra which is $\mathbb{Q}$ in bidegree $(0,0)$ and $\mathbb{Q}$ in bidegree $(n+2,n+1)$. We can regard it as a graded differential $\mathbb{Q}$-coalgebra with trivial differential. 
One obtains a cosimplicial cocommutative $\mathbb{Q}$-coalgebra by using the Dold-Kan correspondence between cochain complexes of graded vector spaces and cosimplicial graded vector spaces. 
This cosimplicial graded $\mathbb{Q}$-coalgebra can be realized as the rational homology of a cosimplicial space which is generated by a point in degree $0$ and an $(n+1)$-sphere in degree $n+2$. 
This cosimplicial space is a weak potential $n$-stage, but it's not a potential $n$-stage because the differential $d_{n+2}$ in its homology spectral sequence is trivial by construction, and hence the spiral exact sequence shows that it has non-trivial higher natural homotopy groups (see the proof of Theorem \ref{weak_potential_stage}). 
\end{remark}

While our goal is to obtain topological realizations of $C$ from totalizations of $\infty$-stages for $C$, the totalization of a potential $n$-stage does not usually yield anything interesting. 

\begin{proposition} Let $\mathbb{F} = \mathbb{F}_p$ where $p > 0$, $C$ a simply-connected unstable coalgebra, and $X^\bu$ a potential $n$-stage for $C$. 
Then ${\rm Tot}(X^\bu)$ is weakly trivial.
\end{proposition}

\begin{proof} By \cite[Theorem 3.6]{Bou: homology SS}, ${\rm Tot}(X^\bu)$ is simply-connected and the homology spectral sequence of $X^\bu$ is strongly convergent to 
$H_*({\rm Tot}(X^\bu))$. Therefore $H_*({\rm Tot}(X^\bu))$ is trivial since $d_{n+2}$ is an isomorphism as shown in the proof of Theorem \ref{weak_potential_stage} above. 
\end{proof}

\begin{definition}   \label{def-stage-extension}
We say that a potential $n$-stage $Y^\bu$ for $C$ \emph{extends or is over a potential $(n-1)$-stage} $X^\bu$ if $\sk^c_n(Y^\bu) \simeq X^\bu$. 
\end{definition}

\begin{theorem}  \label{obstruction-theory-step}
Let $C \in \CA$ and $n \geq 1$. Suppose that $$\alpha \co X^\bullet_{n-1}\to X^\bu$$ is a map in $c\mc{S}^{\mc{G}}$ where $X^\bu_{n-1}$ is a potential $(n-1)$-stage for $C$. 
Then $X^\bu$ is a potential $n$-stage for $C$ over $X^\bu_{n-1}$ if and only if 
there is a homotopy pushout square
\begin{equation} \label{obstruction-theory-step-b}
\xymatrix{
L_C(C[n], n + 1) \ar[r]^(.6){w_n} \ar[d] & X^\bu_{n-1} \ar[d]^{\alpha} \\
L(C,0) \ar[r] & X^\bu
}
\end{equation}
where $w_{n}$ is a map such that the map $\phi(X^\bu_{n-1})(w_n)\co K_C(C[n], n+1) \to H_*(X^\bu_{n-1})$, defined in \emph{Proposition~\ref{Representability of L(C,n)}}, is a weak equivalence.
\end{theorem}

\begin{proof}
Sufficiency is an easy consequence of the long exact sequence associated with the homotopy pushout that is obtained, by Proposition \ref{G-ho po to H-ho po 2}, 
after applying $H_*$ to \eqref{obstruction-theory-step-b} (cf. Proposition \ref{G-cofiber-to-LES-2}). The necessity follows from the discussion after Proposition~\ref{diff. constr. in S}.
\end{proof}

\subsection{The main results} \label{main-results}

For an unstable coalgebra $C$, let $\mc{W}_n(C)$ denote the 
subcategory of $c\mc{S}^{\mc{G}}$ whose objects are potential $n$-stages for $C$ and the morphisms are weak equivalences.
The moduli space of potential $n$-stages $\mc{M}_n(C)$ is defined to be the classifying space of this category. 

The homotopy type of $\mc{M}_0(C)$ has essentially already been determined, but we include the statement here 
again for completeness. 

\begin{theorem} \label{moduli-0-stages}
Let $C$ be an unstable coalgebra. There is a weak equivalence $$\mc{M}_0(C) \simeq B \Aut(C).$$
\end{theorem}

\begin{proof}
This is a reformulation of Proposition \ref{Moduli of L(C,k)} since a potential $0$-stage for $C$ is an object of type $L(C,0)$.
\end{proof}

Note that there is a functor $\sk^c_n\co \mc{W}_n(C) \to \mc{W}_{n-1}(C)$ that sends a potential $n$-stage to its derived $n$-skeleton.
Given a potential $(n-1)$-stage $X^\bu$ for $C$, we denote by $\mc{W}_n(C)_{X^\bu}$ the subcategory of $\mc{W}_n(C)$ over the component of $X^\bu \in \mc{W}_{n-1}(C)$. This may be empty,
of course. We denote its classifying space by $\mc{M}_n(C)_{X^\bu}$. This is the moduli space of potential $n$-stages over $X^\bu$ in the sense of Definition~\ref{def-stage-extension}. 

The following proposition is a refinement of Theorem~\ref{obstruction-theory-step} and is the first step to the description of the map $\sk^c_n\co \mc{M}_n(C) \to \mc{M}_{n-1}(C)$ in Theorem \ref{main-pullback-square} below. 

\begin{proposition} \label{obstruction-theory-step-moduli} 
Let $C$ be an unstable coalgebra and $X^\bu$ a potential $(n-1)$-stage for $C$ with $n\ge 1$. Then there is a homotopy pullback square
\[
\xymatrix{
\mc{M}_n(C)_{X^\bu} \ar[r] \ar[d]^{\sk^c_n} & \mc{M}\bigl(K(C, 0) \twoheadleftarrow K_C(C[n], n+1) \xrightarrow{\simeq} H_*(X^\bu)\bigr) \ar[d] \\
\mc{M}(X^\bu) \ar[r]^{H_*} & \mc{M}(H_*(X^\bu)).
}
\]
\end{proposition}
\begin{proof}
By Theorem~\ref{obstruction-theory-step}, there are functors as follows:
\[
 \xymatrix{
 \mc{W}_n(C)_{X^\bu} \ar@<1ex>[rr]^(.35){\eqref{obstruction-theory-step-b}} \ar[dr] && \mc{W}\bigl(L(C,0) \twoheadleftarrow L_C(C[n], n+1) \Rightarrow X^\bu\bigr) \ar@<1ex>[ll]^(.65){\rm pushout} \ar[dl] \\
 & \mc{W}(X^\bu) &
 }
\]
which make the diagram commute up to natural transformations (see Remark \ref{Moduli-Comparison-of-Spaces-2} for the meaning of $\Rightarrow$ in the diagram). 
The top pair of functors induces an inverse pair of homotopy equivalences
between the classifying spaces; the two composites are connected to the respective identity functors by zigzags of natural transformations. 
Then the result follows directly from Theorem \ref{Moduli-Comparison-of-Spaces} and Remark \ref{Moduli-Comparison-of-Spaces-2}. 
\end{proof}

\begin{theorem} \label{main-pullback-square} 
Let $C$ be an unstable coalgebra. For every $n \geq 1$, there is a homotopy pullback square
\begin{equation} \label{main-square-1}
\xymatrix{
\mc{M}_n(C) \ar[rr] \ar[d]^{\sk^c_n} && \mc{M}\bigl(K_C(C[n],n+2) \twoheadrightarrow K(C,0)\bigr) \ar[d]^{\Delta} \\
\mc{M}_{n-1}(C) \ar[rr]^(.3){H_*} && \mc{M}(K(C,0) \twoheadleftarrow K_C\bigl(C[n],n+2) \twoheadrightarrow K(C,0)\bigr).
}
\end{equation}
where the map $\Delta$ is induced by the functor $(V \twoheadleftarrow U) \mapsto (V \twoheadleftarrow U \twoheadrightarrow V)$. 
\end{theorem}
\begin{proof}
We will check this for each path component of $\mc{M}_{n-1}(C)$ by applying the last proposition. Let $X^\bu$ be a 
potential $(n-1)$-stage for $C$. Then $H_*(X^\bu)$ is an object of $\mc{W}\bigl(K(C,0) + (C[n], n+1)\bigr)$ 
and $\mc{M}(H_*(X^\bu))$ is a path component of the moduli space $\mc{M}\bigl(K(C,0) + (C[n], n+1)\bigr)$. 

By Proposition \ref{Moduli+Diff. = !}, there is a natural weak equivalence 
\begin{equation} \label{auxiliary-id}
\mc{M}\bigl(K(C,0) + (C[n], n+1)\bigr) \simeq \mc{M}\bigl(K(C,0) \twoheadleftarrow K_C(C[n],n+2) \twoheadrightarrow K(C,0)\bigr).
\end{equation}
Using this identification, we obtain the bottom map in \eqref{main-square-1}. The domain of $\Delta$ is path-connected, 
by Proposition \ref{Moduli of K(C,k)}, and its image lies in the component that corresponds to $\mc{M}(K_C(C[n], n+1))$ using the identification in \eqref{auxiliary-id}.

We claim that the following is a homotopy pullback square
\[
\xymatrix{
\text{\footnotesize$\mc{M}(K(C, 0) \twoheadleftarrow K_C(C[n], n+1) \xrightarrow{\sim} H_*(X^\bu))$} \ar[d] \ar[r] & \text{\footnotesize$\mc{M}(K_C(C[n],n+2) \twoheadrightarrow K(C,0))$} \ar[d]^{\Delta} \\
\text{\footnotesize$\mc{M}(H_*(X^\bu))$} \ar[r] & \text{\footnotesize$\mc{M}(K(C,0) \twoheadleftarrow K_C(C[n],n+2) \twoheadrightarrow K(C,0))$}
}
\]
where:
\begin{itemize}
 \item[(i)] The top map is induced by the obvious forgetful functor and the recipe of Proposition \ref{Moduli of K(C,k)}(c).
 \item[(ii)] The bottom map is induced by the recipe of Proposition \ref{Moduli+Diff. = !} as explained above.
 \item[(iii)] The map on the left is induced by the obvious forgetful functor. 
\end{itemize}
The homotopy commutativity and homotopy pullback property of this square can be checked by considering separately the 
following two cases:
\begin{itemize} 
 \item[(a)] $X^\bu$ can be extended to a potential $n$-stage. Then, by Proposition \ref{obstruction-theory-step-moduli}, 
$H_*(X^\bu)$ is weakly equivalent to $K_C(C[n], n+1)$ and hence the moduli space at the top left corner is non-empty. The weak equivalence 
indicated by $\xrightarrow{\sim}$ determines a homotopy between the two compositions in the square. Moreover, since the 
arrow $\xrightarrow{\sim}$ indicates a weak equivalence, it follows easily that the top map in the square is actually a weak 
equivalence. Therefore, the square is a homotopy pullback since the bottom map is up to homotopy the inclusion of a path component and the domain of $\Delta$ is connected. 
\item[(b)] $X^\bu$ cannot be extended to a potential $n$-stage. By Proposition \ref{obstruction-theory-step} and Proposition \ref{Representability of L(C,n)}, 
it follows that $X^\bu$ is not weakly equivalent to $K_C(C[n], n+1)$. Therefore, 
the space at the top left corner is empty and the commutativity of the diagram is a triviality. Moreover, the images of the bottom and 
right maps lie in different connected components because the image of $\Delta$ is in the connected component that corresponds to the 
moduli space $\mc{M}(K_C(C[n], n+1))$. Therefore, the square is a homotopy pullback in this case, too. 
\end{itemize}
Then the result follows from Proposition \ref{obstruction-theory-step-moduli}. 
\end{proof}

Next we will identify the map $\Delta$ in terms of the Andr\'{e}-Quillen cohomology of $C$ and reformulate Theorem \ref{main-pullback-square} in the form that was stated in the Introduction. 

Let $M \in \V C$, $D \in C/\CA$, and $K_C(M,n)$ be a pointed structured $K$-object. Following Theorem~\ref{representability of AQ}, we define the Andr\'{e}-Quillen space of $D$ with coefficients in $M$ as follows
\begin{equation} \label{AQ-space}
   \TAQ^n_C(D;M) : = \map^{\rm der}_{c(C/ \CA)}\bigl(K_C(M,n), cD\bigr).
\end{equation}
Note:
\begin{equation*}
   \pi_s \TAQ^n_C(D;M)\cong\left\{\begin{array}{cl}
                          AQ^{n-s}_C(D;M)& 0\le s\le n \\
                                0       & \text{otherwise}
                                 \end{array}\right. 
\end{equation*}
where the basepoint is given by the structure map $C\to D$.

We can identify the moduli space $\mc{M}\bigl(K(C, 0) \twoheadleftarrow K_C(C[n],n+2) \twoheadrightarrow K(C,0)\bigr)$ in terms of Andr\'e-Quillen cohomology as follows. 
Based on general results about moduli spaces (Theorem \ref{calculus-with-moduli-spaces}), there is a homotopy fiber sequence 
\begin{small}
\[
\xymatrix{
F = \map^{\rm der}_{\twoheadrightarrow}\bigl(K_C(C[n], n+2)), cC\bigr) \ar[r] & \mc{M}\bigl(K(C, 0) \twoheadleftarrow K_C(C[n],n+2) \twoheadrightarrow K(C,0)\bigr) \ar[d] \\
&  \mc{M}\bigl(K(C,0) \twoheadleftarrow K_C(C[n], n+2)\bigr) \times \mc{M}(K(C,0))
}
\]\end{small}

\vspace*{-2mm}\noindent
where the homotopy fiber is the subspace of $\map^{\rm der}\bigl(K_C(C[n], n+2), cC\bigr)$ defined by the maps which induce a $\pi^0$-isomorphism. The basepoint of $F$, which corresponds
to the zero Andr\'e-Quillen cohomology class, maps to the component of the image of $\Delta$. Recall that we have weak equivalences:
$$\mc{M}\bigl(K_C(C[n], n+2) \twoheadrightarrow K(C,0)\bigr) \simeq B \Aut_C(C[n])$$
$$\mc{M}(K(C,0)) \simeq B \Aut^{\rm h}(K(C,0)) \simeq B\Aut(C).$$
In particular, both spaces are path-connected. It follows that the total space of this homotopy fiber sequence is the homotopy quotient of the homotopy fiber
  $$F = \map^{\rm der}_{\twoheadrightarrow}\bigl(K_C(C[n], n+2), cC\bigl)$$ 
by the actions of the homotopy automorphisms of $K_C(C[n], n+2)$, as a structured $K$-object, and the homotopy automorphisms of $cC$. 

The homotopy quotient of $F$ under the action of $\Aut^{\rm h}(cC) \simeq \Aut(C)$ is homotopy equivalent to $\TAQ^{n+2}_C(C; C[n])$. This is obtained essentially by identifying 
sets of connected components of $F$. There is an induced action of the homotopy automorphisms $\Aut^{\rm h}\bigl(K_C(C[n], n+2) \twoheadrightarrow cC\bigr) \simeq  \Aut_C(C[n])$ 
on $\TAQ^{n+2}_C(C; C[n])$ which has a homotopy fixed point at the basepoint. Indeed, considering the associated action on 
$F = \map^{\rm der}_{\twoheadrightarrow}\bigl(K_C(C[n], n+2), cC\bigr)$, the restriction of the action on the basepoint factors through $\Aut(C)$. Let 
$$
\begin{array}{lcl} 
\widetilde{\TAQ}^{n+2}_C(C; C[n]) & : = & \TAQ^{n+2}_C(C; C[n]) //  \Aut^{\rm h}\bigl(K_C(C[n], n+2) \twoheadrightarrow cC\bigr) \\ 
                                        & \simeq & \TAQ^{n+2}_C(C; C[n]) //  \Aut_C(C[n]) \\
                                     
                                         \end{array} 
$$
denote the homotopy quotient. Then we have a canonical weak equivalence
$$\widetilde{\TAQ}^{n+2}_C(C; C[n]) \xrightarrow{\simeq} \mc{M}\bigl(K(C, 0) \twoheadleftarrow K_C(C[n],n+2) \twoheadrightarrow K(C,0)\bigr)$$
such that the square commutes up to canonical homotopy
\[
 \xymatrix{
 \widetilde{\TAQ}^{n+2}_C(C; C[n]) \ar[d] \ar[r]^(.3){\simeq} & \mc{M}\bigl(K(C, 0) \twoheadleftarrow K_C(C[n],n+2) \twoheadrightarrow K(C,0)\bigr) \ar[d] \\
 B \Aut_C(C[n]) \ar[r]^(.4){\simeq} & \mc{M}\bigl(K(C, 0) \twoheadleftarrow K_C(C[n], n+2) \bigr)
 }
\]
where the vertical maps are the canonical projections. 

The inclusion of the basepoint in $\TAQ^{n+2}_C(C; C[n])$, which of course represents the zero Andr\'e-Quillen class, yields a map 
$$[0]: B\Aut_C(C[n]) \rightarrow \widetilde{\TAQ}^{n+2}_C(C; C[n])$$
which is a section up to homotopy of the canonical projection above. Since the induced basepoint
$$\ast \stackrel{0}{\to} \TAQ^{n+2}_C(C; C[n]) \to \mc{M}\bigl(K(C, 0)\twoheadleftarrow K_C(C[n],n+2) \twoheadrightarrow K(C,0)\bigr)$$
factors through the map $\Delta$, it follows that there is a homotopy commutative diagram
\[
 \xymatrix{
  B \Aut_C(C[n]) \ar[r]^(.4){\simeq} \ar[d]^{[0]} & \mc{M}\bigl(K_C(C[n], n+2) \twoheadrightarrow K(C,0)\bigr) \ar[d]^{\Delta} \\
 \widetilde{\TAQ}^{n+2}_C(C; C[n]) \ar[r]^(.3){\simeq} & \mc{M}\bigl(K(C, 0) \twoheadleftarrow K_C(C[n],n+2) \twoheadrightarrow K(C,0)\bigr)
}
\]
Using this identification, we arrive at the following reformulation of Theorem~\ref{main-pullback-square}.

\begin{theorem} \label{main-pullback-square-alternative} 
Let $C$ be an unstable coalgebra. For every $n \geq 1$, there is a homotopy pullback square
\begin{displaymath}
\xymatrix{
\mc{M}_n(C) \ar[rr] \ar[d]^{\sk^c_n} && B\Aut_C(C[n]) \ar[d]^{[0]} \\
\mc{M}_{n-1}(C) \ar[rr]^{H_*} && \widetilde{\TAQ}^{n+2}_C(C; C[n])
}
\end{displaymath}
where the map on the right is defined by the zero Andr\'e-Quillen cohomology class. 
\end{theorem}

The fiber of the map $[0]$ is $\Omega \TAQ^{n+2}_C(C; C[n]) \simeq \TAQ^{n+1}_C(C; C[n]))$, so we obtain the following as a 
corollary. 

\begin{corollary} \label{main-pullback-square2}
Let $X^\bu$ be a potential $n$-stage for an unstable coalgebra $C$. Then there is a homotopy pullback square 
\begin{displaymath}
\xymatrix{
\TAQ^{n+1}_C(C; C[n]) \ar[rr] \ar[d] && \mc{M}_n(C) \ar[d]^{\sk^c_n} \\
\ast \ar[rr]^{\sk^c_n X^\bu} && \mc{M}_{n-1}(C). 
}
\end{displaymath}
\end{corollary}

Our next goal is to compare the tower of moduli spaces $\{\mc{M}_n(C)\}$ with the moduli space $\mc{M}_{\infty}(C)$. The following 
theorem is an application of the results and methods of \cite{DK-classification}; see Appendix \ref{DK-theory}.

\begin{theorem} \label{infty-stages-holim}
Let $C$ be an unstable coalgebra. There is a weak equivalence 
  $$\mc{M}_{\infty}(C) \simeq \holim_n \mc{M}_n(C).$$ 
\end{theorem}
\begin{proof}
If $\mc{W}_{\infty}(C)$ is empty, i.e., if there is no $\infty$-stage, then for every potential $n$-stage 
$X^\bu$ there is an $m > n$ such that no potential $m$-stage $Y^\bu$ satisfies $\sk^c_{n+1}(Y^\bu) \simeq X^\bu$. It follows 
that the homotopy limit is also empty. 

Let $\mathbb{N}$ denote the poset of natural numbers. In order to approximate better the setting of \cite{DK-classification}, we first 
replace $\mc{W}_n(C)$, $0 \leq n \leq \infty$, with a new category of weak equivalences denoted by $\mc{W}'_n(C)$. The objects are functors 
$F\co \mathbf{n} \to c\mc{S}^{\mc{G}}$ (or $F\co \mathbb{N} \to c \mc{S}^{\mc{G}}$ if $n = \infty$) such that:
\begin{itemize} 
\item[(i)] $F(k)$ is a potential $k$-stage for $C$, 
\item[(ii)] for every $k < m$, the map $F(k) \simeq \sk^c_{k+1} F(k) \to \sk^c_{k+1} F(m)$ is a weak equivalence.
\end{itemize}
The skeletal filtration of a potential $n$-stage $X^\bu$:
  $$\cdots \to \sk^c_k(X^\bu) \to \sk^c_{k+1}(X^\bu) \to \cdots$$
defines a functor 
  $$\{\sk^c_{\bu}(-)\}\co \mc{W}_n(C) \rightarrow \mc{W}'_n(C).$$ 
On the other hand, it is easy to see that the homotopy colimit functor defines a functor in the other direction 
  $$\mathrm{hocolim}\co \mc{W}'_n(C) \rightarrow \mc{W}_n(C)$$
and that the two compositions are naturally weakly equivalent to the respective identity functors. 
Hence the classifying spaces of these categories are canonically identified up to homotopy equivalence for all $n$. 

The connected components of $\mc{W}'_{\infty}(C)$ correspond to the equivalence classes of $\infty$-stages. Following 
\cite{DK-classification}, we say that two $\infty$-stages are \emph{conjugate} if their $n$-skeleta are weakly 
equivalent for all $n$. By standard cofinality arguments, it suffices to prove the theorem for the components associated with 
a single conjugacy class of $\mc{W}'_{\infty}(C)$. Such a class is represented by the skeletal filtration $\{ \sk^c_n X^\bu \}$ of a
$\infty$-stage $X^\bu$. Let $\mathrm{Conj}(X^{\bu}) \subseteq \mc{W}'_{\infty}(C)$ be the full subcategory of objects conjugate to $X^\bu$, i.e.\! 
diagrams $F$ such that $F(n) \simeq \sk^c_{n+1}X^\bu$. Let $\mc{W}'_n(C)_X \subseteq \mc{W}'_n(C)$ be the component containing $\{\sk^c_{k+1}(X^\bu)\}_{k \leq n}$. 
We emphasize that while $\mc{W}'_n(C)_X$ is connected, by definition, this will not be true for $\mathrm{Conj}(X^{\bu})$ in general. Then an application of Theorem \ref{DK-Postnikov tower} shows the required weak equivalence
  $$\mc{M}_{\rm conj}(X^\bu) : = B(\mathrm{Conj}(X^{\bu})) \xrightarrow{\simeq} \holim_n B \mc{W}'_n(C)_X \simeq \holim_n \mc{M}(X(n))$$
which is induced by the restriction maps.  
\end{proof}

Finally, we come to the comparison between $\infty$-stages and actual topological realizations of $C$. Let 
$\mc{W}_{\rm Top}(C)$ denote the category of spaces whose objects are topological realizations of $C$, i.e.\! simplicial 
sets $X$ such that $H_*(X) \cong C$ in $\CA$, and morphisms are $\mathbb{F}_p$-homology equivalences between simplicial sets. We define the classifying space of this category $\mc{M}_{\rm Top}(C)$ to be the \emph{realization space of $C$}. 
By Theorem \ref{DK-theory5}, there is a weak equivalence 
\begin{equation} 
\mc{M}_{\rm Top}(C) \simeq \bigsqcup_{X} B \Aut^{\rm h}(X)
\end{equation}
where the index set of the coproduct is the set of equivalence classes of realizations of $C$ with respect to $\mathbb{F}_p$-homology, with $X$ chosen in each class to be fibrant in the Bousfield localization with respect to $\mathbb{F}_p$-homology equivalences, and $\Aut^{\rm h}(X)$ denotes the simplicial monoid of self-homotopy equivalences. 

The meaning of strong convergence in the next theorem should be taken as defined by Bousfield in~\cite[p. 364]{Bou: homology SS}.

\begin{theorem} \label{infty-stages-are-reals}
Let $C$ be an unstable coalgebra. Suppose that the homology spectral sequence of each $\infty$-stage $ X^\bu$ for $C$ converges strongly to $H_*(\Tot X^\bu)$.
Then the totalization functor induces a weak equivalence 
  $$\mc{M}_{\infty}(C) \simeq \mc{M}_{\rm Top}(C).$$ 
\end{theorem}

\begin{proof}
Let $X^\bu$ be an $\infty$-stage and assume that it is Reedy fibrant. The $E^2$-page of the homology spectral sequence
of the cosimplicial space $X^\bu$,
$$\pi^s H_t(X^\bu) \Rightarrow H_{t-s}({\rm Tot}(X^\bu)),$$
is concentrated in the $0$-line and so the spectral sequence collapses at this stage. The spectral sequence converges strongly by assumption and therefore the edge homomorphism
\begin{equation} \label{edge homomorphism}
H_*({ \rm Tot}(X^{\bullet})) \cong E^{\infty}_{0, *} \xrightarrow{\cong} E^2_{0, *} = \pi^0 H_*(X^{\bullet}) \cong C 
\end{equation}
is an isomorphism. Since this isomorphism is induced by a map of spaces (see also \cite[2.4]{Bou: homology SS}):
$${\rm Tot}(X^\bu) \to X^0,$$ 
it follows that it is an isomorphism of unstable coalgebras. Hence ${\rm Tot}(X^\bu)$ is a topological realization of $C$.
This way we obtain a functor 
$${\rm Tot^f}\co \mc{W}_{\infty}(C) \to \mc{W}_{\rm Top}(C)$$
which sends an $\infty$-stage to the totalization of its (Reedy) fibrant replacement. 
On the other hand, there is a functor
$$ c\co \mc{W}_{\rm Top}(C) \to \mc{W}_{\infty}(C)$$
which sends a topological realization $Y$ of $C$ to the constant cosimplicial space $cY$. 
The composite functors ${\rm Tot^f} \circ c$ and $c \circ {\rm Tot^f}$ are connected to the identity functors by natural weak equivalences
$${\rm Id} \to {\rm Tot^f} \circ (\Delta^\bu\times -) \to {\rm Tot^f} \circ c $$
$$c \circ {\rm Tot^f} \leftarrow (\Delta^\bu\times -) \circ {\rm Tot^f} \to {\rm (-)^f} \leftarrow {\rm Id}.$$
Therefore, this pair of functors induces a pair of inverse homotopy equivalences between the associated moduli spaces and hence the result follows.
\end{proof}

Bousfield's convergence theorem of the homology spectral sequence (see \cite[Theorem 3.4]{Bou: homology SS}) applies to the  
$\infty$-stage for $C$ if $C$ is a simply-connected unstable coalgebra, i.e.\! $C_1 = 0$ and $C_0 = \mathbb{F}$. Thus, we obtain the following 

\begin{corollary}
Let $C$ be a simply-connected unstable coalgebra. Then the totalization functor induces a weak equivalence $\mc{M}_{\infty}(C) \simeq \mc{M}_{\rm Top}(C).$
\end{corollary}

\begin{remark}
Other convergence results for the homology spectral sequence were obtained by Shipley~\cite{Shipley:convergence}. For example, the homology spectral sequence of an $\infty$-stage $X^\bu$ converges strongly if $H_*(X^n)$ and $H_*(\mathrm{Tot^f}(X^\bu))$ have finite type and $\mathrm{Tot^f}(X^\bu)$ is $p$-good (see \cite[Theorem 6.1]{Shipley:convergence}). 
\end{remark}

\subsection{Moduli spaces of marked topological realizations} \label{marked-moduli}

All of the moduli spaces considered so far lie over the moduli space of objects of type $C$ (or, equivalently, $K(C,0)$ and $L(C,0)$). That is, we have considered spaces of automorphisms of topological realizations and potential $n$-stages without any restrictions on the induced automorphisms of $C$.
In other words, while certain identifications with $C$ were required to exist, specifying choices of isomorphisms was not. 

In this subsection, we briefly comment on the relation 
with the moduli spaces of objects together with a marking by $C$, i.e.\! a choice of identification with $C$ where applicable. The situation we have considered so far will turn out to 
be the same as for these marked moduli spaces once we pass to the homotopy quotients of the actions by automorphisms of $C$ that identify the different choices of markings.

We define a moduli space of marked topological realizations $\mc{M}'_{\rm Top}(C)$ as follows. Let $\mc{W}'_{\rm Top}(C)$ be the category whose objects are pairs $(X, \sigma)$ 
consisting of a topological realization $X$ of $C$ and an isomorphism $\sigma : C \cong  H_*(X)$. Such a pair is called a \emph{marked topological realization}. A morphism 
$f\co (X, \sigma) \to (Y, \tau)$ in $\mc{W}'_{\rm Top}(C)$ is given by an $H_*$-equivalence $f\co X \to Y$ 
such that $\tau = H_*(f) \sigma$. Then $\mc{M}'_{\rm Top}(C)$ is defined to be the classifying space of $\mc{W}'_{\rm Top}(C)$. 

There is an obvious free action on $\mc{W}'_{\rm Top}(C)$ by $\Aut(C)$ with quotient $\mc{W}_{\rm Top}(C)$. The quotient map is induced by the obvious forgetful functor 
$\mc{W}'_{\rm Top}(C) \to \mc{W}_{\rm Top}(C)$. As a consequence, it is easy to conclude that the functor $H_*\co \mc{W}_{\rm Top}(C) \to \mc{W}(C)$ induces a homotopy fiber sequence 
\begin{equation} \label{marked-realn}
\mc{M}'_{\rm Top}(C) \to \mc{M}_{\rm Top}(C) \to B \Aut(C).
\end{equation}

Let $L(C,0)$ be a fixed cofibrant choice of an object of type $L(C,0)$. A marking of a potential $n$-stage $X^\bu$ for $C$ corresponds to a map $\sigma\co L(C,0) \to X^\bu$ 
which induces an isomorphism on $\pi^0 H_*$ (or, equivalently, a \mc{G}-equivalence $L(C,0) \simeq \sk_1^c(X^\bu)$). By Theorem \ref{Representability of L(C,n)}, this 
structure is homotopically equivalent to providing the $\pi^0$-isomorphism 
  $$\tau = \phi(X^\bu)(\sigma)\co K(C,0) \to H_*(X^\bu).$$ 
Moreover, the latter piece of structure is also 
homotopically equivalent to providing the induced isomorphism $\pi^0(\tau)\co C \cong \pi^0 H_*(X^\bu)$.
The space of markings of a 
potential $n$-stage for $C$ is homotopically discrete and its connected components are in bijection with the automorphisms of $C$. Therefore, we may define the moduli 
space of \emph{marked potential $n$-stages} $\mc{M}'_n(C)$ as the homotopy fiber of the map $\sk^c_1$, i.e., we have a homotopy fiber sequence
\begin{equation} \label{marked-pot-n-stage}
\mc{M}'_n(C) \stackrel{J_n}{\longrightarrow} \mc{M}_n(C) \stackrel{\sk^{c}_1}{\longrightarrow} \mc{M}_0(C) \simeq B{\rm Aut}(C).
\end{equation}
Then the main results of the previous subsection have the following obvious analogues:
\begin{itemize}
\item $\mc{M}'_0(C)$ is contractible.
\item There is a homotopy pullback square 
 \[
  \xymatrix{
  \mc{M}'_n(C) \ar[d] \ar[rr]^-{J_n} && \mc{M}_n(C) \ar[d] \\
  \mc{M}'_{n-1}(C) \ar[rr]^-{J_{n-1}} && \mc{M}_{n-1}(C) 
  }
 \]
Combining this with Theorem \ref{main-pullback-square-alternative} and the discussion that followed, we obtain a homotopy pullback 
\[
 \xymatrix{
\mc{M}'_n(C) \ar[rr] \ar[d] && B {\rm Aut}_C(C[n]) \ar[d] \\
\mc{M}'_{n-1}(C) \ar[rr] && \widetilde{\TAQ}^{n+2}_C(C; C[n])
 }
\]
\item $\mc{M}'_{\infty}(C) \simeq \holim_n \mc{M}'_n(C)$.
\item Under the assumptions of Theorem \ref{infty-stages-are-reals}, $\mc{M}'_{\rm Top}(C) \simeq \mc{M}'_{\infty}(C)$. 
\end{itemize}

Next we may also consider moduli spaces of marked potential $n$-stages $X^\bu$ for $C$ with further additional structure, namely, a choice of a Postnikov invariant of $H_*(X^\bu) \in \mc{W}\bigl(K(C,0) + (C[n+1], n+2)\bigr)$,
$$w_{n+2}\co K_C(C[n+1], n + 3) \to \sk_{n+2} H_*(X^\bu).$$
We recall that this map is generally only specified up to the
action of the automorphisms of the structured $K$-object. Taking $\pi^{n+2}H_*(-)$ defines a functor  
$$\mc{W}_n(C) \to \mc{W}(C; C[n+1])$$ 
and thus we also obtain a map 
$$\mc{M}'_n(C) \stackrel{J_n}{\longrightarrow} \mc{M}_n(C) \to B {\rm Aut}_C(C[n+1]).$$
We define the moduli space of \emph{framed potential $n$-stages} $\mc{M}^f_n(C)$ to be the homotopy fiber of this composite map. 

The relation between these moduli spaces is clear: the original moduli space $\mc{M}_n(C)$ is the homotopy quotient of $\mc{M}^f_n(C)$ by the actions of the groups 
${\rm Aut}(C)$ and ${\rm Aut}_C(C[n+1])$. The analogue of Theorem~\ref{moduli-0-stages} says
$$
\mc{M}^f_0(C) \simeq {\rm Aut}_C(C[1]).
$$
More interestingly, the analogue of Theorem \ref{main-pullback-square-alternative} for the framed moduli spaces is an unwinding of the action by the automorphism groups ${\rm Aut}_C(C[n])$. 
Indeed, in the first place, it is easy to see that the map $\mc{M}_{n-1}(C) \to B {\rm Aut}_C(C[n])$ defined above agrees up to homotopy with the composition coming from Theorem \ref{main-pullback-square-alternative}:
$$\mc{M}_{n-1}(C) \longrightarrow \widetilde{\TAQ}^{n+2}_C(C;C[n]) \longrightarrow B {\rm Aut}_C(C[n]).$$
Moreover, the composite 
$$\mc{M}_n(C) \stackrel{\sk^c_n}{\longrightarrow} \mc{M}_{n-1}(C) \longrightarrow \widetilde{\TAQ}^{n+2}_C(C;C[n]) \longrightarrow B {\rm Aut}_C(C[n])$$
is essentially induced by the functor $\pi^{n+1} H_* (\sk^c_n(-))$ which can also be identified, up to natural isomorphism, with the functor $\pi^{n+2} H_*$: 
\[
 \xymatrix{
 \mc{W}_n(C) \ar[rr]^{\pi^{n+2} H_*} \ar[d]^{\sk^c_n} && \mc{W}(C; C[n+1]) \ar[d]^{\cong} \\
 \mc{W}_{n-1}(C) \ar[rr]^{\pi^{n+1}H_*} && \mc{W}(C; C[n]) 
 }
\]
As a consequence, the homotopy pullback of Theorem \ref{main-pullback-square-alternative}, as well as the one for the marked moduli spaces, lies over $B{\rm Aut}_C(C[n])$ naturally in $n$, where naturality in $n$ is exactly described by the commutativity of the previous square. By taking homotopy fibers, we obtain a homotopy pullback square as follows,
\[
 \xymatrix{
\mc{M}^f_n(C) \ar[rr] \ar[d] && \ast \ar[d]^0 \\
\mc{M}^f_{n-1}(C) \ar[rr] && \TAQ^{n+2}_C(C; C[n])
 }
\]
Finally, we also have the analogue of Theorem \ref{infty-stages-holim}:
$$\mc{M}^f_{\infty}(C) \simeq \holim_n \mc{M}^f_n(C).$$

\subsection{Obstruction theories} \label{obstruction-theories} In this subsection, we spell out in detail how the homotopical description of the realization 
space $\mc{M}_{\rm Top}(C)$ in terms of moduli spaces of potential $n$-stages gives rise to obstruction theories for the 
existence and uniqueness of realizations of $C$. As is usually the case with obstruction theories, these are also described 
as branching processes: at every stage of the obstruction process, an obstruction to passing to the 
next stage can be defined, but this obstruction (and so also its vanishing) depends on the state of the process 
created so far by choices made at earlier stages.

Clearly, $C$ is topologically realizable if and only if $\mc{M}_{\rm Top}(C)$ is non-empty. By Theorems 
\ref{infty-stages-holim} and \ref{infty-stages-are-reals}, this is equivalent to the existence of an object 
in $\mc{W}'_{\infty}(C)$.  Theorem \ref{moduli-0-stages} shows that a potential $0$-stage always exists and 
that it is essentially unique. Suppose that this has been extended to a potential $(n-1)$-stage $X^\bu$. This 
defines a point $[X^\bu] \in \mc{M}_{n-1}(C)$. Theorems \ref{main-pullback-square} and \ref{main-pullback-square-alternative} 
give a map 
  $$w_{n-1} \co \mc{M}_{n-1}(C) \to \widetilde{\TAQ}^{n+2}_C(C; C[n])$$
which associates to a potential $(n-1)$-stage $X^\bu$ the equivalence class of Andr\'{e}-Quillen cohomology classes 
which defines the Postnikov decomposition of 
  $$H_*(X^\bu) \in \mc{W}\bigl(K(C,0) + (C[n], n+1)\bigr).$$ 
This is, in other words, the unique Postnikov $k$-invariant of $H_*(X^\bu)$ in its invariant form. We recall here that the component of $\TAQ^{n+2}_C(C; C[n])$ corresponding to the zero Andr\'{e}-Quillen cohomology class is fixed under the action of $\Aut_C(C[n])$. 
Now Theorem \ref{main-pullback-square-alternative} shows 
\begin{theorem}
The potential $(n-1)$-stage $X^\bu$ is up to weak equivalence the $n$-skeleton of a potential $n$-stage if and only if $w_{n-1}([X^\bu])$ lies in the component of $($the image of$)$ the zero class in $\widetilde{\TAQ}^{n+2}_C(C; C[n])$.  
\end{theorem}

Let us emphasize that the first such Postnikov invariant in the process, 
  $$w_0(C): = w_0([X^\bu]) \in \widetilde{\TAQ}^3_C(C; C[1]),$$
associated to a potential $0$-stage $X^\bu$, is actually an invariant of $C$ because $\mc{M}_0(C)$ is connected. From the general 
comparison results of Baues and Blanc~\cite{Baues-Blanc:comparison}, this invariant seems to relate to an obstruction to endowing $C$ (or its dual) 
with secondary operations. 

Assuming that $w_{n-1}([X^\bu])$ corresponds to (the reduction of) the zero Andr\'{e}-Quillen cohomology class, and so $X^\bu$ can be extended to 
a potential $n$-stage, there will then be choices for an extension given by Corollary \ref{main-pullback-square2}. More specifically, by Proposition 
\ref{obstruction-theory-step-moduli} and Corollary \ref{main-pullback-square2}, there are homotopy fiber sequences as follows
\[
\xymatrix{
\TAQ^{n+1}_C(C; C[n]) \ar@{=}[r] \ar[d] & \TAQ^{n+1}_C(C; C[n]) \ar[d] \\
\mc{M}_n(C)_{X^\bu} \ar[r] \ar[d] & \mc{M}\bigl(K_C(C[n], n+1) \twoheadrightarrow K(C,0)\bigr) \ar[d] \\
B \mathrm{Aut^h}(X^\bu) \ar[r] & B \mathrm{Aut^h}(K_C(C[n], n+1))
}
\]
The set of equivalence classes of potential $n$-stages over $X^\bu$ is exactly the set of components of $\mc{M}_{n}(C)_{X^\bu}$. 
There is an action of the group 
  $$\mathrm{hAut}\bigl(K_C(C[n], n+1)\bigr) : = \pi_1 \bigl(\mathrm{B Aut^h}(K_C(C[n], n+1))\bigr)$$ 
on the path-components of the fiber 
$$\AQ^{n+1}_C(C; C[n])): = \pi_0\bigl(\TAQ^{n+1}_C(C; C[n])\bigr),$$ 
as well as an induced action of $\pi_1 \bigl( \mathrm{B Aut^h}(X^\bu) \bigr)$ by restriction. Using the long exact sequence of homotopy groups, 
we can identify the corresponding set of orbits with the required set of possible equivalence classes of extensions of $X^\bu$.

\begin{proposition} \label{number_extensions}
Let $X^\bu$ be a potential $(n-1)$-stage for an unstable coalgebra $C$. Suppose that $X^\bu$ extends to a potential $n$-stage. Then there is a bijection 
\begin{equation*} 
\pi_0(\mc{M}_{n}(C)_{X^\bu}) \cong \frac{\AQ^{n+1}_C(C; C[n]))}{\pi_1 \bigl(\mathrm{B Aut^h}(X^\bu) \bigr)}.
\end{equation*}
\end{proposition}
For every such choice of potential $n$-stage $Y^\bu$ over $X^\bu$, the same procedure leads to an obstruction 
$$w_n([Y^\bu], [X^\bu]): = w_n([Y^\bu]) \in \widetilde{\TAQ}^{n+3}_C(C; C[n+1])$$ 
for finding a potential $(n+1)$-stage \emph{over} $Y^\bu$ (which in turn is over $X^\bu$). We emphasize 
that the Postnikov invariant of $H_*(Y^\bu)$ is not determined by $X^\bu$. According to Theorems \ref{infty-stages-holim} 
and \ref{infty-stages-are-reals}, $C$ admits a topological realization if and only if 
this procedure can be continued indefinitely to obtain an element in $\mc{W}'_{\infty}(C)$. We summarize this discussion 
in the following statement. 

\begin{theorem}
Let $C$ be a simply-connected unstable coalgebra. There is a space $X$ such that $H_*(X) \cong C$ if and only if there can be made choices 
in the obstruction process above such that an infinite sequence of obstructions can be defined
$$w _0([X_0^\bu]), \  w_1([X_1^\bu], [X_0^\bu]), \ \cdots, w_n([X_n^\bu], [X_{n-1}^\bu]), \cdots,$$
where the first one is $w_0(C)$ and every other is defined if and only if the previous one vanishes. 
\end{theorem}

The problem of uniqueness of topological realizations of $C$ can also be analyzed by considering the tower of the moduli spaces. 
We have that $C$ is uniquely realizable if and only if $\mc{M}_{\rm Top}(C)$ is connected. This means that at each stage of the 
obstruction process, there is exactly \emph{one} choice that can be continued unhindered indefinitely. Assume that a topological 
realization $X$ of $C$ is given, which then defines basepoints in all moduli spaces. Then the 
number of different realizations is given by the Milnor exact sequence 
$$ \star \to \varprojlim{}^1 \pi_1 \left(\mc{M}_n(C), [\sk^c_{n+1}(cX)]\right) \to \pi_0 \mc{M}_{\mathrm{Top}}(C) \to \varprojlim \pi_0\bigl(\mc{M}_n(C)\bigr) \to \star$$
where
$$\pi_1\bigl(\mc{M}_n(C), [\sk^c_{n+1}(cX)]\bigr) \cong \pi_0 \mathrm{Aut^h}(\sk^c_{n+1}(cX)).$$
Moreover, associated to this tower of pointed fibrations, there is a Bousfield-Kan homotopy spectral sequence which starts from Andr\'{e}-Quillen 
cohomology and ends at the homotopy of the realization space. 

We now turn to an obstruction theory for realizing maps between \emph{realized} unstable coalgebras. That is, given a map $\phi\co H_*(X) \to H_*(Y)$, is 
there a realization $f\co X \to Y$, i.e.\! $\phi = H_*(f)$, and if so, what is the moduli space of all such realizations?
This cannot be answered directly from the main theorems above, but an answer can be obtained by similar methods based on the established 
Postnikov decompositions of cosimplicial spaces. 

For the question of existence, we proceed as follows. Under the assumption that $C$ is simply-connected, this realization problem is 
equivalent to finding a realization $f_{\infty}\co cX \to cY$ in $\ho{\mc{S}^{\mc{G}}}$ such that $H_*(f_{\infty}) = c(\phi)$. 
The existence of $f_{\infty}$ is equivalent to the existence of a compatible sequence of maps in $c \mc{S}^{\mc{G}}$:
$$f_n\co \sk^c_n(c X) \to (cY)^f.$$
Let $C = H_*(X)$, $D = H_*(Y)$ and a given map $\phi\co C \to D$.  This map extends trivially to a map of cosimplicial objects 
$$\phi\co K(C, 0) \to H_*(cY)$$
which then gives by Proposition \ref{Representability of L(C,n)}, the starting point of the obstruction theory:
$$f_1\co L(C, 0) \simeq \sk^c_1(cX) \to (cY)^f.$$
Assuming that $f_n$ has been constructed, the obstruction to passing to the next step is expressed in the following diagram:
\[
 \xymatrix{
 & L(C,0) \ar[d] \ar[dl] \ar[dr]^{f_1} & \\
 L_C(C[n], n+1) \ar[d] \ar[r] & \sk^c_n(cX) \ar[d] \ar[r]^{f_n} & (cY)^f \\
 L(C, 0) \ar[r] & \sk^c_{n+1}(cX) \ar@{-->}[ur]
 }
\]
The dotted arrow exists if and only if the top composite factors up to homotopy through $f_1\co L(C,0) \to (cY)^f$. 
By Proposition \ref{Representability of L(C,n)}, this is equivalent to a factorization in $\ho{c\CA^{\mc{E}}}$ 
of the associated map (which depends on $f_n$)
$$w(\phi; f_n)\co K_C(C[n], n+1) \to cH_*(Y)$$
through $\phi\co K(C,0) \to cH_*(Y)$ (in $\ho{c\CA^{\mc{E}}}$). Therefore, $w(\phi; f_n)$ determines an element in 
$\AQ^{n+1}_C(D; C[n])$ whose vanishing is equivalent to the existence of an extension $\sk^c_{n+1}(cX) \to (cY)^f$. 

As long as $\phi$ can be realized, the problem of the uniqueness of realizations can be addressed similarly. In the diagram 
above, the choices of homotopies exhibiting the vanishing of the obstruction $w(\phi; f_n)$, which correspond to elements 
of $\AQ^n_C(D; C[n])$, lead to different extensions in general. The identification of these extensions depends on $f_n$ 
and can be determined using standard obstruction-theoretic methods. The realization is unique if and only if among all such 
intermediate extensions exactly one sequence of choices makes it to the end.

\appendix

\section{The spiral exact sequence}
\label{appsec:spiral}

A spiral exact sequence was first found in the context of $\Pi$-algebras and simplicial spaces by Dwyer-Kan-Stover~\cite{DKSt:bigraded}.
This was generalized and used in the context of moduli problems for 
$E_\infty$-ring spectra by Goerss-Hopkins~\cite{GoHop:moduli2}.
As we will see in this appendix, there is spiral exact sequence associated to each resolution model category.
We work here cosimplicially and in the unpointed case, in contrast to previous instances in the literature, but the generality of the construction will 
be apparent. 

After the sequence is constructed in Subsection~\ref{subsec:contruction}, it is endowed with a module structure over its 
zeroth term in Subsection~\ref{subsec:ses}. The associated spectral sequence is mentioned in Subsection~\ref{subsec:sss}.
All parts here are written in the language of universal algebra and we assume throughout that the conditions of Assumption~\ref{strict-unit} are 
satisfied.
 
\subsection{Constructing the sequence}\label{subsec:contruction}

We follow the outline given in \cite[3.1.1]{GoHop:moduli2}. We write \mc{S} for the category of simplicial sets and $\mc{S}_*$ for pointed simplicial sets.
Let \mc{M} and \mc{G} be as in Assumption~\ref{strict-unit}. In particular, each $G$ in \mc{G} is assumed to be fibrant.
 
For any simplicial set $K$, we have a pair of adjoint functors
   $$ -\otimes^{\rm ext} K\co c\mc{M}\rightleftarrows c\mc{M}:\!\hom^{\rm ext}(K,-)$$
which can be composed with the pair of adjoint functors
   $$ (-)^0\co c\mc{M}\rightleftarrows\mc{M}:\!c .$$

\begin{definition}
We have a functor
  $$ T\co c\mc{M}\times\mc{S}\to\mc{M},\ \ T(X^\bu,K)=(X^\bu\otimes^{\rm ext} K)^0 .$$
For a fixed simplicial set $K$, we abbreviate 
$$T(X^\bu,K)=T_KX^\bu.$$
\end{definition}

\begin{example}\label{T-Delta-n}
There are canonical isomorphisms $T_{\Delta^n}X^\bu\cong X^n$ and $T_{\partial\Delta^n}X^\bu\cong L^nX^\bu$.
\end{example}

\begin{example}
The object $T_{\Lambda^n_k}X^\bu$ is isomorphic to the {\it partial latching object} $L^n_k X^\bu$ of \cite[3.12]{Bou:cos}.
\end{example}

\begin{definition}
Let $K$ be a pointed simplicial set. We define a functor 
  $$C_K\co c\mc{M}\to\mc{M}_{\ast}$$ 
by 
  $$C_KX^\bu := (X^\bu\oltimes K)^0.$$
More generally, as for $T$ above, there is also a functor $C\co c\mc{M}\times\mc{S}_*\to\mc{M}_{\ast}$.
\end{definition}

\begin{remark}\label{TKCK-pushouts}
The functor $C_K$ is left adjoint to the functor
  $$ K\bigl(c(-)\bigr)=\olhom\bigl(K,c(-)\bigr)\co \mc{M}_{\ast}\to c\mc{M},$$
and there are natural pushout diagrams 
\diagr{ T_KX^\bu \ar[d]\ar[r] & C_KX^\bu \ar[d]\ar[r] & \ast\ar[d] \\
        T_LX^\bu  \ar[r] &  C_LX^\bu \ar[r] & C_{L/K}X^\bu }
where $K\subset L$ is an inclusion of pointed simplicial sets.
\end{remark}

\begin{remark}
Both adjoint pairs $(\otimes^{\rm ext},\hom^{\rm ext})$ and $((-)^0,c)$ are simplicially enriched but for different simplicial structures on $c\mc{M}$. In particular, the functors $T_K$ and $C_K$ are {\it not} enriched.
\end{remark}

\begin{lemma}\label{TKCK left Quillen}
We have:
\begin{enumerate}
   \item
For a fixed $($pointed$)$ $K$, the left adjoint functors
  $$T_K\co c\mc{M}\to\mc{M}\ \text{ and }\ C_K\co c\mc{M}\to\mc{M}_* $$
are left Quillen if $c\mc{M}$ is equipped with the Reedy model structure.
   \item
For a fixed Reedy cofibrant $X^\bu$, the functors
  $$T_{(-)}X^\bu\co \mc{S}\to\mc{M}\ \text{ and }\ C_{(-)}X^\bu\co \mc{S}_*\to\mc{M}_* $$
map cofibrations to cofibrations.
\end{enumerate}
\end{lemma}

\begin{proof}
Both follow easily from the more general Proposition~\ref{Reedy-ext-comp}. 
\end{proof}

\begin{definition}
Let $F^\bu$ be a simplicial group and $n\ge 0$. We define the {\it partial matching objects} $M_n^0 F^\bu$ of $F^\bu$ by
  $$ M_n^0F_\bu \colon = \hom(\Lambda^n_0,F_\bu)_0. $$
Here the cotensor is given by the external simplicial structure on the category of simplicial groups.
\end{definition}

The following statement is a special case of \cite[Proposition 3.14]{Bou:cos} and the proof remains valid in our unpointed situation. 
It relies on \cite[Lemma VII.1.26, Proposition VII.1.27]{GoJar:simp2}.

\begin{proposition}\label{latching-matching-isos}
Let $X^\bu$ be a Reedy cofibrant cosimplicial object in $\mc{M}$ and $G\in\mc{G}$. Then for any $n\ge 0$, there is a canonical isomorphism
  $$[L^n_0X^\bu,G]_{\mc{M}}\cong M_n^0[X^\bu,G]_{\mc{M}}. $$
\end{proposition}

\noindent It will be convenient to use the abbreviation $ \Delta^n_0= \Delta^n/\Lambda^n_0$
for any $n \geq -1$. Here $\Delta^{-1}$, $\Lambda^0_0$ and $\partial\Delta^0$ are defined to be the empty (simplicial) set.

\begin{definition}\label{def:CnZn}
Let $X^\bu$ be a cosimplicial object in $\mc{M}$ and $n\ge 1$. We define
  $$ C^nX^\bu= C_{\Delta^n_0}X^\bu\ \ \ \ \ \text{ and }\ \ \ \ \ Z^nX^\bu= C_{\Delta^n/\partial\Delta^n}X^\bu .$$
and for $n = 0, -1$,
  $$Z^0X^\bu=C^0X^\bu=C_{\Delta^0_+}X^\bu= X^0\sqcup\ast\ \ \text{ and }\ \ Z^{-1}X^\bu=\ast.$$
\end{definition}

\medskip

\noindent For $n\ge 0$, the inclusion map $d^0\co\Delta^{n-1}\to\Delta^n$ induces a cofibration
   $$ d^0\co \Delta^{n-1}/\partial\Delta^{n-1}\to\Delta^n_0$$
whose cofiber is isomorphic to $\Delta^n/\partial\Delta^n$. 
Given a Reedy cofibrant $X^\bu \in c\mc{M}$, this induces by Remark~\ref{TKCK-pushouts} and Lemma~\ref{TKCK left Quillen}(2), a cofiber sequence
\begin{equation}\label{spiral-cofiber}
   Z^{n-1}X^\bu \stackrel{d^0}{\longrightarrow} C^nX^\bu \to Z^nX^\bu.
\end{equation}
From this cofiber sequence we obtain for any $G \in \mc{G}$ a long exact sequence as follows,
  $$ \hdots\to[Z^{n-1}X^\bu,\Omega^{q+1} G]\stackrel{\beta}{\longrightarrow}[Z^{n}X^\bu,\Omega^qG]\stackrel{\gamma}{\longrightarrow}[C^{n}X^\bu,\Omega^qG]\stackrel{\delta}{\longrightarrow}[Z^{n-1}X^\bu,\Omega^qG] \to \hdots,$$
where the brackets $[-,-]$ denote homotopy classes of maps in $\mc{M}_*$.
This can be spliced together to form an exact couple:
\diagram{[Z^{n-1}X^\bu,\Omega^{q}G]_{\mc{M}_*} \ar@{.>}[rr]^{\beta}_-{(1,-1)} && [Z^{n}X^\bu,\Omega^{q}G]_{\mc{M}_*} \ar[dl]^-{\gamma_{n,q}}  \\
               &   [C^{n}X^\bu,\Omega^qG]_{\mc{M}_*}=E^{n,q}_1 \ar[ul]^-{\delta_{n,q}} &  }{spiral-les}
where the dotted boundary map $\beta_{n-1,q+1}\co[Z^{n-1}X^\bu,\Omega^{q+1} G]\to[Z^{n}X^\bu,\Omega^qG]$ has bidegree $(1,-1)$. 

\begin{definition}
We denote the differential of this exact couple by
  $$ \ul{d}_{n,q}=\gamma_{n-1,q}\circ\delta_{n,q}\co [C^{n}X^\bu,\Omega^qG]\stackrel{\delta}{\longrightarrow} [Z^{n-1}X^\bu,\Omega^{q}G]\stackrel{\gamma}{\longrightarrow} [C^{n-1}X^\bu,\Omega^qG].$$
\end{definition}

\begin{definition}\label{spiral-ex-sequ}
The first derived couple of the exact couple~\eqref{spiral-les} is called the {\it spiral exact sequence}. The associated spectral sequence will be called the {\it spiral spectral sequence}.
\end{definition}

We proceed to describe the spiral exact sequence and its associated spectral sequence more explicitly.

\begin{definition}\label{simpl-normalization}
Let $F^\bu$ be a simplicial group and $n\ge 1$. We denote by
  $$N_nF_\bu=\bigcap_{i=1}^{n}\ker\left[d_i\co F_n\to F_{n-1}\right]\ \ \text{ and }\ \  N_0F_\bu=F_0$$
the {\it normalized cochain complex} associated to $F^\bu$ with differential $d_0=d_0|_{N_nF_\bu}$ and $d_0|_{F_0}=\ast$.
\end{definition}

\begin{lemma}\label{lem:spiral1}
Let $X^\bullet$ be a Reedy cofibrant cosimplicial object in $\mc{M}$ and $G\in\mc{G}$. 
\begin{enumerate}
  \item
For any $n\ge 0$, there is a natural isomorphism
  $$ [C^nX^\bu, G]_{\mc{M}_{\ast}}\cong N_n[X^\bu,G]_{\mc{M}} .$$
  \item
The differential $\ul{d}$ of the exact couple \emph{(\ref{spiral-les})} fits into the following commutative diagram:
\diagr{ [C^nX^\bu, G]_{\mc{M}_{\ast}}\ar[d]_{\cong}\ar[r]^-{\ul{d}} & [C^{n-1}X^\bu, G]_{\mc{M}_{\ast}}\ar[d]^{\cong} \\
        N_n[X^\bu,G]_{\mc{M}} \ar[r]^-{d_0} & N_{n-1}[X^\bu,G]_{\mc{M}} }
  \item
For any $n\ge 0$, there is a natural exact sequence
  $$ [C^{n+1}X^\bu,G]_{\mc{M}_{\ast}}\xrightarrow{\delta_{n+1}} [Z^nX^\bu,G]_{\mc{M}_{\ast}} \to \naturalpi{n}{X^\bu}{G}\to 0, $$
where the surjection can be identified via the adjunction isomorphisms with the surjection
  $$ [X^\bu,\Omega^ncG]_{c\mc{M}^{\rm Reedy}}\to[X^\bu,\Omega^ncG]_{c\mc{M}^{\mc{G}}}. $$
\end{enumerate}
\end{lemma}
\begin{proof}
The cofiber sequence $ \Lambda^n_0\to\Delta^n\to\Delta^n_0$ induces by Lemma~\ref{TKCK left Quillen} a cofiber sequence in $\mc{M}$
   $$ L^n_0X^\bu\to X^n\to C^nX^\bu.$$
Therefore we obtain a long exact sequence of abelian groups as follows,
   $$ \hdots\to[L^n_0X^\bu,\Omega G]\to[C^nX^\bu,G]\to[X^n,G]\to[L^n_0X^\bu,G].$$
The map $[X^n,G]\to [L^n_0X^\bu,G]$ can be identified with the canonical map
  $$[X^n,G]\to M_n^0[X^\bu,G] $$
using the isomorphism from Proposition \ref{latching-matching-isos}. We recall that this map is surjective for simplicial groups. Moreover, its kernel is $N_n[X^\bu,G]$ by Definition~\ref{simpl-normalization}. Therefore, we obtain short exact sequences
   $$ 0\to [C^nX^\bu,G]\to[X^n,G]\to[L^n_0X^\bu,G]\to 0$$
and (1) follows. To identify the differential in (2), we use Proposition~\ref{latching-matching-isos} repeatedly and obtain 
the following commutative diagram:
  $$ \xymatrix@C=10pt@R=15pt{ [C^nX^\bu,G] \ar@/^22pt/[rr]_-{\ul{d}}\ar[r]_-{d_0}\ar[d]^{\cong} & [Z^{n-1}X^\bu,G] \ar[r]\ar[d] & [C^{n-1}X^\bu,G] \ar[d]^{\cong} \\
   \ker\left\{\mc{X}_n\to[L^n_0X^\bu,G]\right\} \ar[r]_-{d_0}\ar[d]^{\cong} & \ker\left\{\mc{X}_{n-1}\to[L^{n-1}X^\bu,G]\right\}\ar[r] \ar[d] & \ker\left\{\mc{X}_{n-1}\to[L^{n-1}_0X^\bu,G]\right\} \ar[d]^{\cong} \\
   \ker\left\{\mc{X}_n\to M_n^0[X^\bu,G]\right\} \ar[r]_-{d_0}\ar[d]^{\cong} & \ker\left\{\mc{X}_{n-1}\to M_{n-1}[X^\bu,G]\right\}\ar[r] & \ker\left\{\mc{X}_{n-1}\to M_{n-1}^0[X^\bu,G]\right\} \ar[d]^{\cong} \\
   N_n[X^\bu,G] \ar[rr]^-{d_0} & &   N_{n-1}[X^\bu,G] } $$
Here $\mc{X}_k$ stands for $[X^k,G]$.

For (3), we recall from the adjunction isomorphisms from Definition~\ref{def:CnZn}:
\begin{align*}   
   \Hom_{\mc{M}_{\ast}}(Z^nX^\bu,G)&\cong\Hom_{c\mc{M}}(X^\bu,\Omega^n cG) \ \ \text{ and } \\
   \Hom_{\mc{M}_{\ast}}(C^{n+1}X^\bu,G)&\cong \Hom_{c\mc{M}}(X^\bu,\Delta^{n+1}_0 cG).
\end{align*}
Using Proposition~\ref{int-ext-commute}, these induce derived isomorphisms as follows,
\begin{align*}   
   [Z^nX^\bu,G]_{\mc{M}_{\ast}}&\cong [X^\bu,\Omega^n cG]_{c\mc{M}^{\rm Reedy}} \ \ \text{ and } \\
   [C^{n+1}X^\bu,G]_{\mc{M}_{\ast}}&\cong [X^\bu,\Delta^{n+1}_0 cG]_{c\mc{M}^{\rm Reedy}}.
\end{align*}
By our (co-)fibrancy conditions on $X^\bu$ and $G$, the canonical map
  $$ \kappa\co[Z^nX^\bu,G]_{\mc{M}_{\ast}}\cong[X^\bu,\Omega^n cG]_{c\mc{M}^{\rm Reedy}}\to  [X^\bu,\Omega^ncG]_{c\mc{M}^{\mc{G}}}=\naturalpi{n}{X^\bu}{G} $$
is well-defined and surjective. It remains to prove that
$\delta_{n+1}$ maps onto the kernel of $\kappa$. Note that 
$\Delta^{n+1}_0$ is a simplicial cone on $\Delta^n/\partial\Delta^{n}$. Thus, an element in the kernel of $\kappa$ has a representative in $\Hom_{c\mc{M}}(X^\bu,\Omega^ncG)$ that admits a lift as follows:
\diagr{ & \Delta^{n+1}_0(cG) \ar[d] \\
        X^\bu \ar[r]\ar@{.>}[ur] & \Omega^n(cG) }
since the vertical map is a \mc{G}-fibration. Therefore, every element in the kernel of $\kappa$ has a preimage in $[C^{n+1}X^\bu,G]_{\mc{M}_*}$, which concludes the proof of (3). 
\end{proof}

\begin{theorem}\label{thm:how--spiral-exact-seq-looks}
The spiral exact sequence is natural in $X^\bu$ and $G\in\mc{G}$ and takes the form
\begin{align*}
   \hdots\to\pi_{p+1}[X^\bu,G]\toh{b}\naturalpi{p-1}{X^\bu}{\Omega G}\toh{s}\naturalpi{p}{X^\bu}{G}\toh{h}\pi_{p}[X^\bu,G]\to\hdots
\end{align*}
ending with a surjection $\naturalpi{1}{X^\bu}{G}\to\pi_1[X^\bu,G]$. Moreover, there is an isomorphism
  $$ \naturalpi{0}{X^\bu}{G}\cong\pi_0[X^\bu,G]=:\pi_0$$
which is also natural in $X^\bu$ and $G \in \mc{G}$. 
\end{theorem}

\begin{proof}
The naturality of the sequence in both $X^\bu$ and $G\in\mc{G}$ is clear. Without loss of generality, we may assume that $X^\bu$ is 
Reedy cofibrant. Next we identify the terms of the derived couple 
of (\ref{spiral-les}). The first is given by
   $$ E^2_{p,q}=\ker \ul{d}_{p,q}/\im\ul{d}_{p+1,q}\cong\pi_{p}[X^\bu,\Omega^q G],$$
as follows from Lemma~\ref{lem:spiral1}(1)-(2). By Lemma~\ref{lem:spiral1}(3), this receives a map from
   $$   \im\, \{\beta_{p,q}\co[Z^{p}X^\bu,\Omega^{q}G]\to[Z^{p+1}X^\bu,\Omega^{q-1}G]\, \}\cong\coker\delta_{p+1,q} \cong\naturalpi{p}{X^\bu}{\Omega^{q}G} $$
and also maps to
   $$   \im\,\beta_{p-2,q+1}\cong\coker\delta_{p-1,q+1} \cong\naturalpi{p-2}{X^\bu}{\Omega^{q+1}G}, $$
with $\beta_{p-2,q+1}\co[Z^{p-2}X^\bu,\Omega^{q+1}G]\to[Z^{p-1}X^\bu,\Omega^qG]$. The surjection for $p=1$ follows from Lemma~\ref{lem:spiral1}(3) and the fact that $[Z^1 X^\bu, \Omega^q G]$ surjects onto the kernel of $\ul{d}_{1,q} = \delta_{1,q}$.

Lastly, the isomorphism for $p=0$ follows from Lemma~\ref{lem:spiral1}(3) and the fact that $C^0X^\bu=Z^0X^\bu=X^0\sqcup\ast.$
\end{proof}

\begin{definition}
In the spiral exact sequence, we call the maps 
\begin{enumerate}
   \item $h_{p,q}\co\naturalpi{p}{X^\bu}{\Omega^q G}\to\pi_{p}[X^\bu,\Omega^qG]$ the {\it Hurewicz maps}.
   \item $s_{p,q}\co\naturalpi{p}{X^\bu}{\Omega^q G}\to\naturalpi{p+1}{X^\bu}{\Omega^{q-1}G}$ the {\it shift maps}.
   \item $b_{p,q}\co\pi_{p}[X^\bu,\Omega^qG]\to\naturalpi{p-2}{X^\bu}{\Omega^{q+1} G}$ the {\it boundary maps}.
\end{enumerate}
\end{definition}

The names are kept from \cite{DKSt:bigraded}. The name ``Hurewicz map'' was given because it satisfies a ''Hurewicz theorem'', cf. Lemma~\ref{lem:equivalent-formulations-n-connected}. The map $b$ is really a boundary map in view of Proposition~\ref{derived-fiber-sequence}.

\subsection{The spiral exact sequence and $\pi_0$-modules}
\label{subsec:ses}
Under Assumption~\ref{strict-unit} the spiral exact sequence carries extra structure: it is an exact sequence of modules over the \Hun-algebra 
$$\pi_0 : =\naturalpi{0}{X^\bu}{G} \cong \pi_0 [X^\bu, G].$$ 
This structure enters crucially in Propositions~\ref{prop:recognizing-L-objects} and~\ref{alt-def-potential}.  The purpose of this subsection is to give a detailed proof of this property of the spiral exact sequence. The reader is invited to compare with Blanc-Dwyer-Goerss \cite{BlDG:pi-algebra}, and Goerss-Hopkins \cite{GoHop:moduli2, GoHop:moduli}.

We start with the Hurewicz maps. Examples~\ref{exam:ext-mapping-space-constant-target} and \ref{T-Delta-n} show an isomorphism
\begin{equation*}
    \Hom_{\mc{M}}(X^\bullet,G)\cong \map^{\rm ext}(X^\bullet, cG) .
\end{equation*}
The canonical functor $\mc{M}\to\ho{\mc{M}}$ induces a map
  $$\map^{\rm ext}(X^\bu,cG)\cong\Hom_{\mc{M}}(X^\bullet,G)\stackrel{\chi}{\longrightarrow}[X^\bullet,G].$$

\begin{proposition}
The Hurewicz map $h$ is induced by $\chi$. 
\end{proposition}

\begin{proof}
We may assume that $X^\bu$ is Reedy cofibrant. The Hurewicz map $h$ is defined, as the map between cokernels, by
\diagr{  [C^{p+1}X^\bu,G] \ar[r]^-{\cong} \ar[d]_{d^0} & N_{p+1}[X^\bu,G]\ar[d]_-{d_0}\ar[dr]^-{d_0} & \\
        [Z^pX^\bu,G] \ar@{->>}[d]\ar[r]^-{h'} & Z_p[X^\bu,G] \ar@{->>}[d]\ar[r]^-{\subset} & N_{p}[X^\bu,G] \\
        \naturalpi{p}{X^\bu}{G} \ar[r]^-h &  \pi_p[X^\bu,G] & }
where 
  $$Z_p[X^\bu,G]:=\bigcap_{i=0}^p\ker \left\{(d^i)^*=d_i\co[X^p,G]\to[X^{p-1},G]\right\} . $$ 
Then every element $\bar x$ in $\naturalpi{p}{X^\bu}{G}$ has a representative 
  $$x\in \Hom_{\mc{M}_*}(Z^pX^\bu,G)\surj [Z^pX^\bu,G]_{\mc{M}_*}\surj \naturalpi{p}{X^\bu}{G}. $$
Note that the map $h'$ is induced by $\chi$, and therefore the same is true for $h$ as well. 
\end{proof}

\begin{corollary}\label{cor:Hur-is-pi-0}
The Hurewicz map is a map of $\pi_0$-modules.
\end{corollary}

\begin{proof}
The $\pi_0$-module structure on $\pi_p[X^\bu,G]$ is explained in Example~\ref{exam:action-on-bigraded}. The natural homotopy group is
a $\pi_0$-module by Example~\ref{exam:action-on-natural}. Let us assume that $X^\bu$ is Reedy cofibrant.
The map $\chi$ directly yields a morphism of (split) \mc{G}-homotopy fiber sequences:
\diagr{ \Omega^p\map(X^\bu,cG)\ar[r]\ar[d]_{\Omega^p\chi=\chi^p} & \Omega_+^p\map(X^\bu,cG) \ar[r]\ar[d]^{\Omega^p_+\chi=\chi^p_+} & \map(X^\bu,cG) \ar[d]^{\chi} \\
        \Omega^p[X^\bu,G] \ar[r] & \Omega_+^p[X^\bu,G] \ar[r] & [X^\bu,G] }
Applying $\pi_0$ yields the required morphism of $\pi_0$-modules as abelian objects in $\Hun\text{-Alg} / \pi_0$.
\end{proof}

To show that the shift maps 
  $$ s_{p,1}\co \naturalpi{p}{X^\bu}{\Omega G}\to\naturalpi{p+1}{X^\bu}{G} $$
are compatible with the $\pi_0$-module structures, it will be convenient to construct a span as follows, 
  $$\Omega^pc(\Omega G)\stackrel{\simeq}{\longleftarrow}\bar\Omega^pc(\Omega G)\to\Omega^{p+1}cG $$
where the map on the left is a \mc{G}-equivalence. It induces the shift map after applying $[X^\bu,-]_{c\mc{M}^\mc{G}}$. We will see
that this span extends to morphisms of split homotopy fiber sequences over $cG$. This way of constructing the shift map is outlined in \cite{DKSt:bigraded} in the ``dual'' setting of $\Pi$-algebras. The following discussion remains valid for all of the maps $s_{p,q}$ with $q\ge 1$, but we will focus on the case $s_{p,1}$. 

\begin{definition}\label{def:Path-P-Omegabar}
For $G \in \mc{G}$, let $\Path(G)$ denote a functorial path object for $G$. It comes with a fibration $({\rm ev}_0,{\rm ev}_1)\co\Path(G)\to G\times G$ and a section $s_\inter\co G\to\Path(G)$ in $\mc{M}$. These induce maps of group objects in $\ho{\mc{M}}$. We write
  $$PG = \fib\{ {\rm ev}_0\co\Path(G)\to G \}.$$
This object is pointed and weakly trivial in \mc{M}.
It inherits a map $e_1\co PG\to G$ induced by ${\rm ev}_1$ with homotopy fiber $\Omega G$. Note that $e_1$ is a fibration, but not a \mc{G}-injective fibration.

The object $LG$ is defined via the following pullback in \mc{M}
  $$\xymatrix{ LG \ar[r]\ar[d]_{\lambda} & \Path (G) \ar[d]^{({\rm ev}_0,{\rm ev}_1)} \\
                G \ar[r]^-{{\rm diag}} & G\times G. } $$
It can be regarded as an analogue of a free loop space in \mc{M}.
The map $\lambda$ is a fibration because $({\rm ev}_0,{\rm ev}_1)$ is one. The map $G\to\Path (G)$ that is part of the structure of a path object induces a section for $\lambda$. Since all maps in the diagram are morphisms of homotopy abelian group objects, $LG$ splits as $G\times \Omega G$ implying that $\lambda$ is a \mc{G}-injective fibration.

We define an object $\bar{\Omega}^p_{\rm ext}c(\Omega G)$ and maps $\alpha$ and $\omega$ in $c\mc{M}_*$ via the pullback
\diagr{ \bar{\Omega}^pc(\Omega G) \ar[r]^-{\alpha}\ar[d]_{\omega}^{\simeq} &  \Delta^{p+1}_0(cPG) \ar[d]^{\simeq} \\
        \Omega^pc(\Omega G) \ar[r]  & \Omega^pcPG ,}
where the lower map is induced by the fiber inclusion map $\Omega G\to PG$. The vertical maps are \mc{G}-fibrations and Reedy equivalences because the right vertical map 
is a \mc{G}-fibration between objects which are degreewise weakly trivial. Moreover, the targets of these maps are \mc{G}-fibrant. As a consequence, 
$\bar{\Omega}^p_{\rm ext}c(\Omega G)$ is also \mc{G}-fibrant.

We wish to think of the collection of objects $\bar{\Omega}^p_{\rm ext}c(\Omega G)$ as a functor in $G \in \mc{G}$. The defining pullback square of $\bar{\Omega}^p_{\rm ext}c(\Omega G)$ is the left side of the following commutative cube:
\begin{align}\begin{split}
\xy
   \xymatrix"*"@=17pt{ \Omega^pc(\Omega G) \ar'[d] [dd] \ar[rr] & &  \ast  \ar[dd] \\
                                & &                \\
                 \Omega^pc(PG)  \ar'[r]^-{e_1} [rr]   & &  \Omega^pcG }
   \POS(-22,-10)
   \xymatrix@=23pt{  \bar{\Omega}^pc(\Omega G) \ar[rr]^-{\sigma} \ar[dd]_{\alpha} \ar["*"]^-{\omega}_-{\simeq} & &   \Omega^{p+1}cG \ar[dd]  \ar["*"]\\
                                                & &                         \\
             \Delta^{p+1}_0(cPG) \ar[rr]^-{\simeq}_-{e_1} \ar["*"]_-{\simeq}  & & \Delta^{p+1}_0(cG) \ar["*"]}
\endxy
\end{split}\label{proof:Omega-bar}
\end{align}
The map $\sigma$ is induced as the map between pullbacks.
\end{definition}

\begin{lemma}\label{lem:construction-of-sigma}
The front square of Diagram~\emph{(\ref{proof:Omega-bar})}
\diagr{ \bar{\Omega}^pc(\Omega G) \ar[r]^{\sigma}\ar[d]_{\alpha} &  \Omega^{p+1}cG \ar[d] \\
     \Delta^{p+1}_0(cPG) \ar[r]^-{e_1} & \Delta^{p+1}_0(cG)  }
is a pullback square.
\end{lemma}
\begin{proof}
Diagram~(\ref{proof:Omega-bar}) consists of pullback squares on the back, left and right sides of the commutative cube. 
\end{proof}

\begin{remark} Some comments on the previous diagrams are in order.
\begin{enumerate}
   \item
The defining pullback square for $\bar\Omega^pc(\Omega G)$ is obviously a homotopy pullback for the resolution model structure. The pullback square in Lemma~\ref{lem:construction-of-sigma} is not. 
All maps labeled $\simeq$ in Diagram~(\ref{proof:Omega-bar}) are \mc{G}-equivalences (and some are even Reedy equivalences), but neither the back nor the front 
square is a \mc{G}-homotopy pullback. The map $cPG\to cG$ is not a \mc{G}-fibration because the map $e_1\co PG\to G$ is not \mc{G}-injective. Neither the vertical maps nor the maps from left to right are \mc{G}-fibrations.
   \item
Giving a map $X^\bu\to\bar\Omega^pc\Omega G$, where $X^\bu$ is Reedy cofibrant, is essentially the same as giving the following:
\begin{itemize}
   \item a map $f\co Z^{p+1}X^\bu\to G$ in $\mc{M}_*$ together with
   \item a null homotopy of the composite $C^{p+1}X^\bu\to Z^{p+1}X^\bu\to G$ in $\mc{M}_*$.
\end{itemize}
\end{enumerate}
\end{remark}

\begin{definition}\label{def:sigma-ss}
The following span in $c\mc{M}_*$
\begin{equation*}
   \Omega^pc(\Omega G)\stackrel{\simeq}{\longleftarrow}\bar{\Omega}^pc(\Omega G)\stackrel{\sigma}{\longrightarrow}\Omega^{p+1}c(G)
\end{equation*}
yields the map in $\ho{c \mc{M}_*^{\rm Reedy}}$ (and $\ho{c\mc{M}_*^{\mc{G}}})$
$$\ss_{p}\co\Omega^pc(\Omega G)\to\Omega^{p+1}c(G).$$ 
\end{definition}

\begin{proposition}\label{lem:shiftmap-revisited}
The map $\naturalpi{p}{X^\bu}{\Omega G}\to\naturalpi{p+1}{X^\bu}{G}$ that is induced by $\ss_{p}$ agrees with the shift map $s_{p,1}$ in the spiral exact sequence.
\end{proposition}
\begin{proof}
We may assume that $X^\bu$ is Reedy cofibrant. The following diagram commutes:
\diagr{ [Z^pX^\bu,\Omega G] \ar[r]^-{\beta_{p,1}}\ar[d]^{\cong} & [Z^{p+1}X^\bu,G] \ar[d]^{\cong} \\
        [X^\bu,\Omega^pc(\Omega G)]_{c\mc{M}^{\rm Reedy}} \ar[r]^-{\ss_{p,1}} & [X^\bu,\Omega^{p+1}c(G)]_{c\mc{M}^{\rm Reedy}} }
This surjects onto \mc{G}-homotopy classes.
\end{proof}

\begin{corollary}\label{cor:shift-is-pi-0}
The shift map $s_{p,1}$ is a $\pi_0$-module map.
\end{corollary}

\begin{proof}
The description of the $\pi_0$-module structure of the natural homotopy groups is given in Examples~\ref{exam:action-on-natural} and \ref{exam:naturalpi-with-Omega}. The relevant diagram to consider is \eqref{proof:Omega-bar}:
the shift map is induced by $\sigma\omega^{-1}$. We need to exhibit that cube as the fiber of a morphism of cubes whose target cube is constant at $cG$ and that is objectwise a split \mc{G}-fibration. Note however that it suffices to do all this up to Reedy equivalence.

Let us begin by considering the back square of \eqref{proof:Omega-bar}.
Here the necessary split \mc{G}-fibration for $\Omega^pc\Omega G$ has been exhibited in Example \ref{exam:naturalpi-with-Omega}:
  $$ \Omega^pc\Omega G \to P\stackrel{\curvearrowleft}{\longrightarrow} cG . $$
In an analogous way as in Example \ref{exam:naturalpi-with-Omega}, one expresses $\Omega^pcPG$ as the total fiber of the square
\begin{equation*}
 \xymatrix{  c\Path(G)^{S^p} \ar[r]\ar[d] & c\Path(G) \ar[d] \\ 
             (cG)^{S^p} \ar[r] & cG }
\end{equation*}
and, by an appropriate base change, we obtain a split fiber sequence
  $$ \Omega^pcPG \to P_2 \stackrel{\curvearrowleft}{\longrightarrow} cG, $$
where
  $$ P_2=cG\times_{(c\Path(G)\times_{cG}(cG)^{S^p})}c\Path(G)^{S^p}\simeq cG\times_{(cG)^{S^p}}c\Path(G)^{S^p} $$
since $G\simeq \Path(G)$. Up to Reedy equivalence, we could simplify the term on the right further but that would obscure the commutativity of the cube that we are about to construct.
Now the back square of \eqref{proof:Omega-bar} is the fiber of the map from the square 
  $$\xymatrix{ P\ar[r]\ar[d] & cG \ar[d] \\ 
               P_2 \ar[r] & (cG)^{S^p}} $$
to the constant square at $cG$. 

Next consider the bottom square of~\eqref{proof:Omega-bar}. For the corner at $\Delta^{p+1}_0cPG$, we consider a split fiber sequence
  $$ \Delta^{p+1}_0cPG \to P_3 \stackrel{\curvearrowleft}{\longrightarrow} cG, $$
with a simplified
  $$ P_3\simeq cG\times_{(cG)^{\Delta^{p+1}_0}} c\Path (G)^{\Delta^{p+1}_0}.$$
Thus, the commutative square
  $$ \xymatrix{ P_3 \ar[r]\ar[d] & (cG)^{\Delta^{p+1}_0} \ar[d] \\ 
             P_2 \ar[r]^s & (cG)^{S^p} } $$
maps via a split \mc{G}-fibration to the constant cube at $cG$ and has the bottom square of~\eqref{proof:Omega-bar} as fiber. 

The right hand square of the cube~\eqref{proof:Omega-bar} is obviously the fiber of
  $$\xymatrix{ (cG)^{S^{p+1}}\ar[r]\ar[d] & cG\ar[d] \\
              (cG)^{\Delta^{p+1}_0} \ar[r] & (cG)^{S^p} } $$
mapping to the constant square at $cG$. 

Lastly, one defines the total space $P_4$ for the top left front corner at $\bar{\Omega}^pc(\Omega G)$ in~\eqref{proof:Omega-bar} via the pullback
  $$ \xymatrix{ P_4 \ar[r]\ar[d] & P\ar[d] \\  
                P_3 \ar[r] & P_2} $$
By construction, this object sits in a split \mc{G}-fibration sequence
  $$ \bar{\Omega}^pc(\Omega G)\to P_4\stackrel{\curvearrowleft}{\longrightarrow} cG.$$
In summary, the following cube of total spaces
\begin{align*}
\xy
   \xymatrix"*"@=22pt{ P \ar[rr]\ar[dd] & &  cG \ar[dd] \\
                           & &               \\
     P_2 \ar[rr] & &  (cG)^{S^p } }
   \POS(-20,-10)
   \xymatrix@=23pt{  P_4 \ar[dd]_{\alpha'}\ar["*"]^{\omega'}\ar[rr]_<<<<<<<<{\sigma'} & &   cG^{S^{p+1}} \ar[dd]  \ar["*"] \\
                      & &                     \\
     P_3 \ar[rr]\ar["*"] & & (cG)^{\Delta^{p+1}_0} \ar["*"] }
\endxy
\end{align*}
commutes and yields the cube~\eqref{proof:Omega-bar} by taking the fiber over the constant cube at $cG$. The maps $\sigma'$ and $\omega'$ provide the necessary data for the claim that the shift map is a $\pi_0$-module morphism.
\end{proof}

It remains to inspect the boundary maps in the spiral exact sequence. It is useful to observe that most terms of the sequence are representable as functors of 
$X^\bu$ in the resolution model category. We are going to explain that this is the case for the $E_2$-homotopy groups $\pi_p[X^\bu,G]$, $p\neq 1$, too. For $p=0$, 
this is clear from the isomorphism $\pi_0[X^\bu,G]\cong\naturalpi{0}{X^\bu}{G}$.

\begin{definition}\label{def:Psiqp} 
Let $p\ge 2$. We define $\Psi^p(G)$ in $c\mc{M}_{\ast}$ by the following pullback square:
\diagr{ \Psi^p(G) \ar[r]^-{\zeta}\ar[d]^{\tau} &  \bar{\Omega}^{p-2}c(\Omega G)\ar[d]^{\sigma} \\
          \Delta^p_0(cG) \ar[r] & \Omega^{p-1}cG.}
\end{definition}

Note that the lower horizontal map is a \mc{G}-fibration. Hence, $\zeta$ is a \mc{G}-fibration. In particular, $\Psi^p(G)$ is \mc{G}-fibrant. 
By looping the defining pullback square, we get natural \mc{G}-equivalences for $p\ge 2$,
  $$ \Omega^s_{\rm ext}\Psi^p(G)\simeq \Psi^{p+s}(G). $$ 

\begin{remark}\label{rem:hopull-criterium} 
The pullback of a diagram in a model category
  $$ A\stackrel{f}{\longrightarrow} B\stackrel{g}{\longleftarrow} C $$
is weakly equivalent to the respective homotopy pullback if all objects are fibrant and one of the maps $f$ or $g$ is a fibration, see e.g. \cite[Proposition A.2.4.4]{Lurie:higher-topoi}.  
\end{remark}

\begin{definition}\label{def:pitchfork}
For an object $X$ in $\mc{M}_*$, we denote by $X\to \cone(X)$ a functorial cofibration to a weakly trivial object.
Let $\pitchfork^p X^\bu$ be the pushout of
  $$ \cone(C^{p-1}X^\bu)\leftarrow C^{p-1}X^\bu\xrightarrow{d^0} C^pX^\bu,  $$
which is a model for the mapping cone of $d^0$.
\end{definition}

\begin{lemma}\label{lem:what-is-Psip}
For all $p\ge 2$ and every Reedy cofibrant $X^\bu$ we have:
\begin{enumerate}
   \item
Maps in $c\mc{M}$ of the form $X^\bu\to\Psi^p(G)$ correspond bijectively to maps $C^pX^\bu\to G$ in $\mc{M}_*$ together with a null homotopy of the composite
  $$ C^{p-1}X^\bu\stackrel{d^0}{\longrightarrow} C^p X^\bu\to G. $$ 
   \item
There is a natural isomorphism
  $$ [X^\bu,\Psi^p(G)]_{c\mc{M}^{\rm Reedy}}\cong [\pitchfork^p X^\bu,G]_{\mc{M}_*}.$$
   \item
In the Reedy homotopy category, $\Psi^p(G)$ is a group object and the map $\tau_*$, induced by $\tau\co\Psi^p(G)\to\Delta^{p}_0(cG)$, is part of the long exact sequence
  $$ \hdots\to N_{p-1}[X^\bu,\Omega G]_{\mc{M}}\to [X^\bu,\Psi^p(G)]_{c\mc{M}^{\rm Reedy}} \xrightarrow{\tau_*} N_p[X^\bu, G] \xrightarrow{d_0} N_{p-1}[X^\bu, G]\to \hdots $$
   \item
In particular, $\tau_*$ factors through the map 
  $$ \vartheta\co[X^\bu,\Psi^p(G)]_{c\mc{M}^{\rm Reedy}}\surj\bigcap_{i=0}^p\ker \left\{(d^i)^*\co[X^p,G]\to[X^{p-1},G]\right\}= Z_p[X^\bu,G]_{\mc{M}} $$ 
and this map is a surjection.
\end{enumerate}
\end{lemma}

\begin{proof}
Combining~\ref{def:Psiqp} and~\ref{lem:construction-of-sigma} yields a composed pullback square:
\diagram{ \Psi^p(G) \ar[r]^-{\zeta}\ar[d]^{\tau} &  \bar{\Omega}^{p-2}c(\Omega G)\ar[d]^{\sigma} \ar[r]^-{\alpha} & \Delta^{p-1}_0(cPG) \ar[d]^{e_1} \\
        \Delta^{p}_0(cG) \ar[r] & \Omega^{p-1}cG \ar[r] & \Delta^{p-1}_0(cG)}{combined-square-for-Psiqp} 
Examining this large square one deduces the characterization of maps into $\Psi^p(G)$ given in (1). 
Claim (2) is a direct consequence of (1). 

Since $G$ is a group object in \ho{\mc{M}} the isomorphism (2) yields that, for every Reedy cofibrant input $X^\bu$, the functor $[-,\Psi^p(G)]_{c\mc{M}^{\rm Reedy}}$ is group valued. The map $\tau$ is induced by the canonical map $C^pX^\bu\to\pitchfork^pX^\bu$. 
Since the homotopy category has finite products, we can use the Yoneda lemma, and deduce that $\Psi^p(G)$ is a group object in $\ho{c\mc{M}}_{{\rm Reedy}}$ and $\tau$ is a morphism of homotopy group objects. 

The adjunction isomorphism
  $$ \Hom_{\mc{M}_*}(C^p X^\bu, G)\cong\Hom_{c\mc{M}}\bigl(X^\bu,\Delta^{p}_0(cG)\bigr) ,$$
induces an isomorphism
  $$ [C^pX^\bu,G]_{\mc{M}_*}\cong [X^\bu,\Delta^{p}_0(cG)]_{c\mc{M}^{\rm Reedy}}. $$
Since $\Delta^{p-1}_0(cPG)$ is Reedy contractible, the Reedy homotopy fiber sequence
  $$\Psi^p(G) \xrightarrow{\tau} \Delta_0^p(cG) \xrightarrow{d^0} \Delta_0^{p-1}(cG) $$
induces a long exact sequence after applying $[X^\bu, -]_{c\mc{M}^{\rm Reedy}}$. Using~\ref{lem:spiral1}(a), this long exact sequence is the one stated in (3).
Therefore, $[X^\bu, \Psi^p(G)]$ surjects onto $Z_p[X^\bu, G]\cong\ker \{N_p[X^\bu, G] \xrightarrow{d_0} N_{p-1}[X^\bu, G]\}$ which proves (4). 
\end{proof}

Following the method of \cite{DKSt:bigraded}, the proof of the representability of the $E_2$-homotopy groups by the objects $\Psi^p(G)$ uses an explicit construction of an object $P' \Psi^p(G)$ which plays the role of a path object in the $\mc{G}$-model structure. 

\begin{definition}\label{def:PprimePsiG}
Let $p\ge 2$. We define $P' \Psi^p(G)$ by the following pullback square:
  $$ \xymatrix{  P' \Psi^p(G) \ar[r]^-{\psi} \ar[d] & \Psi^p(G) \ar[d]^{\tau} \\ 
                 \Delta^{p+1}_0(cG) \ar[r]^{d^0} & \Delta^{p}_0(c G). } $$
\end{definition}

\begin{remark}
With some care, one can show that $P'\Psi^p(G)$ also fits into the Reedy homotopy pullback square:
  $$ \xymatrix{ P'\Psi^p(G) \ar[r]\ar[d] & \Delta^{p}_0(c\Path(G)) \ar[d]^{({\rm ev}_0,{\rm ev}_1)} \\ 
                \Psi^p(G) \times \Delta^{p+1}_0(cG) \ar[r]^-{(\tau,d^0)} & \Delta^{p}_0(cG)\times\Delta^{p}_0(cG) } $$
This is analogous to the construction of the corresponding object in~\cite[Proposition 7.5]{DKSt:bigraded}. We will not need this description.
\end{remark}

\begin{lemma}\label{lem:P-prime-Psi}
There is an isomorphism $P'\Psi^p(G)\cong\Delta^{p-1}_0(c\Omega G)\times\Delta^{p+1}_0(cG)$.
\end{lemma}

\begin{proof}
Put the pullback squares from~\ref{def:PprimePsiG},~\ref{def:Psiqp}, and~\ref{lem:construction-of-sigma} together:
\begin{equation}\begin{split}\label{eqn:a-lot-of-pullbacks}
   \xymatrix{
P'\Psi^p(G) \ar[r]^-{\psi} \ar[d] & \Psi^p(G) \ar[r]^-{\zeta} \ar[d]^{\tau} & \bar{\Omega}^{p-2}(c \Omega G) \ar[d]^{\sigma}\ar[r]^-{\alpha} & \Delta^{p-1}_0(cPG) \ar[d]^{e_1} \\
\Delta^{p+1}_0(cG) \ar[r]^{d^0} & \Delta^{p}_0(cG) \ar[r]^{i_p^*} & \Omega^{p-1}(cG) \ar[r] & \Delta^{p-1}_0(cG) } \end{split}
\end{equation}
Now observe that the composition of the three lower maps is $d^0\circ d^0=\ast$.
It follows that $P'\Psi^p(G)$ is isomorphic to $\Delta^{p+1}_0(cG)\times\fib (e_1)\cong\Delta^{p+1}_0(cG)\times \Delta^{p-1}_0(c\Omega G)$. 
\end{proof}

\begin{lemma}\label{lem:psi-morphism-of-group-objects}
The map $\psi\co P'\Psi^p(G)\to\Psi^p(G)$ is a morphism of homotopy group objects in the Reedy model structure.
\end{lemma}

\begin{proof}
Analogously to Definition~\ref{def:pitchfork}, there is an adjoint construction to $P'\Psi^p(G)$, based on the pushout:
  $$\xymatrix{ C^{p-1}X^\bu\ar[r]^-{d^0}\ar[d] & C^pX^\bu\ar[r]^-{d^0}\ar[d] & C^{p+1}X^\bu \ar[d] \\                
               \cone(C^{p-1}X^\bu) \ar[r] & \pitchfork^pX^\bu \ar[r]^-{(d^0)'} & \cone'\pitchfork^pX^\bu  }  $$
with a canonical map $\pitchfork^pX^\bu \xrightarrow{(d^0)'} \cone'\pitchfork^pX^\bu$ that induces $\psi$.
An argument as in the proof of Lemma \ref{lem:what-is-Psip}(3) shows that $\psi$ is a morphism of homotopy group objects.
\end{proof}

\begin{lemma}\label{lem:P-prime-Psi-2}
The map $\psi$ is Reedy equivalent to a $\mc{G}$-fibration.
\end{lemma}

\begin{proof}
We extend the left part of~\ref{eqn:a-lot-of-pullbacks} to the following commutative diagram:
  $$ \xymatrix{
\Delta^{p+1}_0(cG) \times \Delta^{p-1}_0(c \Omega G) \ar[r]^-{q} \ar[d]_{\cong} & \Omega^p(cG) \times \Delta^{p-1}_0(c \Omega G) \ar[r] \ar[d]^{\rho} & \Delta^{p-1}_0(c \Omega G) \ar[d]^j \\
P'\Psi^p(G) \ar[r]^-{\psi} \ar[d] & \Psi^p(G) \ar[r]^-{\zeta} \ar[d]^{\tau} & \bar{\Omega}^{p-2}(c \Omega G) \ar[d]^{\sigma} \\
\Delta^{p+1}_0(cG) \ar[r]^-{d^0} & \Delta^{p}_0(c G) \ar[r]^-{i_p^*} & \Omega^{p-1}(cG) 
} $$
The map $j$ is the inclusion of $\fib(\sigma)\cong\fib(e_1)$ into the total space. Hence, the composite of the right vertical maps is constant. It follows that the top right square is a (Reedy homotopy) pullback (using Remark~\ref{rem:hopull-criterium}).

We claim that the map $\psi$ is Reedy equivalent to a \mc{G}-fibration. The top map $q$ is the product $(i_{p+1}^*,\mathrm{id})$, so it is a \mc{G}-fibration. 
Therefore it suffices to show that $\rho$ is Reedy equivalent to a \mc{G}-fibration. Composing with the Reedy equivalence $\omega$ (see Definition \ref{def:Path-P-Omegabar}), we obtain a diagram as follows:
$$
\xymatrix{
\Omega^p(cG) \times \Delta^{p-1}_0(c \Omega G) \ar[r]^(.7){c} \ar[rdd]_{\rho} & E \ar[r] \ar[dd]^{\rho'} &  \Delta^{p-1}_0(c \Omega G) \ar[d] \\
&& \bar{\Omega}^{p-2}(c \Omega G) \ar[d]^{\omega}_{\simeq} \\
& \Psi^p(G) \ar[r]^-{\omega \circ \zeta} & \Omega^{p-2}(c \Omega G) 
}
$$
where $E$ denotes the pullback of the square. The right vertical composition is $i_{p-1}^*$ and therefore it is a \mc{G}-fibration. It follows that the induced map $\rho'$ is a 
\mc{G}-fibration. Moreover, the square is a Reedy homotopy pullback by Remark~\ref{rem:hopull-criterium}. Since $\omega$ is a Reedy equivalence, the induced map $c$ (from the Reedy homotopy pullback of the top right square of the first diagram in this proof) to $E$ is a Reedy equivalence. This achieves the desired factorization of $\rho$ and completes the proof. 
\end{proof}

\begin{proposition}\label{lem:representing-object-for-bigraded-homotopy}
There are natural isomorphisms of $\mc{H}$-algebras
  $$ \pi_p[X^\bu,G]\cong [X^\bu,\Psi^p(G)]_{c\mc{M}^\mc{G}} $$
for all $p\ge 2$. 
\end{proposition}

\begin{proof} 
Since $\Psi^p(G)$ is $\mc{G}$-fibrant, the canonical map 
  $$\pi: [X^\bu,\Psi^p(G)]_{c\mc{M}^{\rm Reedy}} \to [X^\bu,\Psi^p(G)]_{c\mc{M}^\mc{G}}$$ 
is surjective and a group homomorphism by Lemma~\ref{lem:what-is-Psip}. 
By Lemma~\ref{lem:psi-morphism-of-group-objects}, $\psi$ induces a group homomorphism $\psi_*$ in the sequence 
  $$ [X^\bu, P'\Psi^p(G)]_{c\mc{M}^{\rm Reedy}}\xrightarrow{\psi_*} [X^\bu,\Psi^p(G)]_{c\mc{M}^{\rm Reedy}} \xrightarrow{\pi} [X^\bu,\Psi^p(G)]_{c\mc{M}^\mc{G}} \to 0.$$
Since $P'\Psi^p(G)$ is $\mc{G}$-trivial by Lemma~\ref{lem:P-prime-Psi}, one deduces that $\psi_*$ maps into the kernel of $\pi$. Since, moreover, $\psi$ is Reedy equivalent to a \mc{G}-fibration by Lemma~\ref{lem:P-prime-Psi-2}, it can be shown easily that $\psi_*$ maps onto the kernel of $\pi$. So the sequence is exact.

Recall the isomorphisms 
  $$ [X^\bu,\Delta^{p+1}_0(cG)]_{c\mc{M}^{\rm Reedy}}\cong [C^{p+1}X^\bu,G]_{\mc{M}_*}\cong N_{p+1}[X^\bu,G]_{\mc{M}}  $$
from~\ref{def:CnZn} and~\ref{lem:spiral1} for Reedy cofibrant $X^\bu$. 
They yield the following left hand square
  $$\xymatrix{ [X^\bu, P'\Psi^p(G)]_{c\mc{M}^{\rm Reedy}}\ar[r]^-{\psi_*}\ar[d] & [X^\bu,\Psi^p(G)]_{c\mc{M}^{\rm Reedy}}\ar[r]^-{\pi}\ar[d]_{\vartheta}^{{\rm \ref{lem:what-is-Psip}}} & [X^\bu,\Psi^p(G)]_{c\mc{M}^\mc{G}} \ar[r] \ar@{-->}[d]^{\vartheta'} & 0 \\
        N_{p+1}[X^\bu,G]_{\mc{M}}  \ar[r] &  Z_p[X^\bu,G]_{\mc{M}}\ar[r] & \pi_p[X^\bu,G] \ar[r] & 0  
}$$
where left vertical map is induced, using Lemma~\ref{lem:P-prime-Psi}, by the projection
  $$ P'\Psi^p(G)\cong \Delta^{p-1}_0(c\Omega G)\times\Delta^{p+1}_0(cG)\to\Delta^{p+1}_0(cG) $$
and the lower horizontal map is the canonical one.
Commutativity of this square follows since the left square in~\eqref{eqn:a-lot-of-pullbacks} commutes.
We prolong the square to the displayed morphism of exact sequences.
The map $\vartheta$ is surjective by Lemma~\ref{lem:what-is-Psip}. So the induced map $\vartheta'$ between cokernels is surjective as well. 

The left vertical map is surjective because it is a projection. We are going to prove $\ker\vartheta\subset\im\psi_*$. First, one notes $\ker\vartheta=\ker\tau_*$ where $\tau_*$ is part of the long exact sequence
  $$ N_{p-1}[X^\bu,\Omega G]_{\mc{M}}\xrightarrow{\mathfrak{b}} [X^\bu,\Psi^p(G)]_{c\mc{M}^{\rm Reedy}} \xrightarrow{\tau_*} N_p[X^\bu, G] \xrightarrow{d^0} N_{p-1}[X^\bu, G] $$
as shown in Lemma~\ref{lem:what-is-Psip}. Then the left square in~\eqref{eqn:a-lot-of-pullbacks} yields $\ker\tau_*=\im\mathfrak{b}\subset\im\psi_*$.
Finally, a short diagram chase shows that $\vartheta'$ is also injective.
\end{proof}

We now have to promote the isomorphism of \mc{H}-algebras to one of $\pi_0$-modules.
\begin{lemma}\label{lem:pi0-for-bigraded-Psi}
For all $p\ge 2$, the isomorphism from \emph{Proposition~\ref{lem:representing-object-for-bigraded-homotopy}}
  $$ \pi_p[X^\bu,G]\cong [X^\bu,\Psi^p(G)]_{c\mc{M}^\mc{G}} $$
is an isomorphism of $\pi_0$-modules.
\end{lemma}
\begin{proof}
First, we need to explain the $\pi_0$-module structure on the right. This is obtained from a split \mc{G}-fiber sequence 
\begin{equation} \label{split-for-Psi+} 
\Psi^p(G) \to \Psi^p_+(G) \stackrel{\curvearrowleft}{\longrightarrow} cG.
\end{equation}
The total object $\Psi^p_+(G)$ is defined by the following pullback:
\begin{equation}\begin{split}\label{def:psiplus}
  \xymatrix{ \Psi^p_+(G) \ar[r]^{\zeta_+}\ar[d] & P_4 \ar[r]\ar[d] & P_3 \ar[d]^{e_1} \\
               (cG)^{\Delta^p_0} \ar[r] & (cG)^{S^p} \ar[r] & (cG)^{\Delta^{p-1}_0}}
\end{split}
\end{equation}
This diagram is obtained from Diagram~\eqref{combined-square-for-Psiqp} and the terms $P_4$ and 
  $$ P_3\simeq cG\times_{(cG)^{\Delta^{p-1}_0}}c\Path(G)^{\Delta^{p-1}_0}\simeq cG $$
were constructed in the proof of Corollary~\ref{cor:shift-is-pi-0}. The map on the right is induced by ${\rm ev}_1\co\Path(G)\to G$. In this way we obtain the required fiber sequence \eqref{split-for-Psi+} and therefore an associated $\pi_0$-module structure on $[X^\bu, \Psi^p(-)]$. 

Our task is to check that this module structure agrees with the $\pi_0$-module structure on $\pi_p[X^\bu,G]$ as defined in Example~\ref{exam:action-on-bigraded}. To do so we start by discussing the 
necessary analogues of Lemma~\ref{lem:what-is-Psip}(3) and (4) for 
the object $\Psi^p_+(G)$. Consider the Mayer-Vietoris long exact sequence for the functor $[X^\bu,-]_{c\mc{M}^{\rm Reedy}}$ associated to the composite pullback square~\ref{def:psiplus}:
$$
\dots \to [X^\bu, \Psi^p_+(G)] \to [X^\bu, cG^{\Delta^p_0}] \oplus [X^\bu, cG] \stackrel{\eta}{\to} [X^\bu, cG^{\Delta^{p-1}_0}] \to \dots $$
where each term denotes homotopy classes of maps in the Reedy model structure. Let $Z_{p+}[X^\bu, G] : = {\rm ker}(\eta)$. Then we have a
surjective homomorphism 
$$\theta_+ \colon  [X^\bu, \Psi^p_+(G)] \to Z_{p+}[X^\bu, G]$$
which is compatible with the surjective homomorphism $\theta$ from Lemma~\ref{lem:what-is-Psip} and with the projections onto $[X^\bu, cG]$. Note that:
$$[X^\bu, cG^{\Delta^p_0}] \cong N_p [X^\bu, G] \oplus [X^\bu, cG]$$ 
and therefore
$$Z_{p+}[X^\bu, G] \cong Z_p [X^\bu, G] \oplus [X^\bu, cG].$$ 
Similarly to the proof of Proposition \ref{lem:representing-object-for-bigraded-homotopy}, $\theta_+$ 
induces a homomorphism
$$\theta'_+ \colon [X^\bu, \Psi^p_+(G)]_{c \mc{M^G}} \to \pi_0 \Omega^p_+ [X^\bu, G]$$
which is again compatible with $\theta'$ and with the projections onto
$[X^\bu, cG]_{c \mc{M^G}}$, i.e., the following diagram commutes:
$$
\xymatrix{
0 \ar[r] & [X^\bu, \Psi^p(G)]_{c \mc{M^G}} \ar[r] \ar[d]^{\theta'}_{\cong} & [X^\bu, \Psi^p_+(G)]_{c \mc{M^G}} \ar[r] \ar[d]^{\theta'_+} & [X^\bu, cG]_{c \mc{M^G}} \ar[d]^{\cong} \ar[r] & 0 \\ 
0 \ar[r] & \pi_p[X^\bu, G] \ar[r] & \pi_0 \Omega^p_+ [X^\bu, G] \ar[r] & \pi_0[X^\bu, G] \ar[r] & 0
}
$$ 
This diagram is natural in $G$ and therefore the two $\pi_0$-module structures agree. 
\end{proof}

\begin{proposition}\label{derived-fiber-sequence}
The spiral exact sequence
  $$ \hdots\to \naturalpi{2}{X^\bu}{G}\to \pi_2[X^\bu,G]\to \naturalpi{0}{X^\bu}{\Omega G} \to \naturalpi{1}{X^\bu}{G}, $$
is obtained from the homotopy fiber sequence 
  $$ \Psi^2_0(G) \xrightarrow{\zeta} \bar\Omega^0c\Omega G \xrightarrow{\sigma} \Omega^1cG$$
by applying the functor $[X^\bu,-]_{c\mc{M}^\mc{G}}$.
\end{proposition}
\begin{proof}
Note first that this is indeed a homotopy fiber sequence in the resolution model category (see Definition~\ref{def:Psiqp})and therefore yields a long exact sequence after applying the functor $[X^\bu,-]_{c\mc{M}^{\mc{G}}}$. We compare this with the spiral exact sequence using the identification in Proposition~\ref{lem:representing-object-for-bigraded-homotopy}.

In order to make the comparison, we need to show a comparison map first. Equivalently, we need to show that the map $\zeta\co\Psi^p(G)\to \bar\Omega^{p-2}c\Omega G$ really induces the boundary map $b$ in the spiral exact sequence. Going back to the proof of Theorem~\ref{thm:how--spiral-exact-seq-looks}, we recall the isomorphism 
  $$ \pi_p[X^\bu,G]\cong\ker\{ [C^{p}X^\bu,G] \to [C^{p-1}X^\bu,G] \}/\im\{ [C^{p+1}X^\bu,G] \to [C^{p}X^\bu,G] \}.$$ 
The boundary map 
  $$b_{p,0}\co\pi_p[X^\bu,G]\to\naturalpi{p-2}{X^\bu}{\Omega G} $$ 
is induced (and well-defined) by the map $d^0\co Z^{p-1}X^\bu\to C^pX^\bu$ to 
  $$ \naturalpi{p-2}{X^\bu}{\Omega G}\cong\im\beta_{p-2,1}\cong\coker\delta_{p-1,1}\cong\ker\{[Z^{p-1}X^\bu,G] \to [C^{p-1}X^\bu,G]\}. $$
To see that this is induced by $\zeta$ we refine the diagram in Lemma~\ref{lem:psi-morphism-of-group-objects} using Diagram~\eqref{eqn:a-lot-of-pullbacks}: let the left hand side in
  $$\xymatrix{ C^{p-1}X^\bu\ar[r]\ar[d] & Z^{p-1}X^\bu \ar[r]\ar[d] & C^pX^\bu \ar[d] \\                
               \cone(C^{p-1}X^\bu) \ar[r] & \Sigma\bar\Sigma^{p-2}X^\bu \ar[r]^-{\zeta^\dagger} & \pitchfork^pX^\bu }  $$
be a pushout. It is easy to see that $\zeta^\dagger$ and $\zeta$ are adjoint to each other. Unravelling the meaning of these pushouts, it follows that $\zeta^\dagger$ induces $b$.
\end{proof}

\begin{corollary}\label{cor:boundary-is-pi-0}
The boundary map $b$ in the spiral exact sequence is a $\pi_0$-module map.
\end{corollary}

\begin{proof}
In Proposition~\ref{derived-fiber-sequence}, we identified the boundary map $b$ with the map induced by $\zeta\co\Psi^p(G)\to \bar\Omega^{p-2}c\Omega G$. Moreover, we showed in the proof of Lemma~\ref{lem:pi0-for-bigraded-Psi}, precisely in Diagram~\eqref{def:psiplus}, that 
this map is induced by a map of split fiber sequences over $cG$,
  $$\zeta_+\co\Psi^p_+(G)\to P_4.$$
These results imply the required description of $b$ as a $\pi_0$-module 
map.
\end{proof}

Together, Corollaries~\ref{cor:Hur-is-pi-0}, \ref{cor:shift-is-pi-0}
 and \ref{cor:boundary-is-pi-0} imply
\begin{theorem}\label{thm:pi-0-module-structure}
The spiral exact sequence is a long exact sequence of $\pi_0$-modules.
\end{theorem}

\subsection{The spiral spectral sequence}\label{subsec:sss}
The spiral exact sequence can be derived further and yields the {\it spiral spectral sequence}.

\begin{corollary}
The spiral spectral sequence takes the form
\begin{align*}
   E^2_{p,q}(G)=\pi_{p}[X^\bu,\Omega^q G]\ \ \text{ and }\ \ d_2^{p,q}(G)\co\pi_{p}[X^\bu,\Omega^q G]\to\pi_{p-2}[X^\bu,\Omega^{q+1} G]
\end{align*}
The differential $d_2^{p,q}$ is a morphism of \mc{H}-algebras. The terms $E^r_{p,q}$ form $\pi_0$-modules and $d^r_{p,q}$ are morphisms of 
$\pi_0$-modules. If every $G$ admits deloopings $\Omega^qG$ in \mc{G}, for $q<0$, then the spectral sequence converges strongly to $\colim_k\naturalpi{p+k}{X^\bu}{\Omega^{q-k}G}$.
\end{corollary}

\begin{proof}
The differential $d_2$ is an \mc{H}-algebra map by Proposition~\ref{derived-fiber-sequence} 
and even more a $\pi_0$-module morphism by Theorem~\ref{thm:pi-0-module-structure}. Strong Convergence is to be understood in the sense of~\cite[Definition 5.2]{Boa:ccss} and is proved there in Theorem 6.1(a).
\end{proof}

Lemma 2.11 from \cite{Bou:cos} constructs a natural isomorphism
  $$ \Tot\hom^{\rm ext}(K,Y^\bu)\cong\hom_{\mc{M}}(K,\Tot Y^\bu) $$
for all $K$ in \mc{S} and $Y^\bu$ in $c\mc{M}$. This yields an isomorphism
  $$ \Tot\bigl(K(Y^\bu)\bigr)\cong\hom_{\mc{M}_*}(K,\Tot Y^\bu), $$
when both $K$ and $Y^\bu$ are pointed. Consequently, 
  $$ \Tot(\Omega^p_{\rm ext}c\Omega^qG)\cong \Omega^{p+q}G $$
for all $G\in\mc{G}$. Because $\Tot$ is a right Quillen functor for the \mc{G}-resolution model structure by Proposition~\ref{Delta-Tot-adjunction}, there is a map
  $$ t_{p,q}\co\naturalpi{p}{X^\bu}{\Omega^qG}\cong [X^\bu,\Omega^pc\Omega^qG]_{c\mc{M}^{\mc{G}}}\to [\Tot X^\bu,\Omega^{p+q}G], $$
which is compatible with the shift maps, i.e., there is a commutative diagram
  $$ \xymatrix@R=10pt@C=30pt{ \naturalpi{p}{X^\bu}{\Omega^qG}\ar[rr]^{s_{p,q}}\ar[dr]_-{t_{p,q}} && \naturalpi{p+1}{X^\bu}{\Omega^{q-1}G} \ar[dl]^-{\phantom{xxx}t_{p+1,q-1}} \\
      & [\Tot X^\bu,\Omega^{p+q}G] & }.   $$
Commutativity can be shown with the span from Definition~\ref{def:sigma-ss}. One would like to relate the terms $\colim_k\naturalpi{p+k}{X^\bu}{\Omega^{q-k}G}$ and 
$[\Tot X^\bu,\Omega^{p+q}G]$ and prove the convergence of the spiral spectral sequence to the latter object. Such statements are hard to come by in this generality, but in the case 
of cosimplicial spaces and $\F$-GEMs, there are well-known convergence theorems (see also Subsection~\ref{subsec:cohss}). 

\section{$\mc{H}_{{\rm un}}$-algebras and unstable algebras}\label{appsec:H-alg}

In Subsection~\ref{unstable-coalgebras}, we recalled the definition and discussed some basic properties of unstable coalgebras. As explained there, these form a natural target category for 
singular homology with coefficients in $\Fp$ or $\mathbb{Q}$. Dually, unstable algebras form a natural target category for singular cohomology with coefficients in $\Fp$.
For precise definitions, we refer the reader to \cite[1.3 and 1.4]{Schwartz:book}. Let $\UA$ be the category of unstable algebras and $\mc{U}^{\ell}$ the abelian category of unstable left modules. The rational case contains some surprisingly subtleties and will be discussed separately in Subsection~\ref{rational-unstable-algebras}. 

The main result of this appendix is Theorem~\ref{thm:UA-equiv-H-alg} which characterizes unstable algebras as product-preserving functors on the homotopy category of finite $\F$-GEMs. 
This result is mentioned in~\cite[2.1.1 and 5.1.5]{Blanc-Stover} and a somewhat different proof in the case of connected unstable algebras can be found in~\cite{Baues-Jibladze:steenrod-theories}.
So this appendix makes no claim to originality. We provide a detailed proof of the result mainly for completeness, since this characterization is used frequently in the paper 
in order to connect the structure of the general spiral exact sequence from Appendix~\ref{appsec:spiral} with the realization problem for an unstable coalgebra. 

\subsection{The cohomology of Eilenberg-MacLane spaces}
\label{subsec:cohomology-of-EilenbergMacLane}
Let $\F=\Fp$ be a fixed prime field of positive characteristic. Let us abbreviate the notation for Eilenberg-MacLane spaces and write $K_n=K(\Fp,n)$ for $n\ge 0$.

For $n>0$, Cartan's computation~\cite{Cartan} of the $\Fp$-cohomology of $K_n$ yields 
\begin{equation}\label{eqn:Cartan-formula}  
   H^*(K_n)\cong U(L(\Fp[n])), 
\end{equation}
where $U\co\mc{U}^{\ell}\to\UA$ is the free unstable algebra functor, $L\co\Vec\to\mc{U}^{\ell}$ is the Steenrod-Epstein functor which is left adjoint to the forgetful functor, and $\Fp[n]$ is the graded vector space consisting of a copy of $\Fp$ in degree $n$ and $0$ otherwise.

The isomorphism holds also for $n=0$. We have isomorphisms of algebras
  $$ H^0(K_0)=H^0(K(\Fp,0),\Fp)\cong\Hom_{\set}(\Fp,\Fp)\cong \Fp[x]/(x^p-x), $$
and the right side is the free $p$-Boolean algebra on a single generator. We recall that an $\Fp$-algebra $A$ is called $p$-Boolean if it satisfies 
$x^p = x$ for all $x \in A$. By definition, every unstable algebra over $\mathcal{A}_p$ is $p$-Boolean in degree $0$. Therefore, we have for all $A \in \UA$ natural isomorphisms
\begin{equation*}
   \Hom_{\UA}(H^*K(\Fp,0),A)\cong\Hom_{{\rm Alg}_{\Fp}}(\Fp[x]/(x^p-x),A^0)\cong A^0 ,
\end{equation*}
and consequently, (\ref{eqn:Cartan-formula}) holds also for $\F=\Fp$ and $n=0$. 

Altogether this says that the $\Fp$-cohomology of $K_n$ is the free unstable algebra on one generator in degree $n$. Equivalently said, we have established the following
\begin{proposition}\label{prop:identify-representable}
For all primes $p$ and $n\ge 0$, there is a natural isomorphism
  $$\Hom_{\UA}(H^*(K(\Fp,n)),A) \cong A^n.$$
\end{proposition}

\begin{remark} \label{augmented-unstable-algebra-1}
A similar natural isomorphism holds also in the category of augmented unstable algebras. Let $(H^*(K(\Fp,n)), \epsilon)$ be the augmented unstable algebra $H^*(K(\Fp, n))$ 
with augmentation $\epsilon \colon H^*(K(\Fp, n)) \to \underline{\mathbb{F}}_p$ defined by the basepoint of $K(\Fp, n)$. Here $\underline{\mathbb{F}}_p$ denotes the unstable 
algebra which is $\Fp$ in degree $0$ and trivial elsewhere. Then for every augmented unstable algebra $(A, \epsilon_A)$, there is 
a natural isomorphism
$$\Hom_{\UA_{\rm aug}}(H^*(K(\Fp,n)),A) \cong \bar{A}^n$$
between morphisms of augmented unstable algebras and elements in the kernel $\bar{A}$ of the augmentation $\epsilon_A \colon A \to \underline{\mathbb{F}}_p$. 
This is obvious for $n > 0$ and 
it follows easily by inspection of the isomorphisms above for $n = 0$.
\end{remark}

\subsection{Algebraic theories}
We review some definitions and results about algebraic theories from the monograph by Ad\'{a}mek, Rosick\'{y}, and Vitale~\cite{ARV:algebric-cats}.

\begin{definition}
A small category \mc{D} is {\it sifted} if finite products in $\set$ commute with colimits over \mc{D}.
\end{definition}

\begin{definition}\label{def:perf-pres}
An object $A$ in a category \mc{A} is {\it perfectly presentable} if the functor $\Hom_{\mc{A}}(A,-)\co\mc{A}\to\set$ commutes with sifted colimits.
\end{definition}

The following fact is proved in \cite[Theorem 7.7]{ARV:algebric-cats} where also a reference to a more general statement is given.

\begin{theorem}\label{thm:sifted-filteres-reflcoequ}
A functor between cocomplete categories preserves sifted colimits if and only if it preserves filtered colimits and reflexive coequalizers.
\end{theorem}

\begin{definition}\label{def:theories-and-algebras}
An {\it algebraic theory} is a category \mc{T} with finite products. A {\it \mc{T}-algebra} is a set-valued functor from \mc{T} preserving finite products. A morphism of \mc{T}-algebras is a natural transformation of functors.  
The respective category will be denoted by \mc{T}-Alg.
\end{definition}

\begin{definition}\label{def:generators}
If \mc{F} is a set of objects in a category \mc{A}, we use the same symbol \mc{F} to denote the associated full subcategory of \mc{A}. Its opposite category is written $\mc{F}^{\rm op}$ and $\set^{\mc{F}^{\rm op}}$ is the category of contravariant set-valued functors from \mc{F}.
A set \mc{F} of objects in a category \mc{A} is called {\it a set of strong generators for \mc{A}} if the functor
  $$ h_{\mc{F}}\co\mc{A}\to \set^{\mc{F}^{\rm op}}\ , \ \ A\mapsto\Hom_{\mc{A}}(-,A) $$
is faithful and reflects isomorphisms.
\end{definition}

Now suppose that \mc{F} is closed under finite coproducts in $\mc{A}$. Note that $h_{\mc{F}}(A)=\Hom_{\mc{A}}(-,A)$ sends coproducts 
in \mc{F} to products. Put differently, for each $A$ the functor $h_{\mc{F}}(A)\co\mc{F}^{\rm op}\to\set$ preserves products and therefore it is an $\mc{F}^{\rm op}$-algebra. Thus, we actually obtain a functor 

\begin{equation}\label{def:functor-h}   
   h_{\mc{F}}\co\mc{A}\to \mc{F}^{\rm op}\text{-Alg} 
\end{equation}

\begin{theorem}[Ad\'{a}mek-Rosick\'{y}-Vitale\cite{ARV:algebric-cats}] \label{thm:algebraic-category}
Suppose that the category \mc{A} is cocomplete and has a set \mc{F} of perfectly presentable strong generators which is closed under finite coproducts. Then \mc{A} is equivalent, via the functor $h_{\mc{F}}$, to the category of 
$\mc{F}^{\rm op}$-algebras.
\end{theorem}
\begin{proof}
See \cite[Theorem 6.9]{ARV:algebric-cats}.
\end{proof}

\subsection{Unstable algebras are \Hun-algebras}
\label{subsec:unstable-alg-equals-H-alg}
For $n\ge 0$ and some fixed prime $p$ let \mc{G} be the set of Eilenberg-MacLane spaces $K_n=K(\Fp,n)$. Let $\Hun$ be the full sub\-ca\-te\-gory of the homotopy category 
of spaces whose objects are finite products of objects in \mc{G}. These objects are called finite $\Fp$-GEMs. The use of the letters \mc{G} and $\Hun$ here is consistent 
with Subsections~\ref{subsec:algebraic-structure-on-G-homotopy-groups} and~\ref{subsec:ses}. In particular, Definitions~\ref{def:K-algebra-general-unpointed} and \ref{def:theories-and-algebras} coincide; an \Hun-algebra is a product-preserving functor from finite $\Fp$-GEMs to sets.
Let us consider the full subcategory spanned by the set of objects
  $$ \mc{F}=\{H^*(K)\ |\ K\in\Hun\}=H^*(\Hun) \subset \UA$$
Then the representing property~\eqref{representing property of GEMs} of Eilenberg-MacLane spaces gives an iso\-mor\-phism of categories 
$H^*\co\Hun^{\rm op}\stackrel{\cong}{\longrightarrow}\mc{F}.$
We note that \mc{F} is obtained from $$\mc{E}=\{H^*(K(\Fp,n))\,|\,n\ge 0\}$$ by completing with respect to finite products, in the same way that $\Hun$ is obtained from \mc{G}.

\begin{lemma}\label{lem:HK-perf-pres}
For $\F=\Fp$ and every $n \geq 0$, $H^*(K_n)$ is a perfectly presentable object in $\UA$.
\end{lemma}

\begin{proof}
By Theorem~\ref{thm:sifted-filteres-reflcoequ} it suffices to show that the representable functor
$$\Hom_{\UA}(H^*(K_n),-)\co \UA \to \set$$
preserves filtered colimits and reflexive coequalizers for all $n \geq 0$. 
We have seen in Proposition~\ref{prop:identify-representable} that for all $\Fp$ and $n\ge 0$, this representable functor
is isomorphic to the functor $D \mapsto D_n$. This clearly preserves filtered colimits. Let $Q$ denote the colimit of the 
reflexive coequalizer diagram in $\UA$,
\begin{equation}\label{eqn:refl-coequ}
   \xymatrix{D \ar@/^4pt/@<5pt>[rr]^{\alpha} \ar@/_4pt/@<-3pt>[rr]_{\beta} && E \ar@<-1pt>[ll]-|{\sigma} \ar@{-->}[r] & Q.} 
\end{equation}
 The remaining question is whether $Q_n$ is a coequalizer of the induced diagram in $\set$,
 $$ \xymatrix{D_n \ar@/^4pt/@<5pt>[rr]^{\alpha_n} \ar@/_4pt/@<-3pt>[rr]_{\beta_n} && E_n \ar@<-1pt>[ll]-|{\sigma_n} \ar@{-->}[r] & Q_n.} $$
Let $Q'$ denote the colimit of Diagram~(\ref{eqn:refl-coequ}) in graded sets. We will show that $Q'$ is naturally an unstable algebra and has the necessary universal property in \UA. 

For each $d\in D$, we set $\iota(d)=d-\sigma\alpha(d)$. Then $\iota$ is an idempotent morphism of unstable left modules and there is a splitting in $\mc{U}^{\ell}$
  $$ D\cong E\oplus V $$ 
with
  $$ E\cong\im\sigma=\ker\iota\ \ \text{ and }\ \ V=\im\iota=\ker\alpha. $$
Under the splitting, $\alpha$ corresponds to the projection onto $E$. There exists a map $\beta'\co V\to E$ of unstable left modules such that $\beta$ can be written as ${\rm id}_E+\beta'$. Since $\alpha$ is an algebra map, products of the form $ev$ with $e \in E$ and $v \in V=\ker\alpha$ are in $V$. The set quotient 
  $$ Q'\cong E / (e\sim e + \beta'(v)) $$
is the same as the quotient in $\mc{U}^{\ell}$. Moreover, $Q'$ carries an algebra structure inherited from $E$ as one checks that this is well-defined on representatives. Thus, $Q'$ is an object of $\UA$.  The universal property can now 
be verified easily.
\end{proof}

Definition~\ref{def:functor-h} provides a functor $h_{\mc{F}}\co\UA\to\mc{F}^{\rm op}$-Alg that we can compose with the isomorphism $\Hun \cong\mc{F}^{\rm op}$ to obtain a functor
  $$ h\co\UA\to\Hun\text{-Alg}. $$
\begin{theorem}\label{thm:UA-equiv-H-alg}
For $\F=\Fp$ the functor $h$ is an equivalence of categories.
\end{theorem}

\begin{proof}
It is clear that the set $\{H^*(K_n)\ |\ n\ge 0\}$ is a set of strong generators for $\UA$. Lemma~\ref{lem:HK-perf-pres} demonstrates that all its objects are perfectly presentable. 
In particular, \mc{F} is a set of strong generators for $\UA$ which is closed under finite coproducts and whose elements are perfectly presentable. It follows 
from Theorem~\ref{thm:algebraic-category} that $h$ is an equivalence from $\UA$ to $\mc{F}^{\rm op}$-algebras. Finally, cohomology provides an isomorphism $\mc{F}^{\rm op}\cong\Hun$.
\end{proof}

Let $\mc{H}$ denote the full sub\-ca\-te\-gory of the homotopy category of \emph{pointed} spaces whose objects are finite $\Fp$-GEMs with basepoint given by the additive unit using Assumption~\ref{strict-unit}(2). The notation coincides with Definition~\ref{def:K-algebra-general}. 

\begin{remark} \label{augmented-unstable-algebra-2}
Using Remark \ref{augmented-unstable-algebra-1}, similar arguments show an equivalence between the category $\UA_{\rm aug}$ of augmented unstable algebras (or non-unital unstable algebras) 
and the category of $\mc{H}$-algebras. Note that every $\mc{H}$-algebra lifts uniquely to a product-preserving functor with values in pointed sets 
(see Remark~\ref{rem:pointed-values}). 
Using Theorem \ref{thm:UA-equiv-H-alg}, $\UA_{\rm aug}$ is also equivalent to the slice category $\Hun\text{-Alg}/ {\rm Pt}$ where ${\rm Pt}$ denotes the $\Hun$-algebra which is corepresented by $\Delta^0$.
Note that the functor $u$ that we are going to define below admits an analogous version in the setting of \mc{H}-algebras and yields an inverse equivalence with $\UA_{\rm aug}$ (or non-unital unstable algebras). 
\end{remark}

It is useful to describe an inverse functor to $h$. Consider the inclusion functor
  $$ i\co\Hun^{\rm op}\cong\mc{F}\to\UA $$
and its left Kan extension 
  $$ u=Li\co\Hun\text{-Alg}\to\UA $$
along the Yoneda embedding $\Hun^{\rm op}\to\Hun\text{-Alg}$. In the proof of Theorem~\ref{thm:algebraic-category} (see \cite[Theorem 6.9]{ARV:algebric-cats}), it is shown that $Li$ is an 
equivalence and an inverse to $h$. Moreover, the functor $u$ is naturally isomorphic to
  $$ F\in\Hun\text{-Alg}\mapsto \{F(K_n)\}_{n\ge 0},$$
where the graded set $\{F(K_n)\}_{n\ge 0}$ together with all the operations induced by $\Hun$ defines an unstable algebra. The abelian group structure of $K_n \in \ho{\mc{S}}$ induces an 
abelian group structure on $F(K_n)$. Moreover, the morphisms of $\Hun$ that correspond to cohomology classes of finite $\Fp$-GEMs endow this graded abelian group with an algebra structure and an 
action by the Steenrod algebra. To verify the identification of the functor $u$, consider the following  commutative diagram of functors:
  $$ \xymatrix@R=8pt{ && \Hun\text{-Alg} \ar[dd]^{Li=u} \\
      \mc{G}^{\rm op}\ar[r]^-{\subset} &\Hun^{\rm op} \ar[ur]^-{y}\ar[dr]_-{i} \\
                && \UA}  $$               
By \cite[Proposition 4.13]{ARV:algebric-cats}, the Yoneda embedding is a cocompletion under sifted colimits. Hence, to identify $u(F)$, it suffices to examine only the 
representable functors $F$ coming from $\Hun$. By the K\"unneth theorem, it further suffices to examine only representable functors coming 
from \mc{G}. But now we see that for all $m\ge 0$
\begin{align*}
   u(y(K_m))\cong\{H^n(K_m)\}_{n\ge 0} .
\end{align*}
Thus, we have proved
\begin{corollary}\label{cor:inverse-to-h}
Let $\F$ be a prime field of positive characteristic. The functor 
  $$u\co\Hun\text{-Alg}\to\UA\, ,\ u(F)= \{F(K_n)\}_{n\ge 0}$$ 
is an inverse equivalence to $h$.
\end{corollary}

\begin{example}
Let $X$ be a space. The \Hun-algebra $[X,G]$ may be identified with the unstable algebra $H^*(X)$
using the evaluation functor in the previous corollary. Similarly, for a cosimplicial space $X^\bu$ the \Hun-algebra 
$\pi_0[X^\bu,G]$ becomes an unstable algebra. This is a direct consequence of Theorem \ref{thm:UA-equiv-H-alg}.
\end{example}

The identification of the $\pi_0$-modules $G \mapsto \pi_s[X^\bu, G]$ for $s \geq 1$ (Example~\ref{exam:action-on-bigraded}) recovers a familiar structure as well. 

\begin{proposition}\label{lem:pi0-identified}
Let $M$ be a module over the $\Hun$-algebra $A$ in the sense of Definitions~\emph{\ref{def:pi-0-modules}} and \emph{\ref{Hun-algebra-mod}}. Then $u(M)$ is a module over the unstable algebra $u(A)$. Furthermore, for $s \geq 1$, the functor $h$ identifies the 
$\pi_0H^*(X^\bu)$-module $\pi_sH^*(X^\bu)$ with the $\pi_0[X^\bu, -]$-module $\pi_s[X^\bu, -]$.
\end{proposition}

\begin{proof}
First note that $u(M)$ defines an unstable module since $M$ is the 
kernel of an abelian object in $\Hun\text{-Alg}/A$. For each map $f\co G_1\to G_2$ in $\Hun$, one obtains an action map
  $$ \phi_f\co A(G_1)\times M(G_1)\to M(G_2)$$
and the collection of these maps satisfies certain compatibility requirements (see Definition~\ref{Hun-algebra-mod}).
These define an $u(A)$-action on $u(M)$ which makes it into an 
$u(A)$-module.
 
Now set $A=[X^\bu,-]$ and $M=\Omega^s[X^\bu,-]$. 
By Theorem~\ref{thm:UA-equiv-H-alg} and Corollary~\ref{cor:inverse-to-h}, we can identify the unstable algebra $\pi_0H^*(X^\bu)$ with the \Hun-algebra $\pi_0[X^\bu,G]$. Recall that the fiber sequence~(Example~\eqref{exam:action-on-bigraded}):
  $$ \Omega^s[X^\bu,G]\to \Omega_+^s[X^\bu,G]\to [X^\bu,G] $$
admits a section which is induced by the constant map $c:\Delta^s/\partial\Delta^s\to\ast$ and this determines an isomorphism
  $$ [X^\bu,G]\times \Omega^s[X^\bu,G]\cong \Omega_+^s[X^\bu,G].$$

Consider the map $\mu :K_m \times K_n \to K_{m+n}$ which sends the fundamental classes of the factors in the domain to the fundamental class in the target. Then the action of $\pi_0A$ on $\pi_0M$ is induced by the following composition of maps:
\begin{align} \label{first-module-action}
\begin{split}
[X^\bu ,K_m ]&\times \Omega^s [X^\bu ,K_n ]\to \Omega^s_+[X^\bu ,K_m ]\times \Omega^s_+ [X^\bu ,K_n]=  \\
 = &\Omega^s_+[X^\bu ,K_m \times K_n ] \stackrel{\mu_*}{\longrightarrow} \Omega^s_+[X^\bu ,K_{m+n}]\to\Omega^s [X^\bu ,K_{m+n}].
\end{split}
\end{align} 
On the other hand, the action of $\pi_0 H^*(X^\bu)$ on $\pi_sH^*(X^\bu)$ is induced by the following composition:
\begin{equation} \label{second-module-action}
H^m (X^0) \times H^n (X^s) \to H^m (X^s) \times H^n (X^s ) \to 
H^{m+n} (X^s) 
\end{equation}
where the first map comes from the unique degeneracy map and the 
last map is given by the cup product. It is straightforward to check that these actions correspond to each other if we represent classes 
in $\pi_0 A$ and $\pi_0 M$ by elements in $H^*(X^0)$ and in $H^*(X^s)$, respectively.  
\end{proof}

\subsection{Unstable algebras, rationally} \label{rational-unstable-algebras}
At least in positive degrees, this case is much simpler and most of the above applies similarly, since there are no non-trivial unary cohomology operations rationally. 
Let us write $K_n=K(\mathbb{Q},n)$, $n\ge 0$, for the corresponding Eilenberg-MacLane space.

For $n>0$, the computation of the $\mathbb{Q}$-cohomology of $K_n$ yields 
\begin{equation}\label{eqn:Cartan-formula-2}  
   H^*(K_n)\cong U(\mathbb{Q}[n]), 
\end{equation}
where $U(-)$ is the free graded commutative algebra functor. As a consequence, Proposition \ref{prop:identify-representable} holds in this case as well. 

However, there is a subtlety regarding the structure of a rational unstable algebra in degree $0$. On the one hand, there is no obvious analogue of 
the $p$-Boolean property, and on the other hand, the rational cohomology of a space is set-like, i.e.\! isomorphic to the $\mathbb{Q}$-algebra $\mathbb{Q}^X$ of functions $X \to \mathbb{Q}$ for some set $X$ -- which is also the dual algebra of the set-like $\mathbb{Q}$-coalgebra on $X$ (cf. Definition \ref{set-like-deg-0-2}). 

The problem with simply defining rational unstable algebras to be non-negatively graded commutative $\mathbb{Q}$-algebras that are set-like in degree $0$ is that the resulting theory is 
not algebraic in the sense of Theorem \ref{thm:UA-equiv-H-alg}. To see why this definition does not lead to an algebraic
theory, we apply a result of Heyneman and Radford in \cite[3.7]{heyneman-radford}, which says that for a reasonable set $X$, the algebra $\mathbb{Q}^X$ is (co)reflexive, that is, the canonical map 
to the finite (or restricted) dual of $\mathbb{Q}^X$,
  $$\mathbb{Q}(X) \to \Hom_{\set}(X,\mathbb{Q})^{\circ},$$ 
is an isomorphism of $\mathbb{Q}$-coalgebras. We recall that 
  $$(-)^{\circ}\co \mathrm{Alg}_{\mathbb{Q}}^{\op} \to \mathrm{Coalg}_{\mathbb{Q}}$$
is the right adjoint of the functor $(-)^{\dual}\co \mathrm{Coalg}_{\mathbb{Q}} \to \mathrm{Alg}_{\mathbb{Q}}^{\op}$, taking a
coalgebra to its linear dual algebra. See \cite[Chapter IV]{sweedler:hopf} for a good account of the general properties of 
this functor. As a consequence, we have isomorphisms
\begin{equation*}
\Hom_{{\rm Alg}_{\mathbb{Q}}}(\mathbb{Q}^{X}, \mathbb{Q}^Y)\cong \Hom_{{\rm Coalg}_{\mathbb{Q}}}(\mathbb{Q}(Y), \mathbb{Q}(X)) \cong \Hom_{\set}(Y, X)
\end{equation*}
showing that $\set^{\op}$ (for  reasonable sets) is equivalent to set-like algebras and thus the degree $0$ part of the proposed definition is not algebraic.

In relation to this we note several amusing facts. The isomorphisms above show that $H^*(K_0)$ satisfies Proposition \ref{prop:identify-representable} if $A^0$ is set-like. 
However, the set-like property is not preserved under filtered colimits. It follows that $H^0(K_0)$ is not perfectly presentable in the category of set-like $\mathbb{Q}$-algebras. 
Curiously though, mapping out of $H^0(K_0)$ to set-like $\mathbb{Q}$-algebras preserves reflexive coequalizers.

Even in positive characteristic the relation between the $p$-Boolean and set-like properties is not absolutely tight; for example, the free $p$-Boolean $\Fp$-algebra on countably many generators is not set-like. There is, however, a 
duality theorem saying that a $p$-Boolean algebra $A$ is isomorphic to the algebra of continuous $\Fp$-valued maps on $\mathrm{Spec}(A)$, see \cite[Appendix]{kuhn:generic}.

Let $\mc{G}_{\mathbb{Q}}$ denote the set of Eilenberg-MacLane spaces $K_n=K(\mathbb{Q},n)$ for $n\ge 0$, and let $\Hun{}_{,\mathbb{Q}}$ be the full 
subcategory of the homotopy category of spaces whose objects are the finite products of objects in $\mc{G}_{\mathbb{Q}}$. 
As before, $\Hun{}_{,\mathbb{Q}}$ is an algebraic theory. Motivated especially by Theorem \ref{thm:UA-equiv-H-alg}, the following definition takes into account all these issues and 
is particularly suitable to our purposes.

\begin{definition} \label{def:unstable-alg}
An unstable $\mathbb{Q}$-algebra is an $\Hun{}_{,\mathbb{Q}}$-algebra.
\end{definition} 
 
Besides its intrinsic interest, the functorial point of view via algebraic theories, is also desirable here in order 
to enable, also in the rational case, a nice translation of the approach of Appendix~\ref{appsec:spiral} into the language 
of unstable (co)algebras. Note that this definition is formally the same as in the case of positive characteristic.

\begin{example}
A set-like $\mathbb{Q}$-algebra admits an obvious action from the algebra $H^0(K_0)$, the $\mathbb{Q}$-algebra of functions $\mathbb{Q} \to \mathbb{Q}$. In particular, 
the dual of an unstable $\mathbb{Q}$-coalgebra (Def.~\ref{set-like-deg-0-2} ) is an unstable $\mathbb{Q}$-algebra 
in the sense of Definition \ref{def:unstable-alg}.
\end{example}

\begin{example}
Let $X$ be a space. The rational singular cohomology produces an unstable $\mathbb{Q}$-algebra as follows, 
$$\mc{G}_{\mathbb{Q}}\to\set\ , \ \ G \mapsto [X,G].$$ 
Similarly, for a cosimplicial space $X^\bu$, we have an unstable $\mathbb{Q}$-algebra
$$G \mapsto \pi_s[X^\bu,G].$$ 
We note that 
$$\pi_0(H^*(X^\bu))^0 = \mathrm{coeq}\bigl(H^0(X^1) \rightrightarrows H^0(X^0)\bigr)$$
is actually again set-like. 
\end{example}

\subsection{The cohomology spectral sequence}
\label{subsec:cohss}

Let us return to the spiral spectral sequence, which was introduced in Subsection~\ref{subsec:sss} for a general resolution model category,
and specialize to \mc{M} being the category of spaces, and $\mc{G}=\{K(\F,n)\,|\,n\ge 0\}$ where $\F$ is any prime field. 
The reindexed $E^2$-page for $G=K(\F,n)$ is identified as
  $$E^2_{s,t}\cong\pi_s H^{n-t}(X^\bu).$$
The exact couple~(\ref{spiral-les}) of \mc{H}-algebras
  $$ \hdots\to [Z^{m-1}X^\bu,\Omega^{q+1}G]\to [Z^{m}X^\bu,\Omega^{q}G] \to [C^{m}X^\bu,\Omega^qG]\to[Z^{m-1}X^\bu,\Omega^{q}G]\to\hdots $$
is obtained by passing from homology to cohomology in Rector's construction of  the homology spectral sequence \cite{Rector:EMSS}.
Consequently, the spiral spectral sequence is, up to re-indexing, the linear dual of the homology spectral sequence of the cosimplicial space $X^\bu$.    
 In case the terms of the spectral sequence are all of finite type, the strong convergence results of Bousfield \cite[Theorem 3.4]{Bou: homology SS} and 
 Shipley \cite[Theorem 6.1]{Shipley:convergence} for the homology spectral sequence imply analogous results for the convergence of the cohomology spectral 
 sequence to $H^*(\mathrm{Tot^f} (X^\bu))$.
 
From results proved in Subsection~\ref{subsec:ses}, and the identification in Proposition~\ref{lem:pi0-identified}, we obtain the following
 
\begin{proposition} 
The cohomology spectral sequence of a cosimplicial space  $X^\bu$ is a spectral sequence of $\pi_0H^*(X^\bu)$-modules.
\end{proposition}

\section{Moduli spaces in homotopy theory} \label{DK-theory}

The approach to moduli spaces that we take in this paper follows and makes essential use of various results from a series of papers by Dwyer and Kan. 
In this appendix, we recall the necessary background material and give a concise review of those results that are required in the paper. The interested 
reader should consult 
the original papers for a more complete account \cite{DK-classification, DK-function comp, DK-simp loc, DK-calc simp loc} and 
\cite{BlDG:pi-algebra}.

\subsection{Simplicial localization} 
Let $(\mc{C}, \mc{W})$ be a small category $\mc{C}$ together with a subcategory of \emph{weak equivalences} $\mc{W}$ which 
contains the isomorphisms. In their seminal work, Dwyer and Kan introduced two constructions of a simplicially enriched category 
associated to $(\mc{C}, \mc{W})$. Each of them is a refinement of the passage to the homotopy category $\mc{C}[\mc{W}^{-1}]$, 
which is given by formally inverting all weak equivalences, and they uncover in a functorial way the rich homotopy theory encoded 
in the pair $(\mc{C}, \mc{W})$.

The \emph{simplicial localization} $L(\mc{C}, \mc{W})$ \cite{DK-simp loc} is defined by a general free simplicial resolution
$$L(\mc{C}, \mc{W})_n = F_n \mc{C} [(F_n \mc{W})^{-1}]$$
where $n$ indicates the simplicial degree and $F_n$ denotes the free category functor iterated $n+1$ times. As the simplicial set of objects is constant, $L(\mc{C}, \mc{W})$ is a simplicially enriched category. We will follow the standard 
abbreviation and simply call such categories simplicial. 

The following fact, saying that the simplicial localization preserves the homotopy type of the classifying space, will be useful.

\begin{proposition}[Dwyer-Kan \cite{DK-simp loc}] \label{DK-theory1}
There are natural weak equivalences of spaces 
$$ B \mc{C} \stackrel{\sim}{\leftarrow} B(F_{\bullet} \mc{C}) \stackrel{\sim}{\rightarrow} B(L(\mc{C}, \mc{W})_{\bullet}).$$
\end{proposition}
\begin{proof}
See \cite[4.3]{DK-simp loc}.
\end{proof}

\begin{remark}
As a side remark on the last proposition, we like to note that in the equivalent context of $\infty$-categories, one may 
regard the simplicial localization as given by taking a pushout along a disjoint union of copies of the 
inclusion of categories
$$(\bullet \rightarrow \bullet) \to (\bullet \rightleftarrows \bullet),$$
one for each morphism in $\mc{W}$. Since this inclusion is a nerve equivalence (but not a Joyal equivalence), it follows
that the homotopy type of the classifying space is invariant under such pushouts. 
\end{remark}

On the other hand, we also have the \emph{hammock localization} $L^H(\mc{C}, \mc{W})$ \cite[3.1]{DK-function comp}, \cite{DK-calc simp loc} 
whose morphisms in simplicial degree $n$ are given by ``reduced hammocks of width $n$'':
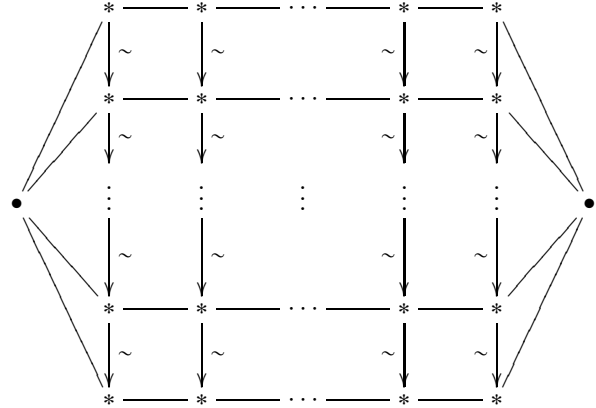
\[
 \xymatrix{
 & \ast \ar@{-}[r] \ar[d]^{\sim} & \ast \ar@{-}[r] \ar[d]^{\sim} & \cdots \ar@{-}[r] & \ast \ar@{-}[r] \ar[d]_{\sim} & \ast \ar[d]_{\sim} & \\
  & \ast \ar@{-}[r] \ar[d]^-{\sim} & \ast \ar@{-}[r] \ar[d]^-{\sim} & \cdots \ar@{-}[r] & \ast \ar@{-}[r] \ar[d]_-{\sim} & \ast \ar[d]_-{\sim} & \\
 \bullet \ar@{-}[ruu] \ar@{-}[ru] \ar@{-}[rd] \ar@{-}[rdd] & \vdots \ar[d]^{\sim} & \vdots \ar[d]^{\sim} & \vdots & \vdots \ar[d]_{\sim} & \vdots \ar[d]_{\sim} &  \bullet \ar@{-}[ldd] \ar@{-}[ld] \ar@{-}[lu] \ar@{-}[luu] \\
 & \ast \ar@{-}[r] \ar[d]^{\sim} & \ast \ar@{-}[r] \ar[d]^{\sim} & \cdots \ar@{-}[r] & \ast \ar@{-}[r] \ar[d]_{\sim} & \ast \ar[d]_{\sim} &\\
 & \ast \ar@{-}[r] & \ast \ar@{-}[r] & \cdots \ar@{-}[r] & \ast \ar@{-}[r] & \ast  & \\
}
\]
This diagram shows $n$ composable morphisms between zigzag diagrams in $(\mc{C}, \mc{W})$ of the same shape, the same source and target, 
and whose components are weak equivalences. The horizontal morphisms in each column go in the same direction; if they go to the left, then 
they are in $\mc{W}$. Moreover, the horizontal morphisms in adjacent columns go in different directions and no column contains only identity maps. 

This is a more explicit definition of a simplicial category and is often more practical than the simplicial localization. The following proposition says 
that the two constructions lead to 
weakly equivalent simplicially enriched categories, i.e., there is zigzag of functors which induces homotopy equivalences of mapping spaces 
and equivalences of homotopy categories. 

\begin{proposition}[Dwyer-Kan \cite{DK-calc simp loc}] \label{DK-theory2}
There are natural weak equivalences of simplicial categories 
$$L^H(\mc{C}, \mc{W}) \stackrel{\sim}{\leftarrow} \mathrm{diag} L^H(F_{\bullet}\mc{C}, F_{\bullet} \mc{W}) \stackrel{\sim}{\rightarrow} L(\mc{C}, \mc{W}).$$
\end{proposition}
\begin{proof}
See \cite[2.2]{DK-calc simp loc}.
\end{proof}

\subsection{Models for mapping spaces} If the pair $(\mc{C}, \mc{W})$ is part of a simplicial model category structure on $\mc{C}$, the simplicial enrichment 
of $L^H(\mc{C}, \mc{W})$ agrees up to homotopy with the (derived) mapping spaces of the simplicial model category. 

\begin{theorem}[Dwyer-Kan \cite{DK-function comp}] \label{DK-theory3}
Let $\mc{C}$ be a simplicial model category with weak equivalences $\mc{W}$ and  $X, Y \in \mc{C}$ where $X$ is 
cofibrant and $Y$ is fibrant. Then there is a weak equivalence 
$$L^H(\mc{C}, \mc{W})(X, Y) \simeq \map_{\mc{C}}(X, Y)$$
which is given by a natural zigzag of weak equivalences. 
\end{theorem}
\begin{proof}
This is \cite[4.7]{DK-function comp} and is a special case of \cite[4.4]{DK-function comp}. The more general case applies also to non-simplicial model 
categories with $\map(-,-)$ replaced in this case by a functorial choice of derived mapping spaces defined in terms of (co)simplicial resolution
of objects. 
\end{proof}

\begin{remark}
There are some obvious set-theoretical issues which need to be addressed here in order to extend the hammock
localization to non-small categories such as model categories. We refer the reader to \cite{DK-function comp} for the details. 
\end{remark}

There is no direct way to compare the compositions in the simplicial category of cofibrant-fibrant objects $\mc{C}^{\rm cf}_{\bullet}$ and the hammock localization $L^H(\mc{C}^{\rm cf}, \mc{W})$. 
However, there is a strengthening of Theorem \ref{DK-theory3} (see \cite[4.8]{DK-function comp}) which says in addition that the weak equivalences of mapping spaces can be upgraded to a weak 
equivalence of simplicial categories. To achieve such a comparison, one considers the hammock localization applied degreewise to $\mc{C}^{\rm cf}_{\bullet}$. Then it is easy to see that that 
the obvious functors
\begin{equation} \label{DK-theory3b}
\mc{C}^{\rm cf}_{\bullet} \stackrel{\sim}{\rightarrow} {\rm diag}L^H(\mc{C}^{\rm cf}_{\bullet}, \mc{W}_{\bullet}) \stackrel{\sim}{\rightarrow} {\rm diag}L^H(\mc{C}_{\bullet}, \mc{W}_{\bullet}) \stackrel{\sim}{\leftarrow} L^H(\mc{C}, \mc{W})
\end{equation}
are weak equivalences of simplicial categories. 

In the presence of a homotopy calculus of fractions, the shapes of zigzag diagrams appearing in the definition of the hammock localization can be reduced to a single zigzag shape of 
arrow length 3. Although this reduction does not respect the composition in the category, which is given by concatenation and thus increases the length of the hammock, this reduction
is useful in obtaining smaller models for the homotopy types of the mapping spaces. 

Let $\mathscr{MAP}(X, Y)$ denote the (geometric realization of the) subspace of the mapping space $L^H(\mc{C}, \mc{W})(X, Y)$, which is 
defined by $0$-simplices of the form 
$$X \stackrel{\sim}{\leftarrow} \bullet \rightarrow \bullet \stackrel{\sim}{\leftarrow} Y.$$
More precisely, $\mc{W}_{\rm Hom}(X, Y)$ will denote the category of such zigzag diagrams from $X$ to $Y$ and weak equivalences 
between them, and $\mathscr{MAP}(X, Y)$ its classifying space.

\begin{proposition}[Dwyer-Kan \cite{DK-calc simp loc}] \label{DK-theory4}
Let $\mc{C}$ be a model category with weak equivalences $\mc{W}$ and  $X, Y \in \mc{C}$. Then there is a natural 
weak equivalence 
$$\mathscr{MAP}(X, Y) \simeq L^H(\mc{C}, \mc{W})(X, Y).$$
\end{proposition}
\begin{proof}
See \cite[6.2 and 8.4]{DK-calc simp loc}. 
\end{proof}

\begin{remark} \label{X-cof-zigzag}
If $X$ is cofibrant, then the smaller category $\mc{W}^{\rm c}_{\rm Hom}(X, Y)$ whose objects are zigzags in which 
the first weak equivalence is the identity map has the same homotopy type (see \cite{DK-calc simp loc} and 
\cite{Dugger:correction}). The analogous statement when $Y$ is fibrant also holds. 
\end{remark}

\subsection{Moduli spaces}\label{subsec:moduli-spaces}
We now turn to the definition of moduli spaces. Let $\mc{C}$ be a simplicial model category and $X \in \mc{C}$. Let $\mc{W}(X)$ denote the subcategory whose objects are the 
objects of $\mc{C}$ which are weakly equivalent to $X$ and the morphisms are weak equivalences between them. The classifying 
space $\mc{M}(X)$ of the category $\mc{W}(X)$ is called the moduli space for objects of type $X$. We remark that 
although $\mc{W}(X)$ is not a small category, it is nevertheless homotopically small which suffices for the purposes of 
extracting a well-defined homotopy type. We refer to \cite{DK-classification, DK-function comp} for more details related to 
this set-theoretical issue. Note that $\mc{M}(X)$ is always path-connected.

Combining the previous results, we obtain the following theorem.

\begin{theorem} [Dwyer-Kan \cite{DK-classification}] \label{DK-theory5}
Let $\mc{C}$ be a simplicial model category and $X \in \mc{C}$ an object which is both cofibrant and fibrant. Then 
there is a weak equivalence 
$$\mc{M}(X) \simeq \mathrm{B Aut^h}(X)$$
which is given by a natural zigzag of weak equivalences. 
\end{theorem}
\begin{proof}
The combination of Propositions \ref{DK-theory1} and \ref{DK-theory2} gives a weak equivalence 
$$\mc{M}(X) \simeq B(L^H(\mc{W}(X), \mc{W}(X))) \simeq B(L^H(\mc{W}(X), \mc{W}(X))(X,X)).$$ 
The map of simplicial monoids
\[
\xymatrix{
L^H(\mc{W}(X), \mc{W}(X))(X,X) \ar[r] & L^H(\mc{C}, \mc{W})(X, X)
}
\]
is the inclusion of those connected components whose $0$-simplices define invertible maps in the homotopy category
(see \cite[4.6]{DK-function comp}). Then from Theorem \ref{DK-theory3} and the weak equivalences of \eqref{DK-theory3b}, 
we can conclude the required weak equivalence 
$$\mathcal{M}(X) \simeq \mathrm{B Aut^h}(X)$$
by restriction to the appropriate components.
\end{proof}

The definition of moduli spaces clearly extends from objects to morphisms, or more general diagrams, by looking at the corresponding model 
category of diagrams. Following \cite{BlDG:pi-algebra}, for two objects $X$ and $Y$ in a model category $\mc{C}$, we 
write
$$\mc{M}(X \rightsquigarrow Y)$$
for the classifying space of a category $\mc{W}(X \rightsquigarrow Y)$ whose objects are maps 
$$U \to V$$
where $U$ is weakly equivalent to $X$ and $V$ is weakly equivalent to $Y$, and whose morphisms are square 
diagrams as follows
\[
 \xymatrix{
 U \ar[r] \ar[d]^{\sim} & V \ar[d]^{\sim} \\
 U' \ar[r] & V'
 }
\]
where the vertical maps are weak equivalences. More generally, for every property $\phi$ which applies to an
object $(U \to V) \in \mc{W}(X \rightsquigarrow Y)$ and whose satisfaction depends only on the respective 
component of $\mc{W}(X \rightsquigarrow Y)$, we define 
$$\mc{M}(X \stackrel{\phi}{\rightsquigarrow} Y)$$
to be the classifying space of the corresponding subcategory. We can clearly combine and extend these 
definitions to define also
moduli spaces of more complicated diagrams, e.g., 
$$\mc{M}(Z \stackrel{\phi}{\leftsquigarrow} X \stackrel{\phi'}{\rightsquigarrow} Y).$$ 

There are various maps connecting these moduli spaces. Among them, we distinguish the following two types of maps.
\begin{itemize}
 \item Maps which forget structure, e.g., the map $\mc{M}(X \rightsquigarrow Y) \to \mc{M}(X) \times \mc{M}(Y)$ which
 is induced by the obvious forgetful functor. 
 \item Maps which neutralize structure, e.g., the map $\mathscr{MAP}(X, Y) \to \mc{M}(X \rightsquigarrow Y)$ which 
 is induced by the functor $(X \stackrel{\sim}{\leftarrow} U \to V \stackrel{\sim}{\leftarrow} Y) \mapsto (U \to V)$.
\end{itemize}

The following theorem from \cite{BlDG:pi-algebra} explains the relations between these types of moduli spaces 
in some important exemplary cases. 

\begin{theorem} \label{calculus-with-moduli-spaces}
Let $\mc{C}$ be a model category and $X, Y, Z \in \mc{C}$. Then:
\begin{itemize} 
 \item[(a)] The sequence 
 $$\mathscr{MAP}(X, Y) \rightarrow \mc{M}(X \rightsquigarrow Y) \rightarrow \mc{M}(X) \times \mc{M}(Y)$$
 is a homotopy fiber sequence.
 \item[(b)] Suppose that $\mc{M}(Z \stackrel{\phi}{\leftsquigarrow} X \stackrel{\phi'}{\rightsquigarrow} Y)$ is non-empty. Then there is a homotopy fiber sequence 
 $$\mathscr{MAP}(X, Y; \phi') \rightarrow \mc{M}(Z \stackrel{\phi}{\leftsquigarrow} X \stackrel{\phi'}{\rightsquigarrow} Y) \rightarrow \mc{M}(Z \stackrel{\phi}{\leftsquigarrow} X) \times \mc{M}(Y)$$
 where $\mathscr{MAP}(X, Y; \phi') \subseteq \mathscr{MAP}(X, Y)$ denotes the subspace of maps which satisfy property $\phi'$.
\end{itemize}
\end{theorem}
\begin{proof}(Sketch)
Part (a) is \cite[Theorem 2.9]{BlDG:pi-algebra}, but the proof seems to contain an error. A correction can be made following \cite{Dugger:correction}
where a similar error from \cite{DK-function comp} is corrected. Consider the category $\mc{W}^{\rm tw}(X \rightsquigarrow Y)$
whose objects are the same as the objects of $\mc{W}(X \rightsquigarrow Y)$ but morphisms are square diagrams as follows
\[
\xymatrix{
U \ar[r] \ar[d]^{\sim} & V  \\
U' \ar[r] & V' \ar[u]^{\sim}
}
\]
where both vertical maps are weak equivalences. Let $\mc{M}^{\rm tw}(X \rightsquigarrow Y)$ denote the classifying space of this 
category. Although there is no natural functor in general that connects the categories $\mc{W}(X \rightsquigarrow Y)$ and 
$\mc{W}^{\rm tw}(X \rightsquigarrow Y)$, nevertheless there is a natural zigzag of 
weak equivalences
$$\mc{M}(X \rightsquigarrow Y) \simeq \mc{M}^{\rm tw}(X \rightsquigarrow Y).$$
To see this, consider an auxiliary double category $\mc{W}_{\square}(X \rightsquigarrow Y)$ defined as follows:
\begin{itemize}
 \item the objects are the same as in $\mc{W}(X \rightsquigarrow Y)$,
 \item the horizontal morphisms are the morphisms of $\mc{W}(X \rightsquigarrow Y)$,
 \item the vertical morphisms are the morphisms of $\mc{W}^{\rm tw}(X \rightsquigarrow Y)$,
 \item the squares are those diagrams that make everything commute. 
\end{itemize}
The nerve of this double category is a bisimplicial set. We write $\mc{M}_{\square}(X \rightsquigarrow Y)$ to denote the 
realization of its diagonal. By construction, there are natural maps 
$$\mc{M}(X \rightsquigarrow Y) \rightarrow \mc{M}_{\square}(X \rightsquigarrow Y) \leftarrow \mc{M}^{\rm tw}(X \rightsquigarrow Y)$$
and an application of \cite[Proposition 3.9]{Dugger:correction} - or an obvious analogue - shows that both maps are weak equivalences.

Similarly, we define $\mc{W}^{\rm tw}_{\rm Hom}(X, Y)$ to be the analogous variant of the category $\mc{W}_{\rm Hom}(X, Y)$ where one of the maps
has the opposite variance. Again, by \cite{Dugger:correction}, we have a natural zigzag of weak equivalences:
\begin{equation} \label{twisted-hom}
B(\mc{W}^{\rm tw}_{\rm Hom}(X,Y)) \simeq \mathscr{MAP}(X, Y).
\end{equation}
Then it suffices to show that the sequence of functors 
$$\mc{W}^{\rm tw}_{\rm Hom}(X, Y) \to \mc{W}^{\rm tw}(X \rightsquigarrow Y) \to \mc{W}(X) \times \mc{W}(Y)^{\rm op}$$
induces a homotopy fiber sequence of classifying spaces. This is an immediate application of Quillen's Theorem B, 
as explained in \cite[2.9]{BlDG:pi-algebra}.

The proof of (b) is similar (cf. \cite[2.11]{BlDG:pi-algebra}). Consider the category $\mc{W}^{\rm tw}(Z \stackrel{\phi}{\leftsquigarrow} X \stackrel{\phi'}{\rightsquigarrow} Y)$
which has the same objects as $\mc{W}(Z \stackrel{\phi}{\leftsquigarrow} X \stackrel{\phi'}{\rightsquigarrow} Y)$ but morphisms are commutative diagrams of the form:
\[
 \xymatrix{
 U \ar[d]^{\sim} & V \ar[d]^{\sim} \ar[r] \ar[l] & W \\
 U' & V \ar[l] \ar[r] & W' \ar[u]_{\sim}
 }
\]
Using the methods of \cite{Dugger:correction}, it can be shown that these two categories have homotopy equivalent classifying spaces. There is a forgetful functor 
$$\mathscr{F}\co \mc{W}^{\rm tw}(Z \stackrel{\phi}{\leftsquigarrow} X \stackrel{\phi'}{\rightsquigarrow} Y) \to \mc{W}(Z \stackrel{\phi}{\leftsquigarrow} X) \times \mc{W}(Y)^{\rm op}.$$
Given $(U \leftarrow V, W) \in \mc{W}(Z \stackrel{\phi}{\leftsquigarrow} X) \times \mc{W}(Y)^{\rm op}$, the comma-category $\mathscr{F} \downarrow (U \leftarrow V, W)$ has as objects 
diagrams of the form:
\[
 \xymatrix{
 U' \ar[d]^{\sim} &  V' \ar[d]^{\sim} \ar[r] \ar[l] & W' \\
 U & V \ar[l] & W \ar[u]_{\sim}
 }
\]
It is easy to see that there are functors which induce a pair of inverse equivalences
$$\mc{W}^{\rm tw}_{\rm Hom}(V, W) \leftrightarrows \mathscr{F} \downarrow (U \leftarrow V, W).$$
Then the required result is an application of Quillen's Theorem B using \eqref{twisted-hom} and Theorems \ref{DK-theory3} and \ref{DK-theory4}.
\end{proof}

\subsection{A moduli space associated with a directed diagram} \label{DK-direct-dia} 
We discuss in detail the case of a moduli space associated with a direct system of objects in a model category. The diagrams that we actually have in mind for our 
applications have the form of a dual Postnikov tower, but we will present the main result in the general case of an arbitrary diagram. In the case of the model 
category of simplicial sets, this result is a special case of the general theorem of \cite{DK-classification}, and in fact the generalization to an arbitrary model 
category can also be treated similarly.

Let $\mathcal{M}$ be a simplicial model category. Let $\mathbf{N}$ denote the poset of natural numbers and 
$\mathbf{n}$ the subposet $\{0 < 1 < \cdots < n \}$. Given a diagram  $X\co \mathbf{N} \to \mathcal{M}$, let 
$X_{\leq n}$ denote its restriction to $\mathbf{n}$. Two diagrams $X, Y\co \mathbf{N} \to \mathcal{M}$ are called 
\emph{conjugate} if their restrictions $X_{\leq n}$ and $Y_{\leq n}$ are pointwise weakly equivalent. More generally,
given a small category $C$ and diagrams $X, Y\co C \to \mathcal{M}$, these are called \emph{conjugate} if for every 
$J\co \mathbf{n} \to C$
the pullback diagrams $X \circ J$ and $Y \circ J$ are pointwise weakly equivalent. 

Let $X\co \mathbf{N} \to \mathcal{M}$ be a diagram. Consider the category $\mathrm{Conj}(X)$ of $\mathbf{N}$-diagrams 
in $\mathcal{M}$ whose objects are the conjugates of $X$ and whose morphisms are pointwise weak equivalences between conjugates. 
The moduli space of conjugates (called \emph{classification complex} in \cite{DK-classification}) of $X$, denoted here by $\mc{M}_{\mathrm{conj}}(X)$, 
is the classifying space of the (homotopically small) category $\mathrm{Conj}(X)$. Similarly we define moduli spaces 
$\mc{M}_{\mathrm{conj}}(X)$ for more general diagrams $X\co C \to \mathcal{M}$. 

We have canonical compatible restriction maps:
$$r_{n}\co \mc{M}_{\mathrm{conj}}(X) \to \mathcal{M}(X_{\leq n}) \simeq \mathrm{B Aut^h}(X_{\leq n}).$$
The last weak equivalence comes from Theorem \ref{DK-theory5} applied to the (Reedy) model category 
$\mathcal{M}^{\mathbf{n}}$. 

\begin{theorem}[after Dwyer-Kan \cite{DK-classification}] \label{DK-Postnikov tower}
Let $\mathcal{M}$ be a simplicial model category and $X\co \mathbf{N} \to \mathcal{M}$ a diagram.
Then there is a weak equivalence $$\mc{M}_{\mathrm{conj}}(X) \stackrel{\sim}{\to} \mathrm{holim}_n \mathcal{M}(X_{\leq n}).$$
\end{theorem}
\begin{proof}
Let $C = \bigsqcup_n \mathbf{n}$ be the coproduct of the posets. There is a functor $s\co C \to C$ which is given by 
the inclusion functors $\mathbf{n} \to \mathbf{n+1}$. The colimit of these inclusions is $\mathbf{N}$ and
there is a pushout
\[
 \xymatrix{
 C \sqcup C \ar[r]^-{\mathrm{id} \sqcup \mathrm{id}}  \ar[d]_{\mathrm{id} \sqcup s} & C \ar[d] \\
 C \ar[r] & \cotower
 }
\]
Let $V = (v_0 \rightarrow v_2 \leftarrow v_1)$ denote the zigzag poset with three elements. Consider the pushout diagram
\[
 \xymatrix{
 C \sqcup C \ar[r]^f \ar[d]_g & C \times V \ar[d] \\
 C \times V \ar[r] & \mathbf{N}^\sharp
 }
\]
where $f$ (respectively $g$) sends the first copy of $\mathbf{n}$ to $\mathbf{n} \times \{v_0\}$ and the second 
copy to $\mathbf{n} \times \{v_1\}$ (respectively $\mathbf{n+1} \times \{v_1\}$). There is an obvious transformation 
from the latter pushout square to the first one. This induces a transformation from the following pullback diagram 
\begin{equation} \label{gluing}
  \xymatrix{
 \mc{M}_{\mathrm{conj}}(X) \ar[r] \ar[d] & \mc{M}_{\mathrm{conj}}(X_{C}) \ar[d] \\
 \mc{M}_{\mathrm{conj}}(X_{C}) \ar[r] & \mc{M}_{\mathrm{conj}}(X_{C \sqcup C})
 }
\end{equation}
where $X_{?}$ are given by pulling back the diagram $X$ to the respective categories and the maps in the square are likewise
given by pullback functors, to the pullback diagram 
\begin{equation} \label{h-gluing}
 \xymatrix{
 \mc{M}_{\mathrm{conj}}(X_{\mathbf{N}^{\sharp}}) \ar[r] \ar[d] & \mc{M}_{\mathrm{conj}}(X_{C \times V}) \ar[d] \\
 \mc{M}_{\mathrm{conj}}(X_{C \times V}) \ar[r] & \mc{M}_{\mathrm{conj}}(X_{C \sqcup C})
 }
\end{equation}
(To make sure that the squares are really pullbacks, we assume here that we work inside a convenient category of topological 
spaces. Alternatively, it is also possible, and makes little difference, to work with simplicial sets throughout.) 

We claim 
that the maps 
$$ \mc{M}_{\mathrm{conj}}(X) \to  \mc{M}_{\mathrm{conj}}(X_{\mathbf{N}^\sharp})$$ 
$$\mc{M}_{\mathrm{conj}}(X_{C}) \to  \mc{M}_{\mathrm{conj}}(X_{C \times V})$$ 
are weak equivalences. To see this note that the functors $C \times V \to C$ and $\mathbf{N}^\sharp \to \mathbf{N}$ admit sections 
$C \to C \times V$, $c \mapsto (c, v_1)$, 
and $\mathbf{N} \to \mathbf{N}^\sharp$, similarly. These give pairs of maps between the respective moduli spaces 
$$ \mc{M}_{\mathrm{conj}}(X_{C}) \leftrightarrows  \mc{M}_{\mathrm{conj}}(X_{C \times V})$$
$$ \mc{M}_{\mathrm{conj}}(X) \leftrightarrows  \mc{M}_{\mathrm{conj}}(X_{\mathbf{N}^\sharp})$$
which are inverse weak equivalences because their composites are either equal to the identity or they can be connected to 
the identity by a zigzag of natural transformations. In fact, this zigzag is actually defined by natural transformations 
already available at the level of the indexing posets: it is essentially the zigzag connecting the identity functor on $V$ 
to the constant functor at $v_1 \in V$. Note that this uses that $X_{C \times V}$ is pulled back from $X\co \mathbf{N}\to \mc{M}$ and so the values of $X_{C \times V}$ in the $V$-direction are weak equivalences. 

Moreover, the square \eqref{h-gluing} is actually a homotopy pullback. The proof of this is similar to 
\cite[Lemma 7.2]{DK-classification}. The homotopy fibers of the vertical maps can be identified with homotopy 
equivalent spaces (the arguments are comparable to Theorem \ref{calculus-with-moduli-spaces}). It follows that 
\eqref{gluing} is also a homotopy pullback. By \cite[8.3]{DK-classification}, we have a weak equivalence
$$ \mc{M}_{\mathrm{conj}}(X_C) \stackrel{\sim}{\to} \prod_n \mathcal{M}(X_{\leq n})$$
and so the homotopy pullback of \eqref{gluing} is weakly equivalent to $\mathrm{holim}_n \mathcal{M}(X_{\leq n})$. This means 
that the map $ \mc{M}_{\mathrm{conj}}(X) \stackrel{\sim}{\to} \mathrm{holim}_n \mathcal{M}(X_{\leq n})$, which is canonically induced 
by the restriction maps $r_n$, is a weak equivalence, as required. 
\end{proof}

\begin{remark}
In our applications the diagram $X$ will correspond to a form of dual Postnikov tower, meaning in particular that each value 
$X(n)$ is determined by $X(m)$ for any $m > n$. More specifically, in such cases, one has an obvious weak equivalence 
$\mathcal{M}(X_{\leq n}) \simeq \mathcal{M}(X(n))$, which simplifies the statement of Theorem \ref{DK-Postnikov tower}.
\end{remark}

\begin{remark}
The methods of \cite{DK-classification} apply more generally to arbitrary diagrams $X\co C \to \mathcal{M}$. The only point of caution is 
related to making sure that certain categories of diagrams in $\mathcal{M}$ are again simplicial model categories so that the results of 
this section, most notably Theorem \ref{DK-theory5}, can be applied. If \mc{M} is cofibrantly generated, then diagram categories over 
\mc{M} can be equipped with such model structures, but as long as there is a model structure on the relevant diagram categories, 
cofibrant generation of the model structure on \mc{M} is not needed. 
\end{remark}



\end{document}